\documentclass[11pt]{article}
\usepackage{amsmath}
\usepackage{amsthm}
\usepackage{amssymb}
\usepackage{graphicx}
\usepackage{adjustbox}
\usepackage{mathtools}
\usepackage{bbm}
\usepackage[margin=2.5cm]{geometry}
\usepackage{setspace}
\usepackage{epstopdf}
\usepackage{subfig}
\usepackage{empheq}
\usepackage{algorithm}
\usepackage{algorithmicx}
\usepackage{algpseudocode}
\usepackage{multirow}
\usepackage{url}
\usepackage[title]{appendix}
\usepackage{enumitem}
\usepackage[svgnames]{xcolor}

\makeatletter

\theoremstyle{plain}
\newtheorem{thm}{\protect\theoremname}
  \theoremstyle{definition}
  \newtheorem{defn}[thm]{\protect\definitionname}
  
  \theoremstyle{plain}
  \newtheorem{lem}[thm]{\protect\lemmaname}
  \newtheorem{prop}[thm]{\protect\propname}
  \newtheorem{cor}[thm]{\protect\corollaryname}
  \newtheorem{example}{Example}
  \newtheorem{assump}{Assumption}
  
  \providecommand{\propname}{Proposition}
  \providecommand{\condname}{Condition}

  \providecommand{\keywords}[1]
{
  \small	
  \textbf{\textit{Keywords---}} #1
}

\makeatother

  \providecommand{\definitionname}{Definition}
  \providecommand{\lemmaname}{Lemma}
  \providecommand{\corollaryname}{Corollary}
\providecommand{\theoremname}{Theorem}

\usepackage{endnotes}
\usepackage{color}
\usepackage{xparse}

\title{The Dyson Equalizer: Adaptive Noise Stabilization for Low-Rank Signal Detection and Recovery}

\author{ Boris Landa${^{1,2,*}}$~~~~Yuval Kluger${^{2,3,4}}$\\
\small{${^1}$Department of Electrical and Computer Engineering, Yale University}\\
\small{${^2}$Program in Applied Mathematics, Yale University}\\
\small{${^3}$Interdepartmental Program in Computational Biology and Bioinformatics, Yale University}\\
\small{${^4}$Department of Pathology, Yale University School of Medicine}\\
\small{${^*}$Corresponding author. Email: boris.landa@yale.edu}
}

\begin{document}

\maketitle

\begin{abstract}
Detecting and recovering a low-rank signal in a noisy data matrix is a fundamental task in data analysis. Typically, this task is addressed by inspecting and manipulating the spectrum of the observed data, e.g., thresholding the singular values of the data matrix at a certain critical level. This approach is well-established in the case of homoskedastic noise, where the noise variance is identical across the entries. However, in numerous applications, the noise can be heteroskedastic, where the noise characteristics may vary considerably across the rows and columns of the data. In this scenario, the spectral behavior of the noise can differ significantly from the homoskedastic case, posing various challenges for signal detection and recovery.
To address these challenges, we develop an adaptive normalization procedure that equalizes the average noise variance across the rows and columns of a given data matrix. Our proposed procedure is data-driven and fully automatic, supporting a broad range of noise distributions, variance patterns, and signal structures. Our approach relies on recent results in random matrix theory, which describe the resolvent of the noise via the so-called Dyson equation. By leveraging this relation, we can accurately infer the noise level in each row and each column directly from the resolvent of the data. We establish that in many cases, our normalization enforces the standard spectral behavior of homoskedastic noise -- the Marchenko-Pastur (MP) law, allowing for simple and reliable detection of signal components. Furthermore, we demonstrate that our approach can substantially improve signal recovery in heteroskedastic settings by manipulating the spectrum after normalization. Lastly, we apply our method to single-cell RNA sequencing and spatial transcriptomics data, showcasing accurate fits to the MP law after normalization.
\end{abstract}


\section{Introduction}
Low-rank approximation, typically realized by PCA or SVD, is a ubiquitous tool for compressing and denoising large data matrices before downstream analysis. A common approach to studying low-rank approximation of noisy data is to assume a signal-plus-noise model, where a low-rank signal matrix is observed under noise and the goal is to detect and recover the signal. In this work, we consider a data matrix $Y\in\mathbb{R}^{m\times n}$ modeled as
\begin{equation}
    Y = X + E, \label{eq: signal-plus-noise}
\end{equation}
where $X$ is a signal matrix of rank $r\ll \min\{m,n\}$ and $E$ is a random noise matrix whose entries $E_{ij}$ are independent with zero mean and variance $S_{ij} = \mathbb{E}[ E_{ij}^2] > 0$. For simplicity of presentation, we assume that $m\leq n$ (otherwise, we can always replace $Y$ with $Y^T$). 
Given the data matrix $Y$, common tasks of interest include identifying the presence of the low-rank signal $X$, estimating its rank $r$, and recovering $X$ from the noisy observations. We refer to these tasks broadly as signal detection and recovery.

Many existing methods for signal detection and recovery rely on inspecting and manipulating the spectrum of the observed data; see, e.g.,~\cite{cattell1966scree,wold1978cross,hoff2007model,owen2009bi,gavish2014optimal,chatterjee2015matrix,fan2014principal,johnstone2018pca,fan2020estimating,hong2020selecting,choi2017selecting,kritchman2008determining,johnstone2017roy} and references therein. In particular, in order to detect the signal and estimate its rank, the singular values of $Y$, or functions thereof, are often compared against analytical or empirical thresholds. Then, to recover the signal matrix $X$, the singular values of the data are typically thresholded or shrunk towards zero, while retaining the original singular vectors.

The above approach for signal detection and recovery is well-established in the case of homoskedastic noise, where the noise variance $S_{ij} = \sigma^2$ is identical across all entries.
In this case, under mild regularity conditions, the noise matrix $E$ satisfies the celebrated Marchenko-Pastur (MP) law~\cite{marvcenko1967distribution}, which describes the eigenvalue density of $E E^T/n$ in the asymptotic regime $m,n\rightarrow\infty$ with $m/n\rightarrow \gamma \in (0,1]$.
Further, in this regime, the largest eigenvalue of $E E^T/n$ converges almost surely to $\beta_+ = \sigma^2(1 + \sqrt{\gamma})^2$ ~\cite{geman1980limit,yin1988limit}, which is the upper edge of the MP density. Consequently, a simple approach for identifying the presence of a signal $X$ and estimating its rank is to count how many eigenvalues of $Y Y^T/n$ exceed $\beta_+$. This approach can be justified further by the BBP phase transition~\cite{baik2005phase,baik2006eigenvalues,benaych2011eigenvalues,nadler2008finite,paul2007asymptotics}. The results therein show that in a suitable signal-plus-noise model and the same asymptotic regime as for the MP law, the eigenvalues of $YY^T/n$ that exceed $\beta_+$ admit a one-to-one analytic correspondence to nonzero eigenvalues of $X X^T/n$, and the respective clean and noisy eigenvectors admit nonzero correlations. These results have been extended to support homoskedastic noise with general distributions under mild moment conditions; see, e.g.,~\cite{ding2021singular,ding2020high}.
Utilizing such results, refined techniques of singular value thresholding and shrinkage were developed for recovering $X$ optimally according to a prescribed loss function; see~\cite{gavish2017optimal,gavish2014optimal,leeb2022optimal,ding2020high} and references therein. 

In many practical situations, the noise can be heteroskedastic, where the noise variance $S_{ij}$ differs between the entries. A notable example is count or nonnegative data, where the entries are typically modeled by, e.g., Poisson, binomial, negative binomial, or gamma distributions. In these cases,  the noise variance inherently depends on the signal, leading to heteroskedasticity. Such data is commonly found in network traffic analysis~\cite{shen2005analysis}, photon imaging~\cite{salmon2014poisson}, document topic modeling~\cite{wallach2006topic}, single-cell RNA sequencing~\cite{hafemeister2019normalization}, spatial transcriptomics~\cite{burgess2019spatial}, and high-throughput chromosome conformation capture~\cite{johanson2018genome}, among many other applications. Heteroskedastic noise also arises when data is nonlinearly transformed, e.g., in natural image processing due to spatial pixel clipping~\cite{foi2009clipped}, or in experimental procedures where conditions vary during data acquisition, such as in spectrophotometry and atmospheric data analysis~\cite{cochran1977statistically,tamuz2005correcting}. Another common reason for heteroskedasticity is when datasets are merged from different sources, e.g., sensors or measurement devices with different levels of technical noise. Lastly, heteroskedastic noise can be caused by abrupt deformations or technical errors during data collection and storage, leading to severe corruption in certain entries of the matrix or even entire rows and columns. Due to the many forms of heteroskedastic noise prevalent in applications, it is important to develop robust methods for signal detection and recovery under general heteroskedastic noise. 

When the noise is heteroskedastic, the spectral behavior of the noise can differ significantly from the homoskedastic case, posing various challenges for signal detection and recovery. First, if the noise variance $S_{ij}$ is abnormally high in a few rows or columns, then some of the noise eigenvalues may depart from the bulk, creating a false impression of signal components. Second, if $S_{ij}$ varies considerably across the rows and columns, then the eigenvalue density of the noise $E E^T/n$ can become much more spread out than the MP law, potentially masking weak signal components of interest. These two fundamental issues are illustrated in Figures~\ref{fig: outlier rows and columns, sorted eigenvalues, before},~\ref{fig: outlier rows and columns, eigenvalue density, before},~\ref{fig: general heteroskedastic noise, sorted eigenvalues, before}, and~\ref{fig: general heteroskedastic noise, eigenvalue density, before}, where we exemplify the sorted eigenvalues of $Y Y^T/n$ and its eigenvalue density for two simulated scenarios with $m=1000$ and $n=2000$. The signal $X$ is identical in both scenarios; it is of rank $r=20$ and contains $10$ strong components and $10$ weak components. In the first scenario, depicted in Figures~\ref{fig: outlier rows and columns, sorted eigenvalues, before} and~\ref{fig: outlier rows and columns, eigenvalue density, before}, the noise matrix $E$ was generated as Gaussian homoskedastic with variance $1$, except that we amplified its last $5$ rows and $5$ columns by factors of $\sqrt{10}$ and $10$, respectively. In the second scenario, depicted in Figures~\ref{fig: general heteroskedastic noise, sorted eigenvalues, before} and~\ref{fig: general heteroskedastic noise, eigenvalue density, before}, the noise variance matrix $S$ was generated randomly in such a way that its entries fluctuate considerably across the rows and columns, while its average across all entries is $1$. More details on these experiments can be found in Appendix~\ref{appendix: reproducibility details}.

\begin{figure} 
  \centering
  	{
  	\subfloat[][Sorted eigenvalues of $Y Y^T/n$]
  	{
    \includegraphics[width=0.35\textwidth]{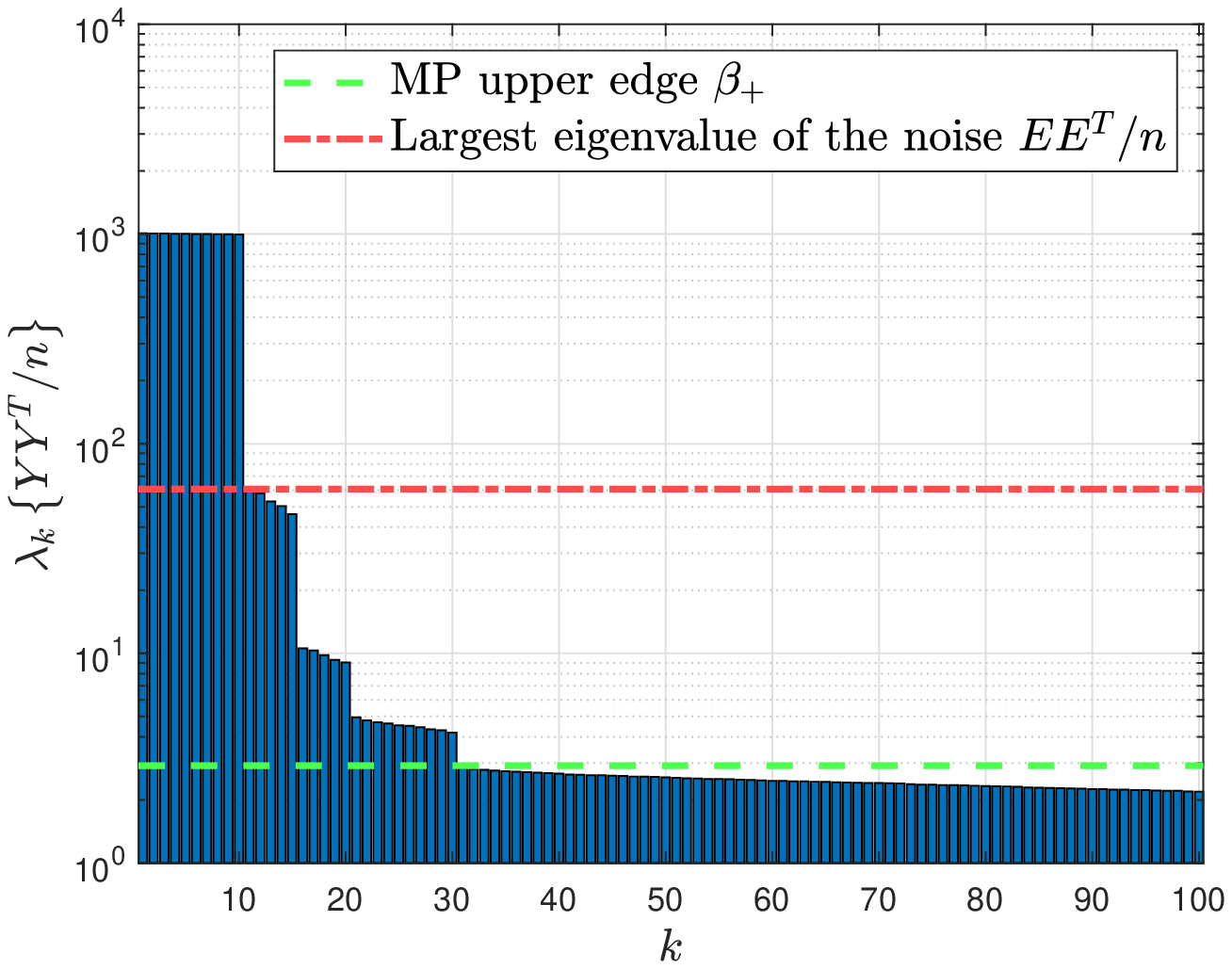} 
    \label{fig: outlier rows and columns, sorted eigenvalues, before}
    }
    \subfloat[][Sorted eigenvalues of $\hat{Y} \hat{Y}^T/n$] 
  	{
    \includegraphics[width=0.35\textwidth]{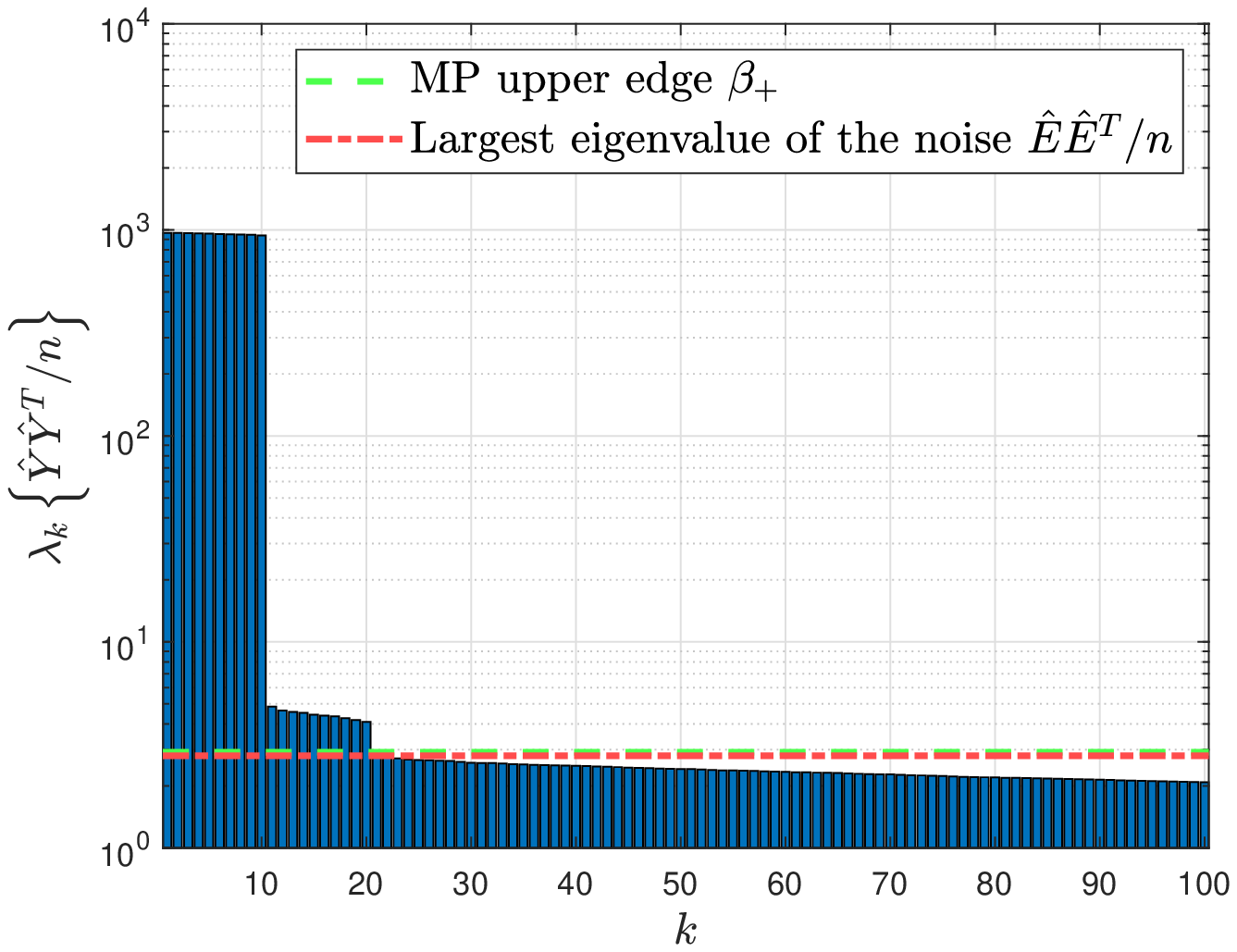} \label{fig: outlier rows and columns, sorted eigenvalues, after}
    }
    \\
    \subfloat[][Eigenvalue density of $Y Y^T/n$]  
  	{
    \includegraphics[width=0.35\textwidth]{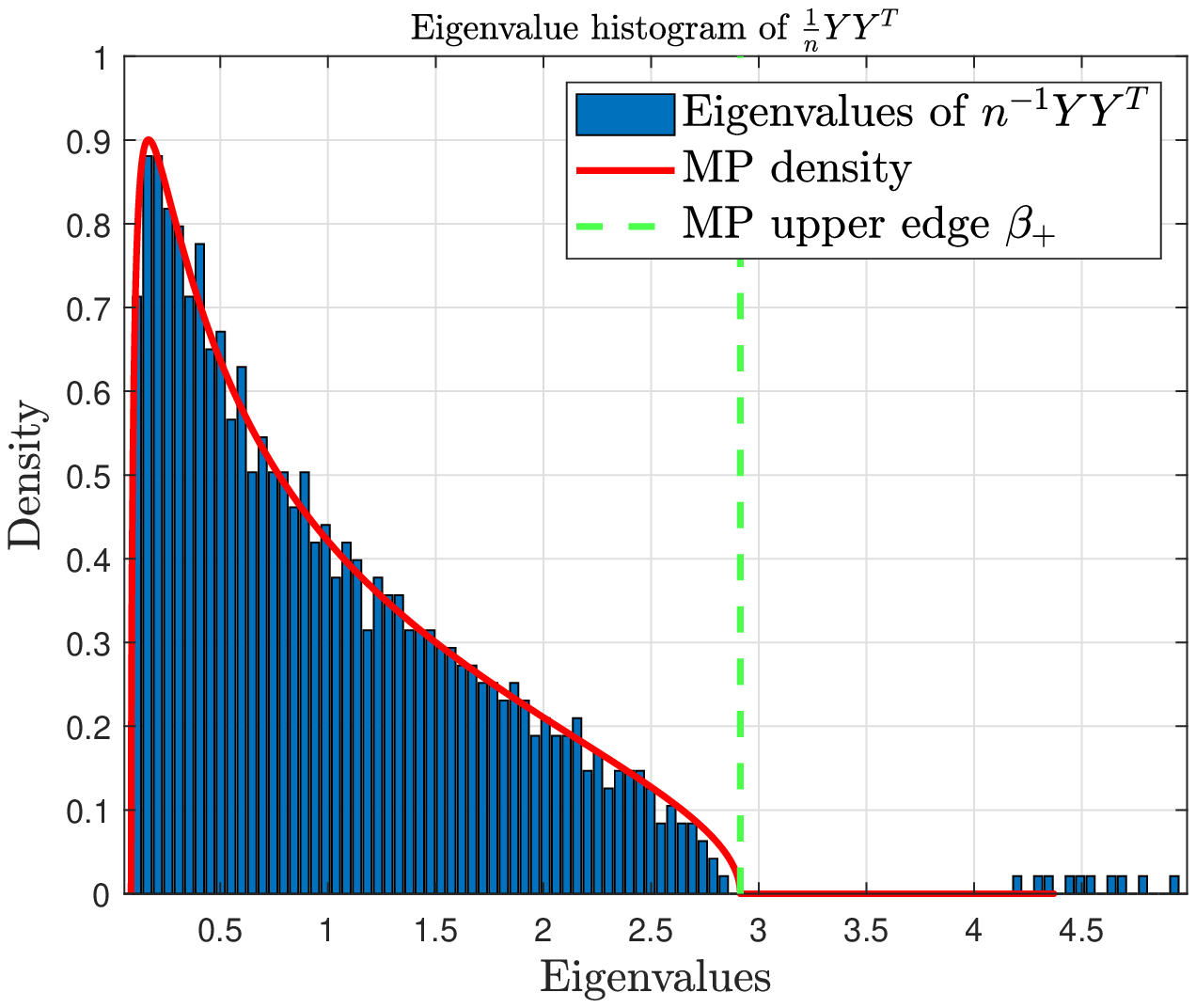} \label{fig: outlier rows and columns, eigenvalue density, before}
    }
    \subfloat[][Eigenvalue density of $\hat{Y} \hat{Y}^T/n$] 
  	{
    \includegraphics[width=0.35\textwidth]{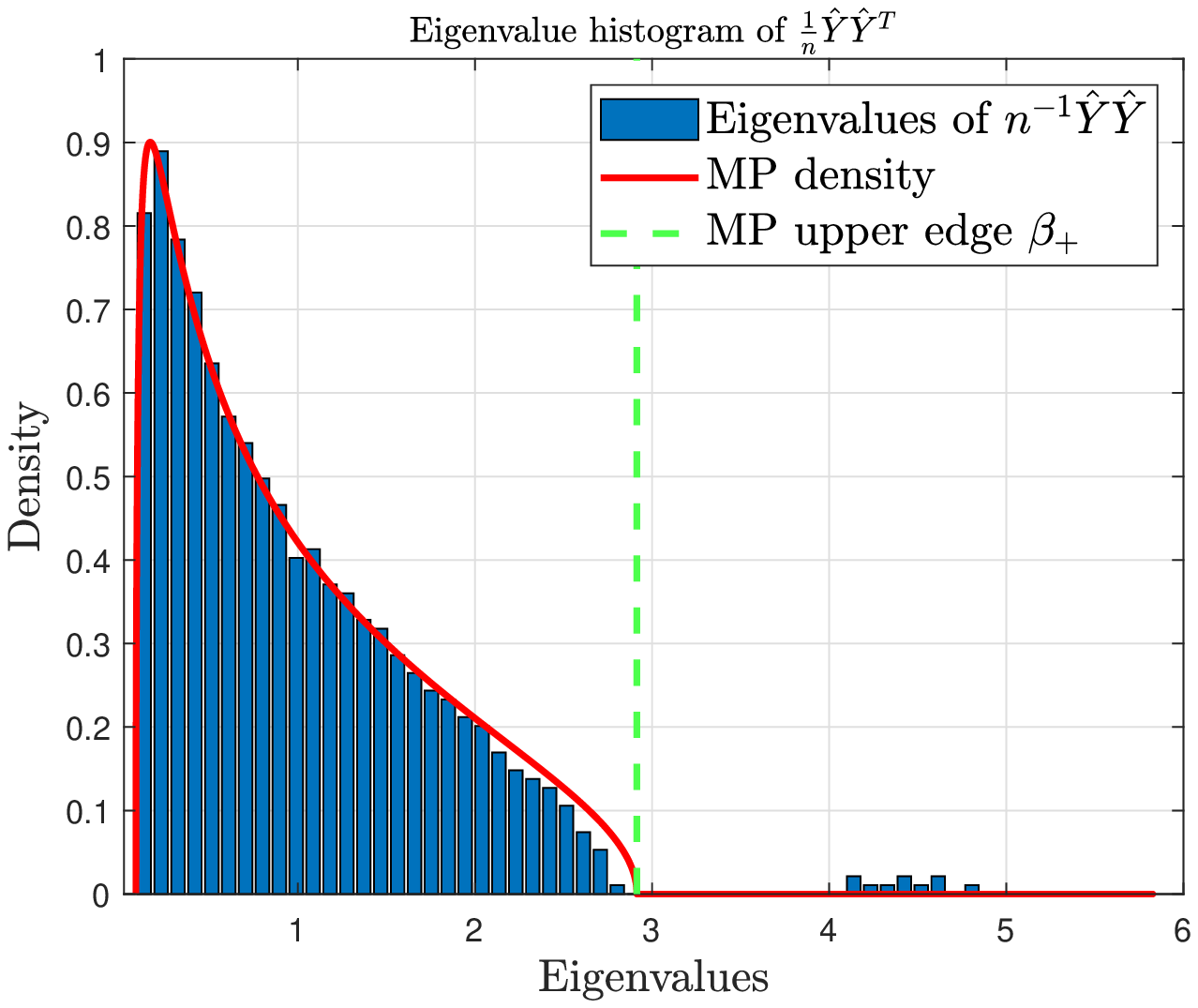} \label{fig: outlier rows and columns, eigenvalue density, after}
    } 

    }
    \caption
    {The spectrum of the original data (panels (a) and (c)) and after applying Algorithm~\ref{alg:noise standardization} (panels (b) and (d)), where $m=1000$, $n=2000$, and $r=20$. The signal $X$ consists of $10$ strong components and $10$ weak components. The noise matrix $E$ is homoskedastic almost everywhere -- except for the last $5$ rows and $5$ columns where the noise is abnormally strong. 
    } 
    \label{fig: spectrum before and after normalization, outlier rows anc columns}
    \end{figure}

\begin{figure} 
  \centering
  	{
  	\subfloat[][Sorted eigenvalues of $Y Y^T/n$]
  	{
    \includegraphics[width=0.35\textwidth]{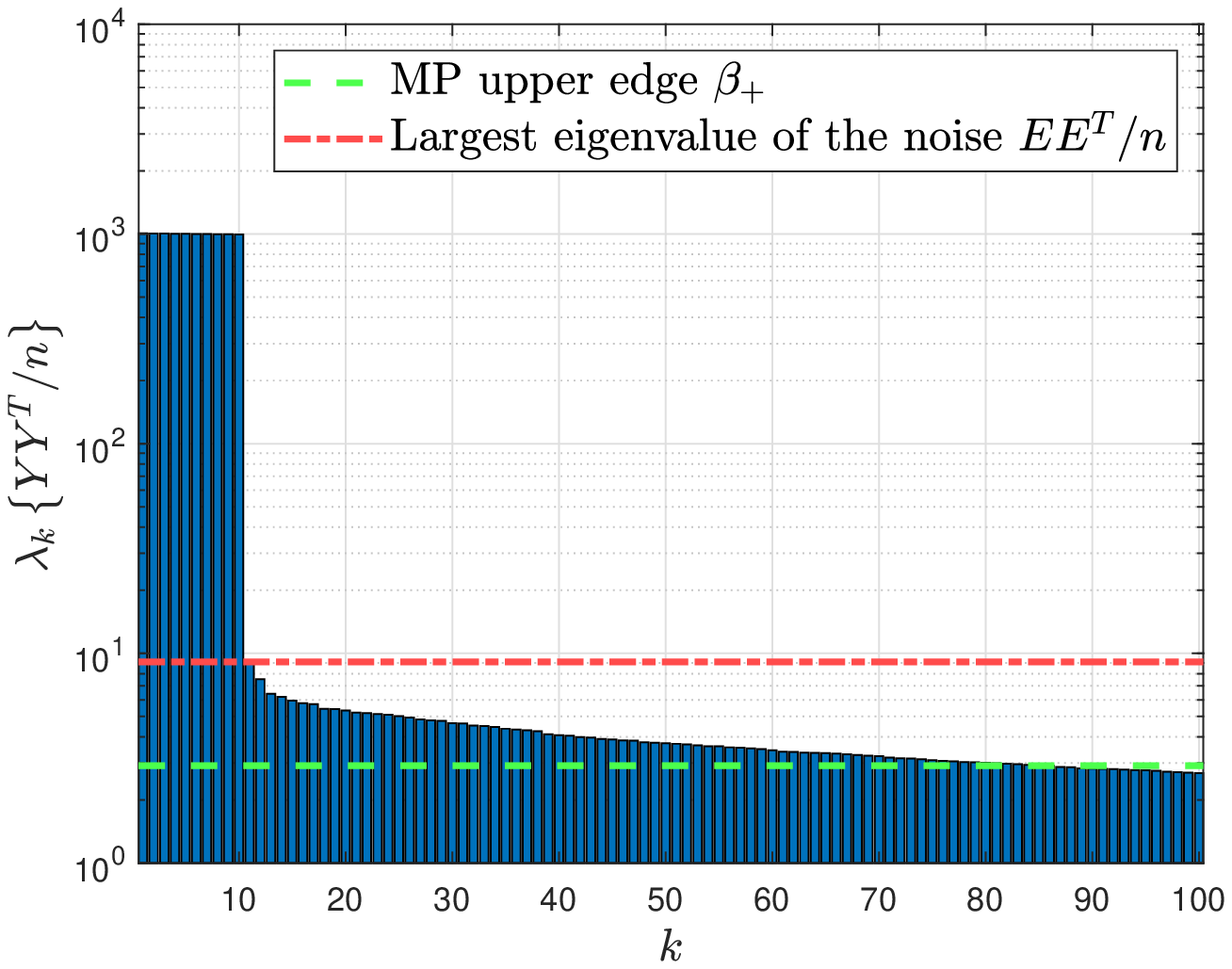} \label{fig: general heteroskedastic noise, sorted eigenvalues, before}
    }
    \subfloat[][Sorted eigenvalues of $\hat{Y} \hat{Y}^T/n$] 
  	{
    \includegraphics[width=0.35\textwidth]{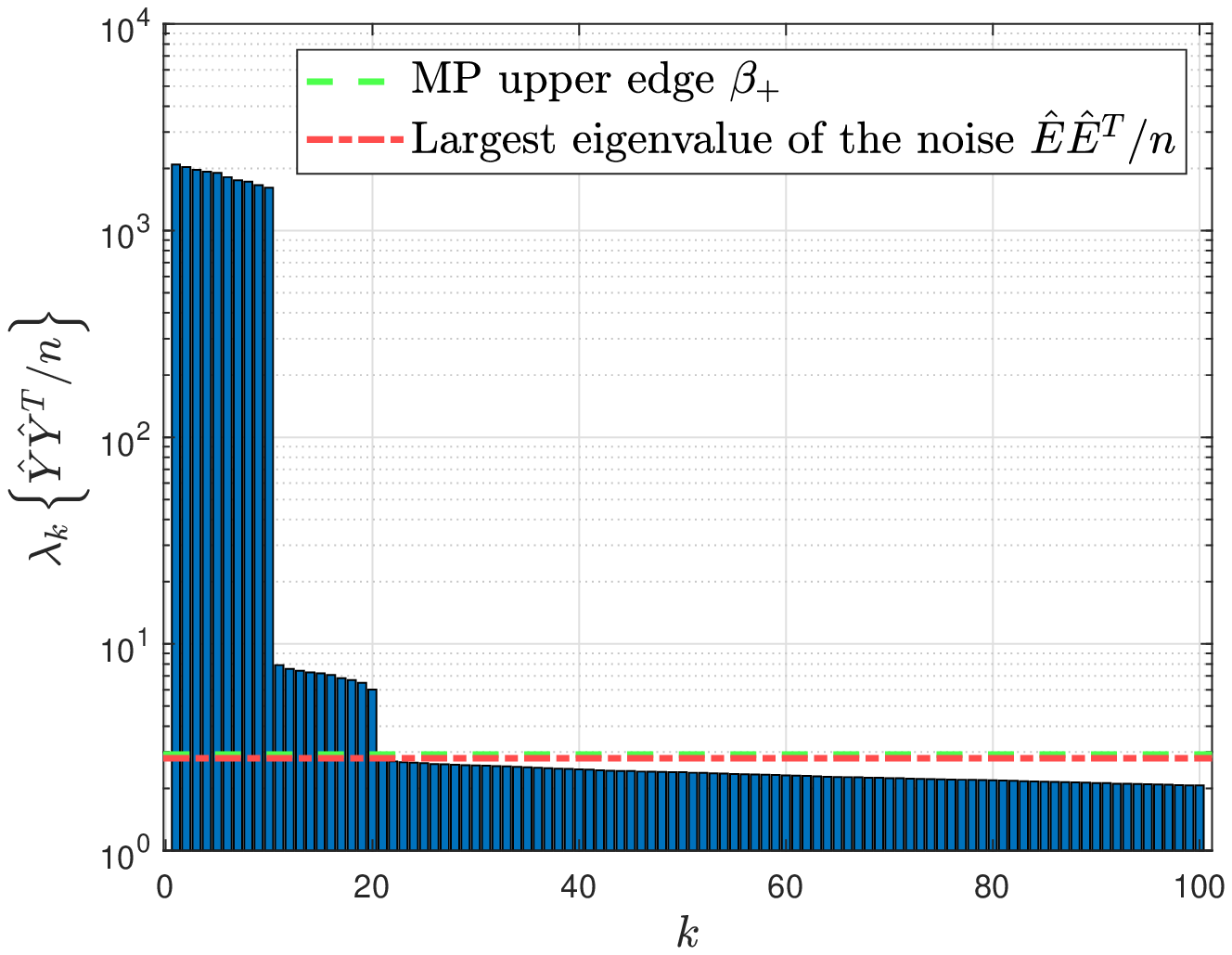} \label{fig: general heteroskedastic noise, sorted eigenvalues, after}
    }
    \\
    \subfloat[][Eigenvalue density of $Y Y^T/n$]  
  	{
    \includegraphics[width=0.35\textwidth]{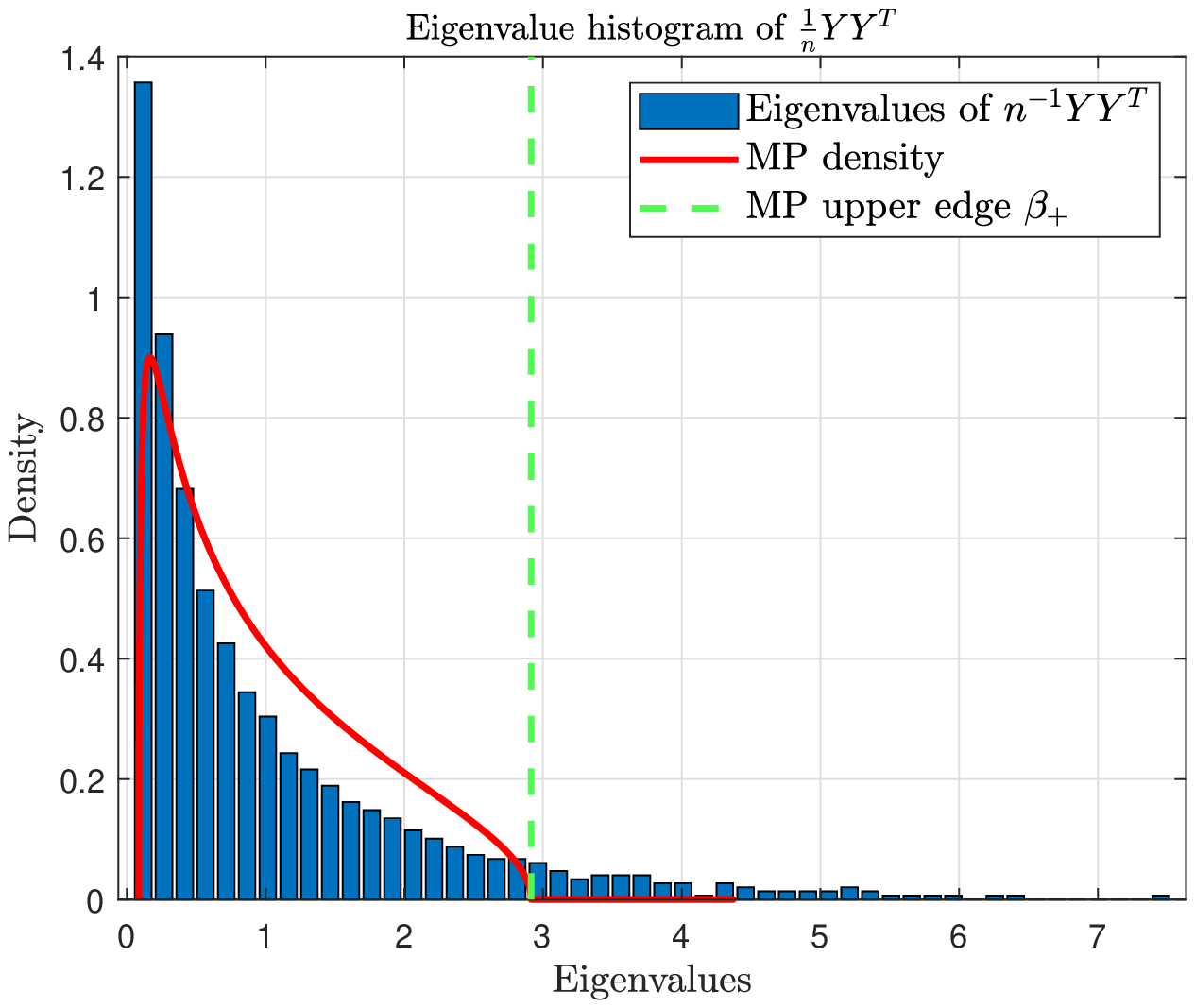}  \label{fig: general heteroskedastic noise, eigenvalue density, before}
    }
    \subfloat[][Eigenvalue density of $\hat{Y} \hat{Y}^T/n$] 
  	{
    \includegraphics[width=0.35\textwidth]{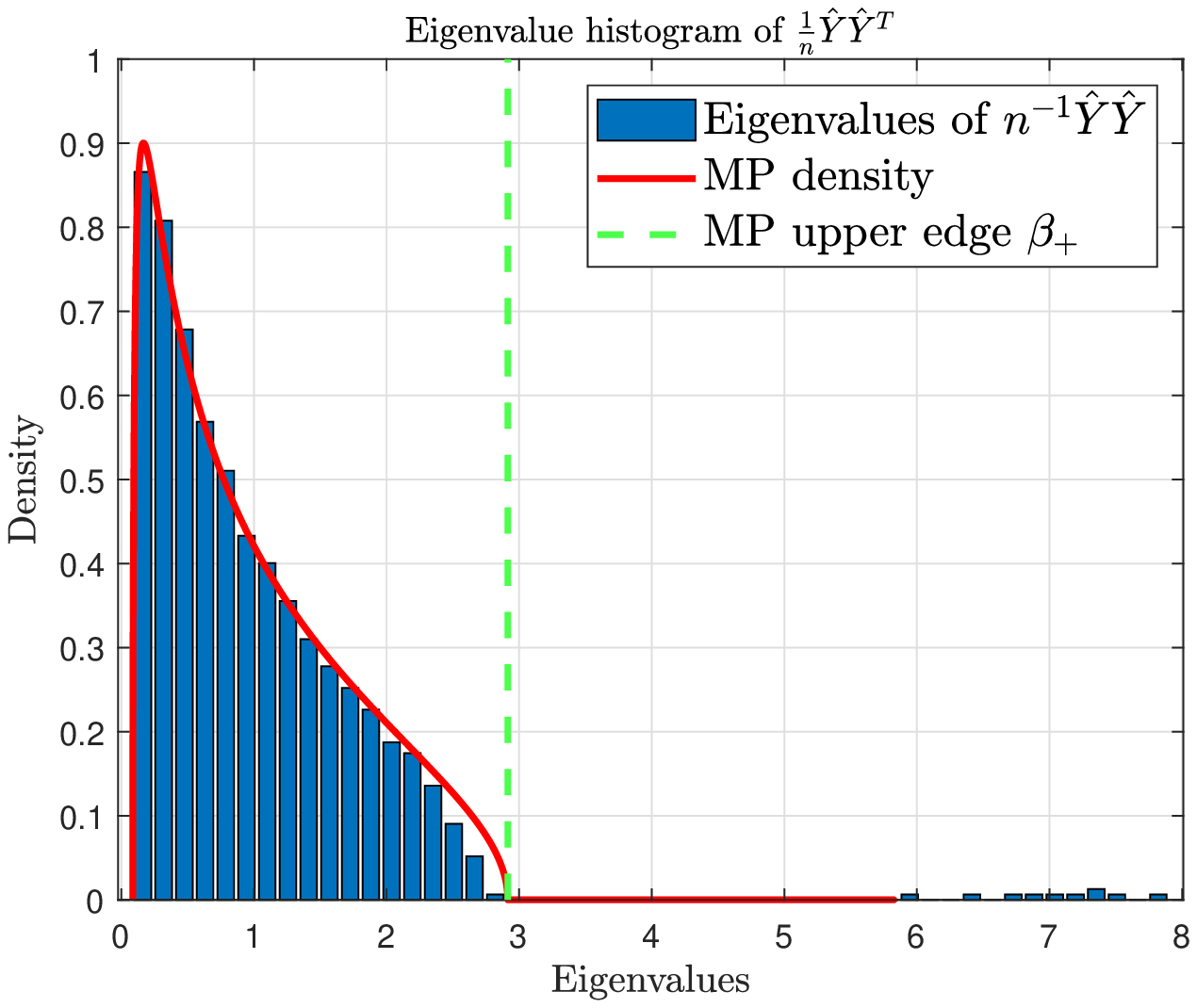} \label{fig: general heteroskedastic noise, eigenvalue density, after}
    } 

    }
    \caption
    {The spectrum of the original data (panels (a) and (c)) and after applying Algorithm~\ref{alg:noise standardization} (panels (b) and (d)), where $m=1000$, $n=2000$, and $r=20$. The signal $X$ is identical to the one in the setting of Figure~\ref{fig: spectrum before and after normalization, outlier rows anc columns}. The noise matrix $E$ is generated as Gaussian heteroskedastic, where $S_{ij}$ varies considerably across the rows and columns.} \label{fig: spectrum before and after normalization, general heteroskedastic noise}
    \end{figure}

It is evident that in the first simulated scenario (Figures~\ref{fig: outlier rows and columns, sorted eigenvalues, before} and~\ref{fig: outlier rows and columns, eigenvalue density, before}), the bulk of the eigenvalues of $Y Y^T/n$ is well-approximated by the MP law since the noise is mostly homoskedastic. However, the severely corrupted rows and columns contribute $10$ significant components to the spectrum of the data. In this case, the first $10$ eigenvalues of $Y Y^T/n$ correspond to the strong signal components in the data, the next $10$ eigenvalues correspond to the $5$ noisy columns and $5$ noisy rows (in this order), and the last $10$ eigenvalues correspond to the weak signal components. Notably, there are $30$ eigenvalues of $Y Y^T/n$ that exceed the MP upper edge, while the signal's rank is only $20$. 
Ideally, we would like to detect and recover all $20$ signal components (weak and strong) while filtering out the $10$ components arising from the corrupted rows and columns. However, this cannot be achieved by thresholding the singular values since the corrupted rows and columns are stronger in magnitude than the weak signal components. For the second simulated scenario (Figures~\ref{fig: general heteroskedastic noise, sorted eigenvalues, before} and~\ref{fig: general heteroskedastic noise, eigenvalue density, before}), we see that the bulk of the eigenvalues no longer fits the MP law and is much more spread out. In particular, the noise covers the weak signal components completely, prohibiting their detection and recovery via traditional approaches. 
Hence, relying on the spectrum of the observed data for signal detection and recovery can be restrictive under heteroskedastic noise. 
This challenge remains valid even if the eigenvalue density of $E E^T/n$ and its largest eigenvalue are known or can be inferred from the data; see, e.g.,~\cite{nadakuditi2014optshrink,gavish2022matrix,donoho2023screenot,ke2021estimation,hong2020selecting,su2022adaptive}.

Instead of using the spectrum of the observed data directly for signal detection and recovery, an alternative approach is to first normalize the data appropriately to stabilize the behavior of the noise.
In the literature on robust covariance estimation under heavy-tailed distributions, it is well known that reweighting the data -- by putting less weight on extremal observations -- can lead to substantial improvement in recovery accuracy~\cite{tyler1987distribution,zhang2014novel,goes2020robust}. Data normalizations were also proposed and analyzed in~\cite{leeb2021optimal,hong2018asymptotic,gavish2022matrix,hong2018optimally} under one-side heteroskedasticity, i.e., across the rows (observations) or columns (features). The results therein show that applying a suitable linear transformation to the heteroskedastic dimension (either the rows or the columns) can be highly beneficial for covariance estimation and matrix denoising. 
To account for more general variance patterns,~\cite{leeb2021matrix} and~\cite{landa2022biwhitening} recently considered normalizing the data by scaling the rows and columns simultaneously. Specifically,
\begin{equation}
    \widetilde{Y} = (D\{\mathbf{{x}}\})^{-1/2} Y (D\{\mathbf{{y}}\})^{-1/2} = \widetilde{X} + \widetilde{E}, \label{eq: scaling rows and columns}
\end{equation}
where $\mathbf{x}\in\mathbb{R}^m$ and $\mathbf{y}\in\mathbb{R}^n$ are positive vectors, $D\{\nu\}$ is a diagonal matrix with $\nu$ on its main diagonal, and $\widetilde{X} = (D\{\mathbf{{x}}\})^{-1/2} X (D\{\mathbf{{y}}\})^{-1/2}$ is the scaled signal while $\widetilde{E} = (D\{\mathbf{{x}}\})^{-1/2} E (D\{\mathbf{{y}}\})^{-1/2}$ is the scaled noise. 

One special case of interest is when the noise variance matrix $S$ is of rank one, which allows for heteroskedasticity across both rows and columns. Such variance matrices lead to separable covariance matrices, whose spectral properties have been extensively studied from the perspective of random matrix theory~\cite{ding2023local,ding2021spiked,yang2019edge}. Importantly, in the case of $S = \mathbf{x}\mathbf{y}^T$, the normalization~\eqref{eq: scaling rows and columns} makes the noise $\widetilde{E}$ completely homoskedastic with variance one. In this case,~\cite{leeb2021matrix} showed that under a suitable signal-plus-noise model with Gaussian noise and a signal that is sufficiently generic and delocalized, the normalization~\eqref{eq: scaling rows and columns} enhances the spectral signal-to-noise ratio, namely, it increases the ratios between the signal's singular values and the operator norm of the noise. Then, to improve signal recovery,~\cite{leeb2021matrix} developed a spectral denoiser for $\widetilde{Y}$ that is optimally tuned for recovering $X$, followed by unscaling the rows and columns of the denoised matrix. Since the noise variance matrix $S = \mathbf{x} \mathbf{y}^T$ is typically unknown in applications,~\cite{leeb2021matrix} proposed to estimate $\mathbf{x}$ and $\mathbf{y}$ directly from the magnitudes of the rows and columns of the data, i.e., from the sums of $Y_{ij}^2$ across the rows and columns. This approach requires the signal $X$ to be sufficiently weak compared to the noise $E$ and spread out across the entries. 

In~\cite{landa2022biwhitening}, the authors proposed to use the normalization~\eqref{eq: scaling rows and columns} for rank estimation under heteroskedastic noise, utilizing the fact that the scaled signal $\widetilde{X}$ preserves the rank of the original signal $X$. It was shown that for general variance matrices $S$ (not necessarily of the form $S = \mathbf{x} \mathbf{y}^T$), the normalization~\eqref{eq: scaling rows and columns} can still enforce the standard spectral behavior of homoskedastic noise, namely the MP law for the spectrum of $\widetilde{E} \widetilde{E}^T/n$ and the convergence of its largest eigenvalue to the MP upper edge. The main idea is that by scaling the rows and columns of $S$ judiciously, we can control its average entry in each row and each column~\cite{sinkhorn1967concerning,sinkhorn1967diagonal}. Specifically, the normalization~\eqref{eq: scaling rows and columns} can make the average entry in each row and each column of $\mathbb{E}[\widetilde{E}_{ij}^2]$ precisely one. This property is sufficient to enforce the spectral behavior of homoskedastic noise under non-restrictive conditions (see Proposition 4.1 in~\cite{landa2022biwhitening}). Then, the rank of $X$ is estimated simply by comparing the eigenvalues of $\widetilde{Y} \widetilde{Y}^T/n$ to the MP upper edge $(1 + \sqrt{m/n})^2$. It was demonstrated in~\cite{landa2022biwhitening} that this method can accurately estimate the rank in challenging regimes, including in severe heteroskedasticity with general variance patterns and when strong signal components are present in the data. If $S$ is known, then the required scaling factors $\mathbf{x}$ and $\mathbf{y}$ can be obtained from $S$ by the Sinkhorn-Knopp algorithm~\cite{sinkhorn1967concerning,sinkhorn1967diagonal}. Otherwise, ~\cite{landa2022biwhitening} derived a procedure to estimate the required $\mathbf{x}$ and $\mathbf{y}$ from the data under the assumption that the variance of $Y_{ij}$ is a quadratic polynomial in its mean. This assumption holds for data sampled from prototypical count and nonnegative random variables such as Poisson, binominal, negative binomial, and Gamma~\cite{morris1982natural}.

\subsection{Our results and contributions}
In this work, we develop a data-driven version of the normalization~\eqref{eq: scaling rows and columns} and utilize it for improved detection and recovery of low-rank signals under general heteroskedastic noise. Similarly to~\cite{leeb2021matrix} and~\cite{landa2022biwhitening}, the scaling factors $\mathbf{x}$ and $\mathbf{y}$ in our setting are designed to equalize the average noise variance across the rows and columns, i.e., make the average entry of \sloppy$\widetilde{S} = (\mathbb{E}[\widetilde{E}_{ij}^2]) = D^{-1}\{\mathbf{x}\} S D^{-1}\{\mathbf{y}\}$ in each row and each column to be precisely one. 
However, distinctly from~\cite{leeb2021matrix} and~\cite{landa2022biwhitening}, our approach allows us to estimate the required $\mathbf{x}$ and $\mathbf{y}$ directly from the observed data $Y$ in a broad range of settings without prior knowledge on the signal or the noise. Specifically, our main contribution is to derive estimators for $\mathbf{x}$ and $\mathbf{y}$ that support general distributions of the noise entries $E_{ij}$, diverse patterns of the variance $S_{ij}$ across the rows and columns, and a signal $X$ whose components can be strong and possibly localized in subsets of the entries; see Section~\ref{sec: method derivation and analysis}. Then, relying on the normalization~\eqref{eq: scaling rows and columns} with our estimated scaling factors, we propose suitable techniques for signal detection and recovery that adapt to general heteroskedastic noise. We provide theoretical justification for our proposed techniques and demonstrate their advantages in simulations; see Section~\ref{sec: application to signal detection and recovery}. 
In Section~\ref{sec: examples on real data}, we exemplify the favorable performance of our normalization procedure on real data from single-cell RNA sequencing and spatial transcriptomics. Lastly, in Section~\ref{sec: discussion}, we discuss our results and some future research directions. 

Our proposed data normalization procedure is described in Algorithm~\ref{alg:noise standardization} below, where $\hat{\mathbf{x}}$ and $\hat{\mathbf{y}}$ denote our estimators for $\mathbf{x}$ and $\mathbf{y}$, respectively. The advantage of this procedure is demonstrated in Figures~\ref{fig: spectrum before and after normalization, outlier rows anc columns} and~\ref{fig: spectrum before and after normalization, general heteroskedastic noise}, where we compare the spectrum of the observed data in the two simulated scenarios discussed previously (in the context of the challenges posed by heteroskedastic noise) to the spectrum after Algorithm~\ref{alg:noise standardization}, i.e.,  $\hat{Y}=(D\{\mathbf{\hat{x}}\})^{-1/2} Y (D\{\mathbf{\hat{y}}\})^{-1/2}$, where we denote the normalized noise as $\hat{E} = (D\{\mathbf{\hat{x}}\})^{-1/2} E (D\{\mathbf{\hat{y}}\})^{-1/2}$. It is evident that in the first scenario, our proposed normalization removes the spurious eigenvalues arising from the severely corrupted rows and columns. As a consequence, only the true signal components exceed the MP upper edge. In the second scenario, our proposed normalization lowers the noise level in the spectrum and reveals the weak signal components under the noise. In both scenarios, our proposed normalization stabilizes the spectral behavior of the noise. In particular, the bulk of the eigenvalues after normalization is described accurately by the MP density, and moreover, the largest eigenvalue of $\hat{E} \hat{E}/n$ is very close to the MP upper edge, which is much smaller than the largest eigenvalue of $E E^T/n$ before normalization. Notably, our approach accurately infers the noise levels in the rows and columns of the data despite the fact that the signal is much stronger than the noise (due to the presence of the strong signal components). 

\begin{algorithm}
\caption{The Dyson Equalizer}\label{alg:noise standardization}
\begin{algorithmic}[1]
\Statex{\textbf{Input:} Data matrix $Y\in \mathbb{R}^{m\times n}$ with $m\leq n$ and no rows or columns that are entirely zero.}
\State \label{alg: step 1}Compute the SVD of $Y$ and let the columns of $U\in\mathbb{R}^{m\times m}$, columns of $V\in\mathbb{R}^{n\times n}$, and $\{\sigma_k\}_{k=1}^m$ denote the left singular vectors, right singular vectors, and singular values of $Y$, respectively.
\State \label{alg: step 2}Set $\eta$ as the median singular value of $Y$, i.e., $\eta = \operatorname{Median}\{\sigma_1,\ldots,\sigma_m\}$.
\State \label{alg:step 3}Compute the vectors $\hat{\mathbf{g}}^{(1)}\in\mathbb{R}^m$ and $\hat{\mathbf{g}}^{(2)}\in\mathbb{R}^n$ given by
\begin{equation}
    \hat{\mathbf{g}}_i^{(1)} = \sum_{k=1}^m \frac{\eta}{\sigma_k^2 + \eta^2} U_{ik}^2, \qquad\qquad \hat{\mathbf{g}}_j^{(2)} = \frac{1}{\eta} + \sum_{k=1}^m \left( \frac{\eta}{\sigma_k^2 + \eta^2} - \frac{1}{\eta} \right) V_{jk}^2,
\end{equation}
for all $i\in [m]$ and $j\in [n]$.
\State \label{alg:step 4}Compute the vectors $\hat{\mathbf{x}}\in\mathbb{R}^m$ and $\hat{\mathbf{y}}\in\mathbb{R}^n$ according to
\begin{equation}
    \hat{\mathbf{x}}_i = \frac{1}{\sqrt{m - \eta \Vert \hat{\mathbf{g}}^{(1)} \Vert_1 }} \left( \frac{1}{\hat{\mathbf{g}}_i^{(1)}} - \eta \right), \qquad\qquad
    \hat{\mathbf{y}}_j = \frac{1}{\sqrt{n - \eta \Vert \hat{\mathbf{g}}^{(2)} \Vert_1 }} \left( \frac{1}{\hat{\mathbf{g}}_j^{(2)}} - \eta \right),
    \end{equation}
for all $i\in [m]$ and $j\in [n]$.
\State \label{alg:step 5}Form the normalized data matrix $\hat{Y} = (D\{\mathbf{\hat{x}}\})^{-1/2} Y (D\{\mathbf{\hat{y}}\})^{-1/2}$.
\end{algorithmic}
\end{algorithm}

Our approach for estimating the scaling factors $\mathbf{x}$ and $\mathbf{y}$ relies on recent results in random matrix theory, which establish that the main diagonal of the resolvent of the noise (see eq.~\eqref{eq: resolvent def}) concentrates around the solution to a nonlinear equation known as the Dyson equation (see eq.~\eqref{eq: Dyson eq}); see~\cite{alt2017local,erdHos2019random} and references therein. Importantly, this phenomenon is universal, i.e., it holds for general noise distributions under certain moment assumptions, where the Dyson equation and its solution depend only on the variance matrix $S$; see Section~\ref{sec: Dyson equations}. Building on these results, we show that in a broad range of settings, there is a tractable relation between the resolvent of the observed data (see eq.~\eqref{eq: resolent of data}) and the variance pattern of the noise. To the best of our knowledge, our approach is the first to exploit this relation to infer the noise variance structure in the signal-plus-noise model~\eqref{eq: signal-plus-noise}. Below is a detailed account of our results and contributions.

We begin by considering the case of a rank-one variance matrix $S = \mathbf{x} \mathbf{y}^T$ in Section~\ref{sec: rank one case}. We observe that in this case, there is an explicit formula for $\mathbf{x}$ and $\mathbf{y}$ in terms of the solution to the aforementioned Dyson equation; see Proposition~\ref{prop: formula for x and y in terms of g}. Since we do not have access to this solution (as it depends on $S$, which is unknown), we propose to estimate it directly from the main diagonal of the resolvent of the observed data, which acts as a surrogate to the resolvent of the noise. One of our main technical contributions is to show that the resolvent of the data is robust to the presence of the low-rank signal $X$ regardless of the signal's magnitude; see Theorem~\ref{thm: robustness of the resolvent} and Corollaries~\ref{cor: convergence of g_tilde to g in l_1} and~\ref{cor: convergence of g_tilde to g in l_infty}. Intuitively, this favorable property follows from the fact that the resolvent of the data is inversely proportional to the singular values of $Y$ (see Proposition~\ref{prop: formula for g_1_tilde and g_2_tilde} and its proof). Hence, strong signal components that correspond to large singular values in the data have a limited influence on the resolvent. Building on these results, we provide estimators for $\mathbf{x}$ and $\mathbf{y}$ via the resolvent of the data and characterize their accuracy in terms of the dimensions of the matrix, the signal's rank, and the localization of the singular vectors of the signal $X$; see Theorem~\ref{thm: convergence of estimated variance factors}. Our theoretical guarantees allow $m$ and $n$ to grow disproportionally, enable the rank of the signal to increase at certain rates with the dimensions, and support regimes where the singular vectors of $X$ concentrate in a vanishingly small proportion of the entries, while the signal's magnitude can be arbitrary. The convergence of our estimators to the true scaling factors in these settings is demonstrated numerically in Figure~\ref{fig: convergense of estimated scaling factors, rank-one case} in Section~\ref{sec: rank one case}.

In Section~\ref{sec: variance matrices with general rank}, we extend the results described above to support more general variance matrices $S$ beyond rank-one. Similarly to~\cite{landa2022biwhitening}, we rely on the fact that any positive variance matrix $S$ can be written as $S = D\{\mathbf{x}\} \widetilde{S} D\{\mathbf{y}\}$, where $\widetilde{S}$ is doubly regular, meaning that the average entry in each row and each column of $\widetilde{S}$ is $1$; see Definition~\ref{def: doubly regular matrix} and Proposition~\ref{prop: matrix scaling}. We investigate which variance matrices $S$ allow us to accurately estimate $\mathbf{x}$ and $\mathbf{y}$ using the procedure developed previously in Section~\ref{sec: rank one case}. We show that the same procedure can be used to estimate $\mathbf{x}$ and $\mathbf{y}$ accurately if $\widetilde{S}$ is sufficiently incoherent with respect to $\mathbf{x}$ and $\mathbf{y}$; see Theorem~\ref{thm: convergence of estimated scaling fators for generl S} and Lemma~\ref{lem: g is close to h under incoherence}. We discuss this incoherence condition and demonstrate that it is satisfied by several prototypical random constructions of $\widetilde{S}$, for which $S$ does not need to be rank-one or even low-rank. We demonstrate numerically the convergence of our proposed estimators in this case in Figure~\ref{fig: convergense of estimated scaling factors, general rank case} in Section~\ref{sec: variance matrices with general rank}.

In Section~\ref{sec: application to signal detection and recovery}, we describe how to utilize our proposed normalization for automating and improving signal detection and recovery under heteroskedastic noise. We first consider the task of rank estimation in Section~\ref{sec: rank estimation}. We establish that under suitable conditions, our proposed normalization from Algorithm~\ref{alg:noise standardization} enforces the standard spectral behavior of homoskedastic noise, i.e., the MP law for the eigenvalue density of $\hat{E}\hat{E}/n$ and the convergence of its largest eigenvalue to the MP upper edge; see Theorem~\ref{thm: MP law for E_hat}. This fact allows for a simple and reliable procedure for rank estimation by comparing the eigenvalues of $\hat{Y}\hat{Y}^T/n$ to the MP upper edge; see Algorithm~\ref{alg:rank estimation} in Section~\ref{sec: rank estimation} and the relevant discussion. We also show that our proposed normalization can significantly reduce the operator norm of the noise and enhance the spectral signal-to-noise ratio, thereby improving the detection of weak signal components under heteroskedasticity; see Figure~\ref{fig: spectrum before and after biwhitening} and the discussion in Section~\ref{sec: rank estimation}.

We proceed in Section~\ref{sec: weighted low-rank approximation} to address the task of recovering the signal $X$. Instead of the traditional approach of thresholding the singular values of the observed data $Y$, we propose to threshold the singular values of the data after normalization, i.e., $\hat{Y}$ from Algorithm~\ref{alg:noise standardization}, followed by unscaling the rows and columns correspondingly; see Algorithm~\ref{alg:weighted low-rank approximation} in Section~\ref{sec: weighted low-rank approximation}. We show that this approach arises naturally from maximum-likelihood estimation of $X$ under Gaussian heteroskedastic noise with variance $S = \mathbf{x} \mathbf{y}^T$. For Gaussian noise with a more general variance matrix $S$ (beyond rank-one), we show that the proposed approach provides an appealing approximation to the true maximum-likelihood estimation problem; see Proposition~\ref{prop: optimal approximate ML estimation} and the relevant discussion. Importantly, our approach is applicable  beyond Gaussian noise and is naturally interpretable; it is equivalent to solving a weighted low-rank approximation problem~\cite{srebro2003weighted} -- seeking the best low-rank approximation to the data in a weighted least-squares sense -- where the weights adaptively penalize the rows and columns according to their inherent noise levels. As a refined alternative to thresholding the singular values after our proposed normalization, we also consider the denoising technique of~\cite{leeb2021matrix}. We demonstrate in simulations that our normalization, when combined with singular value thresholding or the denoising technique of~\cite{leeb2021matrix}, can substantially improve the recovery accuracy of the signal $X$ over alternative approaches under heteroskedasticity; see Figure~\ref{fig: comparison of methods} and the corresponding text.

Lastly, in Section~\ref{sec: examples on real data}, we exemplify our proposed normalization on real data from single-cell RNA sequencing and spatial transcriptomics, comparing our approach to the method of~\cite{landa2022biwhitening}. We demonstrate that our proposed normalization and the method of~\cite{landa2022biwhitening} both result in accurate fits to the MP law when applied to the raw count data; see Figures~\ref{fig: raw scRNA-seq}--\ref{fig: raw ST + Algorithm 1}. However, after applying a certain transformation to the data that is commonly used for downstream analysis, the method of~\cite{landa2022biwhitening} provides unsatisfactory performance, whereas our approach still provides an excellent fit to the MP law; see Figures~\ref{fig: transformed scRNA-seq}--\ref{fig: transformed ST + Algorithm 1}. The reason for this advantage is that our approach supports general noise distributions and is agnostic to the signal, whereas the method of~\cite{landa2022biwhitening} relies on a quadratic relation between the mean and the variance of the data entries, which is hindered by the transformation.

\section{Method derivation and analysis} \label{sec: method derivation and analysis}
\subsection{Preliminaries: the noise resolvent and the quadratic Dyson equation} \label{sec: Dyson equations}
We begin by defining $\mathcal{E},\mathcal{S} \in\mathbb{R}^{(m+n)\times (m+n)}$ as the symmetrized versions of $E$ and $S$, respectively, according to
\begin{equation}
    \mathcal{E} = 
    \begin{bmatrix}
    \mathbf{0}_{m\times m} & E \\
    E^T & \mathbf{0}_{n\times n}
    \end{bmatrix}, \qquad\qquad 
    \mathcal{S} = 
    \begin{bmatrix}
    \mathbf{0}_{m\times m} & S \\
    S^T & \mathbf{0}_{n\times n}
    \end{bmatrix}, \label{eq: definition of E_cal and S_cal}
\end{equation}
where $\mathbf{0}_{m\times m}$ and $\mathbf{0}_{n\times n}$ are $m\times m$ and $n\times n$ zero matrices, respectively. Next, the resolvent of $\mathcal{E}$ is defined as the complex-valued random matrix $R(z) \in \mathbb{C}^{(m+n)\times(m+n)}$ given by
\begin{equation}
    R(z) = (\mathcal{E} - z I)^{-1}, \label{eq: resolvent def}
\end{equation}
where $I$ is the identity matrix and $z\in\mathbb{C}^+$ is a point in the complex upper half plane $\mathbb{C}^+$. A fundamental fact underpinning our approach is that in a broad range of settings, the main diagonal of $R(z)$ concentrates around the solution $\mathbf{f}\in\mathbb{C}^{m+n}$ to the deterministic vector-valued equation
\begin{equation}
    z + \mathcal{S}\mathbf{f} = -\frac{1}{\mathbf{f}}, \label{eq: Dyson eq}
\end{equation} 
for any $z\in \mathbb{C}^+$~\cite{erdHos2019random,alt2017local} (see Lemma~\ref{lem: concentration of bilinear forms} in Section~\ref{sec: analysis for rank-one variance matrices} for a formal statement of this result in our setting). For simplicity of presentation, we use a standard fraction notation, e.g., in the right-hand side of~\eqref{eq: Dyson eq}, to denote entrywise division whenever vectors are involved. Also, the addition of a vector and a scalar, e.g., in the left-hand side of~\eqref{eq: Dyson eq}, describes the entrywise addition of the scalar to all entries of the vector.
The equation~\eqref{eq: Dyson eq} is known as the quadratic Dyson equation and has been extensively studied in the literature on random matrix theory. The following proposition guarantees the existence and uniqueness of the solution $\mathbf{f}$ and states several useful properties (see~\cite{ajanki2019quadratic}). 
\begin{prop} \label{prop: solution to Dyson equation}
There exists a unique holomorphic function $\mathbf{f}:\mathbb{C}^+ \rightarrow \mathbb{C}^{m+n}$ that solves~\eqref{eq: Dyson eq} for all $z\in \mathbb{C}^+$. Moreover, $\operatorname{Im}\{\mathbf{f}(z)\} > 0$ for all $z\in \mathbb{C}^+$ and $\operatorname{Re}\{\mathbf{f}(z)\}$ is an odd function of $\operatorname{Re}\{z\}$.
\end{prop}
In what follows, we focus on the restriction of $\mathbf{f}$ to the upper half of the imaginary axis, i.e., $z = \imath\eta$ for $\eta>0$. In this case, Proposition~\ref{prop: solution to Dyson equation} asserts that the real part of $\mathbf{f}$ vanishes and its imaginary part is strictly positive, hence
\begin{equation}
    \mathbf{f}(\imath \eta) = \imath \mathbf{g}(\eta), \label{eq: definition of g}
\end{equation}
where $\mathbf{g}(\eta) \in \mathbb{R}^{m+n}$ is a positive vector for all $\eta>0$. By plugging~\eqref{eq: definition of g} into the Dyson equation for $z=\imath \eta$ and taking the imaginary part of both sides, we see that $\mathbf{g}$ satisfies the vector-valued equation
\begin{equation}
    \eta + \mathcal{S}\mathbf{g} = \frac{1}{\mathbf{g}}. \label{eq: Dyson eq imag axis}
\end{equation}
Note that according to the definition of $\mathcal{S}$ in~\eqref{eq: definition of E_cal and S_cal}, the equation~\eqref{eq: Dyson eq imag axis} can also be written as the system of coupled equations 
\begin{equation}
    \eta + {S}\mathbf{g}^{(2)} = \frac{1}{\mathbf{g}^{(1)}}, \qquad \qquad \eta + {S^T}\mathbf{g}^{(1)} = \frac{1}{\mathbf{g}^{(2)}}, \label{eq: coupled Dyson eq imag axis}
\end{equation}
where $\mathbf{g}^{(1)} = [\mathbf{g}_1,\ldots,\mathbf{g}_m]^T$ and $\mathbf{g}^{(2)} = [\mathbf{g}_{m+1},\ldots,\mathbf{g}_{m+n}]^T$.
The system of equations in~\eqref{eq: coupled Dyson eq imag axis} involves only real-valued and strictly positive quantities, making it particularly convenient for our subsequent derivations and analysis. 

\subsection{Variance matrices $S$ with rank one} \label{sec: rank one case}
We begin by considering the case of a rank-one variance matrix, i.e.,
\begin{assump} \label{assump: rank one variance matrix}
$S = \mathbf{x} \mathbf{y}^T$ for positive $\mathbf{x}\in\mathbb{R}^m$ and $\mathbf{y}\in\mathbb{R}^n$.
\end{assump}
As mentioned in the introduction, under Assumption~\ref{assump: rank one variance matrix}, the normalization of the rows and columns in~\eqref{eq: scaling rows and columns} makes the noise completely homoskedastic with variance one, i.e., $\mathbb{E}[{\widetilde{E}}_{ij}^2] = 1$ for all $i\in[m]$ and $j\in[n]$. In what follows, we address the task of estimating $\mathbf{x}$ and $\mathbf{y}$ directly from the data matrix $Y$. We begin by deriving our proposed estimators in Section~\ref{sec: method derivation for rank-one variance matrices} and proceed to analyze their convergence to $\mathbf{x}$ and $\mathbf{y}$ in Section~\ref{sec: analysis for rank-one variance matrices}. The main theoretical result in this section is Theorem~\ref{thm: convergence of estimated variance factors}, which provides probabilistic error bounds for our proposed estimators under suitable conditions.

\subsubsection{Method derivation} \label{sec: method derivation for rank-one variance matrices}
To derive our estimators for $\mathbf{x}$ and $\mathbf{y}$, we observe that under Assumption~\ref{assump: rank one variance matrix}, the vectors $\mathbf{x}$ and $\mathbf{y}$ can be written explicitly in terms of $\mathbf{g}^{(1)}$ and $\mathbf{g}^{(2)}$ from~\eqref{eq: coupled Dyson eq imag axis} up to a trivial scalar ambiguity. Specifically, we have the following proposition, whose proof can be found in Appendix~\ref{appendix: proof of formula for x and y in terms of g}.
\begin{prop} \label{prop: formula for x and y in terms of g}
Under Assumption~\ref{assump: rank one variance matrix}, we have
\begin{equation}
    \mathbf{x} = \frac{\alpha}{\sqrt{m - \eta \Vert \mathbf{g}^{(1)}\Vert_1 }} \left( \frac{1}{\mathbf{g}^{(1)}} - \eta \right), 
    \qquad
    \mathbf{y} =  \frac{\alpha^{-1}}{\sqrt{n - \eta \Vert \mathbf{g}^{(2)} \Vert_1}} \left( \frac{1}{\mathbf{g}^{(2)}} - \eta \right),
    \qquad
    \alpha = \sqrt{\frac{\mathbf{x}^T \mathbf{g}^{(1)}}{\mathbf{y}^T \mathbf{g}^{(2)}}}. \label{eq: formulas for x and y}
\end{equation}
\end{prop}
The scalar ambiguity in Proposition~\ref{prop: formula for x and y in terms of g} stems from the fact that $\mathbf{x}$ and $\mathbf{y}$ can always be replaced with $\alpha^{-1} \mathbf{x}$ and $\alpha \mathbf{y}$, respectively, for any $\alpha>0$, resulting in the same noise variance matrix $S = \mathbf{x}\mathbf{y}^T$. To settle this ambiguity, we set $\alpha = 1$ by requiring that
\begin{equation}
    \mathbf{x}^T \mathbf{g}^{(1)} = \mathbf{y}^T \mathbf{g}^{(2)}, \label{eq: setting alpha=1}
\end{equation}
which can be assumed to hold without loss of generality.

We propose to estimate $\mathbf{x}$ and $\mathbf{y}$ by first estimating $\mathbf{g}^{(1)}$ and $\mathbf{g}^{(2)}$ from the observed data and then utilizing the formulas in Proposition~\ref{prop: formula for x and y in terms of g} with $\alpha=1$. As mentioned in Section~\ref{sec: Dyson equations}, under suitable conditions (which will be stated later on in our analysis), the main diagonal of the noise resolvent $R(z)$ concentrates around the solution $\mathbf{f}(z)$ to the Dyson equation~\eqref{eq: Dyson eq}. Since $\mathbf{g}(\eta) = \operatorname{Im}\{\mathbf{f}(\imath \eta)\}$ according to~\eqref{eq: definition of g}, we expect the imaginary part of the main diagonal of the noise resolvent $R(\imath \eta)$ to approximate $\mathbf{g}(\eta)$. However, since we do not have direct access to the noise resolvent, we replace it with the resolvent constructed analogously from the data. To this end, we define $\mathcal{Y}\in\mathbb{R}^{(m+n)\times(m+n)}$ as the symmetrized version of the data matrix $Y$, i.e.,
\begin{equation}
    \mathcal{Y} = 
    \begin{bmatrix}
    \mathbf{0}_{m\times m} & Y \\
    Y^T & \mathbf{0}_{n\times n}
    \end{bmatrix},
\end{equation}
and denote the resolvent of $\mathcal{Y}$ by $\mathcal{R}(z) \in \mathbb{C}^{(m+n)\times(m+n)}$, namely
\begin{equation}
    \mathcal{R}(z) = (\mathcal{Y} - z I)^{-1}, \label{eq: resolent of data}
\end{equation}
for $z\in \mathbb{C}^+$. We then estimate $\mathbf{g}(\eta)$ from the imaginary part of the main diagonal of the data resolvent $\mathcal{R}(\imath \eta)$ according to
\begin{align}
    \hat{\mathbf{g}}_i (\eta) = \operatorname{Im}\{\mathcal{R}_{ii}(\imath \eta) \}, \label{eq: g_hat def}
\end{align}
for all $i\in [m+n]$. For notational convenience, we define the estimates of $\mathbf{g}^{(1)}$ and $\mathbf{g}^{(2)}$ separately as $\hat{\mathbf{g}}^{(1)} = [\hat{\mathbf{g}}_1,\ldots,\hat{\mathbf{g}}_m]^T$ and $\hat{\mathbf{g}}^{(2)} = [\hat{\mathbf{g}}_{m+1},\ldots,\hat{\mathbf{g}}_{m+n}]^T$, respectively, and omit the explicit dependence on $\eta$. 

The following proposition, whose proof can be found in Appendix~\ref{appendix: proof of formula for g_1_tilde and g_2_tilde}, shows that $\hat{\mathbf{g}}^{(1)}$ and $\hat{\mathbf{g}}^{(2)}$ have simple formulas in terms of the SVD of $Y$ (recalling our assumption that $m\leq n$). In particular, these formulas justify steps~\ref{alg: step 1} and~\ref{alg:step 3} in Algorithm~\ref{alg:noise standardization}.
\begin{prop} \label{prop: formula for g_1_tilde and g_2_tilde}
Let the columns of $U\in\mathbb{R}^{m\times m}$, the columns of $V\in\mathbb{R}^{n\times n}$, and $\{\sigma_k\}_{k=1}^m$ denote the left singular vectors, right singular vectors, and singular values of $Y$, respectively. Then, 
\begin{align}
    \hat{\mathbf{g}}_i^{(1)} = \sum_{k=1}^m \frac{\eta}{\sigma_k^2 + \eta^2} U_{ik}^2, \qquad\qquad \hat{\mathbf{g}}_j^{(2)} = \frac{1}{\eta} + \sum_{k=1}^m \left( \frac{\eta}{\sigma_k^2 + \eta^2} - \frac{1}{\eta} \right) V_{jk}^2, \label{eq: formulas for g_hat_1 and g_hat_2 in terms of the SVD}
\end{align}
for all $i\in[m]$ and $j\in [n]$.
\end{prop}

As we shall see in our analysis, an important property of the estimator $\hat{\mathbf{g}}$ of $\mathbf{g}$ is that it is not sensitive to the presence of a possibly strong low-rank signal $X$. This is one of the main advantages of using the data resolvent to infer the structure of the noise variance, as opposed to, e.g., computing the empirical variances of the data matrix $Y$ across its rows and columns, which can be highly sensitive to the presence of strong components in the signal $X$. 

Equipped with an estimator of $\mathbf{g}$, we propose to estimate $\mathbf{x}$ and $\mathbf{y}$ by replacing $\mathbf{g}^{(1)}$ and $\mathbf{g}^{(2)}$ in~\eqref{eq: formulas for x and y} with their estimates $\hat{\mathbf{g}}^{(1)}$ and $\hat{\mathbf{g}}^{(2)}$, respectively, taking $\alpha=1$. Specifically, we define
\begin{equation}
    \hat{\mathbf{x}} = \frac{1}{\sqrt{m - \eta \Vert \hat{\mathbf{g}}^{(1)} \Vert_1 }} \left( \frac{1}{\hat{\mathbf{g}}^{(1)}} - \eta \right), \qquad\qquad
    \hat{\mathbf{y}} = \frac{1}{\sqrt{n - \eta \Vert \hat{\mathbf{g}}^{(2)} \Vert_1 }} \left( \frac{1}{\hat{\mathbf{g}}^{(2)}} - \eta \right), \label{eq: x_hat and y_hat def}
\end{equation}
which agrees with Step~\ref{alg:step 4} in Algorithm~\ref{alg:noise standardization}. 
The following proposition, whose proof can be found in Appendix~\ref{appendix: proof of x and y nonnegative}, establishes the well-posedness of the formulas in~\eqref{eq: x_hat and y_hat def} and the normalization~\eqref{eq: scaling rows and columns} when replacing $\mathbf{x}$ and $\mathbf{y}$ with $\hat{\mathbf{x}}$ and $\hat{\mathbf{y}}$, respectively. In particular, this proposition justifies our requirement at the beginning of Algorithm~\ref{alg:noise standardization}.
\begin{prop} \label{prop: x and y nonnegative}
The vector $\hat{\mathbf{g}}$ from~\eqref{eq: g_hat def} is positive for any $\eta>0$. Further, if $Y$ is not the zero matrix, i.e., $Y\neq \mathbf{0}_{m\times n}$, then $\hat{\mathbf{x}}$ and $\hat{\mathbf{y}}$ from~\eqref{eq: x_hat and y_hat def} are nonnegative vectors for any $\eta>0$, where $\hat{\mathbf{x}}_i = 0$ ($\hat{\mathbf{y}}_j = 0$) if and only if the $i$th row ($j$th column) of $Y$ is entirely zero.
\end{prop}

Lastly, while our derivation in this section does not restrict the value of $\eta>0$, we propose in practice to set $\eta$ as the median singular value of $Y$; see step~\ref{alg: step 2} in Algorithm~\ref{alg:noise standardization}. The purpose of this choice is to adapt our procedure to the global scaling of the noise in the data. In particular, this choice of $\eta$ enforces a certain scaling of the noise that is required for our analysis in the next section; see Assumptions~\ref{assump: noise moment bound} and~\ref{assump: variance boundedness} in the next section and the discussion in Appendix~\ref{appendix: adapting to unknown global scaling of the noise}.

\subsubsection{Convergence analysis} \label{sec: analysis for rank-one variance matrices}
For our analysis in this section and subsequent ones, we consider $\eta$ as a fixed global constant. All other quantities in our setup, such as the matrix dimensions $m$ and $n$, the distribution of the noise entries $E_{ij}$, and the signal $X$ (its rank $r$ and its singular values and vectors) can be arbitrary, pending that they satisfy our assumptions detailed below. Therefore, in what follows, all constants appearing in our results may depend only on $\eta$ and the relevant global constants defined in our assumptions. 

Before delving into the details of our analysis, we provide an overview of its different steps, intermediate results, and assumptions. Our analysis below begins by presenting an auxiliary result on the concentration of bilinear forms of the noise resolvent $R(\imath \eta)$; see Lemma~\ref{lem: concentration of bilinear forms}. This result is fundamental to our subsequent analysis and requires only an assumption on the upper boundedness of the moments of the noise entries $E_{ij}$; see Assumption~\ref{assump: noise moment bound}. Using this result, we characterize the concentration of the main diagonal of the data resolvent $\mathcal{R}(\imath \eta)$ around the main diagonal of the noise resolvent $R(\imath \eta)$ in terms of the signal's rank $r$, its singular vectors, its largest singular value, and the long dimension $n$; see Theorem~\ref{thm: robustness of the resolvent}. For this result, we allow the rank $r$ to grow with $n$ but limit the growth to be slower than $\sqrt{n}$; see Assumption~\ref{assump: rank growth rate}. Importantly, Theorem~\ref{thm: robustness of the resolvent} establishes the robustness of the main diagonal of the data resolvent to low-rank perturbations, irrespective of the signal's strength. Subsequently, we utilize Theorem~\ref{thm: robustness of the resolvent} to prove the stochastic convergence of the errors $\Vert \hat{\mathbf{g}} - \mathbf{g} \Vert_1 \rightarrow 0$ and $\Vert \hat{\mathbf{g}} - \mathbf{g} \Vert_\infty \rightarrow 0$  as $m,n\rightarrow \infty$ with rates, see Corollaries~\ref{cor: convergence of g_tilde to g in l_1} and~\ref{cor: convergence of g_tilde to g in l_infty}, respectively, where the latter requires an additional assumption on the delocalization of the singular vectors of $X$; see Assumption~\ref{assump: signal delocalization} and the discussion and examples that follows. Finally, using Corollary~\ref{cor: convergence of g_tilde to g in l_infty}, we establish the main result of this section, which is the convergence of the estimated factors $(\hat{\mathbf{x}},\hat{\mathbf{y}})$ to the true factors $(\mathbf{x},\mathbf{y})$; see Theorem~\ref{thm: convergence of estimated variance factors}. Specifically, Theorem~\ref{thm: convergence of estimated variance factors} establishes the stochastic convergence of the relative errors $\Vert(\hat{\mathbf{x}} - \mathbf{x})/\mathbf{x}\Vert_\infty\rightarrow 0$ and $\Vert (\hat{\mathbf{y}} - \mathbf{y})/\mathbf{y}\Vert_\infty\rightarrow 0$ as $m,n\rightarrow \infty$ with rates (where $\cdot / \cdot$ refers to entrwywise division). To obtain the convergence of these error terms, we require an assumption on the lower boundedness of the noise variances $S_{ij}$; see Assumption~\ref{assump: variance boundedness}. For the convergence of the error $\Vert (\hat{\mathbf{y}} - \mathbf{y})/\mathbf{y}\Vert_\infty$, we additionally require that the short dimension $m$ grows sufficiently quickly with $n$; see Assumption~\ref{assump: growth of m} and the discussion that follows. 

To formally state the relation between $\mathbf{g}$ from~\eqref{eq: Dyson eq imag axis} and the noise resolvent $R(\imath \eta)$, we require all moments of the scaled noise variables $\sqrt{n} E_{ij}$ to be upper bounded by global constants. That is,
\begin{assump} \label{assump: noise moment bound}
There exist global constants $\{\mu_q\}_{q=1}^\infty$ such that $\mathbb{E}\vert \sqrt{n} E_{ij} \vert^q \leq \mu_q$ for all $i\in[m]$, $j\in[n]$, $q\in\mathbb{N}$.
\end{assump}
Assumption~\ref{assump: noise moment bound} is non-restrictive and covers many standard noise models, including all sub-Gaussian and sub-exponential distributions~\cite{vershynin2018high}. We now have the following lemma, which characterizes the relation between the noise resolvent on the imaginary axis, i.e., $R(\imath \eta)$, and the vector $\mathbf{g}$ from~\eqref{eq: Dyson eq imag axis}.
\begin{lem} \label{lem: concentration of bilinear forms}
Under Assumption~\ref{assump: noise moment bound}, for any $\epsilon>0$ there exist $C^{'},c^{'}(t)>0$ such that for all deterministic unit vectors $\mathbf{a},\mathbf{b}\in\mathbb{C}^{m+n}$ (i.e., $\Vert \mathbf{a} \Vert_2=\Vert \mathbf{b} \Vert_2 = 1$) and $t>0$, with probability at least $1-c^{'}(t) n^{-t}$ we have that
\begin{equation}
    \left\vert \mathbf{a}^T \left[ R(\imath \eta) -  D\{\imath \mathbf{g}(\eta)\} \right] \mathbf{b} \right\vert \leq C^{'} n^{\epsilon-1/2}.
\end{equation}
\end{lem}
The proof of Lemma~\ref{lem: concentration of bilinear forms} can be found in Appendix~\ref{appendix: proof of concentration of biliniear forms} and follows directly from the results in~\cite{erdHos2019random} adapted to our setting. Lemma~\ref{lem: concentration of bilinear forms} establishes the concentration of bilinear forms of the noise resolvent evaluated on the imaginary axis in terms of the deterministic vector $\mathbf{g}$. In particular, it asserts that under Assumption~\ref{assump: noise moment bound} and for large $n$ (the long dimension), the noise resolvent $R(\imath \eta)$ behaves like a diagonal matrix whose main diagonal is $\imath \mathbf{g}$. 
We mention that under Assumption~\ref{assump: noise moment bound}, the entries of $\mathbf{g}$ are always lower bounded away from zero by a global constant (which may depend on $\eta$); see Lemma~\ref{lem: boundedness of g} in Appendix~\ref{appendix: boundedness of g}. Therefore, we have that $\operatorname{Im}\{R_{ii}(\imath \eta) \} \sim  \mathbf{g}_i(\eta)$ almost surely as $n\rightarrow\infty$ for all $i\in [m+n]$, where the convergence rate is almost $n^{-1/2}$. Consequently, for large $n$, the main diagonal of the noise resolvent provides an accurate approximation to the solution $\mathbf{g}$ to the Dyson equation~\eqref{eq: Dyson eq imag axis} on the imaginary axis. 
We note that the statement in Lemma~\ref{lem: concentration of bilinear forms} on the concentration of bilinear forms of a matrix is stronger than a statement on the concentration of the entries. Indeed, the former implies the latter but not vice-versa. This stronger statement is required for our key result below on the concentration of the main diagonal of the data resolvent.

Let us denote the (compact) SVD of the signal matrix $X$ as
\begin{equation}
    X = \widetilde{U}  D \{s_1,\ldots,s_r\} \widetilde{V}^T,
\end{equation}
where $\widetilde{U}\in\mathbb{R}^{m\times r}$, $\widetilde{V}\in\mathbb{R}^{n\times r}$, and $ s_1\geq \ldots \geq s_r$ stand for the left singular vectors, right singular vectors, and singular values of $X$, respectively. 
To establish that $\hat{\mathbf{g}}$ concentrates around $\mathbf{g}$ in the signal-plus-noise model~\eqref{eq: signal-plus-noise}, we make the following assumption on the behavior of the rank $r$ with respect to the long data dimension $n$.
\begin{assump} \label{assump: rank growth rate}
There exist constants $\widetilde{C}_0>0$ and $\delta_0 \in [0,1/2)$ such that $r \leq \widetilde{C}_0 n^{\delta_0}$.
\end{assump}
In other words, the rank $r$ can be constant ($\delta_0 = 0$) but also allowed to grow with $n$ at most with some fractional power smaller than $1/2$. 
We now have the following theorem, which characterizes the concentration of $\hat{\mathbf{g}}$ around $\mathbf{g}$ in terms of the long dimension $n$ and the SVD of the signal $X$.

\begin{thm} \label{thm: robustness of the resolvent}
Under Assumptions~\ref{assump: noise moment bound} and~\ref{assump: rank growth rate},
for any $\epsilon \in (0,1/2 - \delta_0)$ there exist $C_1^{'},C_2^{'},C_3^{'},c^{'}(t)>0$, such that for all $i\in [m]$, $j\in[n]$, and $t>0$, with probability at least $1-c^{'}(t) n^{-t}$ we have that
\begin{align}
    \vert \hat{\mathbf{g}}_i^{(1)} - {\mathbf{g}}_i^{(1)} \vert
    &\leq C_1^{'} n^{\epsilon-1/2} + \frac{C_2^{'}}{\sqrt{s_1^{-2} + C_3^{'}}} \sum_{k=1}^r \widetilde{U}_{ik}^2, \label{eq: g_hat_1 - g_1 error}\\
    \vert \hat{\mathbf{g}}_j^{(2)} - {\mathbf{g}}_j^{(2)} \vert 
    &\leq C_1^{'} n^{\epsilon-1/2} + \frac{C_2^{'}}{\sqrt{s_1^{-2} + C_3^{'}}} \sum_{k=1}^r \widetilde{V}_{jk}^2. \label{eq: g_hat_2 - g_2 error}
\end{align} 
\end{thm}
The proof of Theorem~\ref{thm: robustness of the resolvent} can be found in Appendix~\ref{appendix: proof of robustness of the resolvent}. The main idea therein is to utilize the Woodbury matrix identity together with Lemma~\ref{lem: concentration of bilinear forms} to show that the main diagonal of the data resolvent $\mathcal{R}(\imath \eta)$ concentrates around $\imath \mathbf{g}$ plus an error term. This additive error term has a special algebraic structure involving a certain complex-valued $2r\times 2r$ matrix. We conduct a careful spectral analysis of this error term to prove the required result, where Assumption~\ref{assump: rank growth rate} is utilized to guarantee the boundedness of the aforementioned complex-valued matrix in operator norm. It is important to note that Theorem~\ref{thm: robustness of the resolvent} does not rely on Assumption~\ref{assump: rank one variance matrix}. Consequently, the probabilistic bound in Theorem~\ref{thm: robustness of the resolvent} does not depend on the structure of the noise variance matrix $S$. This fact will be important for our analysis in the next section involving general variance matrices.

From Theorem~\ref{thm: robustness of the resolvent}, we see that the bound on the probabilistic error between the entries of $\hat{\mathbf{g}}$ and $\mathbf{g}$ has two terms. The first term depends only on the long data dimension $n$ and converges to zero with a rate of almost $n^{-1/2}$. The second term depends on the largest singular value $s_1$ and the magnitude of the entries of the singular vectors $\widetilde{U}$ and $\widetilde{V}$. Evidently, if the signal is vanishingly small, i.e., $s_1 \rightarrow 0$, then the second error term vanishes and, as expected from Lemma~\ref{lem: concentration of bilinear forms}, $\hat{\mathbf{g}}$ approaches $\mathbf{g}$ with rate almost $n^{-1/2}$ (with probability rapidly approaching $1$). However, even if some or all of the singular values are arbitrarily large, we always have 
\begin{equation}
    \frac{C_2^{'}}{\sqrt{s_1^{-2} + C_3^{'}}} \sum_{k=1}^r \widetilde{U}_{ik}^2 
    \leq \frac{C_2^{'}}{\sqrt{C_3^{'}}} \sum_{k=1}^r \widetilde{U}_{ik}^2, \qquad
    \frac{C_2^{'}}{\sqrt{s_1^{-2} + C_3^{'}}} \sum_{k=1}^r \widetilde{V}_{jk}^2 
    \leq \frac{C_2^{'}}{\sqrt{C_3^{'}}} \sum_{k=1}^r \widetilde{V}_{jk}^2,
\end{equation}
which depends only on the singular vectors of $X$ and not the singular values. Since each singular vector has a unit Euclidean norm, we have the following immediate corollary of Theorem~\ref{thm: robustness of the resolvent}.
\begin{cor} \label{cor: convergence of g_tilde to g in l_1}
Under Assumptions~\ref{assump: noise moment bound} and~\ref{assump: rank growth rate}, for any $\epsilon \in (0, 1/2 - \delta_0)$ there exist $C^{'},c^{'}(t)>0$ such that for all $t>0$, with probability at least $1-c^{'}(t) n^{-t}$ we have that
\begin{equation}
    \frac{1}{m}\Vert \hat{\mathbf{g}}^{(1)} - \mathbf{g}^{(1)} \Vert_1 \leq C^{'}\max\left\{n^{\epsilon-1/2}, r m^{-1} \right\}, \qquad \frac{1}{n}\Vert \hat{\mathbf{g}}^{(2)} - \mathbf{g}^{(2)} \Vert_1 \leq C^{'}n^{\epsilon-1/2},
\end{equation}
\end{cor}
Therefore, the average absolute error between $\hat{\mathbf{g}}^{(2)}$ and ${\mathbf{g}}^{(2)}$ always converges to zero as $n\rightarrow \infty$ with probability approaching $1$ and rate almost $n^{-1/2}$. This convergence does not depend on $m$ (the short dimension) nor the low-rank signal $X$. If $m$ also grows with $n$ such that $r/m\rightarrow 0$, then the average absolute error between $\hat{\mathbf{g}}^{(1)}$ and ${\mathbf{g}}^{(1)}$ also converges to zero (with probability approaching $1$). In this case, the convergence rate depends on the growth rates of $r$ and $m$ with respect to $n$. Note that when $m,n\rightarrow \infty$, the convergence of the average absolute errors discussed above implies that almost all entries of $\hat{\mathbf{g}}$ converge to the corresponding entries of $\mathbf{g}$ with probability approaching $1$. In particular, the proportion of indices $i\in[m]$ and $j\in[n]$ for which the errors $\vert \hat{\mathbf{g}}_i^{(1)} - \mathbf{g}_i^{(1)} \vert$ and $\vert \hat{\mathbf{g}}_j^{(2)} - \mathbf{g}_j^{(2)} \vert$ do not converge to zero must be vanishing as $m,n\rightarrow \infty$ (otherwise, it would be a contradiction to the convergence of the average absolute errors). Therefore, the main diagonal of the data resolvent $\mathcal{R}(\imath \eta)$ is robust to the presence of the low-rank signal $X$ in the sense that for large $m$ and $n$, we can use it to accurately estimate almost all entries of $\mathbf{g}$ regardless of $X$ and its structure. 

To proceed, it is convenient to have a stronger control of the element-wise errors between $\hat{\mathbf{g}}$ and $\mathbf{g}$. Specifically, it is desirable to guarantee the stochastic convergence of $\Vert \hat{\mathbf{g}} - \mathbf{g} \Vert_\infty$. To this end, we define the quantity
\begin{equation}
    \kappa = \max\left\{\max_{i\in [m]}\left\{ \sum_{k=1}^r \widetilde{U}_{ik}^2\right\}, \max_{j\in [n]}\left\{\sum_{k=1}^r \widetilde{V}_{jk}^2\right\}\right\},
\end{equation}
and make the assumption that $\kappa$ decays at least with some fractional power of $n$:
\begin{assump} \label{assump: signal delocalization}
There exist constants $\widetilde{C},\delta_1 >0$, such that $\kappa \leq \widetilde{C}_1 n^{-\delta_1}$.
\end{assump}
Assumption~\ref{assump: signal delocalization} can be interpreted as a requirement on the delocalization level of the singular vector matrices $\widetilde{U}$ and $\widetilde{V}$ across the rows. In particular, it prohibits the singular vector matrices from having too few non-zero rows, e.g., a fixed number or a logarithmically growing number of such rows. However, the singular vectors matrices can still be sparse or localized in a vanishingly small proportion of rows. We note that the delocalization level of $\widetilde{U}$ and $\widetilde{V}$ across the rows depends implicitly also on the rank $r$ and the behavior of the short dimension $m$. Below we provide a few examples illustrating Assumption~\ref{assump: signal delocalization}. 
\begin{example} \label{example: signal delocalization example 1}
    Suppose that the rank $r$ is fixed and $m$ is proportional to $n$. Suppose further that the singular vectors are fully delocalized, namely that $\max_{i\in [m]} \widetilde{U}^2_{ik} \lesssim m^{-1}$ and $\max_{j\in [n]} \widetilde{V}^2_{jk} \lesssim n^{-1}$ (where $\lesssim$ implies `less than' up to constants and logarithmic factors) for all $k \in [r]$. For instance, this property is satisfied with high probability if the singular vectors are orthonormalized independent Gaussian random vectors, or more generally, if they are the singular vectors of a random matrix with independent entries and certain moment bounds~\cite{ding2019singular}. Another simple model satisfying this property is when $X$ is a block matrix (with identical entries in each block), where the number of blocks is fixed and the dimensions of each block are proportional to the dimensions of the matrix. In all of these cases, since $\kappa \leq r \max\{ \max_{i\in [m]} \widetilde{U}_{ik}^2, \max_{j\in [n]} \widetilde{V}_{jk}^2 \} \lesssim n^{-1}$, Assumption~\ref{assump: signal delocalization} holds with $\delta_1$ arbitrarily close to $1$, e.g., $\delta_1 = 0.99$. Note that $\delta_1$ can be at most $1$ since we always have $\max_{j\in [n]} \widetilde{V}^2_{jk} \geq n^{-1}$, where equality is attained for singular vectors whose entries have identical magnitudes.
\end{example}

\begin{example} \label{example: signal delocalization example 2}
    Suppose that the singular vectors are fully delocalized as in Example~\ref{example: signal delocalization example 1}, but the rank $r$ is growing as $n^{1/4}$ and $m$ is proportional to $n^{3/4}$. Then, $\max_{i\in [m]}\{ \sum_{k=1}^r \widetilde{U}_{ik}^2\} \lesssim rm^{-1} \lesssim n^{-1/2}$ and $\max_{j\in [n]}\{ \sum_{k=1}^r \widetilde{V}_{jk}^2\} \lesssim rn^{-1} \lesssim n^{-3/4}$. Hence, Assumption~\ref{assump: signal delocalization} holds with $\delta_1$ arbitrarily close to $1/2$, e.g., $\delta_1 = 0.49$.
\end{example}

\begin{example} \label{example: signal delocalization example 3}
    Suppose that the singular vector matrices $\widetilde{U}$ and $\widetilde{V}$ have $\lceil {m}^{1/2} \rceil$ and $\lceil {n}^{1/2} \rceil$ non-zero rows, respectively, where $\lceil \cdot \rceil$ denotes the ceiling function. Note that in this case, the singular vector matrices are highly localized since the proportion of the non-zero rows is vanishing as $n\rightarrow \infty$. Suppose further that the singular vectors are delocalized across the non-zero rows, such that $\max_{i\in [m]} \widetilde{U}^2_{ik} \lesssim m^{-1/2}$ and $\max_{j\in [n]} \widetilde{V}^2_{jk} \lesssim n^{-1/2}$.
    If the rank $r$ is fixed and $m$ is proportional to $n$, then as in Example~\ref{example: signal delocalization example 2}, Assumption~\ref{assump: signal delocalization} holds with $\delta_1$ arbitrarily close to $1/2$, e.g., $\delta_1 = 0.49$.
\end{example}

Generally, Assumption~\ref{assump: signal delocalization} allows the singular vectors of $X$ to be highly sparse, since the non-zero entries can be restricted to subsets with cardinality proportional to any (arbitrarily small) fractional power of $n$, as long as the rank grows sufficiently slowly. We note that singular vectors satisfying Assumption~\ref{assump: signal delocalization} do not have to be strictly sparse with delocalized entries on the non-zero subsets (as in Example~\ref{example: signal delocalization example 3}). For instance, they can be generated from random vectors with i.i.d random variables whose tails simulate various degrees of delocalization. Finally, we remark that Assumption~\ref{assump: signal delocalization} requires the short dimension $m$ to grow at least with some fractional power of $n$ times the rank $r$. Specifically, since we always have $\max_{i\in m} \sum_{k=1}^r \widetilde{U}^2_{ik} \geq r m^{-1}$, Assumption~\ref{assump: signal delocalization} implies that $m \geq r n^{\delta_1}/\widetilde{C}_1$.

Under the additional Assumption~\ref{assump: signal delocalization}, we obtain the following immediate corollary of Theorem~\ref{thm: robustness of the resolvent}.
\begin{cor} \label{cor: convergence of g_tilde to g in l_infty}
Under Assumptions~\ref{assump: noise moment bound}--\ref{assump: signal delocalization}, for any $\epsilon \in (0,1/2-\delta_0)$ there exist $C^{'},c^{'}(t)>0$ such that for all $t>0$, with probability at least $1-c^{'}(t) n^{-t}$ we have that
\begin{equation}
    \left\Vert { \hat{\mathbf{g}} - \mathbf{g} } \right\Vert_\infty \leq C^{'} \max\left\{n^{\varepsilon-1/2},n^{-\delta_1} \right\}. 
\end{equation}
\end{cor}
Corollary~\ref{cor: convergence of g_tilde to g in l_infty} establishes the stochastic entrywise convergence of $\hat{\mathbf{g}}$ to $\mathbf{g}$ with a rate that depends on the delocalization level of the singular vector matrices according to Assumption~\ref{assump: signal delocalization}. For the three examples described in the text after Assumption~\ref{assump: signal delocalization}, where $\delta_1 \sim 1$ or $\delta_1 \sim 1/2$, the convergence rate of $\hat{\mathbf{g}}$ to $\mathbf{g}$ would remain almost $n^{-1/2}$. However, the convergence rate can be slower if $\delta_1$ is smaller, e.g., if in Examples~\ref{example: signal delocalization example 1}--\ref{example: signal delocalization example 3} the rank grows faster, the individual singular vectors are sparser, or if $m$ grows more slowly. 

Next, we turn to establish the entrywise concentration of $\hat{\mathbf{x}}$ and $\hat{\mathbf{y}}$ around $\mathbf{x}$ and $\mathbf{y}$, respectively.
To this end, we need two additional assumptions. First, we require that the noise variances $S_{ij}$ are lower bounded by a global constant divided by $n$, i.e.,
\begin{assump} \label{assump: variance boundedness}
There exist a global constant $c>0$ such that $S_{ij} \geq {c}{n}^{-1} $ for all $i\in[m]$ and $j\in [n]$.
\end{assump}
Therefore, together with Assumption~\ref{assump: noise moment bound}, we require the noise variances to be lower and upper bounded according to $cn^{-1} \leq S_{ij} \leq \mu_2 n^{-1}$. Since the constants $\mu_2,c>0$ are arbitrary, e.g., $c = 0.01$ and $\mu_2 = 100$, this boundedness requirement is not restrictive and allows the noise variances to differ substantially across the rows and columns of the data. 
We note that the global scaling  of the noise $S_{ij} \propto n^{-1}$ is convenient for analysis but is not required in practice. In particular, we can enforce it automatically by our choice of $\eta$ in step~\ref{alg: step 2} of Algorithm~\ref{alg:noise standardization}; see the discussion in Appendix~\ref{appendix: adapting to unknown global scaling of the noise}.

Our second assumption is that $m$ (the short dimension) grows sufficiently quickly with respect to $n$.
\begin{assump} \label{assump: growth of m}
There exist global constants $\widetilde{C}_2,\delta_2 >0$ such that $m \geq \widetilde{C}_2 \max \{n^{1/2 + \delta_2}, n^{1-\delta_1 + \delta_2}\}$.
\end{assump}
Assumption~\ref{assump: growth of m} requires that $m$ grows with a rate slightly faster than $\sqrt{n}$ and $n^{1-\delta_1}$, where the latter depends on the delocalization level of the signal's singular vectors according to Assumption~\ref{assump: signal delocalization}. Note that Assumption~\ref{assump: growth of m} is always satisfied if $m$ grows proportionally to $n$, i.e., if $\widetilde{C}_2 n \leq m \leq n$, since we can take any $\delta_2 < \min\{1/2,\delta_1\}$. Therefore, Assumption~\ref{assump: growth of m} would immediately hold for Examples~\ref{example: signal delocalization example 1} and~\ref{example: signal delocalization example 3}. In the case of Example~\ref{example: signal delocalization example 2}, the short dimension $m$ is proportional to $n^{3/4}$ and Assumption~\ref{assump: signal delocalization} holds with $\delta_1$ close to $1/2$, hence Assumption~\ref{assump: growth of m} holds with $\delta_2$ close to $1/4$, e.g., $\delta_2 = 0.249$.

We now have the main result of this section, which provides probabilistic bounds on the relative entrywise deviations between the estimated scaling factors $(\hat{\mathbf{x}},\hat{\mathbf{y}})$ (see~\eqref{eq: x_hat and y_hat def} in Section~\ref{sec: method derivation for rank-one variance matrices} or step~\ref{alg:step 4} in Algorithm~\ref{alg:noise standardization}) and the true factors $(\mathbf{x},\mathbf{y})$ from Assumption~\ref{assump: rank one variance matrix}, respectively. 
\begin{thm} \label{thm: convergence of estimated variance factors}
Under Assumptions~\ref{assump: rank one variance matrix}--\ref{assump: variance boundedness}, for any $\epsilon \in (0,1/2-\delta_0)$ there exist $C^{'},c^{'}(t)>0$, such that for all $t>0$, with probability at least $1-c^{'}(t) n^{-t}$ we have that
\begin{equation}
    \left\Vert \frac{\hat{\mathbf{x}} - \mathbf{x}}{\mathbf{x}} \right\Vert_\infty \leq C^{'} \max\left\{n^{\epsilon-1/2}, n^{-\delta_1} \right\}. \label{eq: x_hat - x relative error}
\end{equation}
If additionally Assumption~\ref{assump: growth of m} holds and $\epsilon<\delta_2$, then with probability at least $1-c^{'}(t) n^{-t}$ we have that
\begin{equation}
    \left\Vert \frac{\hat{\mathbf{y}} - \mathbf{y}}{\mathbf{y}} \right\Vert_\infty \leq C^{'} \max\left\{\frac{n^{\epsilon+1/2}}{m}, \frac{n^{1-\delta_1}}{m} \right\}. \label{eq: y_hat - y relative error}
\end{equation}
\end{thm}
According to Theorem~\ref{thm: convergence of estimated variance factors}, under Assumptions~\ref{assump: rank one variance matrix}--\ref{assump: variance boundedness}, the relative entrywise error between $\hat{\mathbf{x}}$ and $\mathbf{x}$ converges to zero as $n\rightarrow\infty$ with probability approaching $1$ rapidly, regardless of the growth rate of $m$ (the short dimension). The corresponding rate is upper bounded by $n^{-1/2}$ but can be slower -- depending on the parameter $\delta_1$, which is determined by the delocalization of the singular vector matrices $\widetilde{U}$ and $\widetilde{V}$ across the rows according to Assumption~\ref{assump: signal delocalization}. Under the additional Assumption~\ref{assump: growth of m} on the growth rate of $m$, Theorem~\ref{thm: convergence of estimated variance factors} also guarantees that the relative entrywise error between $\hat{\mathbf{y}}$ and $\mathbf{y}$ converges to zero as $m,n\rightarrow \infty$, with probability approaching $1$ rapidly. In this case, the convergence rate depends explicitly on the growth rate of $m$ in addition to $\delta_1$ from Assumption~\ref{assump: signal delocalization}. Importantly, all error bounds are completely oblivious to the singular values of the signal $X$, which can be very large, e.g., growing with $n$, or very small, e.g., decaying with $n$.

The proof Theorem~\ref{thm: convergence of estimated variance factors} can be found in Appendix~\ref{appendix: proof of convergence of estimated variance factors}. It relies on analyzing the relative errors appearing in~\eqref{eq: x_hat - x relative error} and~\eqref{eq: y_hat - y relative error} using a combination of Proposition~\ref{prop: formula for x and y in terms of g}, the estimator formulas~\eqref{eq: x_hat and y_hat def}, and Corollary~\ref{cor: convergence of g_tilde to g in l_infty}. The reason that we consider relative errors in Theorem~\ref{thm: convergence of estimated variance factors} and not absolute errors (i.e., $\Vert \hat{\mathbf{x}} - \mathbf{x}\Vert_\infty$ and $\Vert \hat{\mathbf{y}} - \mathbf{y}\Vert_\infty$) is that the entries of $\mathbf{x}$ and $\mathbf{y}$ tend to zero as $m,n\rightarrow \infty$ under our assumptions; see Lemma~\ref{lem: boundedness of x and y} in Appendix~\ref{appendix: proof of boundedness of x and y}. Hence, relative errors are much more informative in this case than absolute errors. This situation differs from our previous analysis of the deviation between $\hat{\mathbf{g}}$ and $\mathbf{g}$, since the entries of $\mathbf{g}$ are always lower bounded by a positive global constant; see Lemma~\ref{lem: boundedness of g} in Appendix~\ref{appendix: boundedness of g}. 

In Figure~\ref{fig: convergense of estimated scaling factors, rank-one case}, we depict the relative errors from the left-hand sides of~\eqref{eq: x_hat - x relative error} and~\eqref{eq: y_hat - y relative error} as functions of the long dimension $n$ in four simulated scenarios. The results were averaged over $20$ randomized experiments for each scenario. In all scenarios, the signal $X$ has rank $r=10$, where all nonzero singular values are equal to $s$. The noise is Gaussian heteroskedastic with variance matrix $S=\mathbf{x}\mathbf{y}^T$, where the entries of $\mathbf{x}$ and $\mathbf{y}$ are sampled independently and uniformly at random from $[1,10]$. Then, $S$ is normalized by a scalar so that its average entry is $1$. In the first scenario (panel (a)), the dimensions are growing proportionally according to $m = \lceil n/2\rceil$, the magnitudes of the signal's components are fixed at $s^2/n=10$ (i.e., the nonzero eigenvalues of $X X^T/n$ are equal to $10$), and the singular vectors of $X$ are orthonormalized independent Gaussian random vectors. Hence, Assumption~\ref{assump: signal delocalization} holds with high probability for $\delta_1 \sim 1$ for large $n$. The second scenario (panel (b)) is identical to the first scenario except that the signal components are growing according to $s^2/n=10 n$. In this case, the total magnitude of the signal $\Vert X \Vert_F^2 = 100 n^2$ is $200$ times larger than that of the noise $\mathbb{E} \Vert E \Vert_F^2 = mn$. The third scenario is identical to the first scenario except that the dimensions are growing disproportionally, where $m = \lceil 3 n^{3/4} \rceil$. In this case, Assumption~\ref{assump: signal delocalization} holds with high probability for $\delta_1 \sim 3/4$ for sufficiently large $n$. Lastly, the fourth scenario (panel (d)) is identical to the first scenario except that the singular vectors of $X$ are sparse with a diminishing proportion of nonzero entries. Specifically, the nonzero entries of the left and right singular vectors of $X$ are restricted to $\lceil 5{m}^{2/3} \rceil$ and $\lceil 5{n}^{2/3} \rceil$ entries, respectively, implying that Assumption~\ref{assump: signal delocalization} holds with high probability for $\delta_1 \sim 2/3$ for sufficiently large $n$. 

For the first, second, and fourth scenarios, Theorem~\ref{thm: convergence of estimated variance factors} asserts that the convergence rates of the estimation errors of $\mathbf{x}$ and $\mathbf{y}$ are bounded by a rate that is slightly slower than $n^{-1/2}$. For the third scenario, the corresponding rates are $n^{-1/2}$ for the estimation error of $\mathbf{x}$ and $n^{-1/4}$ for the estimation error of $\mathbf{y}$. We see from Figure~\ref{fig: convergense of estimated scaling factors, rank-one case} that in all cases, the empirical estimation errors converge to zero and the rates conform to the bounds in Theorem~\ref{thm: convergence of estimated variance factors}.

\begin{figure} 
  \centering
  	{
  	\subfloat[][$m = \lceil n/2 \rceil$, $s^2/n = 10$, $\delta_1\sim 1$]
  	{
    \includegraphics[width=0.38\textwidth]{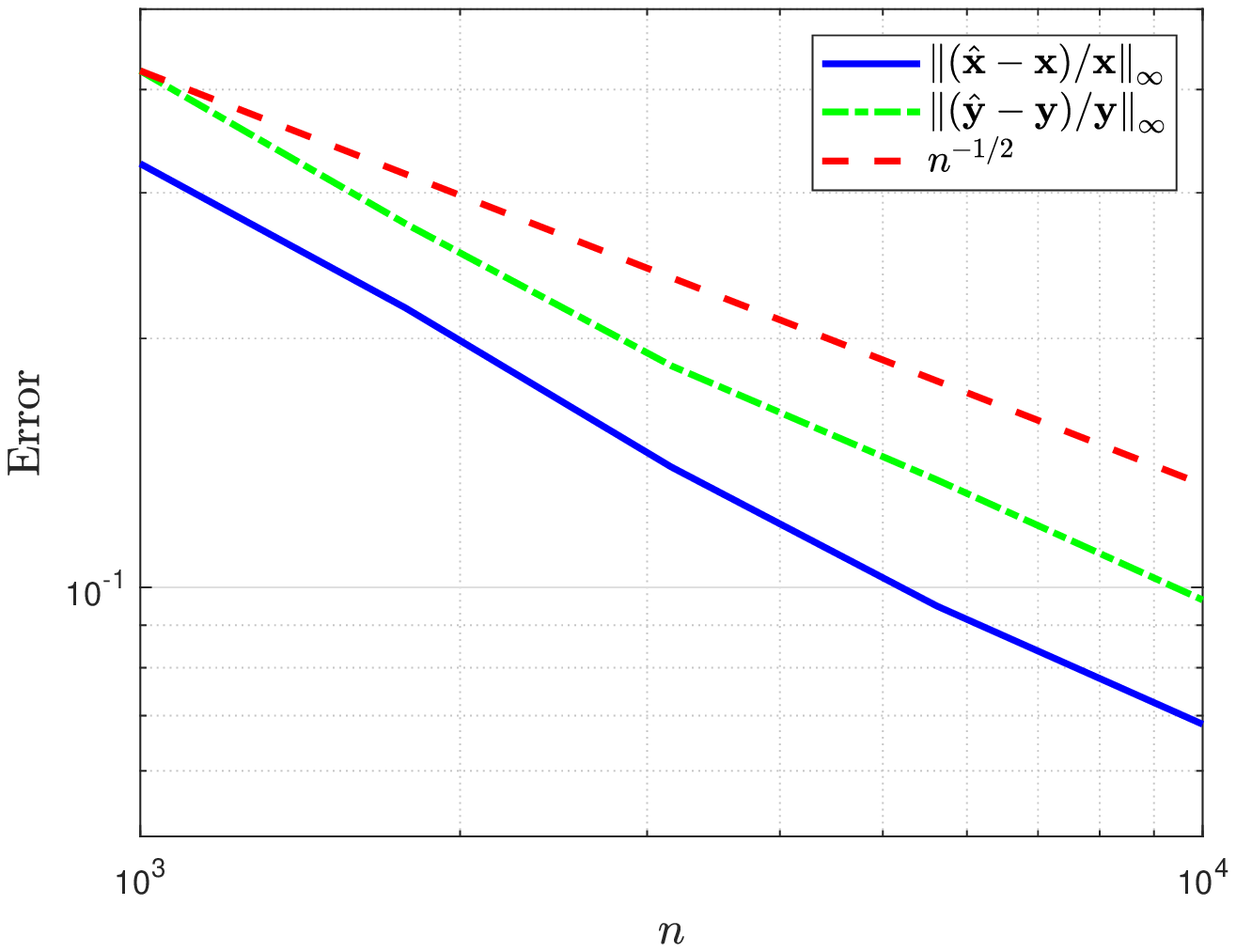} 
    }
    \subfloat[][$m = \lceil n/2 \rceil$, $s^2/n = 10 n$, $\delta_1\sim 1$] 
  	{
    \includegraphics[width=0.38\textwidth]{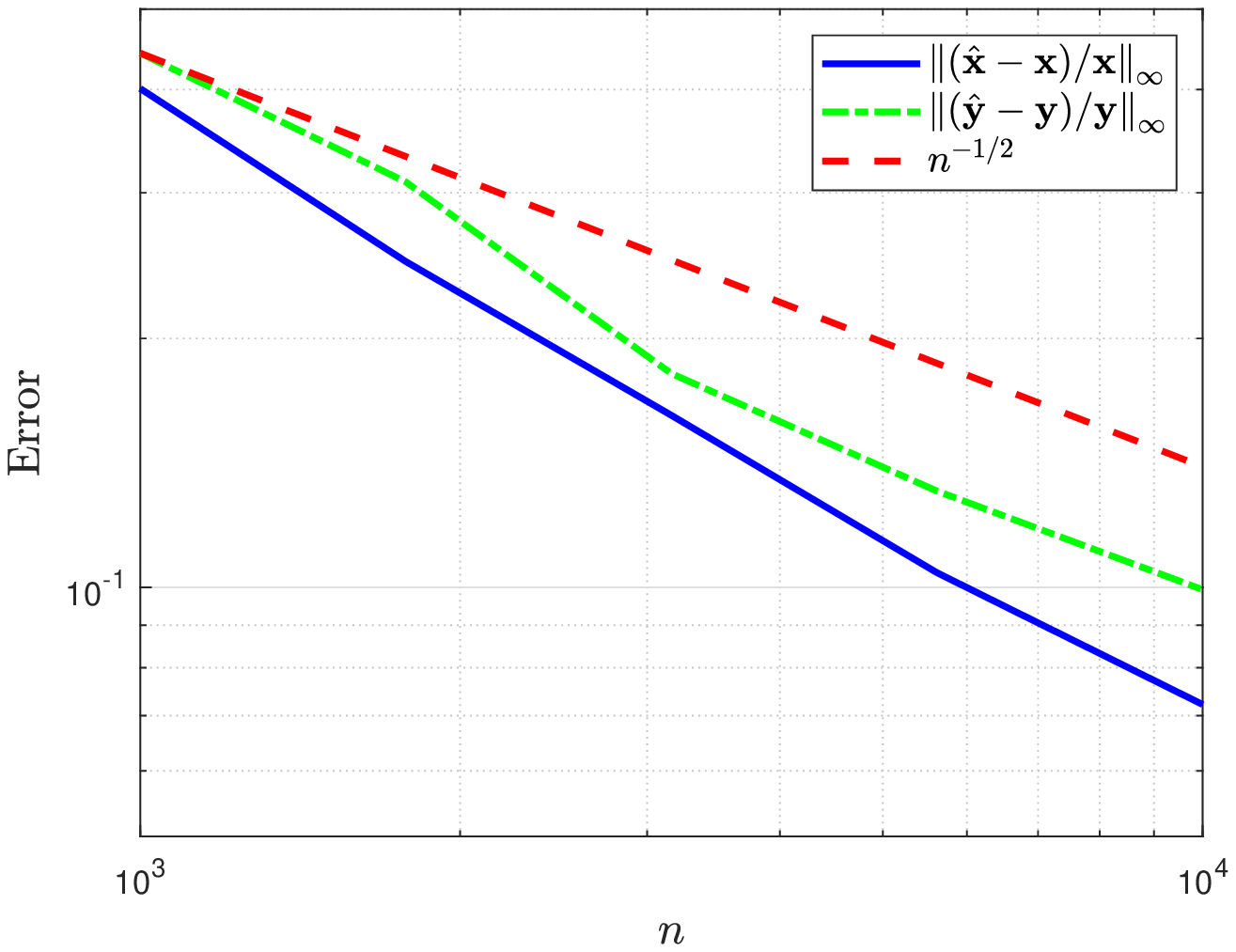} 
    }
    \\
    \subfloat[][$m = \lceil 3 n^{3/4} \rceil$, $s^2/n = 10$, $\delta_1\sim 3/4$]  
  	{
    \includegraphics[width=0.38\textwidth]{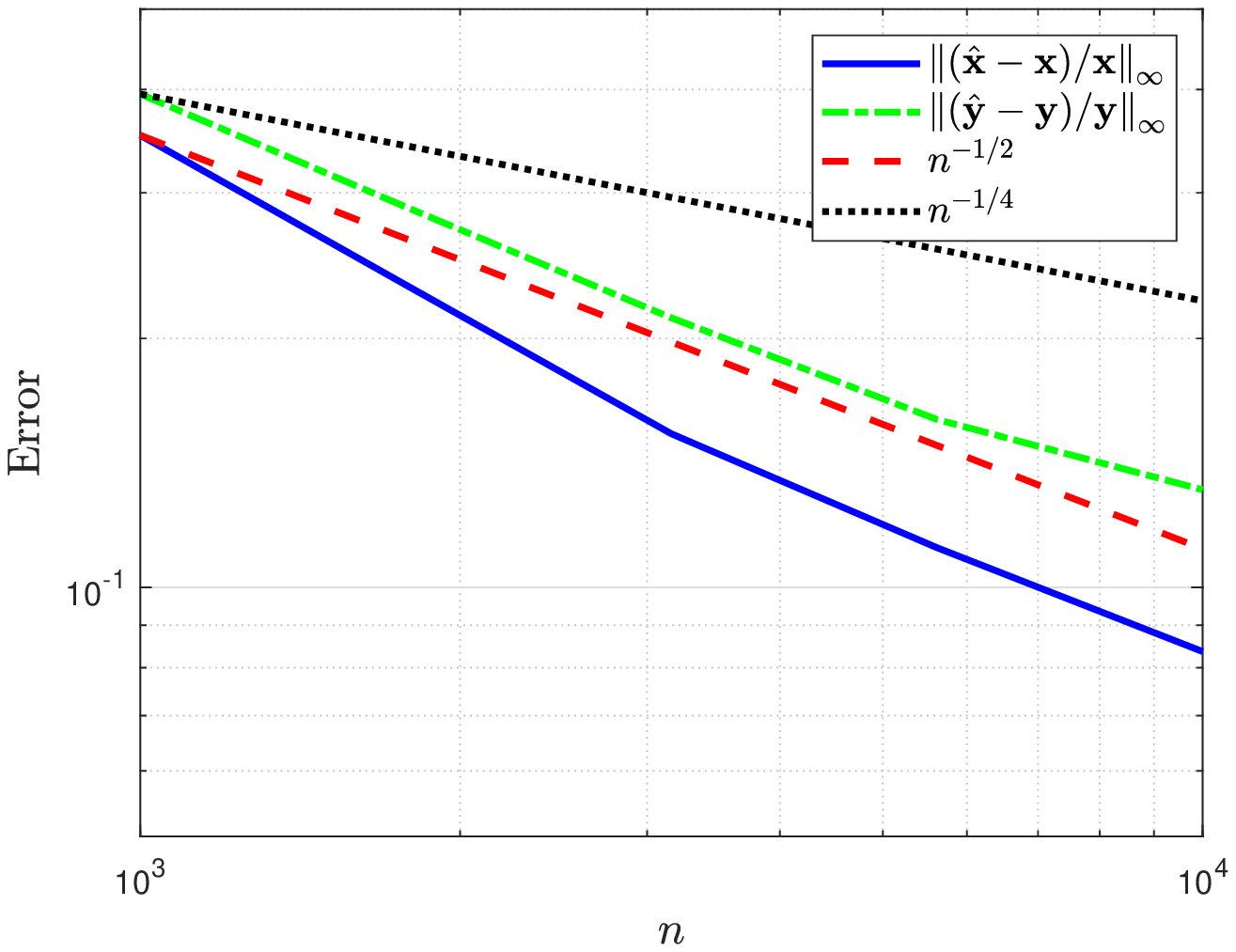}  
    }
    \subfloat[][$m = \lceil n/2 \rceil$, $s^2/n = 10$, $\delta_1\sim 2/3$] 
  	{
    \includegraphics[width=0.38\textwidth]{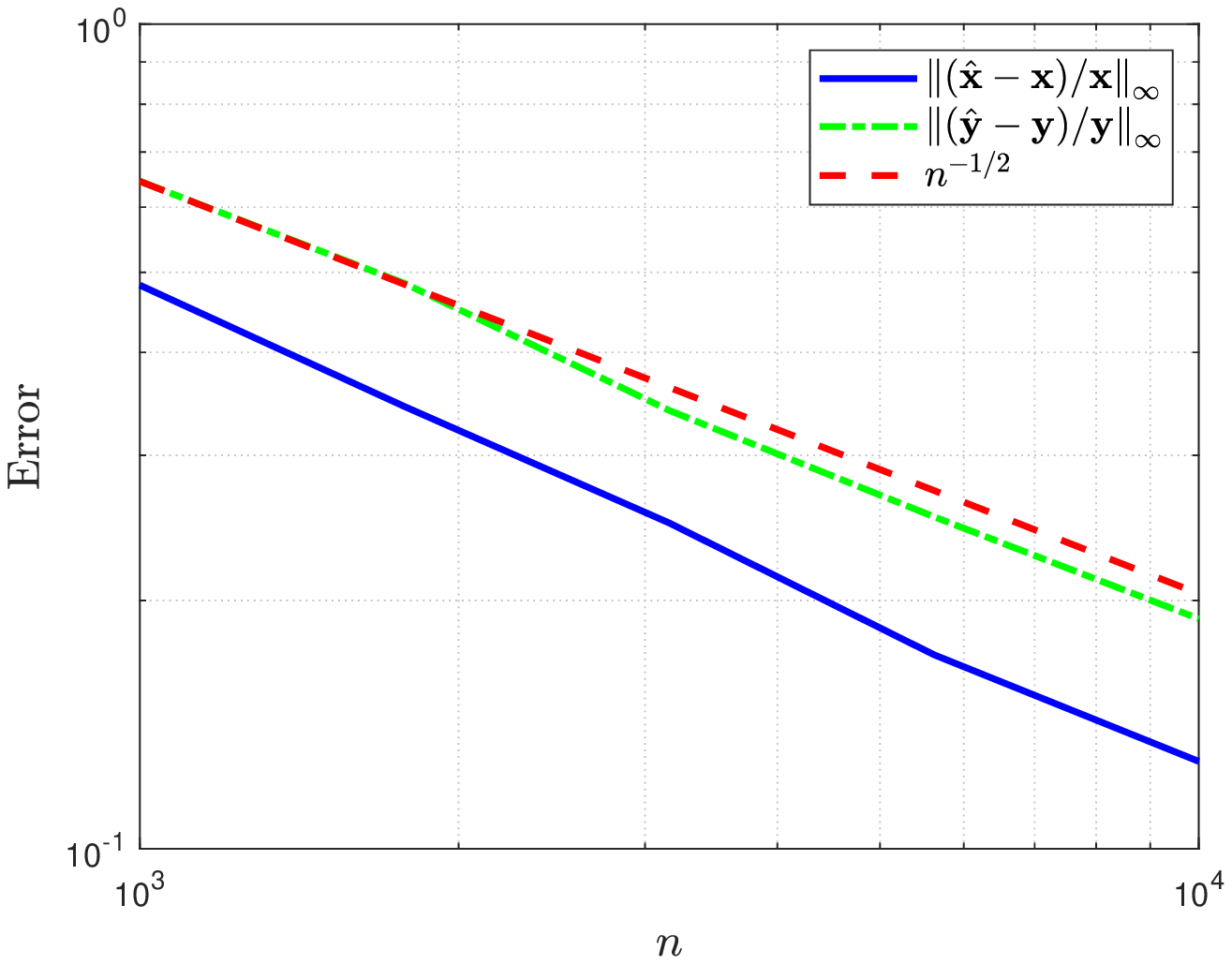} 
    } 

    }
    \caption
    {Maximal relative errors between $\hat{\mathbf{x}}$ and $\mathbf{x}$ and between $\hat{\mathbf{y}}$ and $\mathbf{y}$, as functions of $n$ in four different scenarios, where $S = \mathbf{x} \mathbf{y}^T$. 
    } \label{fig: convergense of estimated scaling factors, rank-one case}
    \end{figure}

\subsection{Variance matrices $S$ with arbitrary rank} \label{sec: variance matrices with general rank}
In the previous section, we established that if $S = \mathbf{x}\mathbf{y}^T$, then $\mathbf{x}$ and $\mathbf{y}$ can be estimated accurately from $Y$ if its dimensions are sufficiently large and under mild delocalization conditions on the singular vectors of the signal $X$, even if the signal's singular values are arbitrarily large.
In this section, we alleviate the rank-one assumption on $S$ and consider more general variance matrices.

We begin with the following definition of a doubly regular matrix.
\begin{defn} \label{def: doubly regular matrix}
A matrix $A\in\mathbb{R}^{m\times n}$ is called \textit{doubly regular} if its average entry in each row and in each column is precisely one, i.e., 
\begin{equation}
    \frac{1}{n} A \mathbf{1}_n = \mathbf{1}_m, \qquad \frac{1}{m} A^T \mathbf{1}_m = \mathbf{1}_n. 
\end{equation}
\end{defn}

A fundamental fact crucial to our approach is that any positive matrix can be made doubly regular by appropriate scaling of its rows and columns. Specifically, since the variance matrix $S$ is positive (see~\eqref{eq: signal-plus-noise} and the subsequent text), we have the following proposition, which is an immediate consequence of~\cite{sinkhorn1967diagonal}. 
\begin{prop} \label{prop: matrix scaling}
There exist positive vectors $\mathbf{x}\in\mathbb{R}^m$ and $\mathbf{y}\in\mathbb{R}^{n}$ such that
\begin{equation}
    S = D\{\mathbf{x}\} \widetilde{S} D\{\mathbf{y}\}, \label{eq: scaling factors def}
\end{equation}
where $\widetilde{S}\in\mathbb{R}^{m\times n}$ is a positive and doubly regular matrix. Moreover, the pair $(\mathbf{x},\mathbf{y})$ is unique up to the trivial scalar ambiguity, i.e., it can only be replaced with $(\alpha \mathbf{x}, \alpha^{-1} \mathbf{y})$ for any $\alpha>0$.
\end{prop}
We shall refer to $\mathbf{x}$ and $\mathbf{y}$ as the \textit{scaling factors} of $S$; see~\cite{idel2016review} and the references therein for an extensive review of the topic of matrix scaling.
Note that the case of $S = \mathbf{x} \mathbf{y}^T$, which was investigated in Section~\ref{sec: rank one case}, is a special case of~\eqref{eq: scaling factors def} where $\widetilde{S}$ is a matrix of ones, i.e., $\widetilde{S} = \mathbf{1}_{m\times n}$. In this case, the normalization~\eqref{eq: scaling rows and columns} makes all of the noise variances in the matrix identical. More generally, we cannot make all noise variances identical but we can stabilize the noise variances simultaneously across rows and columns by making the noise variance matrix $(\mathbb{E}[\widetilde{E}_{ij}^2]) = \widetilde{S}$ doubly regular. Hence, after the normalization~\eqref{eq: scaling rows and columns}, the average noise variance in each row and in each column is precisely $1$. As we shall see in Section~\ref{sec: application to signal detection and recovery}, this normalization is particularly advantageous for signal detection and recovery under general heteroskedastic noise.

We now provide a high-level overview of the results and rationale of our analysis in the remainder of this section. Our goal is to investigate which structures of $S$ allow us to accurately estimate the scaling factors $\mathbf{x}$ and $\mathbf{y}$ of $S$ using the previously-derived estimators $\hat{\mathbf{x}}$ and $\hat{\mathbf{y}}$. To this end, we consider the solution to a surrogate Dyson equation obtained by replacing the true variance matrix $S$ with the rank-one variance matrix $\mathbf{x}\mathbf{y}^T$ (where $\mathbf{x}$ and $\mathbf{y}$ are the scaling factors of $S$). We denote this surrogate solution as $\mathbf{h}\in\mathbb{R}^{m+n}$. Our main argument is that the estimators $\hat{\mathbf{x}}$ and $\hat{\mathbf{y}}$ should be accurate as long the surrogate solution $\mathbf{h}$ is sufficiently close to the true solution $\mathbf{g}$ (to the Dyson equation with the variance matrix $S$), regardless of the actual rank of $S$. To identify such cases, we characterize the error $\Vert \mathbf{g} - \mathbf{h}\Vert_\infty$ in terms of the doubly regular matrix $\widetilde{S}$ and certain quantities that can be directly related to the scaling factors $\mathbf{x}$ and $\mathbf{y}$; see Lemma~\ref{lem: g is close to h under incoherence} and Proposition~\ref{prop: relation between x,y and w_1,w_2}. These results show that besides the case of a rank-one variance matrix $S$, the vector $\mathbf{h}$ is close to $\mathbf{g}$ in other cases, such as when $\mathbf{x}$ and $\mathbf{y}$ are close to scalar multiples of all-ones vectors (irrespective of $\widetilde{S}$), and more generally when $\mathbf{x}$ and $\mathbf{y}$ are sufficiently incoherent with respect to $\widetilde{S}$. We discuss such cases in detail and provide several examples using random priors imposed on $\widetilde{S}$; see Examples~\ref{example: random S example 1}--\ref{example: random S example 3}. Finally, under an assumption on the decay of the error $\Vert \mathbf{g} - \mathbf{h}\Vert_\infty$ as $n$ grows (see Assumption~\ref{assump: decay rate of g-h}) and an adjusted assumption on the growth rate of the short dimension $m$ (see Assumption~\ref{assump: growth of m for general S}, which replaces Assumption~\ref{assump: growth of m}), we extend Theorem~\ref{thm: convergence of estimated variance factors} to support more general variance matrices $S$ beyond rank-one; see Theorem~\ref{thm: convergence of estimated scaling fators for generl S}. In particular, Theorem~\ref{thm: convergence of estimated scaling fators for generl S} establishes the stochastic convergence of the relative errors $\Vert(\hat{\mathbf{x}} - \mathbf{x})/\mathbf{x}\Vert_\infty\rightarrow 0$ and $\Vert (\hat{\mathbf{y}} - \mathbf{y})/\mathbf{y}\Vert_\infty\rightarrow 0$ as $m,n\rightarrow \infty$ with rates, for variance matrices $S$ with general rank, including full-rank, as long as $S$ implicitly satisfies Assumption~\ref{assump: decay rate of g-h}.

We begin with several definitions. Let $\mathbf{h} \in\mathbb{R}^{(m+n)}$ be the positive vector that solves~\eqref{eq: Dyson eq imag axis} when replacing $S$ with $\mathbf{x}\mathbf{y}^T$. In particular, $\mathbf{h}^{(1)} = [\mathbf{h}_1,\ldots,\mathbf{h}_m]^T$ and $\mathbf{h}^{(2)} = [\mathbf{h}_{m+1},\ldots,\mathbf{h}_{m+n}]^T$ satisfy
\begin{equation}
    \eta + \mathbf{x} \mathbf{y}^T\mathbf{h}^{(2)} = \frac{1}{\mathbf{h}^{(1)}}, \qquad \qquad \eta + \mathbf{y}\mathbf{x}^T\mathbf{h}^{(1)} = \frac{1}{\mathbf{h}^{(2)}}, \label{eq: coupled Dyson eq imag axis rank one}
\end{equation}
where $\mathbf{x}$ and $\mathbf{y}$ are the scaling factors of $S$ from~\eqref{eq: scaling factors def}.
Note that $\mathbf{h}$ depends only on $\mathbf{x}$ and $\mathbf{y}$, whereas $\mathbf{g}$ from~\eqref{eq: Dyson eq imag axis} may additionally depend on $\widetilde{S}$. Next, analogously to~\eqref{eq: setting alpha=1}, we settle the scalar ambiguity in the definition of $\mathbf{x}$ and $\mathbf{y}$ by requiring (without loss of generality) that
\begin{equation}
    \mathbf{x}^T \mathbf{h}^{(1)} = \mathbf{y}^T \mathbf{h}^{(2)}. \label{eq: setting alpha=1 for general variance matrices}
\end{equation}
Note that $\mathbf{g} = \mathbf{h}$ if $S=\mathbf{x} \mathbf{y}^T$. Therefore, Proposition~\ref{prop: formula for x and y in terms of g} implies that 
\begin{equation}
    \mathbf{x} = \frac{1}{\sqrt{m - \eta \Vert \mathbf{h}^{(1)}\Vert_1 }} \left( \frac{1}{\mathbf{h}^{(1)}} - \eta \right), 
    \qquad
    \mathbf{y} =  \frac{1}{\sqrt{n - \eta \Vert \mathbf{h}^{(2)} \Vert_1}} \left( \frac{1}{\mathbf{h}^{(2)}} - \eta \right).
 \label{eq: formulas for x and y in terms of h}
\end{equation}
 In the previous section, we established that $\mathbf{g}$ can be estimated accurately from the data matrix $Y$ using $\hat{\mathbf{g}}$ from~\eqref{eq: g_hat def}. Importantly, the estimation accuracy of $\mathbf{g}$ does not depend on the structure of the noise variance matrix $S$ (see Theorem~\ref{thm: robustness of the resolvent} and Corollaries~\ref{cor: convergence of g_tilde to g in l_1} and~\ref{cor: convergence of g_tilde to g in l_infty}, which do not rely on Assumption~\ref{assump: rank one variance matrix}).  Hence, in cases where $\mathbf{g}$ is close to $\mathbf{h}$, it is natural to employ the formulas in~\eqref{eq: x_hat and y_hat def} to estimate $\mathbf{x}$ and $\mathbf{y}$. Consequently, our primary focus in this section is to bound the discrepancy $\Vert \mathbf{g} - \mathbf{h} \Vert_\infty$ and describe how this discrepancy controls the  accuracy of recovering $\mathbf{x}$ and $\mathbf{y}$. As we shall see, even though the estimators $\hat{\mathbf{x}}$ and $\hat{\mathbf{y}}$ in~\eqref{eq: x_hat and y_hat def} were derived in the case where ${S}$ is of rank one, they can be used to accurately estimate the scaling factors $\mathbf{x}$ and $\mathbf{y}$ of $S$ in many other cases.

Let us define the positive vectors $\mathbf{w}^{(1)} \in \mathbb{R}^{m}$ and $\mathbf{w}^{(2)} \in \mathbb{R}^{n}$ according to
\begin{equation}
     \mathbf{w}^{(1)} = D\{\mathbf{x} \} \mathbf{h}^{(1)} = \frac{\mathbf{x}}{\eta + a\mathbf{x}}, \qquad \mathbf{w}^{(2)} = D\{\mathbf{y} \} \mathbf{h}^{(2)} = \frac{\mathbf{y}}{\eta + a\mathbf{y} },
    \label{eq: W_1 and W_2 def}
\end{equation}
where $a = \mathbf{x}^T \mathbf{h}^{(1)} = \mathbf{y}^T \mathbf{h}^{(2)}$; see~\eqref{eq: coupled Dyson eq imag axis rank one} and~\eqref{eq: setting alpha=1 for general variance matrices}.
Note that $\mathbf{w}^{(1)}$ and $\mathbf{w}^{(2)}$ depend only on $\mathbf{x}$ and $\mathbf{y}$ and not on $\widetilde{S}$. To further clarify the correspondence between the pairs $(\mathbf{x},\mathbf{y})$ and $(\mathbf{w}^{(1)},\mathbf{w}^{(2)})$, we have the following proposition, whose proof can be found in Appendix~\ref{appendix: proof of relation between x,y and w_1,w_2}.
\begin{prop} \label{prop: relation between x,y and w_1,w_2}
The sign of $\mathbf{w}_i^{(1)} - \mathbf{w}_j^{(1)}$ ($\mathbf{w}_i^{(2)} - \mathbf{w}_j^{(2)}$) is identical to the sign of $\mathbf{x}_i - \mathbf{x}_j$ ($\mathbf{y}_i - \mathbf{y}_j$), and
\begin{equation}
    \left\vert \frac{\mathbf{w}_i^{(1)} - \mathbf{w}_j^{(1)}}{\mathbf{w}_j^{(1)}} \right\vert \leq \left\vert \frac{\mathbf{x}_i - \mathbf{x}_j}{\mathbf{x}_j} \right\vert, \qquad 
    \left\vert \frac{\mathbf{w}_i^{(2)} - \mathbf{w}_j^{(2)}}{\mathbf{w}_j^{(2)}} \right\vert \leq \left\vert \frac{\mathbf{y}_i - \mathbf{y}_j}{\mathbf{y}_j} \right\vert,
\end{equation}
for all $i,j\in[m]$ for the first inequality and $i,j\in[n]$ for the second inequality. Moreover, equality holds in the two inequalities above only if $x_i = x_j$ and $y_i = y_j$, respectively.
\end{prop}
Proposition~\ref{prop: relation between x,y and w_1,w_2} shows that the entries of $\mathbf{w}^{(1)}$ ($\mathbf{w}^{(2)}$) preserve the same ranking as the entries of $\mathbf{x}$ ($\mathbf{y}$). Moreover, the pairwise relative differences between the entries of $\mathbf{w}^{(1)}$ ($\mathbf{w}^{(2)}$) are always smaller or equal to the corresponding relative pairwise differences between the entries of $\mathbf{x}$ ($\mathbf{y}$). In this sense, the entries of $\mathbf{w}^{(1)}$ and $\mathbf{w}^{(2)}$ are always closer to being constant than the entries of $\mathbf{x}$ and $\mathbf{y}$, respectively.

We now have the following lemma, which bounds the discrepancy between $\mathbf{g}$ and $\mathbf{h}$ in terms of $\widetilde{S}$ and the vectors $\mathbf{w}^{(1)}$ and $\mathbf{w}^{(2)}$, where $\langle \cdot \rangle$ denotes the average entry in a vector.  
\begin{lem} \label{lem: g is close to h under incoherence}
Under Assumptions~\ref{assump: noise moment bound} and~\ref{assump: variance boundedness}, there exists a constant $C^{'}>0$ such that
\begin{equation}
    \Vert \mathbf{g} - \mathbf{h} \Vert_\infty \leq C^{'} \max\left\{ \frac{1}{\sqrt{n}}\left\Vert \left( \widetilde{S} -1\right)  \frac{\mathbf{w}^{(2)} - \langle \mathbf{w}^{(2)} \rangle}{\Vert \mathbf{w}^{(2)} \Vert_2} \right\Vert_\infty, \frac{\sqrt{m}}{n}\left\Vert \left( \widetilde{S} -1\right)^T \frac{\mathbf{w}^{(1)} - \langle \mathbf{w}^{(1)} \rangle}{\Vert \mathbf{w}^{(1)} \Vert_2} \right\Vert_\infty 
     \right\}. \label{eq: g-h upper bound}
\end{equation}
\end{lem}
The proof of Lemma~\ref{lem: g is close to h under incoherence} is found in Appendix~\ref{appendix: proof of g is close to h under incoherence} and relies on a stability analysis of the Dyson equation on the imaginary axis, which may be of independent interest; see Lemma~\ref{lem: stability of Dyson equation} in Appendix~\ref{appendix: stability of Dyson equation}. First, Lemma~\ref{lem: g is close to h under incoherence} shows that the error $\Vert \mathbf{g} - \mathbf{h} \Vert_\infty$ is small if $\widetilde{S}$ is sufficiently close to the matrix of all ones $\mathbf{1}_{m\times n}$, or alternatively, if $\mathbf{w}^{(1)}$ and $\mathbf{w}^{(2)}$ are sufficiently close to being constant vectors (i.e., vectors whose entries are nearly identical). In particular, if $\mathbf{x}$ and $\mathbf{y}$ are multiples of the all-ones vectors $\mathbf{1}_m$ and $\mathbf{1}_n$, respectively, then according to Proposition~\ref{prop: relation between x,y and w_1,w_2}, $\mathbf{w}^{(1)}$ and $\mathbf{w}^{(1)}$ are also multiples of the all-ones vectors $\mathbf{1}_m$ and $\mathbf{1}_n$, in which case we have $\mathbf{g} = \mathbf{h}$ by Lemma~\ref{lem: g is close to h under incoherence}. In other words, we have $\mathbf{g} = \mathbf{h}$ whenever $S$ is a scalar multiple of a doubly regular matrix, regardless of its rank. Second, Lemma~\ref{lem: g is close to h under incoherence} shows that the error $\Vert \mathbf{g} - \mathbf{h} \Vert_\infty$ is small 
if $\mathbf{w}^{(1)}$ and $\mathbf{w}^{(2)}$ are sufficiently incoherent with respect to $(\widetilde{S}-1)$. Below we provide two examples demonstrating this scenario using random priors imposed on $\widetilde{S}$.
\begin{example} \label{example: random S example 1}
Consider a matrix $\widetilde{S}$ given by
\begin{equation}
    \widetilde{S}_{ij} = 1 + z_{ij}  - \frac{1}{m}\sum_{i=1}^m z_{ij} - \frac{1}{n}\sum_{j=1}^n z_{ij} + \frac{1}{mn}\sum_{i=1}^m \sum_{j=1}^n z_{ij}, \label{eq: example for S_tilde}
\end{equation}
where $\{z_{ij}\} \in [b_1-1,b_2-1]$ are independent random variables with zero means and $b_2>1>b_1>0$ are fixed constants. By construction, $\widetilde{S}$ is always doubly regular and is also positive with high probability for all sufficiently large $m$ and $n$. Importantly, if $\mathbf{x}$ and $\mathbf{y}$ are either deterministic or random but generated independently of $\{z_{ij}\}$, then by standard concentration arguments~\cite{hoeffding1963probability} together with the union bound, we have
\begin{align}
    \frac{1}{\sqrt{n}}\left\Vert \left( \widetilde{S} -1\right)  \frac{\mathbf{w}^{(2)} - \langle \mathbf{w}^{(2)} \rangle}{\Vert \mathbf{w}^{(2)} \Vert_2} \right\Vert_\infty 
    &\leq C^{''} \frac{ \left\Vert \mathbf{w}^{(2)} - \langle \mathbf{w}^{(2)} \rangle \right\Vert_2}{\Vert \mathbf{w}^{(2)} \Vert_2} \sqrt{\frac{{\log n}}{{n}}} \leq C^{''} \sqrt{\frac{{\log n}}{{n}}} , \label{eq: example 1 random S_tilde bound 1}\\ 
    \frac{\sqrt{m}}{n}\left\Vert \left( \widetilde{S} -1\right)^T \frac{\mathbf{w}^{(1)} - \langle \mathbf{w}^{(1)} \rangle}{\Vert \mathbf{w}^{(1)} \Vert_2} \right\Vert_\infty 
    &\leq C^{''} \frac{ \left\Vert \mathbf{w}^{(1)} - \langle \mathbf{w}^{(1)} \rangle \right\Vert_2}{\Vert \mathbf{w}^{(1)} \Vert_2} \sqrt{\frac{\log n}{n}}  
    \leq C^{''} \sqrt{\frac{{\log n}}{{n}}}, \label{eq: example 1 random S_tilde bound 2}
\end{align}
with probability that approaches $1$ as $m,n \rightarrow \infty$, where $C^{''}>0$ is a global constant. Here, we also used the fact that $m\leq n$, $\left\Vert \mathbf{w}^{(2)} - \langle \mathbf{w}^{(2)} \rangle \right\Vert_2 \leq \left\Vert \mathbf{w}^{(2)} \right\Vert_2$, and $\left\Vert \mathbf{w}^{(1)} - \langle \mathbf{w}^{(1)} \rangle \right\Vert_2 \leq \left\Vert \mathbf{w}^{(1)} \right\Vert_2$.
We see that in this case, the upper bound in~\eqref{eq: g-h upper bound} is of the order of $\sqrt{\log n / n}$ with high probability, regardless of $\mathbf{w}^{(1)}$ and $\mathbf{w}^{(2)}$. If, in addition, the entries of $\mathbf{w}^{(1)}$ and $\mathbf{w}^{(2)}$ are approximately constant, then the quantities ${ \left\Vert \mathbf{w}^{(2)} - \langle \mathbf{w}^{(2)} \rangle \right\Vert_2}/{\Vert \mathbf{w}^{(2)} \Vert_2}$ and ${ \left\Vert \mathbf{w}^{(1)} - \langle \mathbf{w}^{(1)} \rangle \right\Vert_2}/{\Vert \mathbf{w}^{(1)} \Vert_2}$ will be small and the upper bound in~\eqref{eq: g-h upper bound} will further improve. According to Proposition~\ref{prop: relation between x,y and w_1,w_2}, the entries of $\mathbf{w}^{(1)}$ and $\mathbf{w}^{(2)}$ are approximately constant if the entries of $\mathbf{x}$ and $\mathbf{y}$ are approximately constant, respectively.
\end{example}

\begin{example} \label{example: random S example 2}
Consider a matrix $\widetilde{S}$ given by
\begin{equation}
    \widetilde{S} = 1 + \sum_{k=1}^{\rho-1} \left(\nu_k - \langle \nu_k \rangle\right)   \left(\xi_k - \langle \xi_k \rangle\right)^T,
\end{equation}
where $\{\nu_k\}_{k=1}^{\rho-1}\in\mathbb{R}^m$ and $\{\xi_k\}_{k=1}^{\rho-1}\in\mathbb{R}^n$ are random vectors whose entries are independent, have zero means, and are bounded such that $\sum_{k}(\nu_k \xi_k^T)_{ij} \in [b_1-1,b_2-1]$ for all $i\in[m]$ and $j\in[n]$ for some fixed constants $b_2>1>b_1>0$ and $\rho \in \{1,\ldots,m\}$. As in Example~\ref{example: random S example 1}, $\widetilde{S}$ is doubly regular by construction and is positive with high probability for sufficiently large $m$ and $n$. In this case, if $\mathbf{x}$ and $\mathbf{y}$ are either deterministic or random but generated independently of $\{\nu_k\}$ and $\{\xi_k\}$, then by the same concentration arguments used in Example~\ref{example: random S example 1}, the bounds~\eqref{eq: example 1 random S_tilde bound 1} and~\eqref{eq: example 1 random S_tilde bound 2} also hold here (with probability approaching $1$ as $m,n \rightarrow \infty$). 
\end{example}

The main distinction between the two examples above is that in Example~\ref{example: random S example 2}, the matrix  $\widetilde{S}$ is constructed to be low-rank and its entries are highly dependent, whereas in Example~\ref{example: random S example 1}, the variables $\{z_{ij}\}$ are independent and the resulting $\widetilde{S}$ is full-rank with probability one.
Examples~\ref{example: random S example 1} and~\ref{example: random S example 2} show that the error $\Vert \mathbf{g} - \mathbf{h}\Vert_\infty$ is small for a wide range of noise variance matrices $S = D\{\mathbf{x}\} \widetilde{S} D\{\mathbf{y}\}$ for which $\widetilde{S}$ is sufficiently generic and incoherent with respect to $\mathbf{x}$ and $\mathbf{y}$. Moreover, in these cases, the bound on the error $\Vert \mathbf{g} - \mathbf{h}\Vert_\infty$ is further reduced if $\mathbf{w}^{(1)}$ and $\mathbf{w}^{(2)}$ are close to being constant vectors, which is determined directly by the variability of the entries of $\mathbf{x}$ and $\mathbf{y}$ according to Proposition~\ref{prop: relation between x,y and w_1,w_2}. 

Lemma~\ref{lem: g is close to h under incoherence} can also be applied to variance matrices $S$ that are more structured across the rows and columns. For instance, the example below shows that  $\mathbf{g}$ is close to $\mathbf{h}$ for certain random variance matrices with a block structure, as long as the dimensions of the blocks are not too large.
\begin{example} \label{example: random S example 3}
    Suppose that $\widetilde{S}$ is a block matrix with identical entries in each block. Specifically, the rows (columns) of $\widetilde{S}$ can be partitioned into $M$ ($N$) disjoint subsets $\Omega^{\text{r}}_1,\ldots,\Omega^{\text{r}}_M$ ($\Omega^{\text{c}}_1,\ldots,\Omega^{\text{c}}_N$) such that $\widetilde{S}_{ij} = \overline{S}_{kl}$ for all $i\in \Omega^{\text{r}}_k$ and $j\in \Omega^{\text{c}}_\ell$, where $k\in[M]$ and $\ell\in [N]$. Suppose further that $\mathbf{x}$ and $\mathbf{y}$ have identical entries across the same subsets of rows and columns, where $\mathbf{x}_i = \overline{\mathbf{x}}_k$ for all $i\in \Omega^{\text{r}}_k$ and $\mathbf{y}_j = \overline{\mathbf{y}}_\ell$ for all $j\in \Omega^{\text{c}}_\ell$. Thus, $S = D\{\mathbf{x}\} \widetilde{S} D\{\mathbf{y}\}$ is a block matrix. We denote the cardinality of $\Omega^{\text{r}}_k$ and $\Omega^{\text{c}}_\ell$ as $m_k = \vert \Omega^{\text{r}}_k \vert$ and $n_\ell = \vert \Omega^{\text{c}}_\ell \vert$, respectively. In this case, Lemma~\ref{lem: g is close to h under incoherence} asserts that
    \begin{align}
    \Vert \mathbf{g} - \mathbf{h} \Vert_\infty \leq C^{'} \max\bigg\{ &\frac{1}{\sqrt{n}}\left\Vert \left( \overline{S} -1\right) D\{[n_1,\ldots,n_N]\} \frac{\overline{\mathbf{w}}^{(2)} - \langle \overline{\mathbf{w}}^{(2)} \rangle}{\Vert \overline{\mathbf{w}}^{(2)} \Vert_2} \right\Vert_\infty, \nonumber \\
    &\frac{\sqrt{m}}{n}\left\Vert \left( \overline{S} -1\right)^T D\{[m_1,\ldots,m_M]\} \frac{\overline{\mathbf{w}}^{(1)} - \langle \overline{\mathbf{w}}^{(1)} \rangle}{\Vert \overline{\mathbf{w}}^{(1)} \Vert_2} \right\Vert_\infty 
     \bigg\},
\end{align}
where $\overline{S} = (\overline{S}_{k\ell}) \in \mathbb{R}^{M\times N}$, and $\overline{\mathbf{w}}^{(1)}\in\mathbb{R}^M$ and $\overline{\mathbf{w}}^{(2)}\in\mathbb{R}^N$ are defined analogously to ${\mathbf{w}}^{(1)}$ and ${\mathbf{w}}^{(2)}$ from~\eqref{eq: W_1 and W_2 def} when replacing $\mathbf{x}$ and $\mathbf{y}$ with $\overline{\mathbf{x}} = [\overline{\mathbf{x}}_1,\ldots,\overline{\mathbf{x}}_M]^T$ and $\overline{\mathbf{y}} = [\overline{\mathbf{y}}_1, \ldots, \overline{\mathbf{y}}_N]^T$, respectively. If the entries of $\overline{S}$ are generated randomly analogously to the entries of $S$ in Example~\ref{example: random S example 1} or~\ref{example: random S example 2} (while $\overline{\mathbf{x}}$ and $\overline{\mathbf{y}}$ are deterministic or random but independent of $\overline{S}$), and if the number of blocks across the rows and columns is increasing, i.e., $M,N \rightarrow \infty$ as $m,n\rightarrow\infty$, then similarly to~\eqref{eq: example 1 random S_tilde bound 1} and~\eqref{eq: example 1 random S_tilde bound 2}, it can be verified that
\begin{align}
    \frac{1}{\sqrt{n}}\left\Vert \left( \overline{S} -1\right) D\{[n_1,\ldots,n_N]\} \frac{\overline{\mathbf{w}}^{(2)} - \langle \overline{\mathbf{w}}^{(2)} \rangle}{\Vert \overline{\mathbf{w}}^{(2)} \Vert_2} \right\Vert_\infty 
    &\leq C^{'''} \frac{\max_{\ell\in [N]}\{n_\ell\}}{\sqrt{n}} \sqrt{\log (\max\{M,N\})}, \\
    \frac{\sqrt{m}}{n}\left\Vert \left( \overline{S} -1\right)^T D\{[m_1,\ldots,m_M]\} \frac{\overline{\mathbf{w}}^{(1)} - \langle \overline{\mathbf{w}}^{(1)} \rangle}{\Vert \overline{\mathbf{w}}^{(1)} \Vert_2} \right\Vert_\infty 
    &\leq C^{'''} \frac{\max_{k\in [M]}\{m_k\}}{\sqrt{n}} \sqrt{\log (\max\{M,N\})},
\end{align}
with probability approaching $1$ as $m,n\rightarrow \infty$, where $C^{'''}>0$ is some global constant. Consequently, if the dimensions of the blocks are sufficiently small compared to $\sqrt{n}$, e.g., if they are growing with some fractional power of $n$ smaller than $1/2$, then $\Vert \mathbf{g} - \mathbf{h} \Vert_\infty$ is guaranteed to be small for sufficiently large $m$ and $n$ with high probability.
\end{example}

In cases where $\mathbf{g}$ is close to $\mathbf{h}$, we expect the formulas in~\eqref{eq: x_hat and y_hat def} to provide accurate estimates of $\mathbf{x}$ and $\mathbf{y}$. To guarantee the convergence of $\hat{\mathbf{x}}$ and $\hat{\mathbf{y}}$ from~\eqref{eq: x_hat and y_hat def} to $\mathbf{x}$ and $\mathbf{y}$, respectively, and provide corresponding rates, it is convenient to make the assumption that $\mathbf{g}$ approaches $\mathbf{h}$ at least with some fractional power of $n$. Specifically, we assume that
\begin{assump} \label{assump: decay rate of g-h}
There exist constants $\widetilde{C}_3,\delta_3>0$ such that $\Vert \mathbf{g} - \mathbf{h} \Vert_\infty \leq \widetilde{C}_3 n^{-\delta_3}$.
\end{assump}
For instance, if $m$ and $n$ are sufficiently large and the noise variance matrix $S$ is generated according to Example~\ref{example: random S example 1} or Example~\ref{example: random S example 2}, then Assumption~\ref{assump: decay rate of g-h} is satisfied with high probability using, e.g., $\delta_3 = 0.49$. If $S$ is generated according to Example~\ref{example: random S example 3} and the block dimensions grow with some fractional power of $n$ smaller than $1/2$, then Assumption~\ref{assump: decay rate of g-h} is satisfied with any $\delta_3$ that is less than $1/2$ minus that power. Aside from these examples, Assumption~\ref{assump: decay rate of g-h} allows for more general classes of variance matrices $S$ where $\delta_3$ can be an arbitrarily small positive constant. 

To state our main result in this section, which extends Theorem~\ref{thm: convergence of estimated variance factors}, we replace Assumption~\ref{assump: growth of m} from Section~\ref{sec: rank one case} with the following assumption, which similarly requires that $m$ grows sufficiently quickly with $n$. 
\begin{assump} \label{assump: growth of m for general S}
There exist constants $\widetilde{C}_4,\delta_4 >0$ such that $m \geq \widetilde{C}_4 \max \{n^{1/2 + \delta_4}, n^{1-\delta_1 + \delta_4}, n^{1-\delta_3 + \delta_4}\}$.
\end{assump}
Recall that $\delta_1 > 0$ is from Assumption~\ref{assump: signal delocalization}, which controls the delocalization level of the signal's singular vector matrices $\widetilde{U}$ and $\widetilde{V}$ across the rows. Similarly to Assumption~\ref{assump: growth of m} in the previous section, Assumption~\ref{assump: growth of m for general S} always holds if $m$ is proportional to $n$ (since one can take any $\delta_4 < \min\{\delta_1,\delta_3,0.5\}$). Now, we can extend Theorem~\ref{thm: convergence of estimated variance factors} to cover general variance matrices $S$ as long as $\mathbf{g}$ is close to $\mathbf{h}$ for large $m$ and $n$.

\begin{thm} \label{thm: convergence of estimated scaling fators for generl S}
Under Assumptions~\ref{assump: noise moment bound},\ref{assump: rank growth rate},\ref{assump: signal delocalization},\ref{assump: variance boundedness},\ref{assump: decay rate of g-h}, for any $\epsilon\in (0,1/2-\delta_0)$ there exist $C^{'},c^{'}(t)>0$ such that for all $t>0$, with probability at least $1-c^{'}(t) n^{-t}$ we have that
\begin{equation}
    \left\Vert \frac{\hat{\mathbf{x}} - \mathbf{x}}{\mathbf{x}} \right\Vert_\infty \leq C^{'} \max\left\{n^{\epsilon-1/2}, n^{-\delta_1}, n^{-\delta_3} \right\}, \label{eq: relative estimation error of x for general rank}
\end{equation}
If additionally Assumption~\ref{assump: growth of m for general S} holds and $\epsilon<\delta_4$, then with probability at least $1-c^{'}(t) n^{-t}$ we have that
\begin{equation}
    \left\Vert \frac{\hat{\mathbf{y}} - \mathbf{y}}{\mathbf{y}} \right\Vert_\infty \leq C^{'} \max\left\{\frac{n^{\epsilon+1/2}}{m}, \frac{n^{1-\delta_1}}{m}, \frac{n^{1-\delta_3}}{m} \right\}. \label{eq: relative estimation error of y for general rank}
\end{equation}
\end{thm}
The proof of Theorem~\ref{thm: convergence of estimated scaling fators for generl S} is an immediate extension of the proof of Theorem~\ref{thm: convergence of estimated variance factors} under Assumptions~\ref{assump: decay rate of g-h} and~\ref{assump: growth of m for general S}; see Appendix~\ref{appendix: proof of convergence of estimated scaling fators for generl S}. 
Theorem~\ref{thm: convergence of estimated scaling fators for generl S} is identical to Theorem~\ref{thm: convergence of estimated variance factors} except that the bounds here depend additionally on $\delta_3$, which controls the closeness of $\mathbf{g}$ and $\mathbf{h}$. Intuitively, the bound on the estimation accuracy improves if the solution $\mathbf{g}$ of the true Dyson equation~\eqref{eq: coupled Dyson eq imag axis} is close to the solution $\mathbf{h}$ of the rank-one Dyson equation~\eqref{eq: coupled Dyson eq imag axis rank one}. According to Lemma~\ref{lem: g is close to h under incoherence} and Proposition~\ref{prop: relation between x,y and w_1,w_2}, $\mathbf{g}$ and $\mathbf{h}$ are close if $\widetilde{S}$ is sufficiently incoherent with respect to $\mathbf{x}$ and $\mathbf{y}$, or if the entries of $\mathbf{x}$ and $\mathbf{y}$ are approximately constant (i.e., $\mathbf{x}$ and $\mathbf{y}$ are close to multiples of the all-ones vector). As noted after Assumption~\ref{assump: decay rate of g-h}, if ${S}$ is generated as described in Examples~\ref{example: random S example 1} or~\ref{example: random S example 2} for sufficiently large $m$ and $n$, then Assumption~\ref{assump: decay rate of g-h} holds with high probability for $\delta_3 = 0.49$. In this case, the bounds in Theorem~\ref{thm: convergence of estimated scaling fators for generl S} imply almost the same rates as in Theorem~\ref{thm: convergence of estimated variance factors}. 

In Figure~\ref{fig: convergense of estimated scaling factors, general rank case}, we illustrate the relative errors in the left-hand sides of~\eqref{eq: relative estimation error of x for general rank} and~\eqref{eq: relative estimation error of y for general rank} for the same experimental setup used for Figure~\ref{fig: convergense of estimated scaling factors, rank-one case}, except that the variance matrix is $S = D\{\mathbf{x}\} \widetilde{S} D\{\mathbf{y}\}$, where $\widetilde{S}$ was generated according to Example~\ref{example: random S example 1}. Specifically, $\{z_{ij}\}$ are i.i.d and sampled from $5(\operatorname{Bernoulli}(0.1) - 0.1)$, i.e., $z_{ij}$ takes the value of $-0.5$ with probability $0.9$ and the value of $4.5$ with probability $0.1$. In this case, as mentioned previously, Assumption~\ref{assump: decay rate of g-h} is satisfied with high probability with $\delta_3 \sim 0.5$ for sufficiently large $m$ and $n$. Therefore, according to Theorem~\ref{thm: convergence of estimated scaling fators for generl S}, we expect the convergence rates to be the same as for the corresponding rates from Figure~\ref{fig: convergense of estimated scaling factors, rank-one case}. Indeed, we observe from Figure~\ref{fig: convergense of estimated scaling factors, general rank case} that in this case, the empirical error rates conform to the same bounds from Figure~\ref{fig: convergense of estimated scaling factors, rank-one case}.
\begin{figure} 
  \centering
  	{
  	\subfloat[][$m = \lceil n/2 \rceil$, $s^2/n = 10$, $\delta_1\sim 1$, $\delta_3 \sim 0.5$]
  	{
    \includegraphics[width=0.38\textwidth]{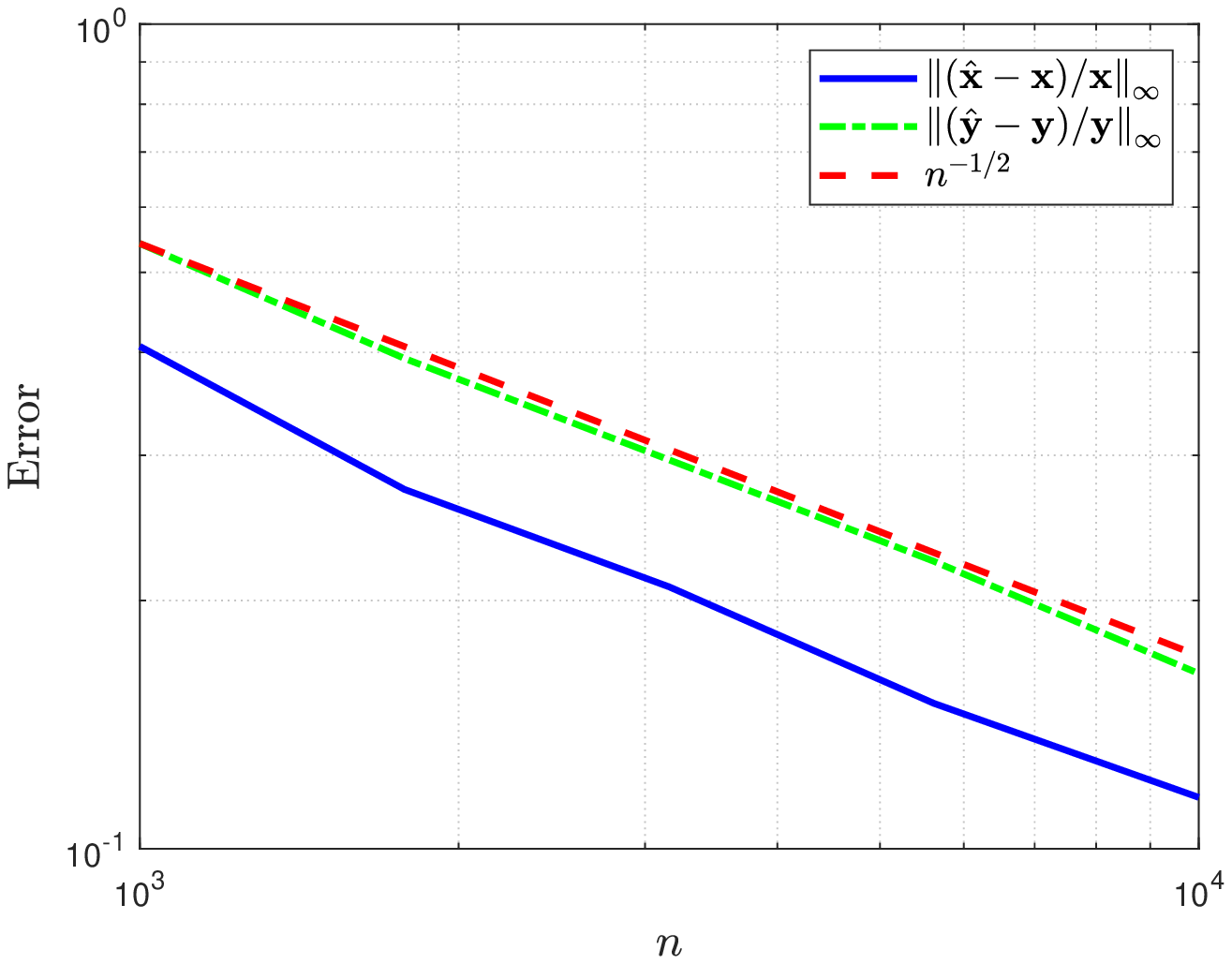} 
    }
    \subfloat[][$m = \lceil n/2 \rceil$, $s^2/n = 10 n$, $\delta_1\sim 1$, $\delta_3 \sim 0.5$] 
  	{
    \includegraphics[width=0.38\textwidth]{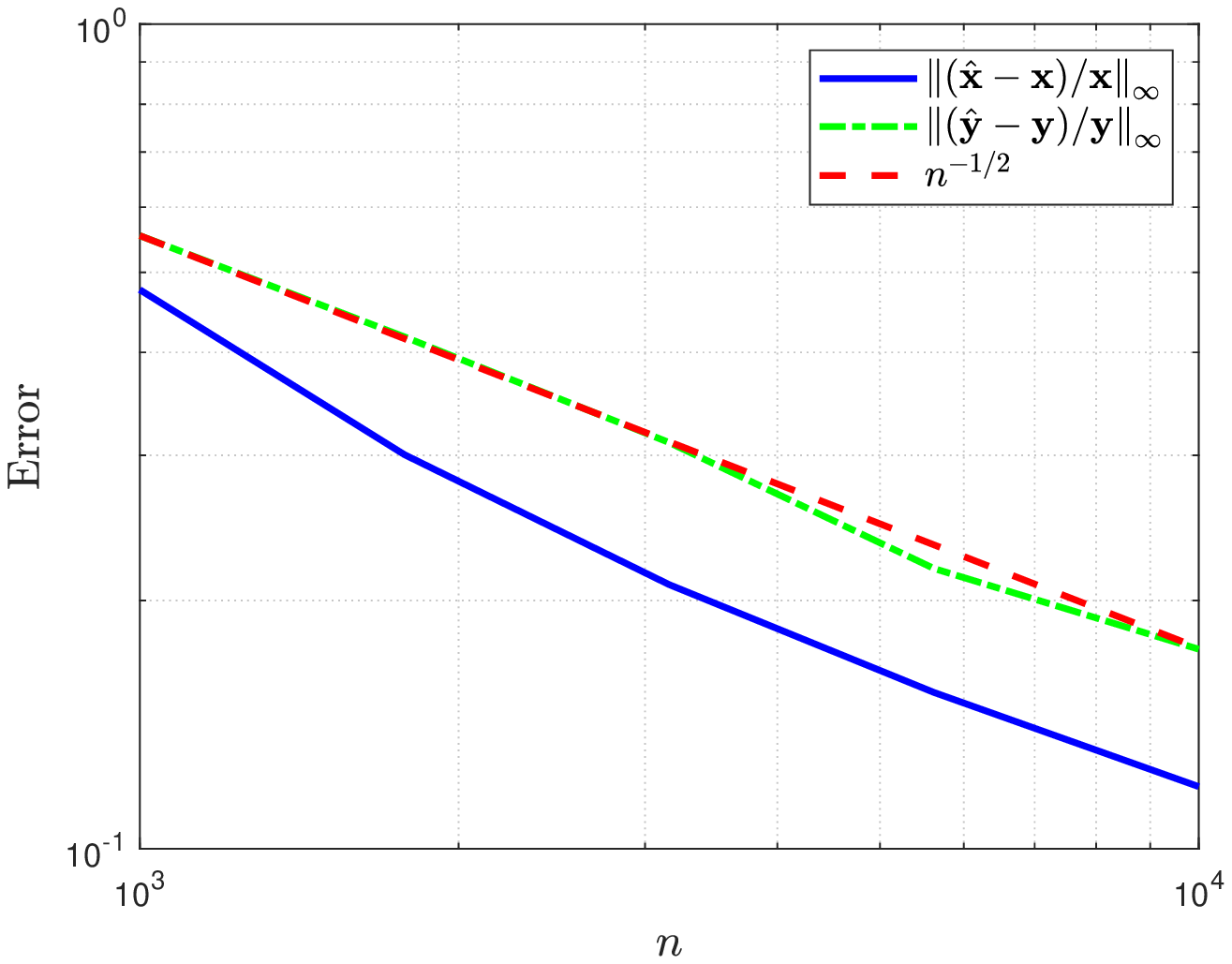} 
    }
    \\
    \subfloat[][$m = \lceil 3 n^{3/4} \rceil$, $s^2/n = 10$, $\delta_1\sim 1$, $\delta_3 \sim 0.5$]  
  	{
    \includegraphics[width=0.38\textwidth]{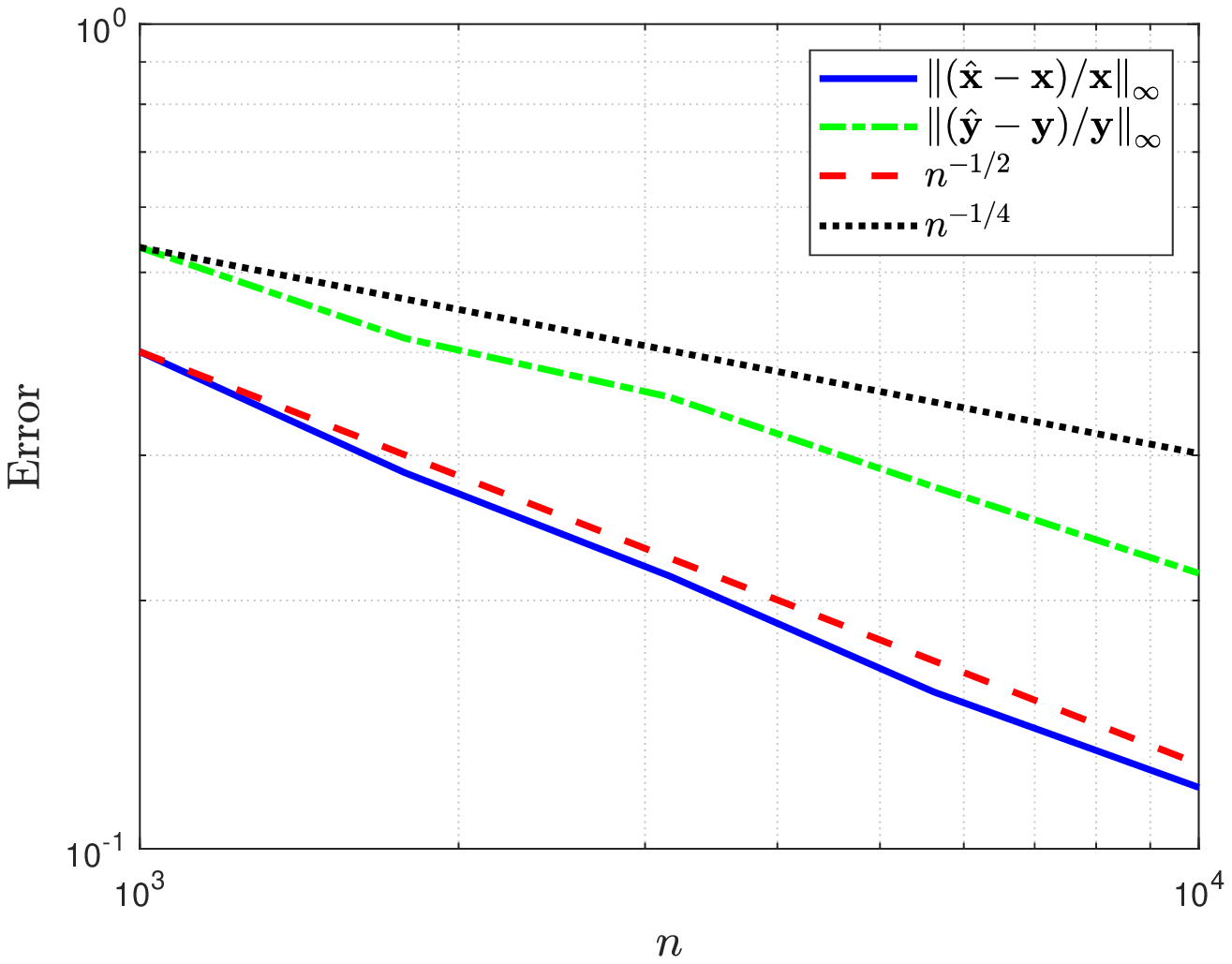}  
    }
    \subfloat[][$m = \lceil n/2 \rceil$, $s^2/n = 10$, $\delta_1\sim 2/3$, $\delta_3 \sim 0.5$] 
  	{
    \includegraphics[width=0.38\textwidth]{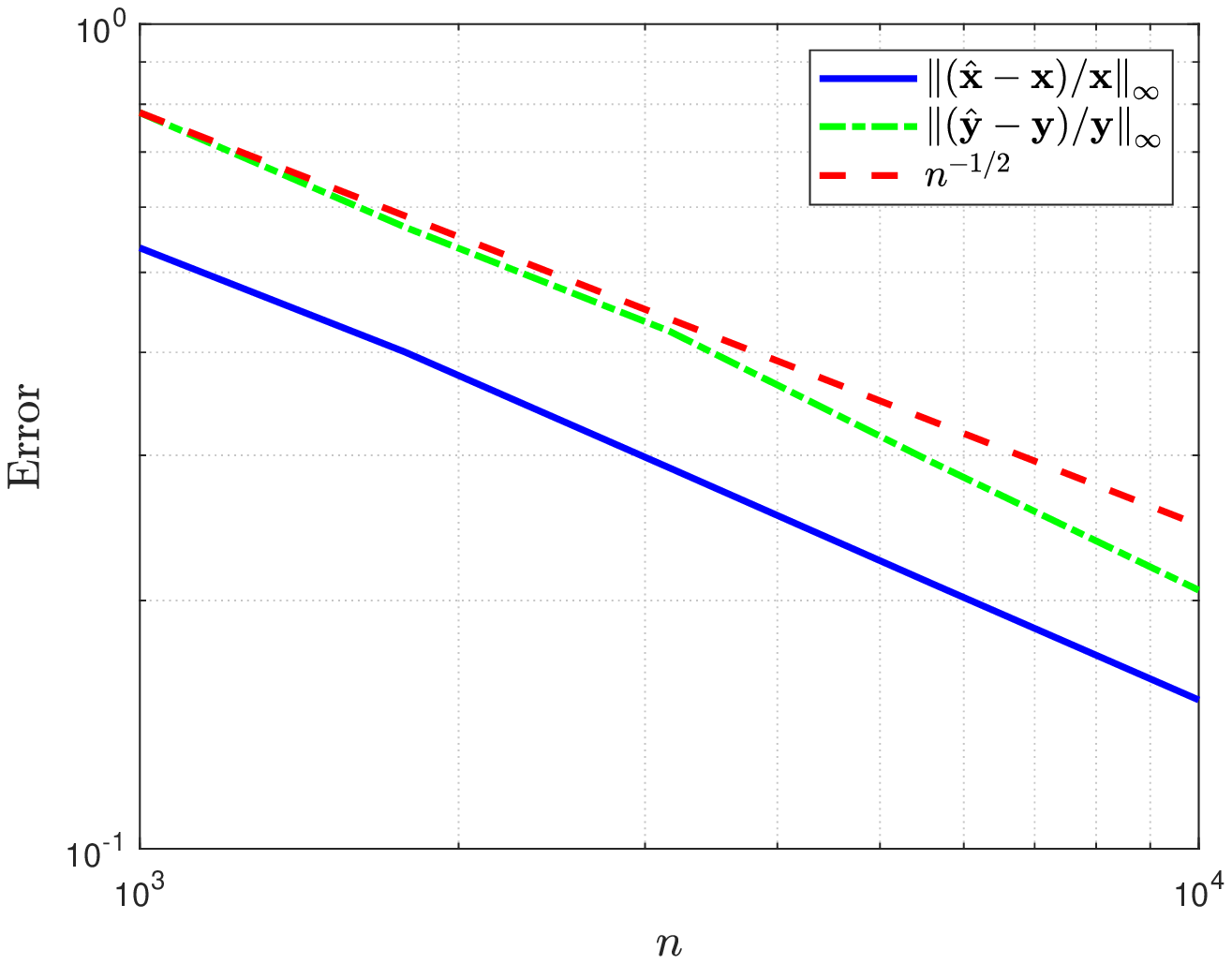} 
    } 

    }
    \caption
    {Maximal relative errors between $\hat{\mathbf{x}}$ and $\mathbf{x}$ and between $\hat{\mathbf{y}}$ and $\mathbf{y}$, as functions of $n$ in four different scenarios. The variance matrix here is given by $S = D\{\mathbf{x}\} \widetilde{S} D\{\mathbf{y}\}$, where $\widetilde{S}$ was generated according to Example~\ref{example: random S example 1}. 
    } \label{fig: convergense of estimated scaling factors, general rank case}
    \end{figure} 

\section{Application to signal detection and recovery} \label{sec: application to signal detection and recovery}
\subsection{Signal identification and rank estimation} \label{sec: rank estimation}
We begin by reviewing the spectral behavior of homoskedastic noise. Letting $\Sigma =  {E} {E}^T/n$, we define the \textit{Empirical Spectral Distribution} (ESD) of $\Sigma$ as
\begin{equation}
    F_{\Sigma} (\tau)= \frac{1}{m} \sum_{i=1}^m \mathbbm{1}\left(\lambda_i \{ \Sigma \} \leq \tau \right), \label{eq: empirical spectral distribution}
\end{equation}
where $\mathbbm{1}(\cdot)$ is the indicator function and $\lambda_i \{ \Sigma \}$ is the $i$th largest eigenvalue of $\Sigma$. If the noise is homoskedastic, i.e., $S = \sigma^2 \mathbf{1}_{m\times n}$ for some $\sigma>0$, then as $m,n\rightarrow \infty$ and $m/n\rightarrow \gamma \in (0,1]$, the empirical spectral distribution $F_{\Sigma}(\tau)$ converges almost surely to the Marchenko-Pastur (MP) distribution $F_{\gamma,\sigma}(\tau)$~\cite{marvcenko1967distribution}, which is the cumulative distribution function of the MP density
\begin{equation}
    dF_{\gamma,\sigma}(\tau) = \frac{\sqrt{(\beta_+ - \tau)(\tau - \beta_-)}}{2\pi \sigma^2 \gamma \tau} \mathbbm{1}\left( \beta_- \leq \tau \leq \beta_+\right), \label{eq:MP density}
\end{equation}
where $\beta_{\pm} = \sigma^2 (1\pm \sqrt{\gamma})^2$. Moreover, for many types of noise distributions~\cite{geman1980limit,yin1988limit}, the largest eigenvalue of $\Sigma$ converges almost surely (in the same asymptotic regime) to the upper edge of the MP density, i.e., $\lambda_1\{\Sigma\} \overset{a.s.}{\longrightarrow} \beta_+$.

If the noise variance matrix $S$ is of rank one, i.e., $S = \mathbf{x}\mathbf{y}^T$, then the normalization~\eqref{eq: scaling rows and columns} makes the noise completely homoskedastic, imposing the standard spectral behavior described above on $\widetilde{\Sigma} = \widetilde{E}\widetilde{E}^T/n$. For more general variance matrices, the normalization~\eqref{eq: scaling rows and columns} makes the noise variance matrix doubly regular (see Proposition~\ref{prop: matrix scaling} and Definition~\ref{def: doubly regular matrix}). Importantly, it was shown in~\cite{landa2022biwhitening} (see Proposition 4.1 therein), that under mild regularity conditions, this normalization is sufficient to enforce the standard spectral of homoskedastic noise.
Hence, we propose to use our estimates of $\mathbf{x}$ and $\mathbf{y}$ (described and analyzed in Section~\ref{sec: method derivation and analysis}) to normalize the rows and columns of the data, thereby making the noise variance matrix close to doubly regular. Specifically, we define $\hat{Y}\in\mathbb{R}^{m\times n}$, $\hat{X}\in\mathbb{R}^{m\times n}$, and $\hat{E}\in\mathbb{R}^{m\times n}$ according to
\begin{equation}
    \hat{Y} = (D\{\mathbf{\hat{x}}\})^{-1/2} Y (D\{\mathbf{\hat{y}}\})^{-1/2} = \hat{X} + \hat{E}, \label{eq: scaling rows and columns with estimated scaling factors}
\end{equation}
where $\hat{X} = (D\{\mathbf{\hat{x}}\})^{-1/2} X (D\{\mathbf{\hat{y}}\})^{-1/2}$ and $\hat{E} = (D\{\mathbf{\hat{x}}\})^{-1/2} E (D\{\mathbf{\hat{y}}\})^{-1/2}$. 
We now have the following theorem.
\begin{thm} \label{thm: MP law for E_hat}
Suppose that the assumptions of Theorem~\ref{thm: convergence of estimated scaling fators for generl S} hold and consider the asymptotic regime $m,n \rightarrow\infty$ with $m/n\rightarrow\gamma \in (0,1]$. Then, the empirical spectral distribution of $\hat{\Sigma} = \hat{E} \hat{E}^T / n$, given by $F_{\hat{\Sigma}}(\tau)$ (see~\eqref{eq: empirical spectral distribution}), converges almost surely to the Marchenko-Pastur distribution with parameter $\gamma$ and variance $1$, i.e., $F_{\gamma,1}(\tau)$, for all $\tau \in \mathbb{R}$. Moreover, the largest eigenvalue of $\hat{\Sigma}$ converges almost surely to the upper edge of the MP density $\beta_+ = (1 + \sqrt{\gamma})^2$. 
\end{thm}
The proof of Theorem~\ref{thm: MP law for E_hat} can be found in Appendix~\ref{appendix: proof of MP law for E_hat} and closely follows the proof of Theorem 2.6 in~\cite{landa2022biwhitening}. Theorem~\ref{thm: MP law for E_hat} establishes that under suitable conditions, $\hat{E}$ enjoys the same asymptotic spectral behavior described above for homoskedastic noise, namely the same limiting spectral distribution and the limit of the largest eigenvalue. This fact allows for a simple procedure for rank estimation. If no signal is present, i.e., if $X$ is the zero matrix, then the largest eigenvalue of $\hat{Y}\hat{Y}^T/n$ should be close to $(1+\sqrt{m/n})^2$ for sufficiently large $m$ and $n$. Alternatively, if an eigenvalue of $\hat{Y}\hat{Y}^T/n$ exceeds the threshold $(1+\sqrt{m/n})^2$ by a small positive constant for sufficiently large $m$ and $n$, then we can deduce that a signal is present in the data. Moreover, since the rank of $\hat{X}$ is identical to the rank of ${X}$ (as they differ by a positive diagonal scaling), we can estimate the rank of $X$ by counting how many eigenvalues of $\hat{Y}\hat{Y}^T/n$ exceed $(1+\sqrt{m/n})^2$ by a small positive constant $\varepsilon$, i.e., 
\begin{equation}
    \hat{r}_\varepsilon = \left\vert \left\{i\in[m]: \; \lambda_i\{\frac{1}{n} \hat{Y} \hat{Y}^T\} > \left( 1 + \sqrt{\frac{m}{n}}\right)^2  + \varepsilon  \right\} \right\vert.
\end{equation}
Here, $\varepsilon$ accounts for finite-sample fluctuations of the largest eigenvalue of $\hat{E}\hat{E}^T/n$, which diminish as $m,n \rightarrow\infty$.
It can be shown that for any fixed $\varepsilon>0$, under the conditions in Theorem~\ref{thm: convergence of estimated scaling fators for generl S}, the rank estimator $\hat{r}_\varepsilon$ provides a consistent lower bound on the true rank $r$ as $n\rightarrow \infty$, namely
\begin{equation}
    \operatorname{Pr}\{ r < \hat{r}_\varepsilon \} \underset{n \rightarrow \infty}{\longrightarrow} 0,
\end{equation}
see Theorem 2.3 in~\cite{landa2022biwhitening} and its proof, which we do not repeat here for the sake of brevity. In other words, all signal components that are detected in this way are true signal components asymptotically. In what follows, we take $\varepsilon \rightarrow 0$ for simplicity and denote $\hat{r} = \hat{r}_0$. Algorithm~\ref{alg:rank estimation} summarizes our proposed rank estimation procedure, noting that comparing the eigenvalues of $\hat{Y}\hat{Y}^T/n$ to $\beta_+$ is equivalent to comparing the singular values of $\hat{Y}$ to $\sqrt{m} + \sqrt{n}$. 

\begin{algorithm}
\caption{Rank estimation}\label{alg:rank estimation}
\begin{algorithmic}[1]
\Statex{\textbf{Input:} Data matrix $Y\in \mathbb{R}^{m\times n}$ with $m\leq n$ and no rows or columns that are entirely zero.}
\State \label{alg: step 1}Apply Algorithm~\ref{alg:noise standardization} to obtain $\hat{Y}$.
\State \label{alg: step 2}Compute $\hat{r}$  by counting the number of singular values of $\hat{Y}$ that exceed $\sqrt{m} + \sqrt{n}$.
\end{algorithmic}
\end{algorithm}

An important advantage of our proposed rank estimation technique is that the normalization of the rows and columns can enhance the spectral signal-to-noise ratio, namely, improve the ratios between the singular values of the signal and the operator norm of the noise. Specifically, it was shown in~\cite{leeb2021matrix} (see Proposition 6.3 therein) that in the case of $E_{ij} \sim \mathcal{N}(0,S_{ij})$, $S = \mathbf{x}\mathbf{y}^T$, and if the singular vectors of $X$ satisfy certain genericity and weighted orthogonality conditions (see~\cite{leeb2021matrix} for more details), then
\begin{equation}
    \frac{\sigma_i^2\{\widetilde{X}\}}{\Vert \widetilde{E} \Vert_2^2} \geq \tau \frac{\sigma_i^2\{{X}\}}{\Vert {E} \Vert_2^2},
\end{equation}
almost surely as $m,n\rightarrow\infty$ with $m/n\rightarrow\gamma$, for all $i=1,\ldots,r$, where $\sigma_i\{\cdot\}$ denotes the $i$th largest singular value of a matrix and
\begin{equation}
    \tau = \left( \frac{1}{m} \sum_{i=1}^m \mathbf{x}_i \right) \left( \frac{1}{m} \sum_{i=1}^m \frac{1}{\mathbf{x}_i} \right) \left( \frac{1}{n} \sum_{j=1}^n \mathbf{y}_j \right) \left( \frac{1}{n} \sum_{j=1}^n \frac{1}{\mathbf{y}_j} \right).
\end{equation}
Here, we always have $\tau \geq 1$, where $\tau = 1$ if and only if $\mathbf{x}$ and $\mathbf{y}$ are constant vectors (i.e., vectors with identical entries); see~\cite{leeb2021matrix}. In this setting, if the noise variances are not identical across rows and/or columns, the normalization~\eqref{eq: scaling rows and columns} will improve the signal-to-noise ratio for each signal component. Moreover, the improvement increases with the level of heteroskedasticity, i.e., with the amount of variability of the entries in $\mathbf{x}$ and $\mathbf{y}$ as encoded by $\tau$. Importantly, even if some of the signal components are originally below the spectral noise level, i.e., ${\sigma_i\{{X}\}} < {\Vert {E} \Vert_2}$, the corresponding signal components after the normalization~\eqref{eq: scaling rows and columns} can exceed it if $\tau$ is sufficiently large. Under the conditions in Theorem~\ref{thm: convergence of estimated variance factors} in Section~\ref{sec: rank one case}, which allow us to estimate $\mathbf{x}$ and $\mathbf{y}$ accurately for large matrix dimensions, our proposed normalization~\eqref{eq: scaling rows and columns with estimated scaling factors} will enhance the spectral signal-to-noise ratio analogously.

More generally, the normalization~\eqref{eq: scaling rows and columns} can significantly reduce the operator norm of the noise with respect to the average noise variance in the data. This advantage holds for general noise distributions and variance patterns regardless of the signal.
To explain this advantage, suppose for simplicity that the average noise variance across all entries in the original data is one, i.e., $\sum_{ij}S_{ij}/(mn) = 1$. This property also holds after the normalization~\eqref{eq: scaling rows and columns} since the corresponding noise variance matrix is doubly regular. On the one hand, before the normalization, we have the inequality
\begin{equation}
    \frac{1}{n} \mathbb{E} \Vert {E} \Vert^2_2 \geq \max \left\{ \max_{i\in[n]} \left\{ \frac{1}{n} \sum_{j=1}^n S_{ij} \right\}, \left( \frac{m}{n}\right) \max_{j\in [m]} \left\{ \frac{1}{m} \sum_{i=1}^m S_{ij} \right\}\right\},
\end{equation}
which follows from the fact that the operator norm of a matrix is lower-bounded by the Euclidean norm of any of its rows or columns. Hence, the quantity $\mathbb{E}\Vert {{E}}  \Vert^2_2/n$ will be large if the average noise variance in at least one of the rows or columns is large. For instance, if $m/n = \gamma = 0.5$ and the average noise variance in one of the rows or columns is $100$, then $\mathbb{E}\Vert {{E}}  \Vert^2_2/n$ must be at least $50$ (see Figure~\ref{fig: spectrum before and after normalization, outlier rows anc columns} in the introduction for an illustration). On the other hand, after the normalization~\eqref{eq: scaling rows and columns}, the quantity $\mathbb{E}\Vert \widetilde{{E}}  \Vert^2_2/n = \mathbb{E}[ \lambda_1\{\widetilde{E} \widetilde{E}^T / n\} ]$ approaches the MP upper edge $(1+\sqrt{\gamma})^2$ in the asymptotic regime $m,n\rightarrow\infty$ with $m/n\rightarrow\gamma$. Crucially, $(1+\sqrt{\gamma})^2$ is upper bounded by $4$ for all $\gamma \in (0,1]$. Therefore, the operator norm of the noise can be substantially reduced by the normalization~\eqref{eq: scaling rows and columns} while keeping the average noise variance in the data the same. This advantage holds similarly for the normalization~\eqref{eq: scaling rows and columns with estimated scaling factors} whenever $\mathbf{x}$ and $\mathbf{y}$ are accurately estimated by $\hat{\mathbf{x}}$ and $\hat{\mathbf{y}}$, respectively (see Section~\ref{sec: variance matrices with general rank}).

Figure~\ref{fig: spectrum before and after biwhitening} depicts the spectrum of a matrix $Y$ of size $1000 \times 2000$ before and after the normalizations~\eqref{eq: scaling rows and columns} and~\eqref{eq: scaling rows and columns with estimated scaling factors}, where the true scaling factors $\mathbf{x}$ and $\mathbf{y}$ were computed by the Sinkhorn-Knopp algorithm~\cite{sinkhorn1967diagonal,sinkhorn1967concerning} applied to $S$. Here, the signal matrix $X$ is of rank $r=10$ with singular values $s_1,\ldots,s_{10} = \sqrt{10n}$, where the singular vectors were drawn from random Gaussian vectors followed by orthonormalization. The noise entries $E_{ij}$ were sampled independently from $\mathcal{N}(0,S_{ij})$, where $S$ is of rank $30$  and was generated according to $S = AB$, $A\in\mathbb{R}^{m\times 30}$, $B\in\mathbb{R}^{30\times n}$, $A_{ij}, B_{ij} \sim \operatorname{exp}\{\mathcal{N}(0,t^2)\}$, and $t$ is the heteroskedasticity parameter ($t=0$ corresponds to homoskedastic noise). We then fixed $t=2$ and normalized $S$ by a scalar so that its average entry is $1$. Since the entries of $A$ and $B$ are generated from a log-normal distribution, the entries of $S$ can differ substantially across rows and columns. 

\begin{figure} 
  \centering
  	{
  	\subfloat[][Singular values of $Y$]
  	{
    \includegraphics[width=0.3\textwidth]{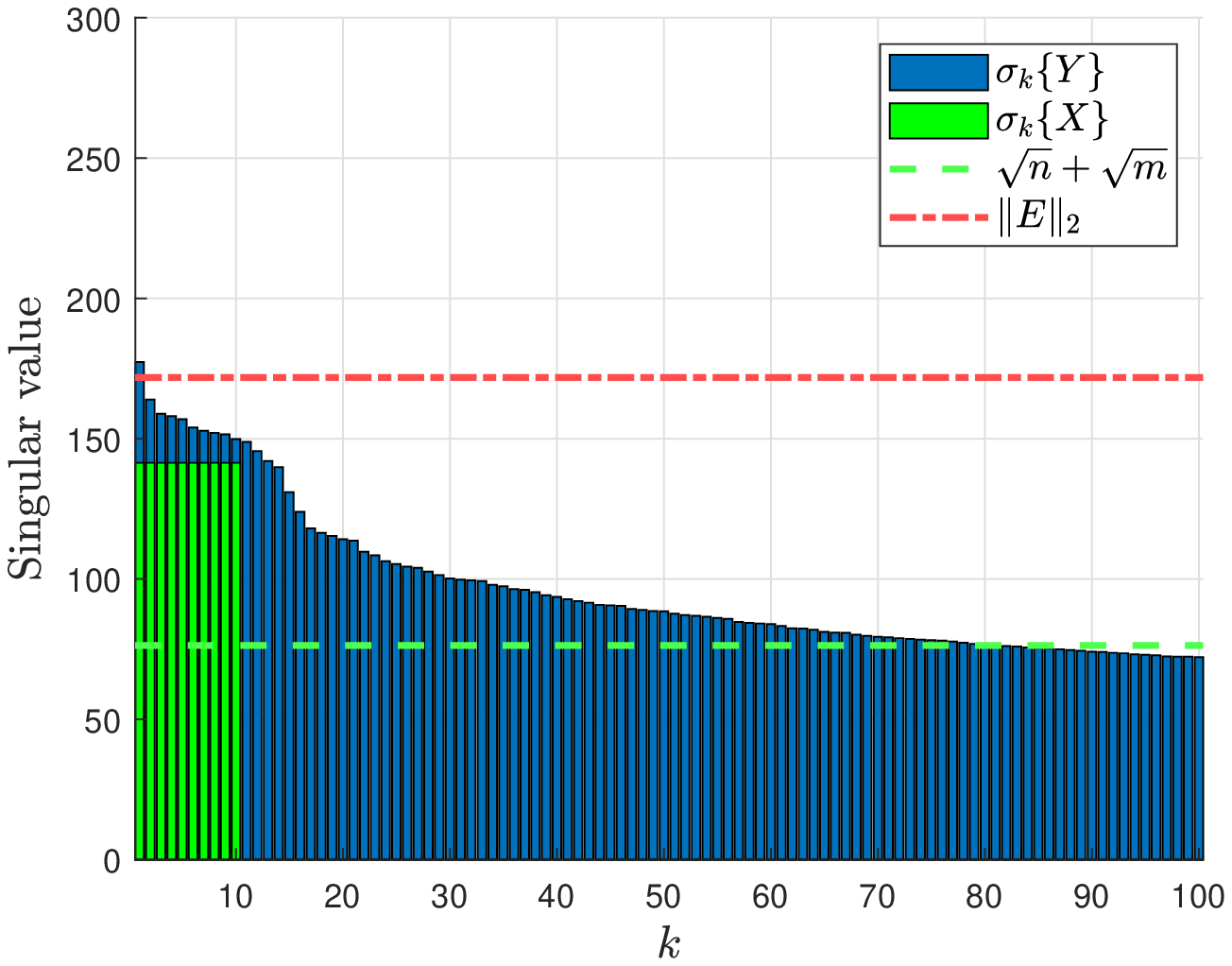} \label{fig: sorted eigenvalues before biwhitening}
    }
    \subfloat[][Singular values of $\widetilde{Y}$] 
  	{
    \includegraphics[width=0.3\textwidth]{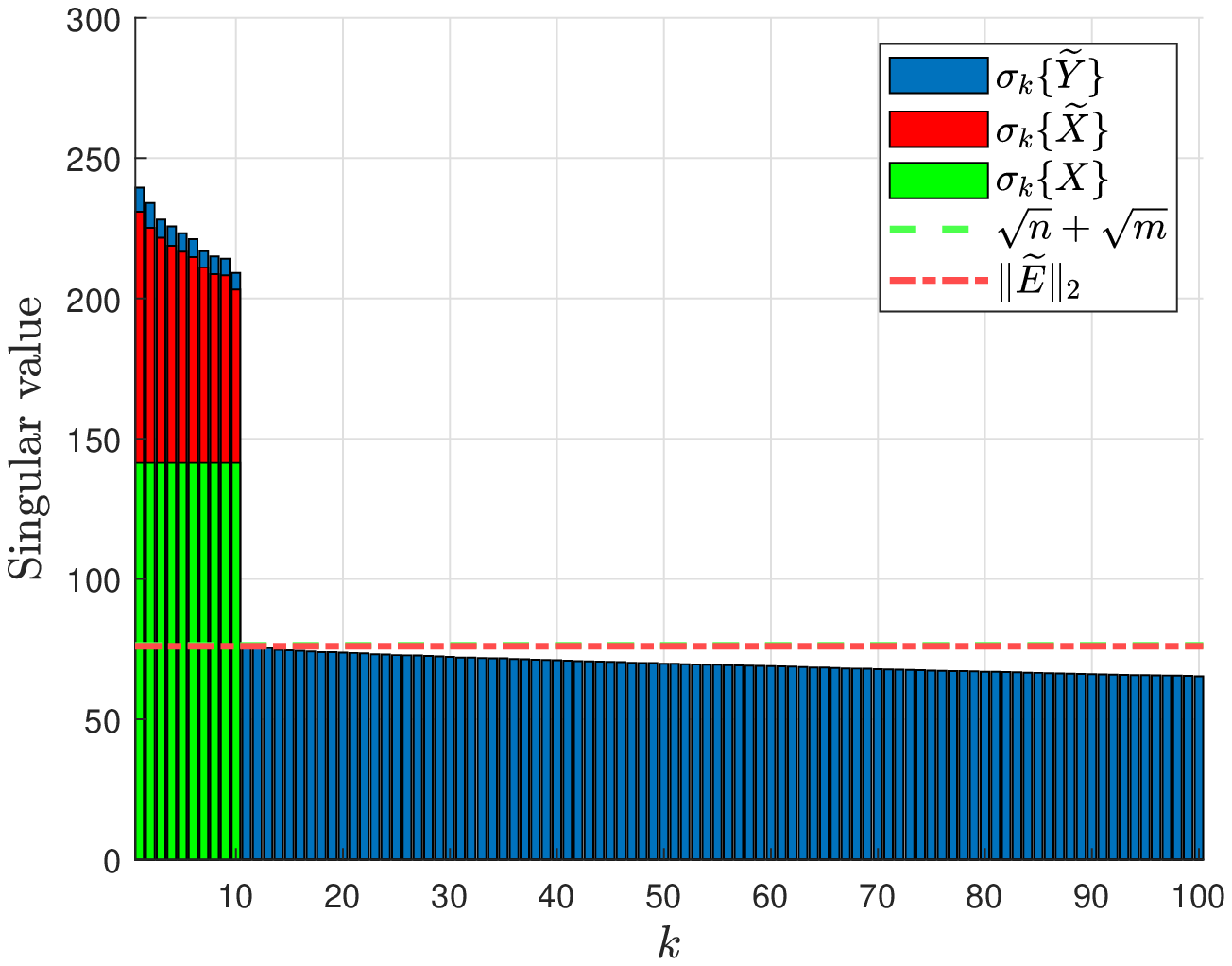} 
    }
     \subfloat[][Singular values of $\hat{Y}$] 
    {
    \includegraphics[width=0.3\textwidth]{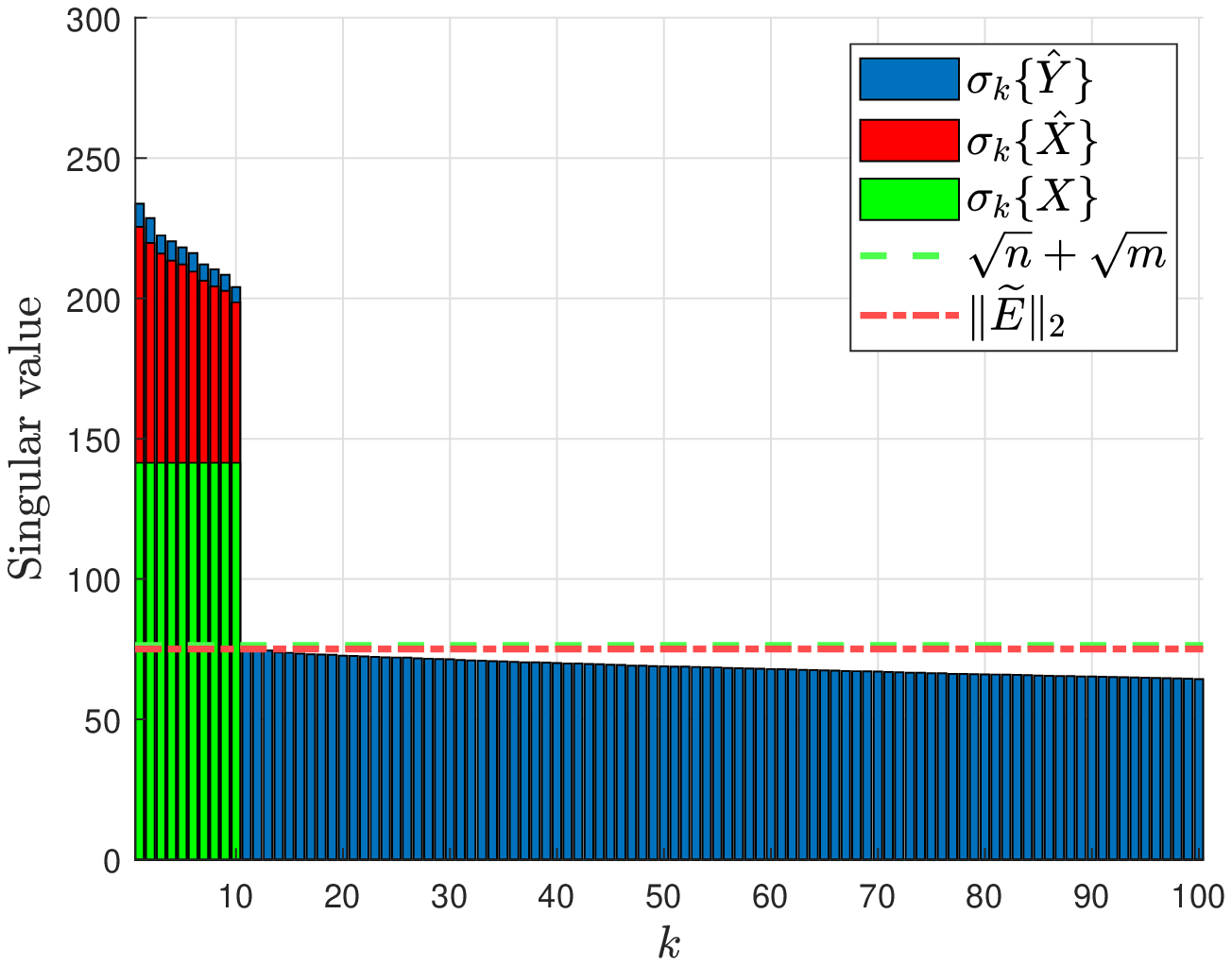} 
    }
    \\
    \subfloat[][Eigenvalue density of $Y Y^T/n$]  
  	{
    \includegraphics[width=0.3\textwidth]{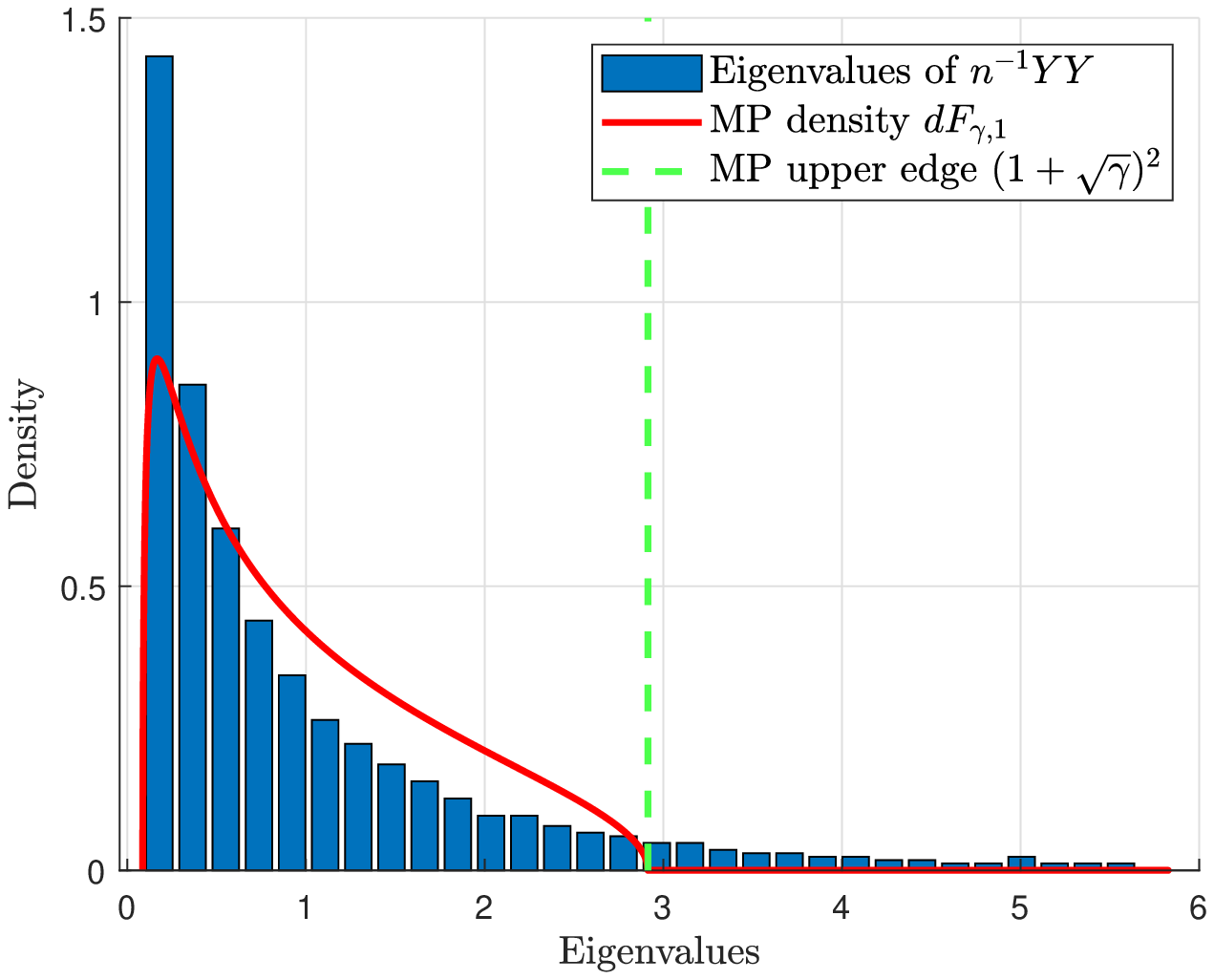}  
    }
    \subfloat[][Eigenvalue density of $\widetilde{Y} \widetilde{Y}^T/n$] 
  	{
    \includegraphics[width=0.3\textwidth]{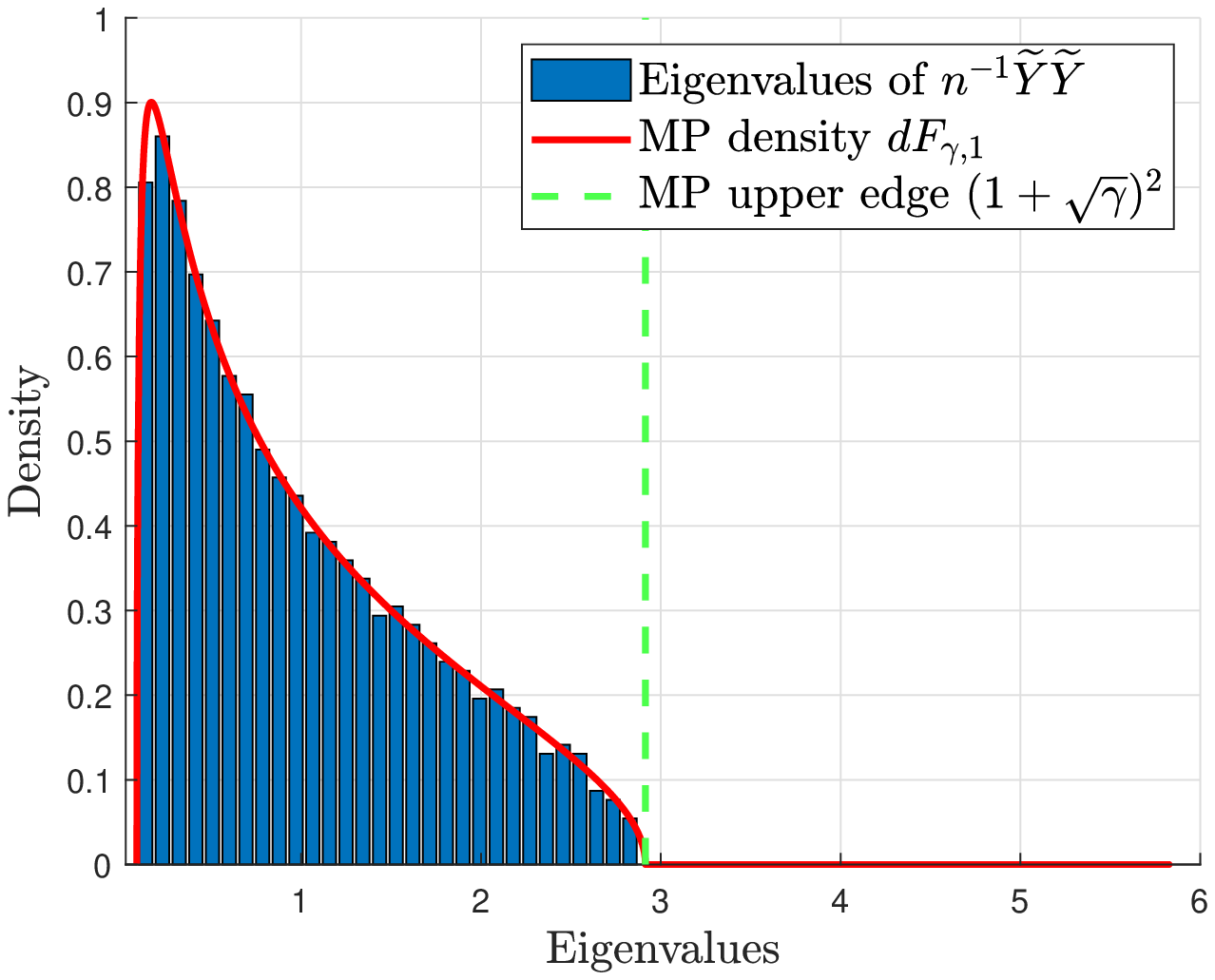} 
    } 
    \subfloat[][Eigenvalue density of $\hat{Y} \hat{Y}^T/n$] 
  	{
    \includegraphics[width=0.3\textwidth]{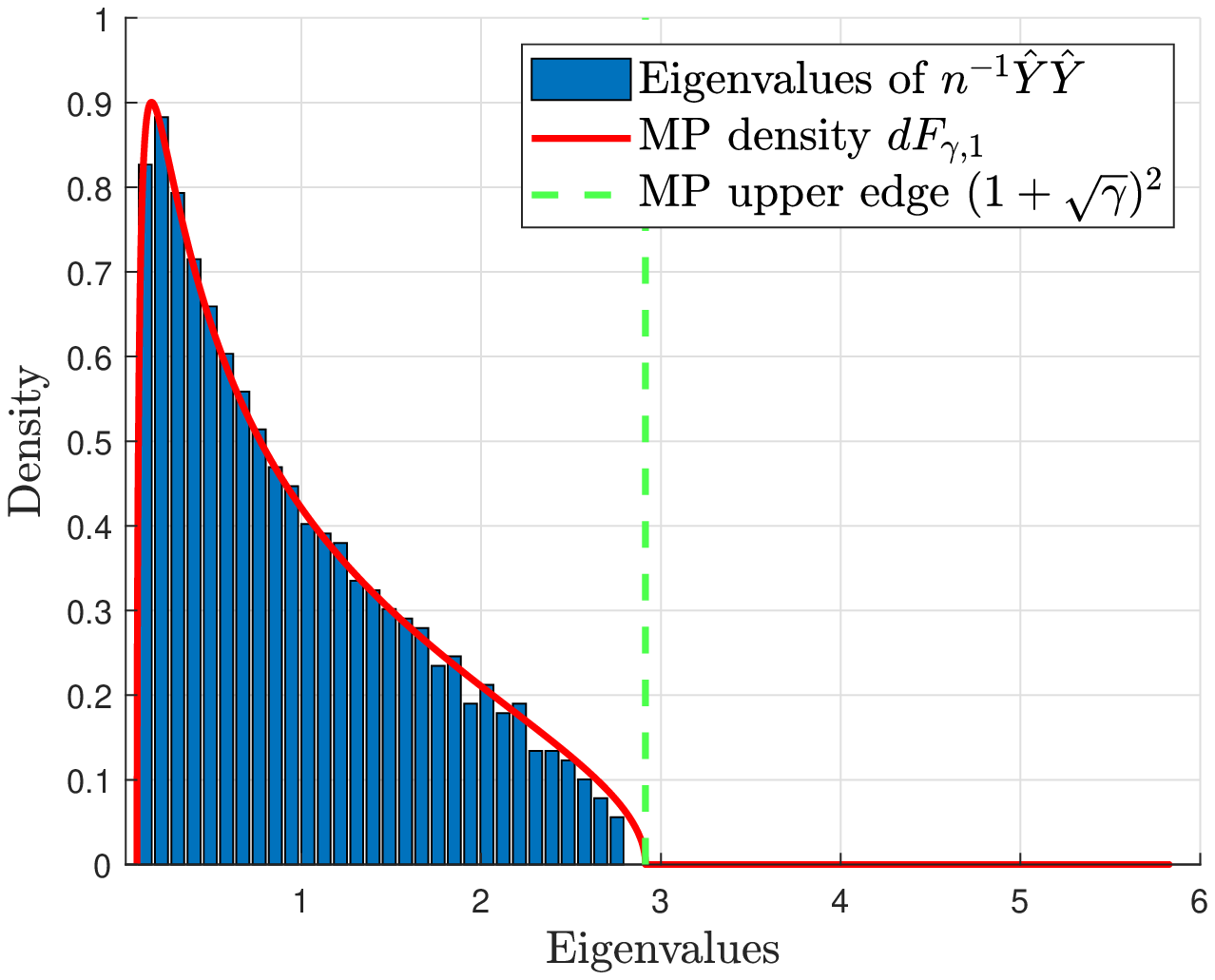} 
    } 

    }
    \caption
    {The spectrum of the original data matrix (panel (a)) and after applying the normalizations~\eqref{eq: scaling rows and columns} (panel(b)) and~\eqref{eq: scaling rows and columns with estimated scaling factors}, where $m=1000$, $n=2000$, and $r=10$. 
    } \label{fig: spectrum before and after biwhitening}
    \end{figure} 
    
Figure~\ref{fig: spectrum before and after biwhitening} shows that for the original data, the spectral noise level $\Vert E\Vert_2$ exceeds the signal singular values $\sigma_i\{X\}$. Moreover, the largest singular value of $Y$ corresponds to the noise level $\Vert {E} \Vert_2$ almost perfectly. Hence, the signal components are undetectable by traditional approaches that compare the singular values of $Y$ to the noise level. Moreover, the data singular values do not admit any noticeable gap after the $10$'th singular value, and the eigenvalue density of $Y Y^T/n$ does not fit the MP law. In particular, the eigenvalue density of $Y Y^T/n$ is much more spread out than the MP density, with a large number of eigenvalues that exceed the upper edge $\beta_+$. These large eigenvalues arise due to the heteroskedasticity of the noise. However, after our proposed normalizations, the signal singular values increase and the noise level decreases simultaneously, revealing the signal components and allowing for their easy detection. Specifically, we immediately see a large gap after the $10$'th singular value of $\widetilde{Y}$ and $\hat{Y}$, while the $11$'th singular values is just below the critical threshold $\sqrt{n} + \sqrt{m}$, which is very close to $\Vert {\widetilde{E}} \Vert_2$ and $\Vert {\hat{E}} \Vert_2$. Furthermore, the eigenvalue densities of $\widetilde{Y} \widetilde{Y}^T/n$ and $\hat{Y} \hat{Y}^T/n$ are very close to the MP law, even though the noise variances after the normalizations~\eqref{eq: scaling rows and columns} and~\eqref{eq: scaling rows and columns with estimated scaling factors} are not identical (since $S$ is not rank-one). 

The results above suggest that for heteroskedastic noise, traditional signal detection and rank estimation techniques that rely on inspecting the spectrum of the observed data without normalization can perform suboptimally. In particular, methods that compare the observed leading singular values of the data to the largest singular value of the noise (estimated by, e.g., parallel analysis and its variants~\cite{dobriban2020permutation,dobriban2019deterministic,hong2020selecting}) will not detect $9$ out of the $10$ signal components in the experiment outlined above (Figure~\ref{fig: sorted eigenvalues before biwhitening}); see also the experiments in~\cite{landa2022biwhitening} that further illustrate this phenomenon. On the other hand, our normalization stabilizes the spectral behavior of the noise and can substantially reduce its operator norm (relative to the average noise variance), thereby allowing for improved detection of weak signal components. 

\subsection{Low-rank signal recovery} \label{sec: weighted low-rank approximation}
In the previous section, we showed that normalizing the rows and columns of the data matrix according to~\eqref{eq: scaling rows and columns with estimated scaling factors} can be highly beneficial for signal detection. In particular, the normalization standardizes the spectral behavior of the noise and can improve the spectral signal-to-noise ratio. In this section, we propose to leverage these advantages for improved recovery of the low-rank signal matrix $X$. Specifically, instead of applying a low-rank approximation to the original data matrix $Y$, where the noise can be highly heteroskedastic, we advocate applying a low-rank approximation to the normalized matrix $\hat{Y}$ and then un-normalize the rows and columns of the resulting matrix. We motivate and justify this approach through maximum-likelihood estimation and weighted low-rank approximation. We then demonstrate the practical advantages of this approach in simulations.

Let us consider the case of Gaussian heteroskedastic noise $E_{ij}\sim \mathcal{N}(0,S_{ij})$. In this case, we have $Y_{ij} \sim \mathcal{N}(X_{ij},S_{ij})$, where we treat $X$ as a deterministic parameter to be estimated. The negative log-likelihood of observing $Y$ is given by
\begin{equation}
    \mathcal{L}_S(X) = \frac{1}{2} \sum_{i=1}^m\sum_{j=1}^n \left[ \frac{\vert X_{ij} - Y_{ij} \vert^2}{S_{ij}} + \log (2\pi S_{ij})\right]. \label{eq: negative log-likelihood}
\end{equation}
If the signal's rank $r$ is known, then the maximum-likelihood estimate of $X$ is obtained by minimizing $\mathcal{L}_S(\Theta)$ over all matrices $\Theta \in \mathbb{R}^{m\times n}$ of rank $r$. This is equivalent to minimizing
\begin{equation}
    \widetilde{\mathcal{L}}_S(\Theta) = \sum_{i=1}^m\sum_{j=1}^n \frac{\vert \Theta_{ij} - Y_{ij} \vert^2}{S_{ij}}, \label{eq: L_tilde maximum-likelihood def}
\end{equation}
over all rank-$r$ matrices, a problem known as weighted low-rank approximation~\cite{srebro2003weighted}. Specifically, the goal is to find an $m\times n$ matrix of rank $r$ that best fits the data $Y$ in a weighted least-squares sense, where the weights are inversely proportional to the noise variances in the data. 

If the noise variance matrix $S$ is of rank one, i.e., $S = \mathbf{x}\mathbf{y}^T$, then the weighted low-rank approximation problem admits a simple closed-form solution that can be written via the normalization~\eqref{eq: scaling rows and columns}. Specifically, the minimizer of $\widetilde{\mathcal{L}}_{\mathbf{x}\mathbf{y}^T}(\Theta)$ over all matrices $\Theta$ of rank $r$ is given by (see, e.g.,~\cite{irani2000factorization,srebro2003weighted})
\begin{equation}
    (D\{\mathbf{x}\})^{1/2} T_r\{\widetilde{Y}\} (D\{\mathbf{y}\})^{1/2}, \label{eq: soltuion to rank-one weighted low-rank approximation problem}
\end{equation}
where $\widetilde{Y}$ is from~\eqref{eq: scaling rows and columns} and $T_r\{\widetilde{Y}\}$ is the closest rank-$r$ matrix to $\widetilde{Y}$ in Frobenius norm. In other words, if $S=\mathbf{x}\mathbf{y}^T$, then the solution to the weighted low-rank approximation problem is given by first scaling the rows and columns of the data matrix $Y$ according to~\eqref{eq: scaling rows and columns}, then truncating the SVD of the normalized matrix to its $r$ largest components, and lastly, unscaling the rows and columns of the resulting rank-$r$ matrix. Since $\mathbf{x}$ and $\mathbf{y}$ are unknown, we replace them with their estimates $\hat{\mathbf{x}}$ and $\hat{\mathbf{y}}$ from Algorithm~\ref{alg:noise standardization}. Further, we replace the true rank $r$ with its estimate $\hat{r}$ from Section~\ref{sec: rank estimation} (see Algorithm~\ref{alg:rank estimation}). Overall, we obtain the signal estimate
\begin{equation}
    \overline{X} = (D\{\hat{\mathbf{x}}\})^{1/2} T_{\hat{r}}\{\hat{{Y}}\} (D\{\hat{\mathbf{y}}\})^{1/2}. \label{eq: estimated solution to weighted low-rank approximation problem}
\end{equation}

If the variance matrix $S$ is not rank-one, then the weighted low-rank approximation problem 
does not admit a closed-form solution~\cite{srebro2003weighted}. Since it is a non-convex optimization problem, solving it can be highly prohibitive from a computational perspective, particularly for large matrix dimensions. 
Therefore, it is advantageous to replace the variance matrix $S$ in the weighted low-rank approximation problem with a rank-one surrogate. As we show below, replacing $S$ in $\mathcal{L}_S$ from~\eqref{eq: negative log-likelihood} with $\mathbf{x}\mathbf{y}^T$, where $\mathbf{x}$ and $\mathbf{y}$ are the scaling factors of $S$ from Proposition~\ref{prop: matrix scaling}, has a favorable interpretation from the viewpoint of approximate maximum-likelihood estimation in the Gaussian noise model. Specifically, we have the following proposition, whose proof can be found in Appendix~\ref{appendix: proof of optimal approximate ML estimation}. 
\begin{prop} \label{prop: optimal approximate ML estimation}
For any fixed signal $X$, the matrix $\mathbf{x}\mathbf{y}^T$ is the unique minimizer of $\mathbb{E}[\mathcal{L}_A(X)]$ over all positive matrices $A\in\mathbb{R}^{m\times n}$ of rank one, where $\mathbf{x}$ and $\mathbf{y}$ are from Proposition~\ref{prop: matrix scaling} and the expectation is over $\{Y_{ij}\sim \mathcal{N}(X_{ij},S_{ij})\}$. If we alleviate the rank-one requirement, then the corresponding minimizer is $S$.
\end{prop}
In other words, among all positive rank-one matrices $A$ that act as surrogates for $S$, taking $A=\mathbf{x}\mathbf{y}^T$ in $\mathcal{L}_A(X)$ provides the best approximation, on average, to the true negative log-likelihood $\mathcal{L}_S(X)$ of observing the data in the Gaussian noise model. Hence, 
if we want to replace $S$ with a rank-one surrogate for approximate maximum-likelihood estimation, i.e., when minimizing $\mathcal{L}_A(\Theta)$ over all rank-$r$ matrices, it is natural to use $A=\mathbf{x}\mathbf{y}^T$. Note that $\mathbf{x}\mathbf{y}^T$ is generally not the best rank-one approximation to $S$ in Frobenius norm.

As explained above, in the Gaussian noise model, the weighted low-rank approximation problem of minimizing $\widetilde{\mathcal{L}}_{\mathbf{x}\mathbf{y}^T}(\Theta)$ over all matrices $\Theta$ of rank $r$ is a natural surrogate for the true maximum-likelihood estimation problem when $S$ is not rank-one. More broadly, the rationale for solving this problem is to penalize noisy rows and columns in the data while allowing for a simple closed-form solution to the weighted low-rank approximation problem. In this context, the noise levels in the rows and columns are encoded by the scaling factors $\mathbf{x}$ and $\mathbf{y}$ from Proposition~\ref{prop: matrix scaling}, which can be estimated accurately by $\hat{\mathbf{x}}$ and $\hat{\mathbf{y}}$ from~\eqref{eq: x_hat and y_hat def}, respectively, in a broad range of settings (see Section~\ref{sec: variance matrices with general rank}). 
Consequently, even if the noise is highly heteroskedastic with a general variance pattern, it is advantageous to use~\eqref{eq: estimated solution to weighted low-rank approximation problem} to estimate the low-rank signal $X$. Algorithm~\ref{alg:weighted low-rank approximation} summarizes our proposed approach for signal recovery. 

\begin{algorithm}
\caption{Adaptively weighted low-rank approximation}\label{alg:weighted low-rank approximation}
\begin{algorithmic}[1]
\Statex{\textbf{Input:} Data matrix $Y\in \mathbb{R}^{m\times n}$ with $m\leq n$ and no rows or columns that are entirely zero.}
\State \label{alg: step 1}Apply Algorithms~\ref{alg:noise standardization} and~\ref{alg:rank estimation} to $Y$, obtaining $\hat{Y}$, $\hat{\mathbf{x}}$, $\hat{\mathbf{y}}$, and $\hat{r}$.
\State \label{alg: step 3}Compute $T_{\hat{r}}\{\hat{{Y}}\}$ by truncating the SVD of $\hat{Y}$ to its $\hat{r}$ largest components.
\State \label{alg: step 4}Form the estimated signal matrix $\overline{X} = (D\{\hat{\mathbf{x}}\})^{1/2} T_{\hat{r}}\{\hat{{Y}}\} (D\{\hat{\mathbf{y}}\})^{1/2}$.
\end{algorithmic}
\end{algorithm}

The solution~\eqref{eq: soltuion to rank-one weighted low-rank approximation problem} to the weighted low-rank approximation problem with $S=\mathbf{x}\mathbf{y}^T$ involves the operator $T_r$, which truncates the SVD of the input matrix to its $r$ leading components. This operator can be replaced with a more general matrix denoising operator. For the Gaussian noise model with $S = \mathbf{x}\mathbf{y}^T$,~\cite{leeb2021matrix} proposed to use $\mathcal{T}_r\{A\} = \sum_{i=1}^r \sum_{j=1}^r u_i \theta_{ij} v_j$, where $\{u_i\}$ and $\{v_j\}$ are the left and right singular vectors of $A$, respectively, and $\{\theta_{ij}\}$ are tunable parameters. This matrix denoising operator generalizes singular value shrinkage~\cite{gavish2017optimal} (which manipulates the singular values of a matrix) by allowing for all possible cross-products of left and right singular vectors. \cite{leeb2021matrix} derived the optimally-tuned parameters $\{\theta_{ij}\}$ that minimize the weighted loss
\begin{equation}
    \left\Vert \left( D\{\mathbf{x}\} \right)^{1/2} \left( \mathcal{T}_r\{\widetilde{Y}\} - \widetilde{X} \right) \left( D\{\mathbf{y}\} \right)^{1/2}\right\Vert_F^2 = \left\Vert \left( D\{\mathbf{x}\} \right)^{1/2} \mathcal{T}_r\{\widetilde{Y}\} \left( D\{\mathbf{y}\} \right)^{1/2} - X \right\Vert_F^2
\end{equation}
asymptotically as $m,n\rightarrow \infty$ and $m/n\rightarrow\gamma$ under suitable conditions. The optimal parameters $\{\theta_{ij}\}$ can be computed by Algorithm 4.1 in~\cite{leeb2021matrix} (where they are denoted by the matrix $\hat{\mathbf{B}}\in\mathbb{R}^{r\times r}$).
The corresponding estimator of $X$ is given by \sloppy$(D\{\mathbf{x}\} )^{1/2} \mathcal{T}_r\{\widetilde{Y}\} (D\{\mathbf{y}\})^{1/2}$, which is analogues to~\eqref{eq: soltuion to rank-one weighted low-rank approximation problem} when replacing $T_r$ with $\mathcal{T}_r$. We refer to this approach as weighted-loss denoising (WLD). Analogously to our adaptively weighted low-rank approximation approach, we propose to replace $\mathbf{x}$ and $\mathbf{y}$ with our estimates $\hat{\mathbf{x}}$ and $\hat{\mathbf{y}}$ from~\eqref{eq: x_hat and y_hat def}, respectively. 

In Figure~\ref{fig: comparison of methods}, we demonstrate the advantages of the normalization~\eqref{eq: scaling rows and columns with estimated scaling factors} for recovering $X$ -- either by truncating the SVD after our normalization (Algorithm~\ref{alg:weighted low-rank approximation}) or via WLD with our estimated scaling factors $\hat{\mathbf{x}}$ and $\hat{\mathbf{y}}$. Specifically, we compare several methods for recovering $X$: a) Oracle singular value thresholding (SVT) of $Y$, i.e., $T_{r^{'}}\{Y\}$, where $r^{'}$ is set to minimize the error with respect to $X$ in Frobenius norm ($X$ is provided for optimally tuning $r^{'}$); b) Oracle singular value shrinkage (SVS) of $Y$, i.e., $\mathcal{T}_r\{Y\}$, where the parameters $\{\theta_{ij}\}$ are nonzero only for $i=j$ and are set to minimize the error with respect to $X$ in Frobenius norm ($X$ is provided for optimally tuning $\{\theta_{ii}\}$); c) WLD~\cite{leeb2021matrix}, where $\mathbf{x}$ and $\mathbf{y}$ are estimated as described in~\cite{leeb2021matrix} from the Euclidean norms of the rows and columns of $Y$ (listed as ``traditional normalization'' in the legend of Figure~\ref{fig: comparison of methods}); d) SVT at rank $\hat{r}$ after our proposed normalization~\eqref{eq: scaling rows and columns with estimated scaling factors}, i.e., Algorithm~\ref{alg:weighted low-rank approximation}; and (e) WLD~\cite{leeb2021matrix} combined with our proposed normalization~\eqref{eq: scaling rows and columns with estimated scaling factors}, namely, replacing $T_{\hat{r}}$ in~\eqref{eq: estimated solution to weighted low-rank approximation problem} with $\mathcal{T}_{\hat{r}}$. In this experiment, the dimensions are $m=1000$, $n=2000$, and the noise $E$ is Gaussian heteroskedastic as described for Figure~\ref{fig: spectrum before and after biwhitening} (see details in Section~\ref{sec: rank estimation}). The signal $X$ has $r=20$ nonzero singular values equal to $s$ and the singular vectors are orthonormalized Gaussian random vectors whose support was restricted to half the rows and half the columns. Hence, the signal $X$ is localized to a quarter of the entries. We then evaluated the relative mean squared error (MSE) of the recovery for each method, i.e., the squared Frobenius norm of the difference between each method's output and $X$, divided by $\Vert X \Vert_F^2$, followed by averaging the results over $10$ randomized experiments.

In Figure~\ref{fig: MSE vs. signal level}, we depict the relative MSE versus the signal strength $s$ while the noise heteroskedasticity level is fixed at $t=2$ (see the details of the noise generation in Section~\ref{sec: rank estimation}). It is evident that our proposed normalization combined with WLD provides the best performance. Very similar performance is achieved by Algorithm~\ref{alg:weighted low-rank approximation}, i.e., when truncating the singular values after our normalization, with a slight advantage to WLD for weak signals. However, when applying WLD with the traditional normalization proposed in~\cite{leeb2021matrix}, which relies on $(Y_{ij}^2)$ to estimate $\mathbf{x}$ and $\mathbf{y}$, the performance deteriorates significantly for strong signals since the corresponding estimators for $\mathbf{x}$ and $\mathbf{y}$ become highly inaccurate. Moreover, we see that truncating or shrinking the singular values of the original data $Y$, even if optimally tuned using an oracle that provides the true signal $X$, is always suboptimal to normalizing the data. This is most notable for weak signals, e.g., $s=2\sqrt{n}$, where the difference between the relative MSEs before and after normalization is substantial. 

In Figure~\ref{fig: MSE vs. hetero level}, we depict the relative MSE versus the noise heteroskedasticity level $t$ while the signal strength is fixed at $s=3\sqrt{n}$. It is evident that when no heteroskedasticity is present, i.e., $t=0$, all methods perform similarly. In this case, oracle SVS is almost identical to our normalization + WLD, while oracle SVT is almost identical to our normalization + SVT. Indeed, if the noise is homoskedastic, then no normalization should be performed, i.e., $\hat{\mathbf{x}}$ and $\hat{\mathbf{y}}$ should be close to vectors of all-ones (up to a possible global factor). Hence, we see that our normalization does not degrade the performance in this case. However, as $t$ increases, the noise becomes more heteroskedastic, and all normalization-based approaches significantly outperform traditional SVT or SVS applied to the original (un-normalized) data. 

The empirical results above demonstrate that for heteroskedastic noise, normalizing the data before applying thresholding or shrinkage of the singular value can provide a significant improvement in signal recovery. Intuitively, the reason for this advantage is that the normalization can reduce the spectral noise level while increasing the spectral gap between the signal components and the noise (see Section~\ref{sec: rank estimation} for an example and discussion), thereby enhancing the accuracy of the singular vectors, especially those corresponding to the small singular values (weak signals) that are below or near the original spectral noise level.  

\begin{figure} 
  \centering
  	{
  	\subfloat[][Relative MSE vs. signal strength]  
  	{
    \includegraphics[width=0.48\textwidth]{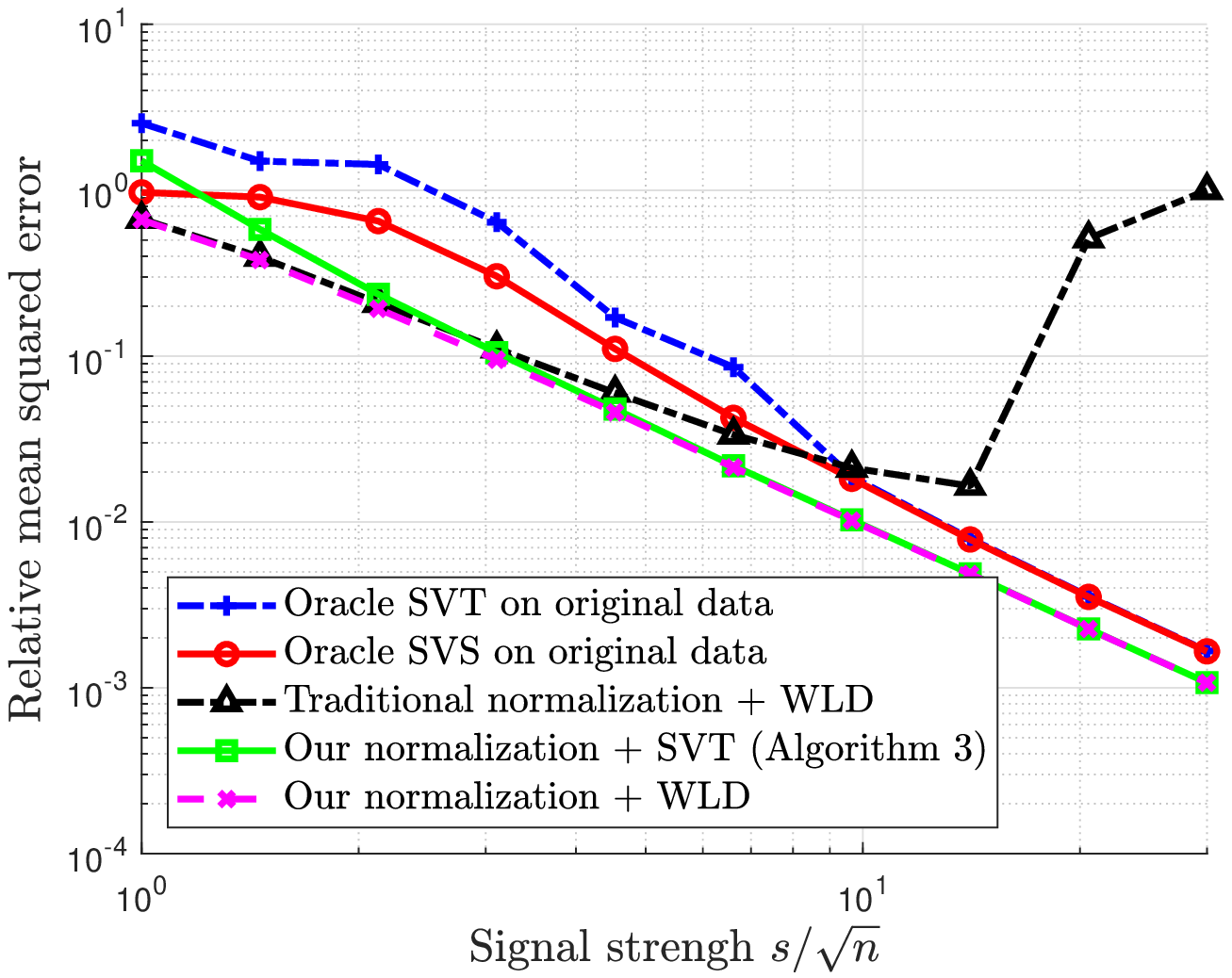} \label{fig: MSE vs. signal level}
    } 
    \subfloat[][Relative MSE vs. heteroskedasticity level] 
  	{
    \includegraphics[width=0.48\textwidth]{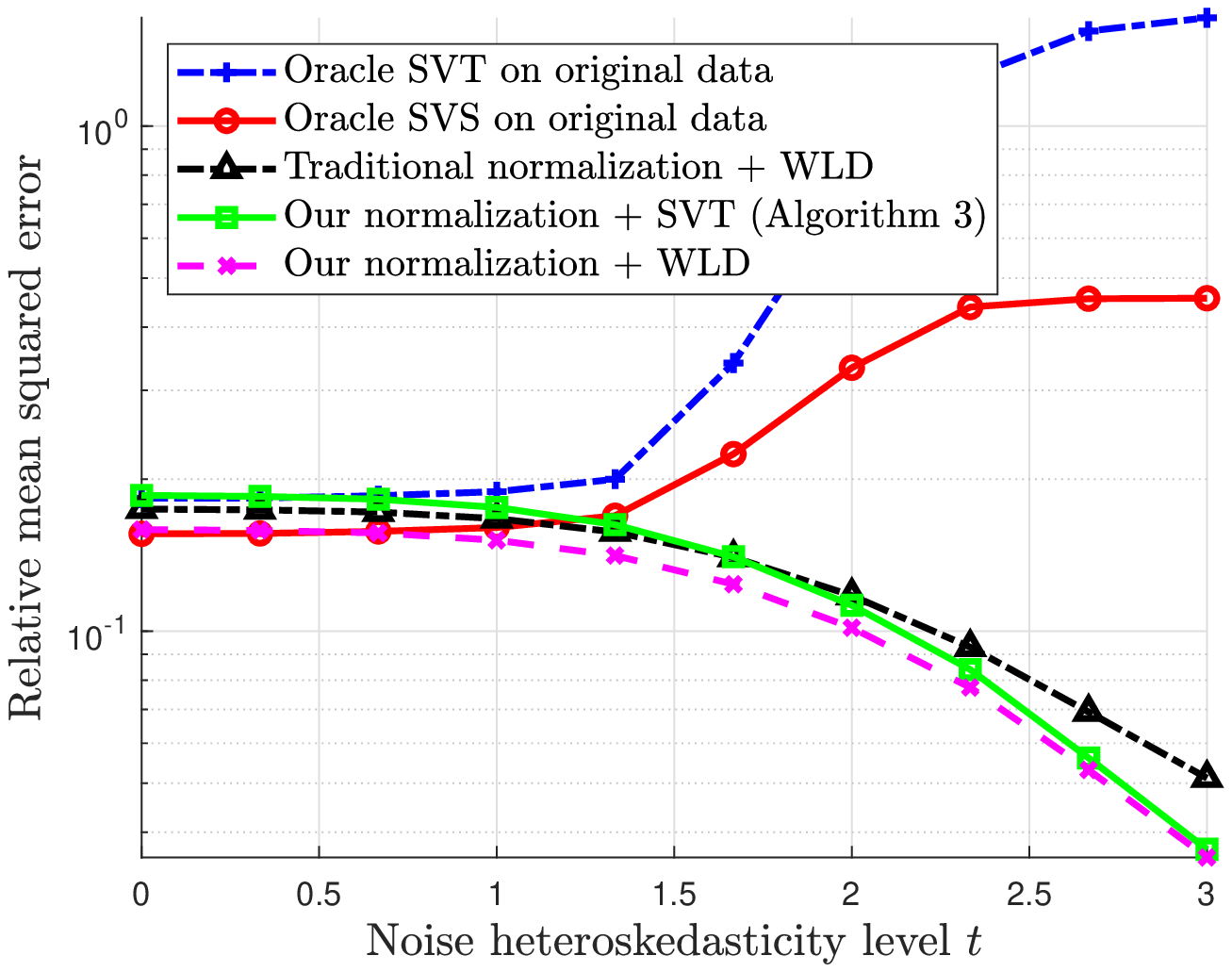} \label{fig: MSE vs. hetero level}
    }
    }
    \caption
    {Comparison of the relative MSE of recovering $X$ using different methods, where $m=1000$, $n=2000$, and $r=20$. The noise here is Gaussian heteroskedastic, generated as described in Section~\ref{sec: rank estimation} for Figure~\ref{fig: spectrum before and after biwhitening}. Panel (a) depicts the relative MSE versus the signal strength $s/\sqrt{n}$ while fixing the heteroskedasticity level at $t=2$. Panel (b) depicts the relative MSE versus the heteroskedasticity level $t$ while fixing the signal strength at $s=3 \sqrt{n}$.}  \label{fig: comparison of methods}
    \end{figure} 

\section{Examples on real data} \label{sec: examples on real data}
Next, we demonstrate the effect of our proposed normalization on the spectrum of experimental genomics data from Single-Cell RNA Sequencing (scRNA-seq) and Spatial Transcriptomics (ST). Specifically, we used the scRNA-seq dataset of purified peripheral blood mononuclear cells (PBMCs) of Zheng et al.~\cite{zheng2017massively} and the Visium ST dataset of human dorsolateral prefrontal cortex from Maynard et. al.~\cite{maynard2021transcriptome}. We compare our approach to the method of~\cite{landa2022biwhitening}, which assumes that the data satisfies a quadratic variance function (QVF), i.e., $S_{ij} = a + b X_{ij} + c X_{ij}^2$ with some coefficients $a,b,c$, for estimating the scaling factors $\mathbf{x}$ and $\mathbf{y}$. 
Since the entries of scRNA-seq and ST data are integers, such data are typically modeled by count random variables such as the Poisson or negative binomial, which satisfy the QVF assumption (with $a=0$, $b = 1$, $c=0$ for Poisson and $a=0$, $b=1$, and some $c>0$ for negative binomial).

We applied several preliminary preprocessing steps to the datasets to control their size and sparsity. For the scRNA-seq data, which originally contained $\sim 95000$ rows (cells) and $\sim 33000$ columns (genes), we randomly sampled $5000$ rows and then removed all rows and columns with less than $50$ nonzero entries, resulting in a matrix of size $\sim 5000\times 7500$. For the ST data, which originally contained $\sim 3600$ rows (pixels) and $\sim 33000$ columns (genes), we did not apply any downsampling, but we did apply an analogous sparsity filtering step with a threshold of $20$ nonzeros per row and column, resulting in a matrix of size $\sim 3600 \times 15000$.

Figures~\ref{fig: raw scRNA-seq},~\ref{fig: raw scRNA-seq + QVF scaling}, and~\ref{fig: raw scRNA-seq + Algorithm 1}, depict the empirical eigenvalue densities of the raw scRNA-seq data, the corresponding data after the normalization of~\cite{landa2022biwhitening} (where we used the procedure described therein to adaptively find the QVF parameters), and the data after Algorithm~\ref{alg:noise standardization}, respectively. Analogously, Figures~\ref{fig: raw ST},~\ref{fig: raw ST + QVF scaling}, and~\ref{fig: raw ST + Algorithm 1}, depict the corresponding empirical eigenvalue densities for the ST data. We mention that for visualization purposes, we normalized the raw scRNA-seq and ST data by a scalar to match the median of the eigenvalues of $Y Y^T/n$ to the median of the MP law. It is evident that the eigenvalue densities of the raw scRNA-seq and ST data are very different from the MP density and are much more spread out. On the other hand, after normalization of the rows and columns using the estimated scaling factors of either~\cite{landa2022biwhitening} or our proposed approach here, we obtain accurate fits to the MP law. 

Next, we applied a sequence of two common transformations to the scRNA-seq and ST data matrices. First, we normalized each row of the data matrices so that their entries sum to $1$. This step is known as library normalization and is used to mitigate the influence of technical variability of counts (also known as ``read depth'') across cell populations~\cite{vieth2019systematic,cole2019performance}. Second, we removed the empirical mean of each column (gene), which is a common step used for principal components analysis. The resulting data matrices contain negative entries, so we cannot apply the method of~\cite{landa2022biwhitening} directly. Instead, we used the absolute values of the transformed data entries to estimate the scaling factors $\mathbf{x}$ and $\mathbf{y}$ via the method of~\cite{landa2022biwhitening}. 

Figures~\ref{fig: transformed scRNA-seq},~\ref{fig: transformed scRNA-seq + QVF scaling},~\ref{fig: transformed scRNA-seq + Algorithm 1},~\ref{fig: transformed ST},~\ref{fig: transformed ST + QVF scaling}, and~\ref{fig: transformed ST + Algorithm 1} depict the empirical eigenvalue densities of the transformed data analogously to Figures~\ref{fig: raw scRNA-seq},~\ref{fig: raw scRNA-seq + QVF scaling},~\ref{fig: raw scRNA-seq + Algorithm 1},~\ref{fig: raw ST},~\ref{fig: raw ST + QVF scaling}, and~\ref{fig: raw ST + Algorithm 1}. 
Similarly to the raw data, we see that the eigenvalue density of the transformed data does not fit the MP law and is much more spread out. Yet, for the transformed data, the method of~\cite{landa2022biwhitening} no longer provides accurate fits to the MP law, arguably because the transformed data does not satisfy the required QVF condition. On the other hand, our proposed approach still provides excellent fits to the MP law for the transformed data of scRNA-seq and ST. It is worthwhile to note that the transformations applied to the data (row-normalization and mean subtraction) may introduce some dependence between the noise entries. Therefore, the empirical results also suggest that our approach is robust to certain violations of the model assumptions. In this case, the transformations applied to the data rely on empirical averages along the rows or the columns of the matrix, which are robust to independent noise and should converge to their population counterparts for large matrix dimensions. Hence, when using these quantities for the data transformation (library normalization and mean centering), the amount of induced statistical dependence between the noise entries is small. Finally, in Figure~\ref{fig: robustness to choice of eta} in Appendix~\ref{appendix: adapting to unknown global scaling of the noise}, we demonstrate that our approach is robust to the choice of $\eta$ in Algorithm~\ref{alg:noise standardization}, and in particular that other quantiles can be used instead of the median. Figure~\ref{fig: robustness to choice of eta} also shows that the estimated scaling factors $\hat{\mathbf{x}}$ and $\hat{\mathbf{y}}$ for the transformed purified PBMC dataset vary by orders of magnitude, suggesting that the data is severely heteroskedastic. 

The results above demonstrate that our method is applicable to real genomics data with strong heteroskedastic noise such as scRNA-seq and ST data, and can support common transformations of the data for downstream analysis. The main advantage of our approach is that it does not rely on the particular distributions of the noise entries nor their relation to the signal, whereas~\cite{landa2022biwhitening} requires a quadratic relation between the mean and the variance. Instead, our approach requires some degree of delocalization for the signal's singular vectors (see Assumption~\ref{assump: signal delocalization}), and variance matrices $S = D\{\mathbf{x}\} \widetilde{S} D\{\mathbf{y}\}$ that are either approximately rank-one or exhibit a degree of incoherence between $\widetilde{S}$ and the scaling factors $\mathbf{x}$ and $\mathbf{y}$ (see Assumption~\ref{assump: decay rate of g-h} and the discussion preceding it). 

We conclude this section by briefly discussing the above-mentioned assumptions in the context of scRNA-seq data to illustrate how they might be satisfied. To this end, consider a simple prototypical model where the data entries $Y_{ij}$ are sampled independently from $\operatorname{Poisson}(X_{ij})$ for some (entrywise) positive Poisson parameter matrix $X$. In this case, we have $X = S = D\{\mathbf{x}\} \widetilde{S} D\{\mathbf{y}\}$, where $\widetilde{S}$ is doubly regular (see Definition~\ref{def: doubly regular matrix}). The scaling factors $\mathbf{x}$ and $\mathbf{y}$ are associated with a quantity known as \textit{sequencing depth}~\cite{choudhary2022comparison,chen2019single}, describing the technical variation of count levels across cells (rows) or genes (columns), which depends on experimental conditions.
The matrix $\widetilde{S}$ describes the underlying normalized expression profiles of the cells across the genes and vice-versa, irrespective of experimental conditions that may vary the sequencing depth. Consequently, it is reasonable to expect $\widetilde{S}$ to be incoherent with respect to the factors $\mathbf{x}$ and $\mathbf{y}$. Since the Poisson parameter matrix $X$ is commonly assumed to be low-rank~\cite{linderman2022zero,kharchenko2021triumphs}, the normalized matrix $\widetilde{S}$ should also be low-rank. For instance, if $\widetilde{S}$ was generated randomly according to the low-rank model in Example~\ref{example: random S example 2} in Section~\ref{sec: variance matrices with general rank} while $\mathbf{x}$ and $\mathbf{y}$ are deterministic and depend only on the particular experimental setup, then both Assumptions~\ref{assump: signal delocalization} and~\ref{assump: decay rate of g-h} would be satisfied with high probability for a sufficiently large number of sequenced cells and genes.

\begin{figure} 
  \centering
  	{
  	\subfloat[][Raw scRNA-seq data]
  	{
    \includegraphics[width=0.3\textwidth]{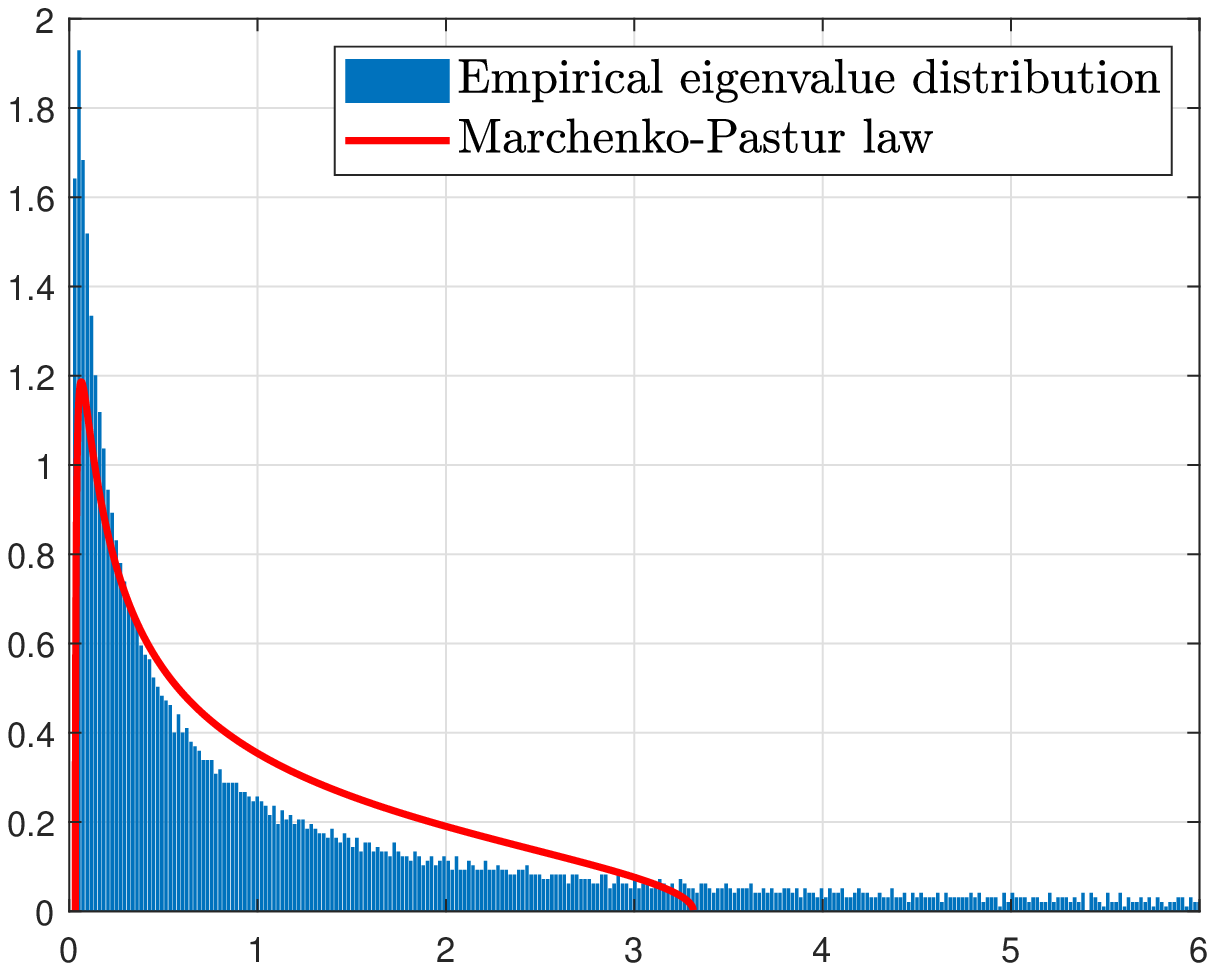} \label{fig: raw scRNA-seq}
    } \hspace{5pt}
    \subfloat[][Raw scRNA-seq + QVF scaling~\cite{landa2022biwhitening}] 
  	{
    \includegraphics[width=0.3\textwidth]{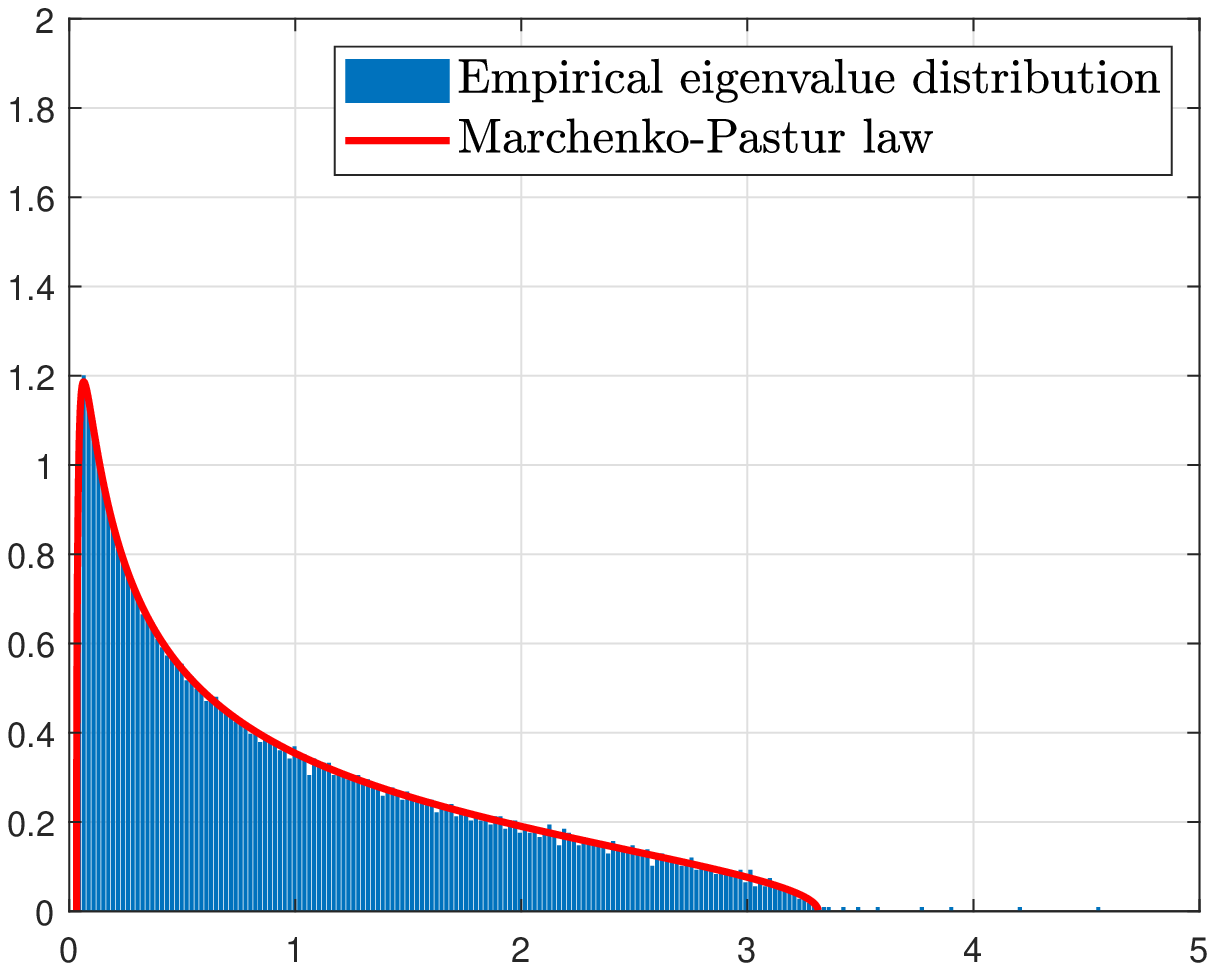} \label{fig: raw scRNA-seq + QVF scaling}
    } \hspace{5pt}
     \subfloat[][Raw scRNA-seq + Algorithm~\ref{alg:noise standardization}] 
    {
    \includegraphics[width=0.3\textwidth]{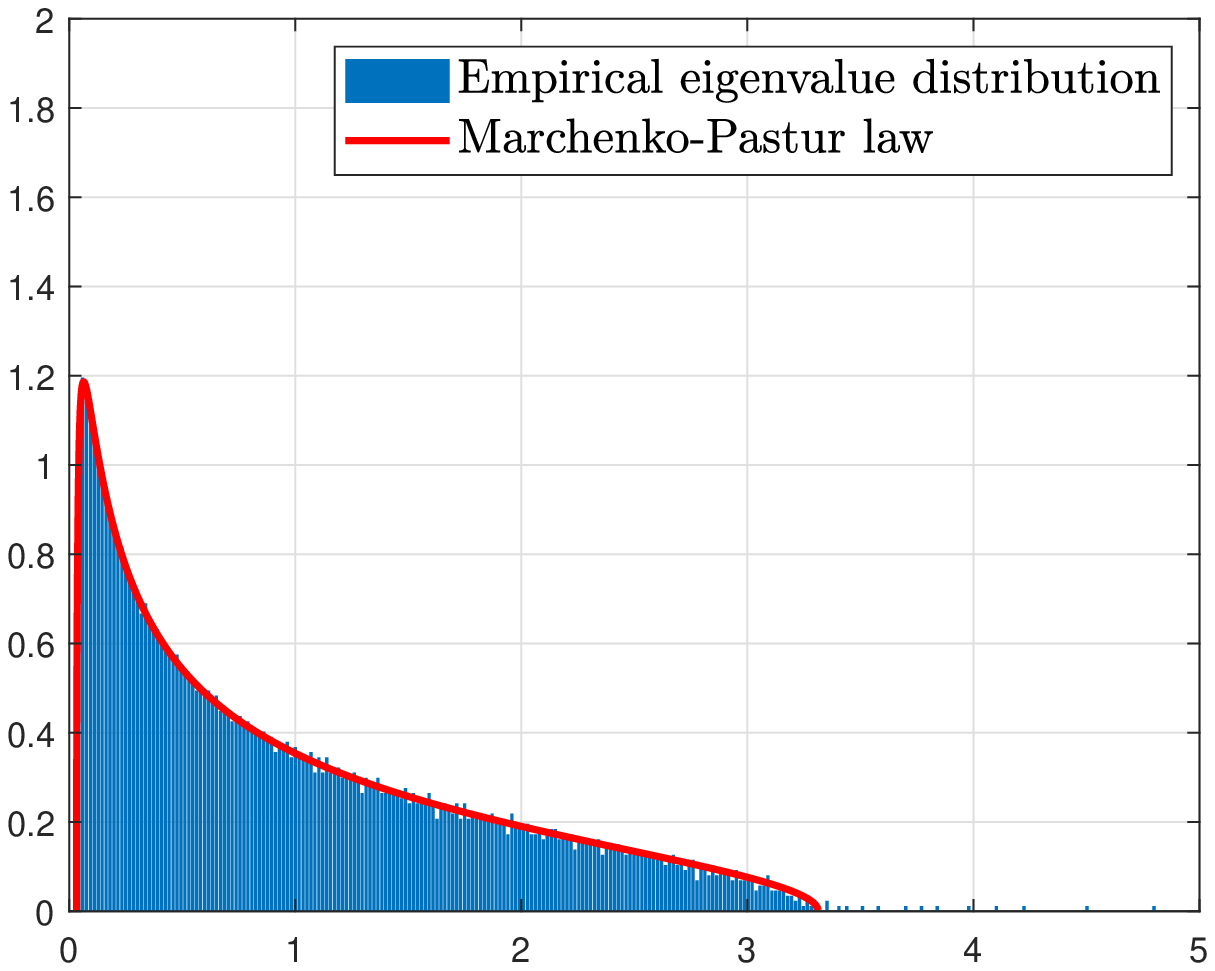} \label{fig: raw scRNA-seq + Algorithm 1}
    }
    \\
    \subfloat[][Raw ST data]  
  	{
    \includegraphics[width=0.3\textwidth]{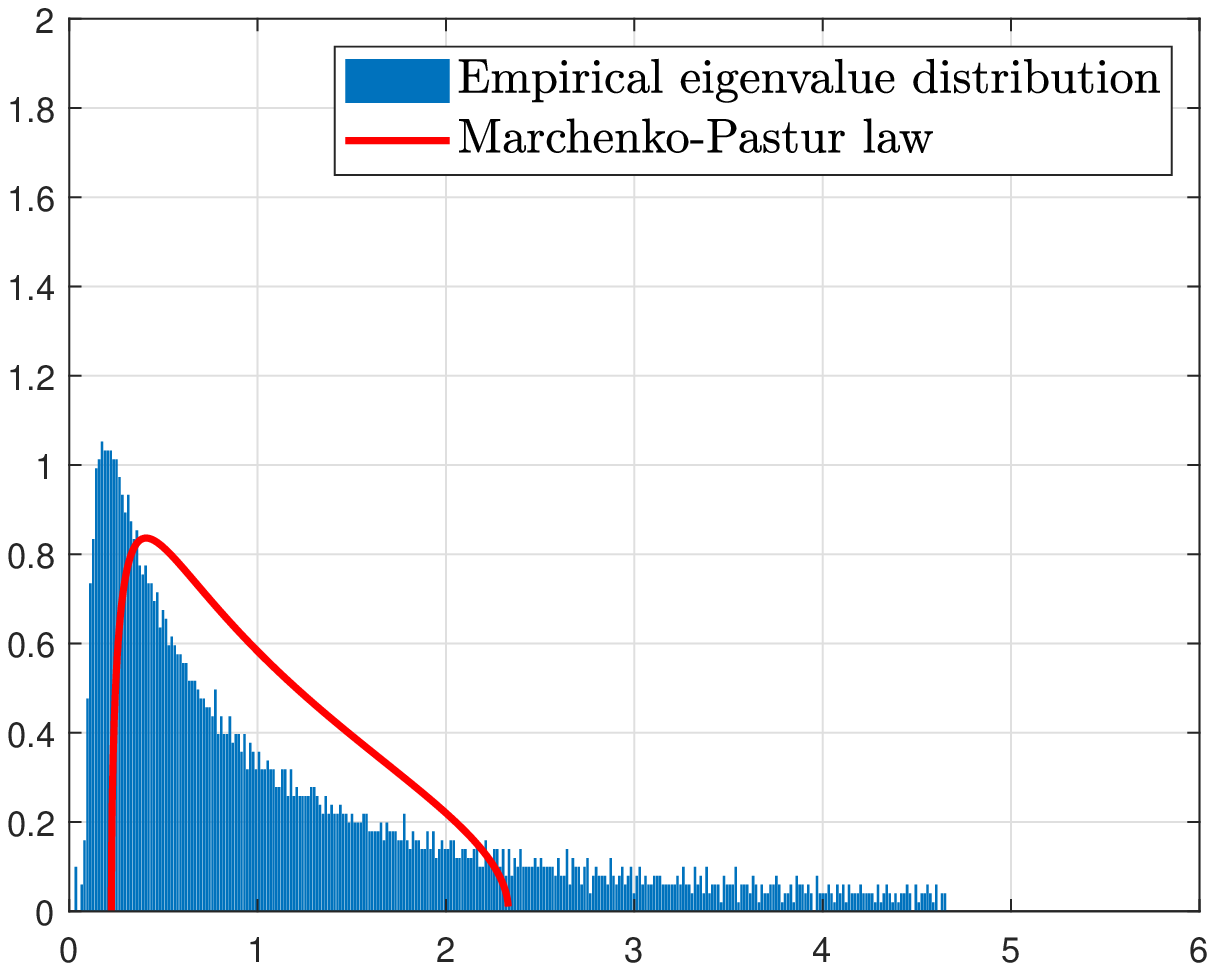}  \label{fig: raw ST}
    } \hspace{5pt}
    \subfloat[][Raw ST + QVF scaling~\cite{landa2022biwhitening}] 
  	{
    \includegraphics[width=0.3\textwidth]{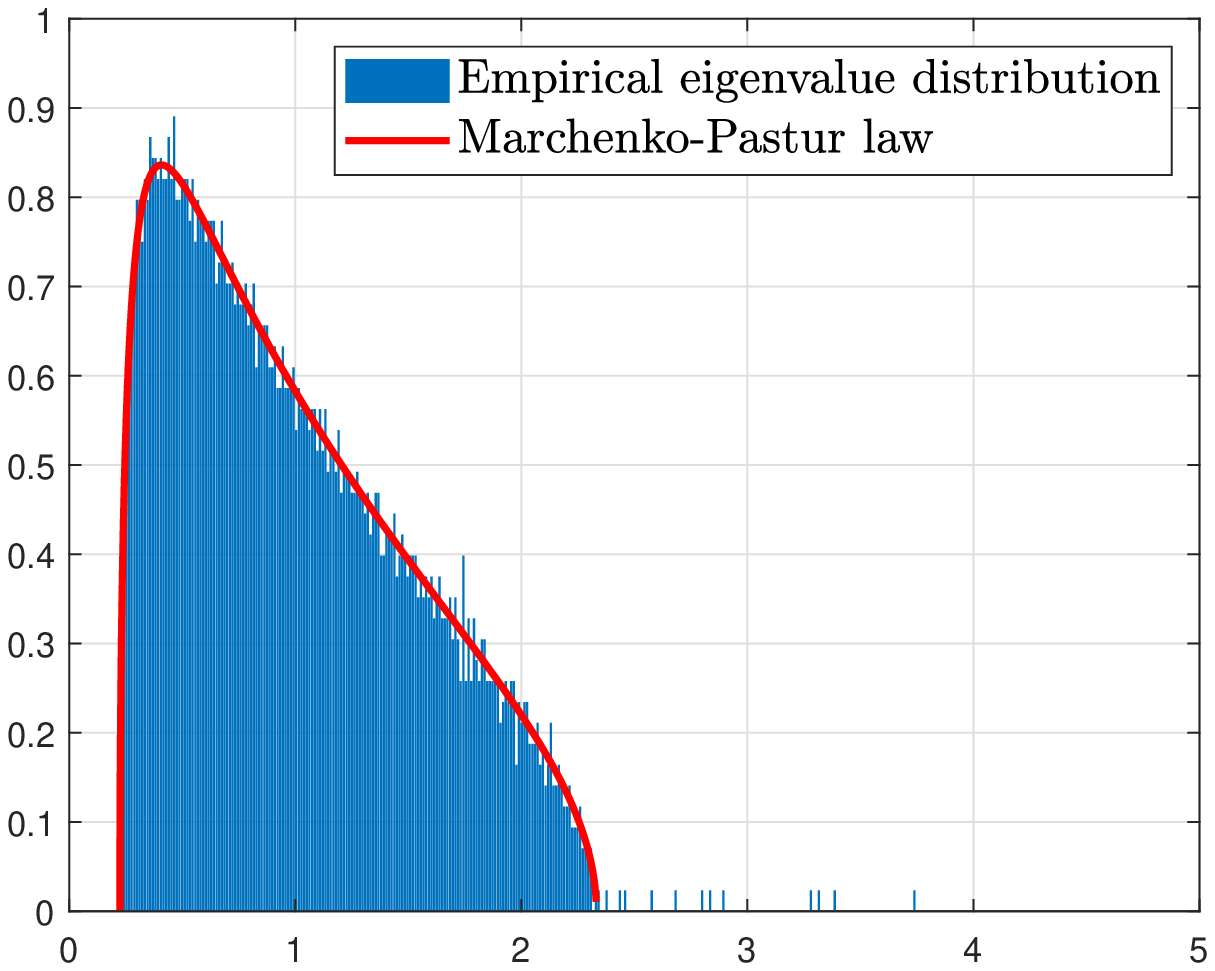} \label{fig: raw ST + QVF scaling}
    } \hspace{5pt}
    \subfloat[][Raw ST + Algorithm~\ref{alg:noise standardization}] 
  	{
    \includegraphics[width=0.3\textwidth]{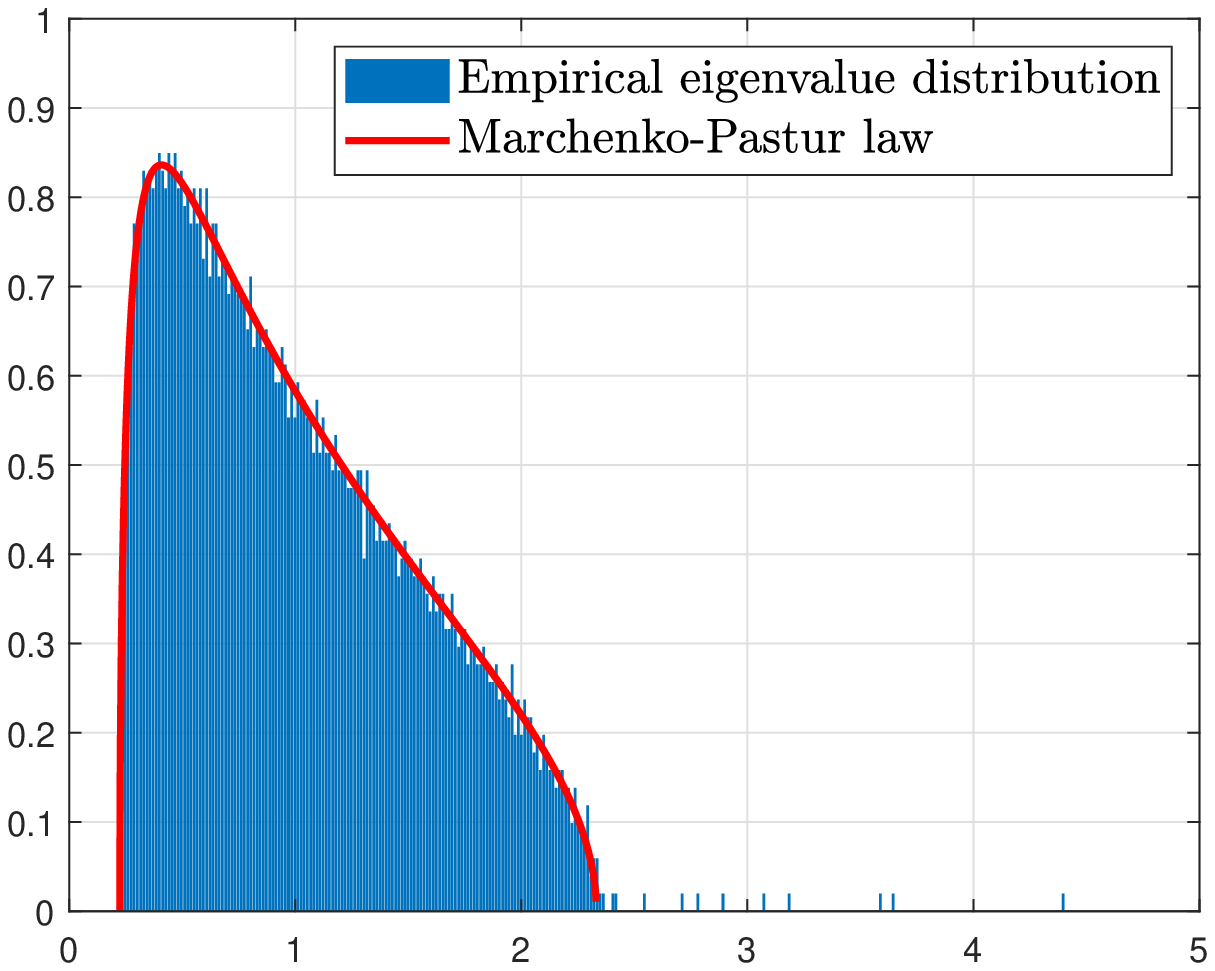} \label{fig: raw ST + Algorithm 1}
    } 
    \\
    \subfloat[][Transformed scRNA-seq data]  
  	{
    \includegraphics[width=0.3\textwidth]{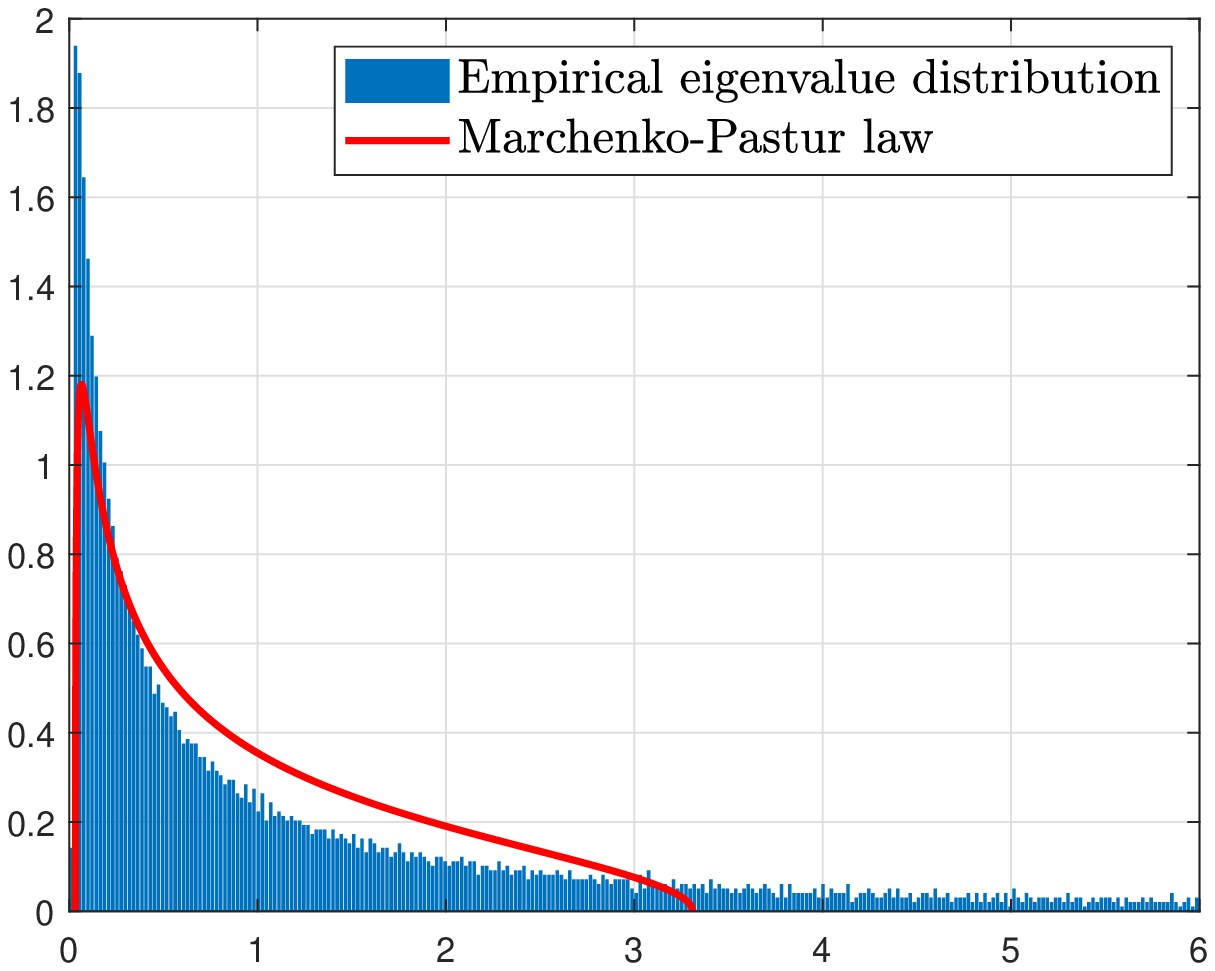}  \label{fig: transformed scRNA-seq}
    } \hspace{5pt}
    \subfloat[][Transformed scRNA-seq + QVF scaling~\cite{landa2022biwhitening}] 
  	{
    \includegraphics[width=0.3\textwidth]{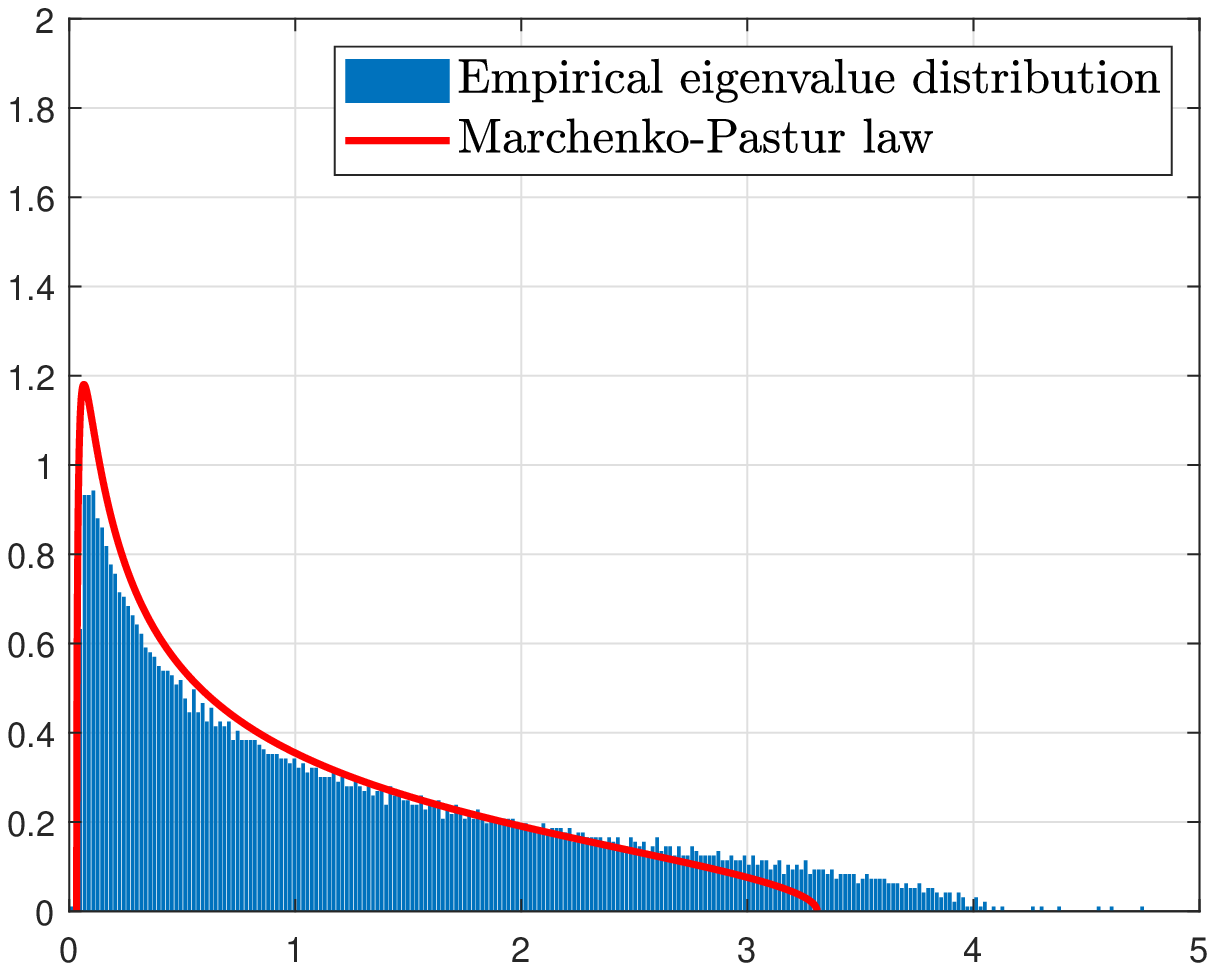} \label{fig: transformed scRNA-seq + QVF scaling}
    }  \hspace{5pt}
    \subfloat[][Transformed scRNA-seq + Algorithm~\ref{alg:noise standardization}] 
  	{
    \includegraphics[width=0.3\textwidth]{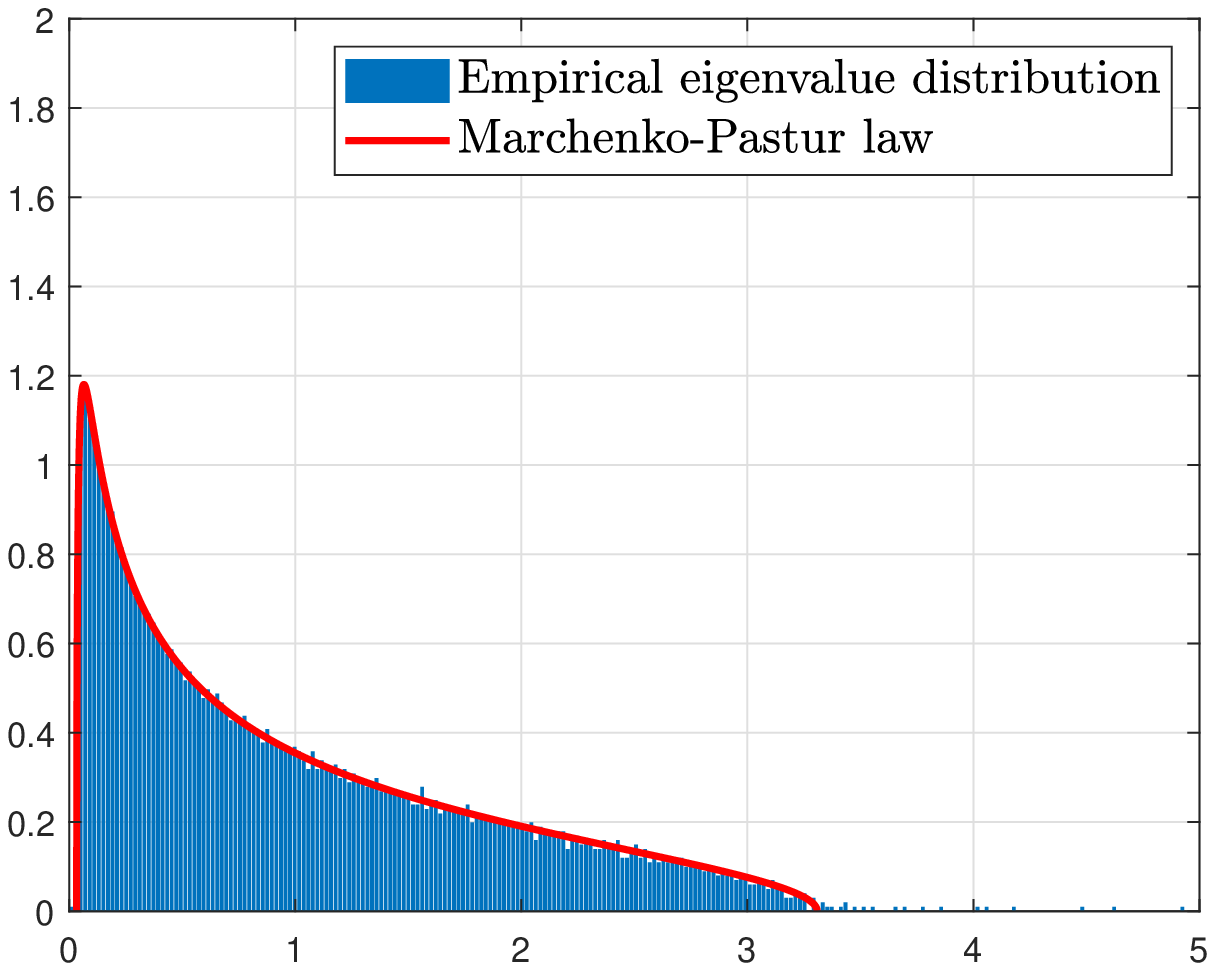} \label{fig: transformed scRNA-seq + Algorithm 1}
    } 
    \\
    \subfloat[][Transformed ST data]  
  	{
    \includegraphics[width=0.3\textwidth]{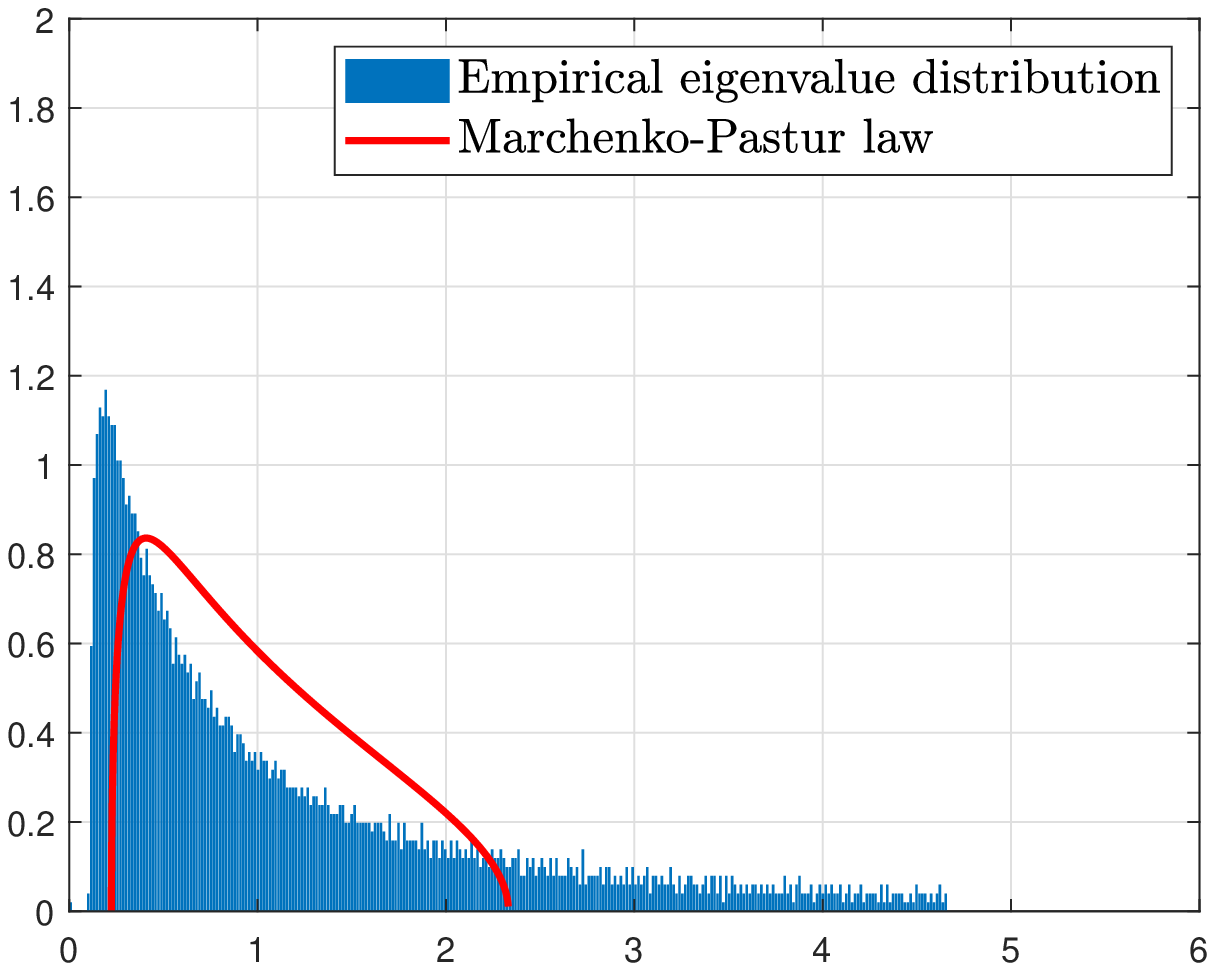}   \label{fig: transformed ST}
    } \hspace{5pt}
    \subfloat[][Transformed ST + QVF scaling~\cite{landa2022biwhitening}] 
  	{
    \includegraphics[width=0.3\textwidth]{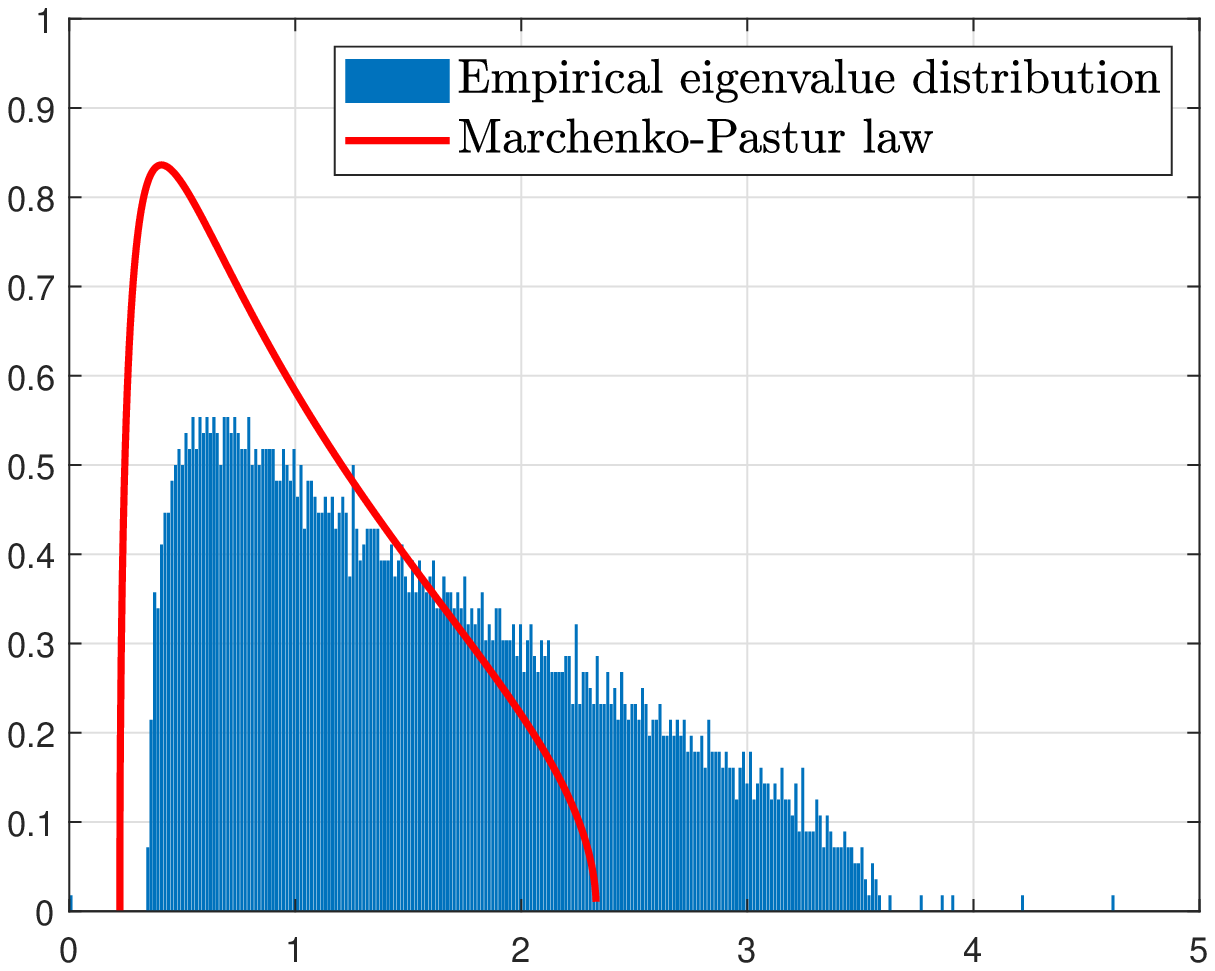} \label{fig: transformed ST + QVF scaling}
    }  \hspace{5pt}
    \subfloat[][Transformed ST + Algorithm~\ref{alg:noise standardization}] 
  	{
    \includegraphics[width=0.3\textwidth]{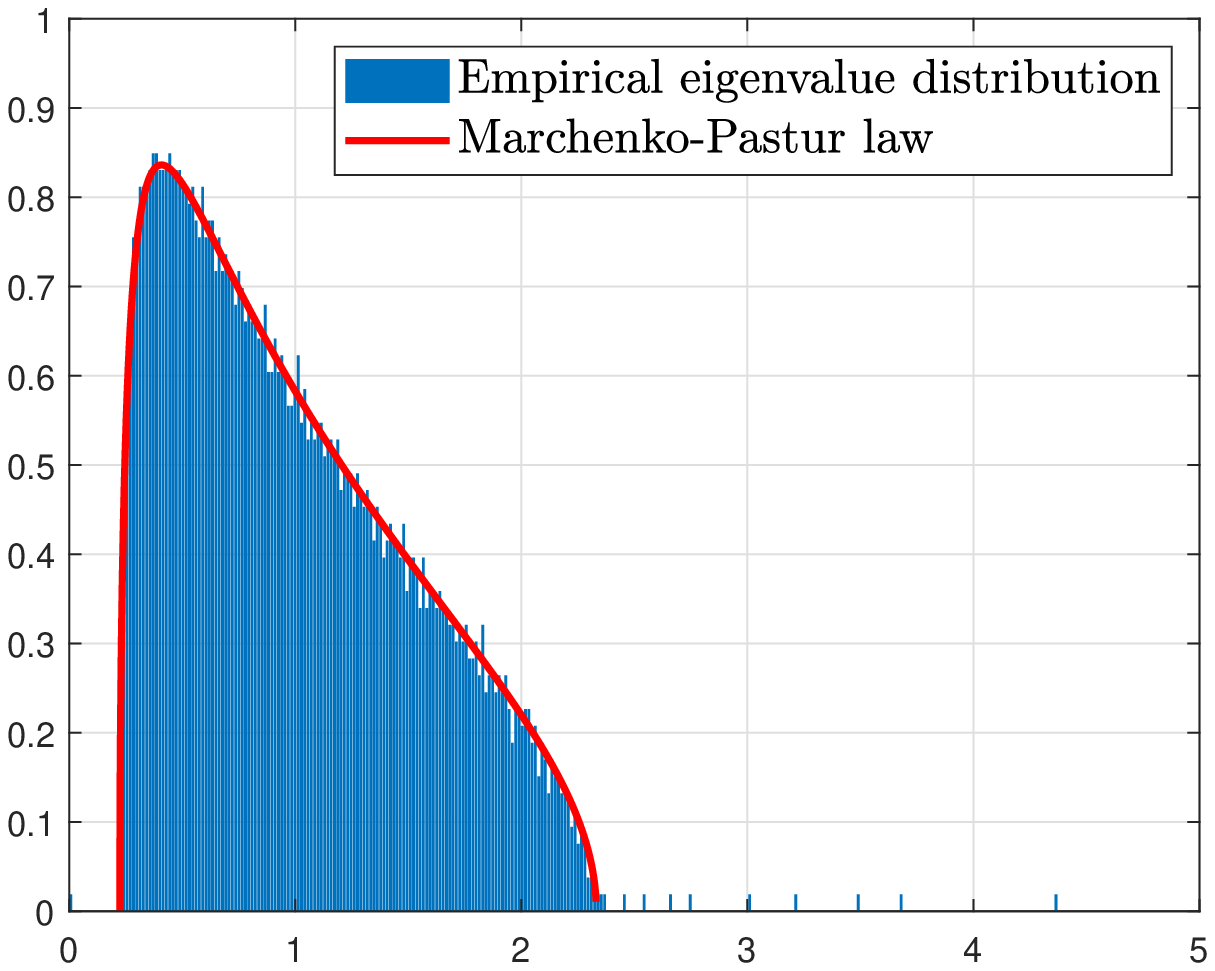} \label{fig: transformed ST + Algorithm 1}
    } 

    }
    \caption
    {Empirical eigenvalue densities of scRNA-seq~\cite{zheng2017massively} and ST~\cite{maynard2021transcriptome} data, before and after the normalization of the rows and columns. The scaling factors for the normalization were estimated using two approaches: 1) the QVF-based method of~\cite{landa2022biwhitening}; and 2) our proposed method (see Algorithm~\ref{alg:noise standardization}). The data was used either in its raw count form (panels (a) -- (f)) or after a standard transformation (library normalization followed by mean subtraction; see the corresponding text). } \label{fig:MP fit for real data}
    \end{figure}

\section{Summary and discussion} \label{sec: discussion}
In this work, we developed an adaptive bi-diagonal scaling procedure for equalizing the average noise variance across the rows and columns of a given data matrix. We analyzed the accuracy of our proposed procedure in a wide range of settings and provided theoretical and empirical evidence of its advantages for signal detection and recovery under general heteroskedastic noise. Our approach is particularly appealing from a practical standpoint: it is fully automatic and data-driven, it supports general noise distributions and a broad range of signals and noise variance structures, and perhaps most importantly, it provides an empirical validation of our model assumptions via the accuracy of the spectrum's fit to the MP law after normalization.

It is worthwhile to briefly discuss our delocalization assumption on the signal's singular vectors (see Assumption~\ref{assump: signal delocalization} in Section~\ref{sec: rank one case} and the subsequent text). Currently, this assumption prohibits highly localized singular vectors, e.g., with a finite number of nonzeros or a logarithmically-growing number of them (with respect to the growing dimensions of the matrix). A natural question is whether this assumption is required in practice and whether it can be relaxed. We conjecture that some amount of delocalization of the signal singular vectors is always necessary to treat heteroskedastic noise with a general variance matrix $S$. Otherwise, a single row or column in the data with abnormally strong noise can always be considered as a rank-one component in the signal (see, e.g., Figure~\ref{fig: spectrum before and after normalization, outlier rows anc columns}). In this case, there is no fundamental way of telling whether such a row or column belongs to the signal or the noise. Therefore, our approach implicitly assumes that highly localized singular vectors (i.e., whose energy is concentrated in a very small number of entries) belong to the noise and are thus suppressed by the normalization of the rows and columns. We believe that this is a desirable property for many applications.

Another question of interest is whether it is possible to accurately estimate the scaling factors $\mathbf{x}$ and $\mathbf{y}$ for general variance matrices $S$ without relying on their incoherence with respect to $\widetilde{S}$ (see Assumption~\ref{assump: decay rate of g-h} and Lemma~\ref{lem: g is close to h under incoherence} in Section~\ref{sec: variance matrices with general rank}). One potential direction is to consider an iterative application of our proposed normalization, i.e., to apply Algorithm~\ref{alg:noise standardization} consecutively to its own output. Our results in Section~\ref{sec: variance matrices with general rank} already show that the estimation error of the scaling factors improves if the scaling factors are close to vectors of all-ones. This fact implies that an iterative application of our normalization may improve the estimation accuracy if the previous round of normalization made the variance matrix closer to being doubly regular. We conducted preliminary numerical experiments of this idea, not reported in this work, which suggest that iterative application of Algorithm~\ref{alg:noise standardization} can indeed improve the quality of normalization in challenging regimes where $S$ does not abide by our incoherence requirements. However, a comprehensive investigation of this idea, especially the numerical stability and convergence of the iterative procedure, is beyond the scope of this work. 

Lastly, we discuss several possible future research directions. First, it is of interest to further investigate the advantages of the normalization~\eqref{eq: scaling rows and columns} for signal detection and recovery in general heteroskedastic settings. In particular, it is desirable to investigate the influence of this normalization on the spectral signal-to-noise ratio beyond the case of Gaussian noise with $S=\mathbf{x}\mathbf{y}^T$ (which was considered in~\cite{leeb2021matrix}) and characterize settings where signal detection and rank estimation improve. Second, it is of interest to analyze the recovery accuracy of the signal $X$ by singular value thresholding or shrinkage before and after the normalization~\eqref{eq: scaling rows and columns}, identifying scenarios where the improvement in recovery accuracy is substantial. Finally, it is desirable to further refine the presented theoretical results and identify potential for improved algorithms. One such direction is to refine Theorem~\ref{thm: MP law for E_hat} to characterize the fluctuations of the largest eigenvalue around $\beta_+$. To this end, a promising direction is to use the results derived in~\cite{ding2022tracy} to establish the Tracy-Widom law, allowing for a refined rank estimation procedure. Another direction of interest is to investigate the dependence of the estimation errors in $\hat{\mathbf{x}}$ and $\hat{\mathbf{y}}$ on the parameter $\eta$ and explore whether the results can be enhanced by optimizing over $\eta$, possibly also using multiple values of $\eta$ simultaneously. We leave such extensions for future work.

\section*{Acknowledgements}
The authors acknowledge funding support from NIH grants R01GM131642, UM1PA051410, R33DA047037, U54AG076043, U54AG079759, U01DA053628, P50CA121974, and R01GM135928.

\section*{Data availability statement}
No new data were generated or analyzed in support of this research.

\begin{appendices}

\section{Experiment reproducibility details} \label{appendix: reproducibility details}
\subsection{Figure~\ref{fig: spectrum before and after normalization, outlier rows anc columns}}
In this figure, the signal $X$ is of rank $20$ with $10$ identical singular values given by $\sqrt{10^3 n}$ and $10$ identical singular values given by $\sqrt{3n}$. The singular vectors were obtained by orthonormalizing independent random Gaussian vectors of suitable size. The noise matrix $E$ was generated as Gaussian homoskedastic with variance $1$, except that we amplified its last $5$ rows and $5$ columns by factors of $\sqrt{10}$ and $10$, respectively. Specifically, $E_{ij} \sim \mathcal{N}(0,S_{ij})$, $S = \mathbf{x}\mathbf{y}^T$, where $\mathbf{x} = [1,\ldots,1,10,10,10,10,10]^T$ and $\mathbf{y} = [1,\ldots,1,100,100,100,100,100]^T$. 

\subsection{Figure~\ref{fig: spectrum before and after normalization, general heteroskedastic noise}}
In this figure, the signal $X$ was generated in the same way as described for Figure~\ref{fig: spectrum before and after normalization, outlier rows anc columns}. The noise was generated according to $E_{ij} \sim \mathcal{N}(0,S_{ij})$, $S = A B$, where $A\in\mathbb{R}^{m\times 10}$, $B\in\mathbb{R}^{10\times n}$, and  $A_{ij},B_{ij}\sim \operatorname{exp}\{\mathcal{N}(0,2)\}$. We then normalized $S$ by a scalar so that its average entry is $1$. 

\section{Adaptive choice of $\eta$} \label{appendix: adapting to unknown global scaling of the noise}

The analysis in Section~\ref{sec: method derivation and analysis} relies on Assumptions~\ref{assump: noise moment bound} and~\ref{assump: variance boundedness}, which require a specific scaling of the moments of $E_{ij}$ with respect to the dimensions of the matrix. In particular, they require that the noise variances $S_{ij}$ scale like $n^{-1}$. If we further assume that $ 0 < d_1 \leq m/n \leq d_2 <1 $ for some global constants $d_1,d_2>0$, then with high probability, all singular values of $E$ are bounded from above and from below by positive global constants, i.e.,
\begin{equation}
    c_1 \leq \sigma_i\{E\} \leq c_2, \label{eq: singular value boundedness of the noise}
\end{equation}
for some constants $c_1,c_2>0$; see, e.g.,~\cite{alt2017local,ding2022tracy} and references therein. 

The scaling requirement on the noise variances $S_{ij} \propto n^{-1}$ may not naturally hold in applications. To account for other scalings of the noise, suppose that instead of $Y$, we observe
\begin{equation}
    Y^{'} = \alpha(m,n) Y,
\end{equation}
where $\alpha(m,n)$ is a scalar with an unknown dependence on the matrix dimensions $m$ and $n$.

Since $\alpha(m,n)$ is unknown, we propose to estimate it from the data. According to~\eqref{eq: singular value boundedness of the noise}, the singular values of the scaled noise $\alpha(m,n) E$ are bounded with high probability from above and from below by constant multiples of $\alpha(m,n)$. Hence, we can correct for the unknown scaling $\alpha(m,n)$, up to a global constant, if we know the singular values of $\alpha(m,n) E$. Since we do not have access to these singular values, we estimate $\alpha(m,n)$ directly from the scaled data $Y^{'}$. To this end, we propose to exploit the fact that the signal $X$ is low-rank by employing the median singular value of $Y^{'}$, which is robust to low-rank perturbations of $Y^{'}$ for sufficiently large $m$ and $n$. Specifically, since $Y^{'} = \alpha(m,n)X + \alpha(m,n)E$, where the rank of $X$ is of fixed and independent of $m$ and $n$, Weyl's inequality for singular values of sums of matrices (see Theorem 3.3.16 in~\cite{horn1994topics}) together with~\eqref{eq: singular value boundedness of the noise} imply that
\begin{equation}
    c_1 \alpha(m,n) \leq \operatorname{Median}\{\sigma_1\{Y^{'}\},\ldots,\sigma_m\{Y^{'}\}\} \leq c_2 \alpha(m,n), \label{eq: boundeness of the median singular value}
\end{equation}
with probability approaching $1$ as $n\rightarrow\infty$. Consequently, the median singular value of $Y^{'}$ is proportional to $\alpha(m,n)$ with high probability for large $m$ and $n$. We can then divide $Y{'}$ by its median singular value to attain the scaling of the noise required for the analysis in Section~\ref{sec: method derivation and analysis}. Alternatively, we can simply take 
\begin{equation}
    \eta = \operatorname{Median}\left\{\sigma_1\{Y^{'}\},\ldots,\sigma_m\{Y^{'}\}\right\},
\end{equation}
and proceed as usual -- treating $Y^{'}$ as $Y$. By doing so, $\hat{\mathbf{x}}$ and $\hat{\mathbf{y}}$ from~\eqref{eq: x_hat and y_hat def} would automatically account for the factor $\alpha(m,n)$, and would concentrate around $\alpha(m,n) \mathbf{x}$ and $\alpha(m,n) \mathbf{y}$, respectively. This follows from the fact that according to~\eqref{eq: formulas for g_hat_1 and g_hat_2 in terms of the SVD}, if we multiply $Y$ and $\eta$ each by $\alpha(m,n)$, then $\hat{\mathbf{g}}^{(1)}$ and $\hat{\mathbf{g}}^{(2)}$ would be divided by $\alpha(m,n)$, and consequently, $\hat{\mathbf{x}}$ and $\hat{\mathbf{y}}$ would each be multiplied by $\alpha(m,n)$. Overall, this discussion justifies Step~\ref{alg: step 2} in Algorithm~\ref{alg:noise standardization}.

We remark that instead of the median of the singular values, any other quantile can be used, as long as it is bounded from below away from $0$ and bounded from above away from $1$. Such quantiles would also satisfy the boundedness property~\eqref{eq: boundeness of the median singular value}. Therefore, our justification in this section for taking $\eta$ as the median of the singular values of the data would also go though for other quantiles. Our choice of using the median in Algorithm~\ref{alg:noise standardization} is primarily motivated by its simplicity and its popularity in the literature to estimate the global noise level (see, e.g.,~\cite{gavish2014optimal}). From a practical perspective, the median is a natural choice for making the method less sensitive to finite-sample fluctuations near the edges of the spectral distribution, as well as the influence of signal components or rows/columns with anomalous noise variances (which may lead to abnormally small or large singular values).

In Figure~\ref{fig: robustness to choice of eta} we illustrate the robustness of our approach to the choice of $\eta$. Specifically, we evaluated the estimated scaling factors $\hat{\mathbf{x}}$ and $\hat{\mathbf{y}}$ and the resulting fit to the MP law for the transformed purified PBMC data (see Section~\ref{sec: examples on real data}) for two different choices of $\eta$: one using the quantile $q=0.5$ (median) of the data's singular values and another using the quantile $q=0.25$. It is evident that an excellent fit to the MP law is attained similarly for the two choices of $\eta$. Moreover, the estimated scaling factors are very similar at the level of their individual entries, which is to be expected from our theoretical guarantees on the convergence of the estimated scaling factors to the true scaling factors. 

\begin{figure} [!t]
  \centering
  	{
  	\subfloat[][Fit to MP law, $q=0.5$]
  	{
    \includegraphics[width=0.38\textwidth]{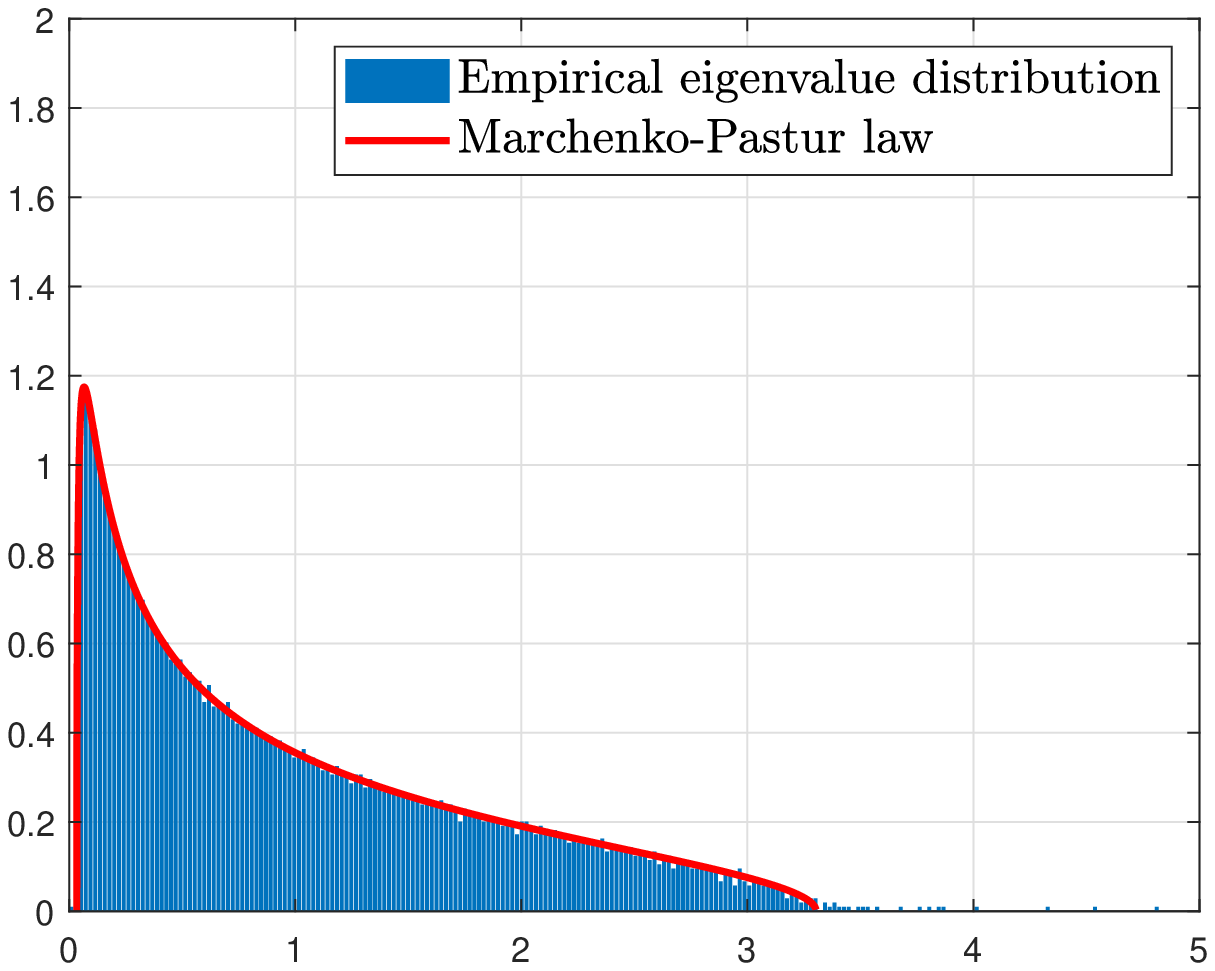} 
    } \hspace{20pt}
    \subfloat[][Fit to MP law, $q=0.25$] 
  	{
    \includegraphics[width=0.38\textwidth]{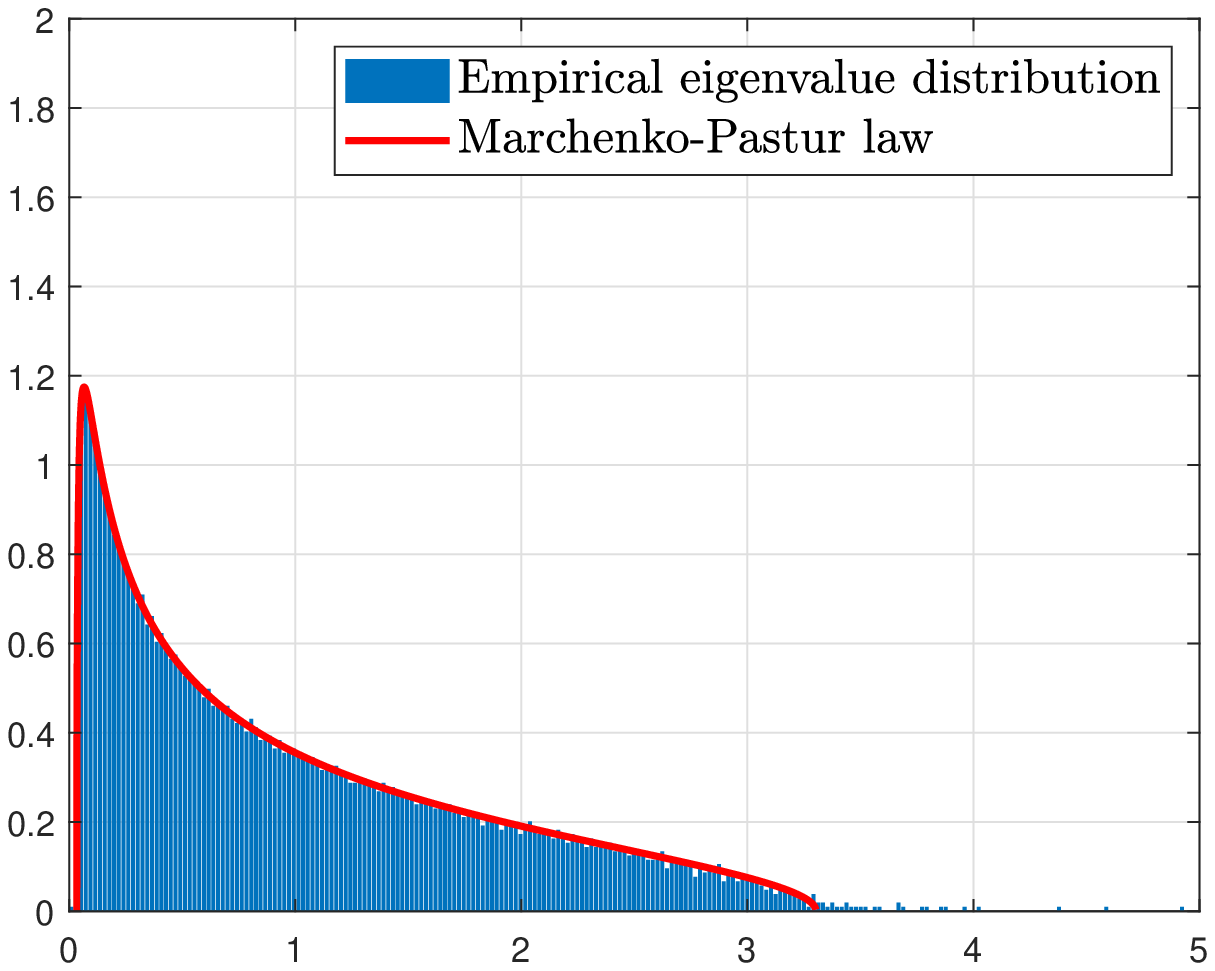} 
    }
    \vspace{20pt}
    \\
    \subfloat[][$\hat{\mathbf{x}}$ for $q=0.5$ vs. $\hat{\mathbf{x}}$ for $q=0.25$]  
  	{
    \includegraphics[width=0.38\textwidth]{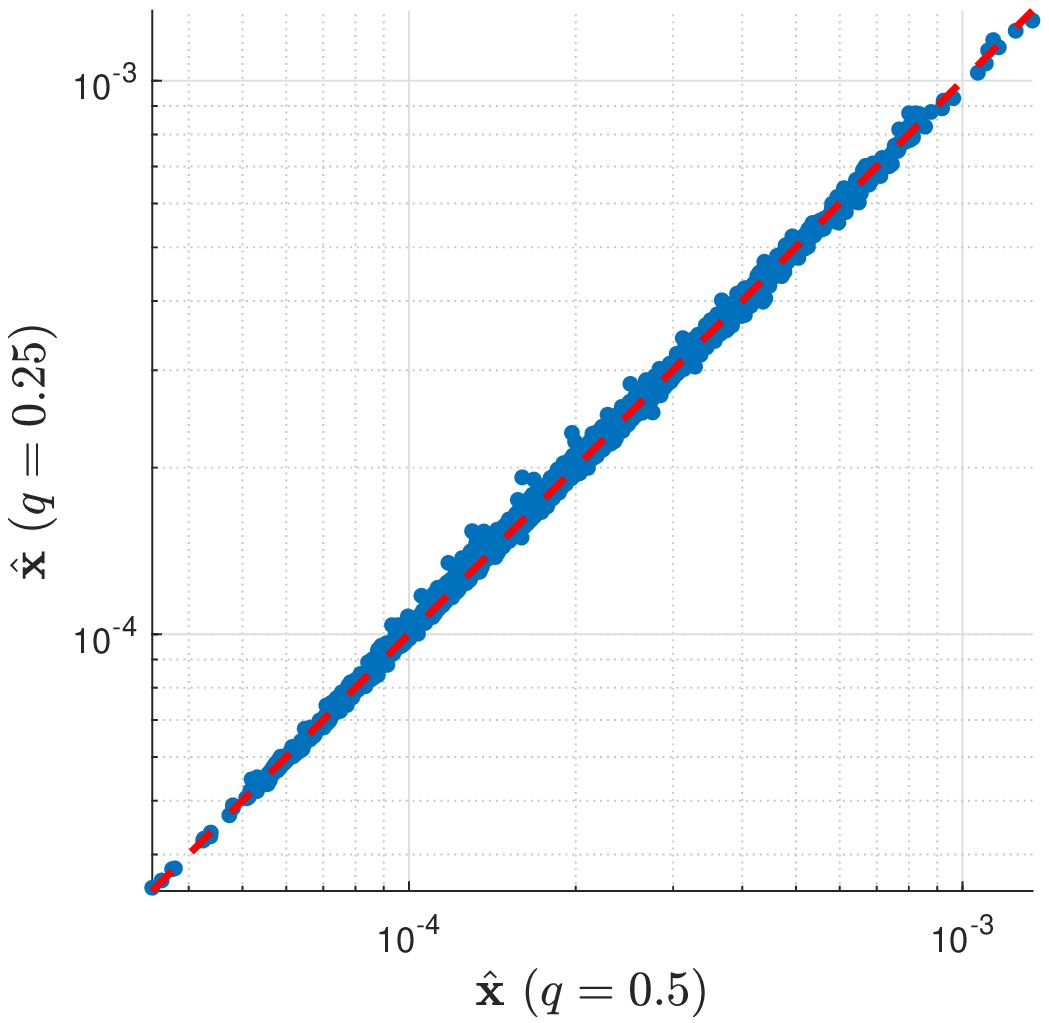}  
    } \hspace{20pt}
    \subfloat[][$\hat{\mathbf{y}}$ for $q=0.5$ vs. $\hat{\mathbf{y}}$ for $q=0.25$] 
  	{
    \includegraphics[width=0.38\textwidth]{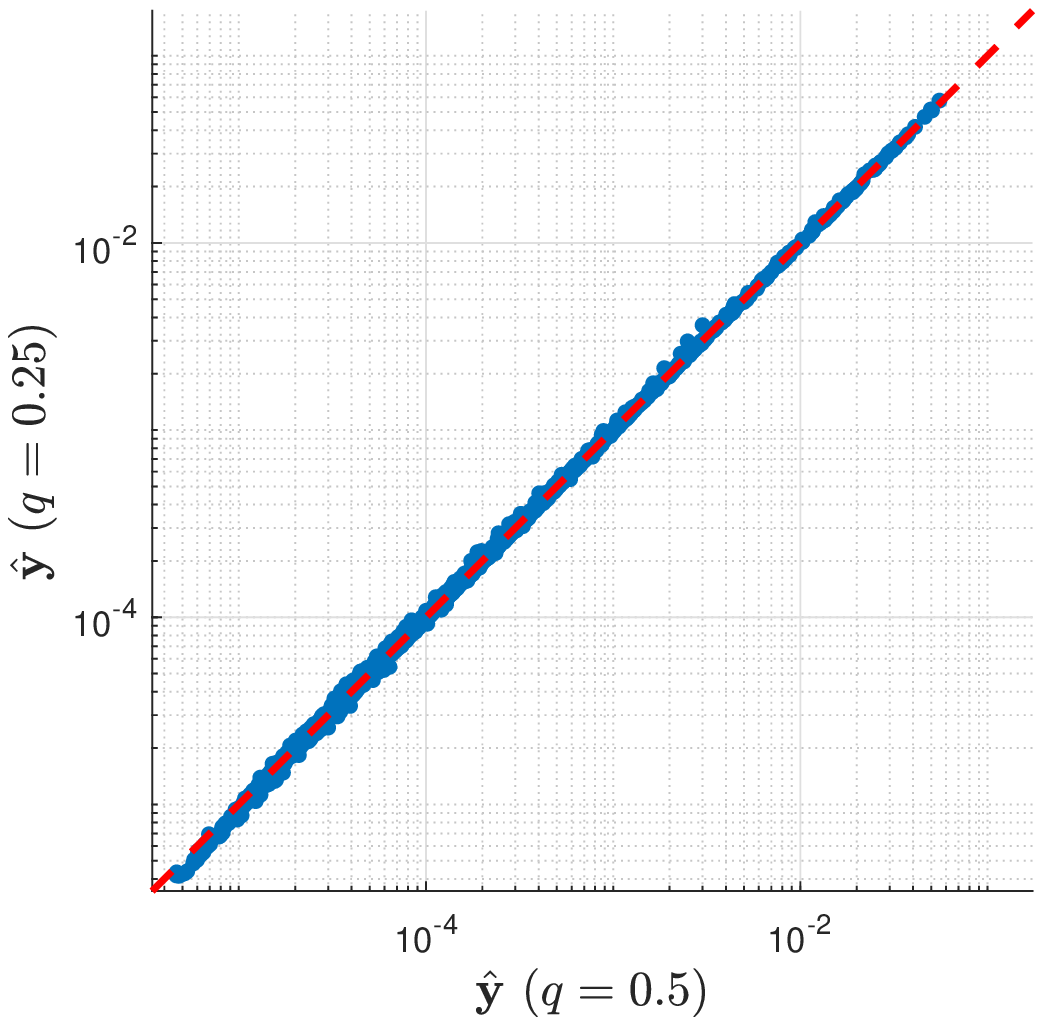} 
    } 
    }
    \caption
    {\textbf{Top row:} Empirical eigenvalue density vs. the MP law (solid red line) for the transformed scRNA-seq data~\cite{zheng2017massively} (see Section~\ref{sec: examples on real data}), where $\eta$ is chosen according to the $q$th quantile of the empirical singular values, for $q=0.5$ (left panel) and $q=0.25$ (right panel). \textbf{Bottom row:} Comparison of the entries of the estimated scaling factors $\hat{\mathbf{x}}$ (left panel) and $\hat{\mathbf{y}}$ (right panel) in logarithmic scale using the above two choices of $\eta$ (x-axis for $q=0.5$ and y-axis for $q=0.25$). The dashed red line (y-axis = x-axis) describes ideal perfect correspondence.
    } \label{fig: robustness to choice of eta}
    \end{figure} 

Finally, we briefly discuss the case of $m=n$ (or if $m$ is sufficiently close to $n$). In this case, eq.~\eqref{eq: singular value boundedness of the noise} does not hold in general, and some of the singular values of the noise $\sigma_i\{E\}$ may approach zero. Therefore, instability may arise in Algorithm~\ref{alg:noise standardization} if $\eta$ is close to zero. Note that in our analysis in Sections~\ref{sec: analysis for rank-one variance matrices} and~\ref{sec: variance matrices with general rank}, where we assume that the noise variances $S_{ij}$ scale like $n^{-1}$, we consider $\eta$ as a fixed global constant. Therefore, as long as $\eta$ is bounded from below away from zero for all sufficiently large matrix dimensions, our results would hold also for the case of $m=n$. Our proposed choice of $\eta$ via the median of the empirical singular values (or other quantiles as noted above) would enforce this property automatically, even if $m=n$, as long as the median of the noise singular values is bounded from below away from zero under the correct noise scaling $S_{ij} \propto n^{-1}$, in which case eq.~\eqref{eq: boundeness of the median singular value} would still hold.

\section{Auxiliary results and definitions}
\subsection{Boundedness of $\mathbf{g}$} \label{appendix: boundedness of g}
\begin{lem} \label{lem: boundedness of g}
Let $m\leq n$ and suppose that there exists a global constant $C > 0$ such that $ S_{ij} \leq Cn^{-1}$ for all $i,j$. Then,
\begin{equation}
    \left( \eta + \frac{C}{\eta} \right)^{-1} < \mathbf{g}_{i} < \frac{1}{\eta} \label{eq: boundedness of g_1 and g_2},
\end{equation}
for all $i\in[m+n]$, where $\mathbf{g}$ is the solution to~\eqref{eq: Dyson eq imag axis}.
\end{lem}
\begin{proof}
According to~\eqref{eq: Dyson eq imag axis} we have
\begin{equation}
    \eta < \eta + (\mathcal{S} \mathbf{g})_i = \frac{1}{\mathbf{g}_i}, \label{eq: boundedness of g}
\end{equation}
for all $i\in [m+n]$, where we used the fact that $\mathcal{S}\mathbf{g}$ is a positive vector. Therefore, we immediately obtain the upper bound in~\eqref{eq: boundedness of g_1 and g_2}. Plugging this upper bound in~\eqref{eq: Dyson eq imag axis} we obtain the lower bound
\begin{equation}
     \mathbf{g}_i  > \left({\eta + \frac{(\mathcal{S}\mathbf{1}_{m+n})_i}{\eta}} \right)^{-1}, 
\end{equation}
for all $i=1,\ldots,m+n$, where $\mathbf{1}_{m+n}$ is a vector of $m+n$ ones. Lastly, the assumption $S_{ij} \leq C n^{-1}$ implies that $(\mathcal{S} \mathbf{1}_{m+n})_i \leq C \max\{1,m/n\} \leq C $ for all $i=1,\ldots,m+n$ (since $m\leq n$), which completes the proof.
\end{proof}

\subsection{Boundedness of $\mathbf{x}$ and $\mathbf{y}$ for $S = \mathbf{x}\mathbf{y}^T$} \label{appendix: proof of boundedness of x and y}
\begin{lem} \label{lem: boundedness of x and y}
Let $m\leq n$ and $S = \mathbf{x} \mathbf{y}^T$. Suppose that there exists global constants $C,c>0$ such that $c n^{-1} \leq S_{ij} \leq C n^{-1}$ for all $i\in[m]$ and $j\in [n]$. Then, under the normalization~\eqref{eq: setting alpha=1}, we have 
\begin{align}
    \frac{1}{\sqrt{m}} \left( \frac{c}{\sqrt{C}} \right) \left( 1 + \frac{C}{\eta^2} \right)^{-3/2} &\leq \mathbf{x}_i \leq \frac{1}{\sqrt{m}} \left( \frac{C}{\sqrt{c}} \right) \left( 1 + \frac{C}{\eta^2} \right)^{3/2}, \\
    \frac{\sqrt{m}}{n} \left( \frac{c}{\sqrt{C}} \right) \left( 1 + \frac{C}{\eta^2} \right)^{-3/2} &\leq \mathbf{y}_j \leq \frac{\sqrt{m}}{n} \left( \frac{C}{\sqrt{c}} \right) \left( 1 + \frac{C}{\eta^2} \right)^{3/2},
\end{align}
for all $i\in[m]$ and $j\in [n]$.
\end{lem}

\begin{proof}
According to~\eqref{eq: setting alpha=1} and Lemma~\ref{lem: boundedness of g}, we have
\begin{equation}
     c \left( \eta + \frac{C}{\eta} \right)^{-1}  \leq \frac{c}{n} \sum_{j=1}^n\mathbf{g}_j^{(2)} \leq \left( S \mathbf{g}^{(2)}\right)_i = \mathbf{x}_i \left( \mathbf{y}^T \mathbf{g}^{(2)} \right) =  \mathbf{x}_i \left( \mathbf{x}^T \mathbf{g}^{(1)} \right) \leq \frac{\mathbf{x}_i}{\eta} \sum_{k=1}^m \mathbf{x}_k.
\end{equation}
Summing the above over $i=1,\ldots,m$ and taking the square root of both sides gives
\begin{equation}
    \sum_{i=1}^m \mathbf{x}_i \geq m^{1/2}\sqrt{\frac{c\eta }{\eta + C \eta^{-1}}}.
\end{equation}
Similarly, we have
\begin{equation}
    \left( \eta + \frac{C}{\eta} \right)^{-1} \mathbf{x}_i \sum_{k=1}^m \mathbf{x}_k\leq \mathbf{x}_i \left( \mathbf{x}^T \mathbf{g}^{(1)} \right) = \mathbf{x}_i \left( \mathbf{y}^T \mathbf{g}^{(2)} \right) =  \left( S \mathbf{g}^{(2)}\right)_i \leq \frac{C}{n} \sum_{j=1}^n \mathbf{g}_j^{(2)} \leq  \frac{C}{\eta},
\end{equation}
implying that
\begin{equation}
    \sum_{i=1}^m \mathbf{x}_i \leq m^{1/2} \sqrt{\frac{C}{\eta} \left( \eta + \frac{C}{\eta} \right) }.
\end{equation}
Combining all of the above and after some manipulation, we obtain that
\begin{equation}
    \frac{1}{\sqrt{m}} \left( \frac{c}{\sqrt{C}} \right) \left( 1 + \frac{C}{\eta^2} \right)^{-3/2} \leq \mathbf{x}_i \leq \frac{1}{\sqrt{m}} \left( \frac{C}{\sqrt{c}} \right) \left( 1 + \frac{C}{\eta^2} \right)^{3/2},
\end{equation}
for all $i\in[m]$. Next, we derive analogous bounds for $\mathbf{y}_j$. First, we have 
\begin{equation}
     c\left(\frac{m}{n} \right) \left( \eta + \frac{C}{\eta} \right)^{-1}  \leq \frac{c}{n} \sum_{i=1}^m\mathbf{g}_i^{(1)} \leq \left( [\mathbf{g}^{(1)}]^T S\right)_j = \left( [\mathbf{g}^{(1)}]^T \mathbf{x} \right) \mathbf{y}_j =  \left( [\mathbf{g}^{(2)}]^T \mathbf{y} \right) \mathbf{y}_j \leq \frac{\mathbf{y}_j}{\eta} \sum_{k=1}^n \mathbf{y}_k,
\end{equation}
which asserts that
\begin{equation}
    \sum_{j=1}^n \mathbf{y}_j \geq m^{1/2}\sqrt{\frac{c\eta }{\eta + C \eta^{-1}}}.
\end{equation}
Second, we can write
\begin{equation}
\left( \eta + \frac{C}{\eta} \right)^{-1} \mathbf{y}_j \sum_{k=1}^n \mathbf{y}_k\leq   \left( [\mathbf{g}^{(2)}]^T \mathbf{y} \right) \mathbf{y}_j = \left( [\mathbf{g}^{(1)}]^T \mathbf{x} \right) \mathbf{y}_j = \left( [\mathbf{g}^{(1)}]^T S\right)_j \leq \left(\frac{C}{n} \right) \sum_{i=1}^m \mathbf{g}_i^{(1)} \leq  \frac{C}{\eta} \left( \frac{m}{n} \right),
\end{equation}
which leads to
\begin{equation}
    \sum_{j=1}^n \mathbf{y}_j \leq m^{1/2} \sqrt{\frac{C}{\eta} \left( \eta + \frac{C}{\eta} \right) }.
\end{equation}
Finally, combining the last four equations gives
\begin{equation}
    \frac{\sqrt{m}}{n} \left( \frac{c}{\sqrt{C}} \right) \left( 1 + \frac{C}{\eta^2} \right)^{-3/2} \leq \mathbf{y}_j \leq \frac{\sqrt{m}}{n} \left( \frac{C}{\sqrt{c}} \right) \left( 1 + \frac{C}{\eta^2} \right)^{3/2},
\end{equation}
for all $j\in[n]$.
\end{proof}

\subsection{Boundedness of $\mathbf{x}$ and $\mathbf{y}$ for general $S$} \label{appendix: proof of boundedness of x and y for general S}
\begin{lem} \label{lem: boundedness of x and y for general S}
Let $m\leq n$ and suppose there exists global constants $C,c>0$ such that $c n^{-1} \leq S_{ij} \leq C n^{-1}$ for all $i\in[m]$ and $j\in [n]$. Then, under the normalization~\eqref{eq: setting alpha=1 for general variance matrices}, we have 
\begin{align}
    \frac{1}{\sqrt{m}} \left( \frac{c^{5/2}}{{C^2}} \right) \left( 1 + \frac{C^2}{c\eta^2} \right)^{-3/2} &\leq \mathbf{x}_i \leq \frac{1}{\sqrt{m}} \left( \frac{C^{5/2}}{{c^2}} \right) \left( 1 + \frac{C^2}{c\eta^2} \right)^{3/2}, \\
    \frac{\sqrt{m}}{n} \left( \frac{c^{5/2}}{{C^2}} \right) \left( 1 + \frac{C^2}{c\eta^2} \right)^{-3/2} &\leq \mathbf{y}_j \leq \frac{\sqrt{m}}{n} \left( \frac{C^{5/2}}{{c^2}} \right) \left( 1 + \frac{C^2}{c\eta^2} \right)^{3/2},
\end{align}
for all $i\in[m]$ and $j\in [n]$.
\end{lem}
\begin{proof}
According to~\eqref{eq: scaling factors def} we can write
\begin{equation}
    \widetilde{S}_{ij} = \frac{S_{ij}}{\mathbf{x}_i \mathbf{y}_j},
\end{equation}
where the sum of each row of $\widetilde{S}$ is $n$ and the sum of each column of $\widetilde{S}$ is $m$.
Therefore, applying Lemma 4.1 in~\cite{landa2022scaling} together with the assumptions in Lemma~\ref{lem: boundedness of x and y for general S}, we obtain 
\begin{equation}
    \frac{c}{C^2} n \leq \frac{1}{\mathbf{x}_i\mathbf{y}_j} \leq  \frac{C}{c^2} n,
\end{equation}
implying that
\begin{equation}
    \frac{c^2}{C} n^{-1} \leq {\mathbf{x}_i\mathbf{y}_j} \leq  \frac{C^2}{c} n^{-1}.
\end{equation}
Note that the normalization~\eqref{eq: setting alpha=1 for general variance matrices} is equivalent to~\eqref{eq: setting alpha=1} under Assumption~\ref{assump: rank one variance matrix} ($S = \mathbf{x}\mathbf{y}^T$). Since the proof of Lemma~\ref{lem: boundedness of x and y} relies on the assumption $cn^{-1} \leq \mathbf{x}_i \mathbf{y}_j \leq Cn^{-1}$, it also holds in the setting of this lemma if we replace $c$ and $C$ in the statement of Lemma~\ref{lem: boundedness of x and y} with $c^2/C$ and $C^2/c$, respectively, which provides the required result.
\end{proof}

\subsection{Stability of the Dyson equation on the imaginary axis} \label{appendix: stability of Dyson equation}
\begin{lem} \label{lem: stability of Dyson equation}
Let $\widetilde{\mathbf{g}}\in\mathbb{R}^{m+n}$ be a positive vector satisfying
\begin{equation}
    \frac{1}{\widetilde{\mathbf{g}}} = \eta + \mathcal{S} \widetilde{\mathbf{g}} + \mathbf{e}, \label{eq: g_tilde and e def}
\end{equation}
for some $\mathbf{e}\in\mathbb{R}^{m+n}$. Then, 
\begin{equation}
    \left\Vert \frac{\widetilde{\mathbf{g}} - \mathbf{g}}{\mathbf{g}} \right\Vert_\infty \leq \Vert \mathbf{e} \Vert_\infty \max_i \{ \widetilde{\mathbf{g}}_i\} \max \left\{ 1, \frac{2}{\eta \min_{i} \{ \widetilde{\mathbf{g}}_i \}} \right\},
\end{equation}
where $\mathbf{g}$ is the solution to~\eqref{eq: Dyson eq imag axis}.
\end{lem}

\begin{proof}
According to~\eqref{eq: g_tilde and e def} we can write
\begin{align}
    1 - D\{\widetilde{\mathbf{g}}\} \mathbf{e} 
    &= \eta \widetilde{\mathbf{g}} + D\{\widetilde{\mathbf{g}}\} \mathcal{S} \widetilde{\mathbf{g}} 
    = \eta \widetilde{\mathbf{g}} + D\left\{\frac{\widetilde{\mathbf{g}}}{\mathbf{g}}\right\} D\{{\mathbf{g}}\} \mathcal{S} D\{{\mathbf{g}}\} \left[ \frac{\widetilde{\mathbf{g}}}{\mathbf{g}} \right] \nonumber \\
    &\geq \eta \widetilde{\mathbf{g}} + \min_i \left\{ \frac{\widetilde{\mathbf{g}}_i}{\mathbf{g}_i} \right\} D\left\{\frac{\widetilde{\mathbf{g}}}{\mathbf{g}}\right\} D\{{\mathbf{g}}\} \mathcal{S} \mathbf{g} \nonumber \\
    &= \eta \widetilde{\mathbf{g}} \left( 1 - \min_i \left\{ \frac{\widetilde{\mathbf{g}}_i}{\mathbf{g}_i} \right\} \right) + \left[ \frac{\widetilde{\mathbf{g}}}{\mathbf{g}} \right] \min_i \left\{ \frac{\widetilde{\mathbf{g}}_i}{\mathbf{g}_i} \right\}, \label{eq: 1 - g_tilde e lower bound}
\end{align}
where we used $\mathcal{S} \mathbf{g} = 1/\mathbf{g} - \eta$ (see~\eqref{eq: Dyson eq imag axis}) in the last transition. Similarly, we have 
\begin{equation}
    1 - D\{\widetilde{\mathbf{g}}\} \mathbf{e} \leq \eta \widetilde{\mathbf{g}} \left( 1 - \max_i \left\{ \frac{\widetilde{\mathbf{g}}_i}{\mathbf{g}_i} \right\} \right) + \left[ \frac{\widetilde{\mathbf{g}}}{\mathbf{g}} \right] \max_i \left\{ \frac{\widetilde{\mathbf{g}}_i}{\mathbf{g}_i} \right\}. \label{eq: 1 - g_tilde e upper bound}
\end{equation}
Let us denote
\begin{equation}
    k = \operatorname{argmax}_{i\in [m+n]}\left\{ \frac{\widetilde{\mathbf{g}}_i}{\mathbf{g}_i} \right\}, \qquad \ell = \operatorname{argmin}_{i\in [m+n]}\left\{ \frac{\widetilde{\mathbf{g}}_i}{\mathbf{g}_i} \right\}.
\end{equation}
Using~\eqref{eq: 1 - g_tilde e lower bound} we obtain 
\begin{equation}
    1 + \widetilde{\mathbf{g}}_k \vert \mathbf{e}_k\vert 
    \geq 1 - \widetilde{\mathbf{g}}_k \mathbf{e}_k 
    \geq \eta \widetilde{\mathbf{g}}_k \left( 1 - \min_{i} \left\{ \frac{\widetilde{\mathbf{g}}_i}{\mathbf{g}_i} \right\} \right) +  \max_{i} \left\{ \frac{\widetilde{\mathbf{g}}_i}{\mathbf{g}_i} \right\} \min_{i} \left\{ \frac{\widetilde{\mathbf{g}}_i}{\mathbf{g}_i} \right\},
\end{equation}
and analogously, using~\eqref{eq: 1 - g_tilde e upper bound} we have 
\begin{equation}
    1 - \widetilde{\mathbf{g}}_\ell \vert \mathbf{e}_\ell\vert 
    \leq 1 - \widetilde{\mathbf{g}}_\ell \mathbf{e}_\ell 
    \leq \eta \widetilde{\mathbf{g}}_\ell \left( 1 - \max_{i} \left\{ \frac{\widetilde{\mathbf{g}}_i}{\mathbf{g}_i} \right\} \right) +  \max_{i} \left\{ \frac{\widetilde{\mathbf{g}}_i}{\mathbf{g}_i} \right\} \min_{i} \left\{ \frac{\widetilde{\mathbf{g}}_i}{\mathbf{g}_i} \right\}.
\end{equation}
Subtracting one of the two equations above from the other implies that
\begin{equation}
    \widetilde{\mathbf{g}}_k \left( 1 - \min_{i} \left\{ \frac{\widetilde{\mathbf{g}}_i}{\mathbf{g}_i} \right\} \right) + \widetilde{\mathbf{g}}_\ell \left( \max_{i} \left\{ \frac{\widetilde{\mathbf{g}}_i}{\mathbf{g}_i} \right\} - 1 \right) \leq  \frac{\widetilde{\mathbf{g}}_k \vert \mathbf{e}_k\vert + \widetilde{\mathbf{g}}_\ell \vert  \mathbf{e}_\ell\vert}{\eta}, 
\end{equation}
and consequently,
\begin{equation}
    \left( 1 - \min_{i} \left\{ \frac{\widetilde{\mathbf{g}}_i}{\mathbf{g}_i} \right\} \right) + \left(  \max_{i} \left\{ \frac{\widetilde{\mathbf{g}}_i}{\mathbf{g}_i} \right\} - 1 \right) 
    \leq \frac{2 \Vert \mathbf{e} \Vert_\infty}{\eta} \left( \frac{\max_{i} \{ \widetilde{\mathbf{g}}_i \}}{\min_{i} \{ \widetilde{\mathbf{g}}_i \}} \right). \label{eq: varpesdef}
\end{equation}
Next, we consider four different cases that correspond to the possible signs of the two summands in the left-hand side of~\eqref{eq: varpesdef}:
\begin{enumerate}
    \item \underline{$\min_{i} \left\{ {\widetilde{\mathbf{g}}_i}/{\mathbf{g}_i} \right\} \leq 1$ and $\max_{i} \left\{ {\widetilde{\mathbf{g}}_i}/{\mathbf{g}_i} \right\} \geq 1$:} In this case, both summands in the left-hand side of~\eqref{eq: varpesdef} are nonnegative, thus
    \begin{equation}
        - \frac{2 \Vert \mathbf{e} \Vert_\infty}{\eta} \left( \frac{\max_{i} \{ \widetilde{\mathbf{g}}_i \}}{\min_{i} \{ \widetilde{\mathbf{g}}_i \}} \right) \leq \min_{i} \left\{ \frac{\widetilde{\mathbf{g}}_i}{\mathbf{g}_i} \right\} -1 \leq \frac{\widetilde{\mathbf{g}}_i}{\mathbf{g}_i} - 1
        \leq \max_{i} \left\{ \frac{\widetilde{\mathbf{g}}_i}{\mathbf{g}_i} \right\} -1 \leq  \frac{2 \Vert \mathbf{e} \Vert_\infty}{\eta} \left( \frac{\max_{i} \{ \widetilde{\mathbf{g}}_i \}}{\min_{i} \{ \widetilde{\mathbf{g}}_i \}} \right),
    \end{equation}
    for all $i\in[m+n]$.
    \item \underline{$\min_{i} \left\{ {\widetilde{\mathbf{g}}_i}/{\mathbf{g}_i} \right\} \leq 1$ and $\max_{i} \left\{ {\widetilde{\mathbf{g}}_i}/{\mathbf{g}_i} \right\} < 1$:} In this case we have $\widetilde{\mathbf{g}}_i<\mathbf{g}_i$ for all $i\in [m+n]$, and therefore we can use~\eqref{eq: Dyson eq imag axis} to write
    \begin{equation}
        \frac{1}{\mathbf{g}} = \eta + \mathcal{S}\mathbf{g} > \eta + \mathcal{S}\widetilde{\mathbf{g}} = \frac{1}{\widetilde{\mathbf{g}}} - \mathbf{e},
    \end{equation}
    where we used~\eqref{eq: g_tilde and e def} in the last transition. Consequently, we obtain
    \begin{equation}
        1 > \frac{\widetilde{\mathbf{g}}_i}{\mathbf{g}_i} > 1 - \widetilde{\mathbf{g}}_i \mathbf{e}_i \geq 1 - \Vert \mathbf{e} \Vert_\infty \max_i \{ \widetilde{\mathbf{g}}_i\},
    \end{equation}
    for all $i\in[m+n]$.
    \item \underline{$\min_{i} \left\{ {\widetilde{\mathbf{g}}_i}/{\mathbf{g}_i} \right\} > 1$ and $\max_{i} \left\{ {\widetilde{\mathbf{g}}_i}/{\mathbf{g}_i} \right\} \geq 1$:}
    In this case we have $\widetilde{\mathbf{g}}_i>\mathbf{g}_i$ for all $i\in [m+n]$, and therefore we can use~\eqref{eq: Dyson eq imag axis} to write
    \begin{equation}
        \frac{1}{\mathbf{g}} = \eta + \mathcal{S}\mathbf{g} < \eta + \mathcal{S}\widetilde{\mathbf{g}} = \frac{1}{\widetilde{\mathbf{g}}} - \mathbf{e},
    \end{equation}
    where we used~\eqref{eq: g_tilde and e def} in the last transition. Consequently, we obtain
    \begin{equation}
        1 < \frac{\widetilde{\mathbf{g}}_i}{\mathbf{g}_i} < 1 - \widetilde{\mathbf{g}}_i \mathbf{e}_i \leq 1 + \Vert \mathbf{e} \Vert_\infty \max_i \{ \widetilde{\mathbf{g}}_i\},
    \end{equation}
    for all $i\in[m+n]$.
    \item \underline{$\min_{i} \left\{ {\widetilde{\mathbf{g}}_i}/{\mathbf{g}_i} \right\} > 1$ and $\max_{i} \left\{ {\widetilde{\mathbf{g}}_i}/{\mathbf{g}_i} \right\} < 1$:} This case is clearly infeasible.
\end{enumerate}
Overall, considering all possible cases mentioned above, we conclude that
\begin{equation}
    \left\vert \frac{\widetilde{\mathbf{g}}_i}{\mathbf{g}_i} -1 \right\vert \leq \Vert \mathbf{e} \Vert_\infty \max_i \{ \widetilde{\mathbf{g}}_i\} \max \left\{ 1, \frac{2}{\eta \min_{i} \{ \widetilde{\mathbf{g}}_i \}} \right\},
\end{equation}
for all $i\in [m+n]$.

\end{proof}

\subsection{Operator norm bound on the inverse of certain complex matrices}
The following lemma is useful for the proof of Theorem~\ref{thm: robustness of the resolvent}.
\begin{lem} \label{lem: operator norm bound on inverse of complex matrix}
    Let $A_R,A_I\in\mathbb{R}^{N\times N}$ be symmetric and non-singular matrices, where $A_I$ is positive definite, and define the complex-valued matrix $A = A_R + \imath A_I \in \mathbb{C}^{N\times N}$. Then, $A$ is invertible and
    \begin{equation}
        \Vert A^{-1}\Vert_2 \leq \frac{1}{\sqrt{\Vert A_R^{-1} \Vert_2^{-2} + \Vert A_I^{-1} \Vert_2^{-2}}},
    \end{equation}
    where $\Vert \cdot \Vert_2$ denotes the operator norm over $\mathbb{C}^N$.
\end{lem}

\begin{proof}
    Since $A_I$ is positive definite and $A_R$ is symmetric, we can simultaneously diagonalize $A_R$ and $A_I$. To this end, we first write $A_I = Q \Lambda Q^T$, where $Q$ is orthogonal and $\Lambda$ is a diagonal matrix containing the (positive) eigenvalues of $A_I$. We then define $P = Q \Lambda^{-1/2}$. Since $A_R$ is symmetric, the matrix $P^T A_R P$ is also symmetric, and we write its eigen-decomposition as $ V D V^T = P^T A_R P$, where $V$ is orthogonal and $D$ is a diagonal matrix containing the eigenvalues of $P^T A_R P$. Overall, the matrices $A_R$ and $A_I$ are simultaneously diagonalizable using the matrix $PV$ as
    \begin{equation}
        (PV)^T A_I (PV) = I_N, \qquad (PV)^T A_R (PV) = D, \label{eq: simultaneous diagonalization}
    \end{equation}
    where $I_N$ is the $N\times N$ identity matrix. Note that $PV = Q \Lambda^{-1/2} V$ is invertible with $(PV)^{-1} = V^T \Lambda^{1/2} Q^T$. Therefore
    \begin{equation}
        A_I = (PV)^{-T} (PV)^{-1}, \qquad A_R = (PV)^{-T} D (PV)^{-1}, \label{eq: simultaneous diagonalization inverse}
    \end{equation}
    where $(PV)^{-T}$ denotes the inverse transpose of $PV$.
    Consequently, $A$ is invertible according to
    \begin{equation}
        A^{-1} = (A_R + \imath A_I)^{-1} = \left[ (PV)^{-T} (D+\imath I_N) (PV)^{-1} \right]^{-1} = (PV) (D+\imath I_N)^{-1} (PV)^T.
    \end{equation}
We can now write
    \begin{equation}
        \Vert A^{-1} \Vert_2 = \Vert (PV) (D+\imath I_N)^{-1} (PV)^T \Vert_2 \leq \Vert PV \Vert_2^2 \Vert (D+\imath I_N)^{-1} \Vert_2 = \frac{\Vert \Lambda^{-1/2} \Vert_2^2}{ \sqrt{\min_i \{ D_{i,i}^2 \} + 1}},
    \end{equation}
where we used the fact that $PV =  Q \Lambda^{-1/2} V$ with $Q$ and $V$ orthogonal, and also the fact that $(D+\imath I_N)^{-1}$ is a diagonal matrix whose operator norm is the largest absolute value on the main diagonal. 
Since $D = (PV)^T A_R (PV)$ with $A_R$ and $PV$ non-singular, it follows that $D$ is non-singular and
\begin{equation}
    \min_i \{ D_{i,i}^2 \} = \frac{1}{\Vert D^{-1} \Vert_2^2} \geq \frac{1}{\Vert (PV)^{-1} \Vert_2^4 \Vert A_R^{-1} \Vert_2^2} = \frac{1}{\Vert \Lambda^{1/2} \Vert_2^4 \Vert A_R^{-1} \Vert_2^2} = \frac{\Vert \Lambda^{-1/2} \Vert_2^4}{\Vert A_R^{-1} \Vert_2^2}. 
\end{equation}
Finally, by combining the previous two inequalities, we obtain
\begin{equation}
    \Vert A^{-1} \Vert_2 \leq \frac{\Vert \Lambda^{-1/2} \Vert_2^2}{ \sqrt{\Vert \Lambda^{-1/2} \Vert_2^4 \Vert A_R^{-1} \Vert_2^{-2}  + 1}} = \frac{1}{\sqrt{\Vert A_R^{-1} \Vert_2^{-2} + \Vert \Lambda^{-1/2} \Vert_2^{-4}}} = \frac{1}{\sqrt{\Vert A_R^{-1} \Vert_2^{-2} + \Vert A_I^{-1} \Vert_2^{-2}}},
\end{equation}
where we used the fact that $\Vert A_I^{-1} \Vert^{-2}_2 = \Vert Q^T \Lambda^{-1} Q \Vert^{-2}_2 = \Vert \Lambda^{-1} \Vert^{-2}_2 = \Vert \Lambda^{-1/2} \Vert^{-4}_2$.
\end{proof}

\subsection{Order with high probability}
For simplicity of presentation and brevity of our proofs, we will make use of the following definition.
\begin{defn} \label{def: order with high probability}
Let $X$ be a complex-valued random variable. We say that $X = \mathcal{O}_{n}\left(f(m,n)\right)$ if there exist $C^{'}, c^{'}(t)> 0$ (which may depend on $\eta$ and other global constants), such that for all $t>0$, with probability at least $1-c^{'}(t) n^{-t}$, we have 
\begin{equation}
    |X| \leq C^{'} f(m,n).
\end{equation}
\end{defn}

The following two facts make it easy to analyze expressions involving multiple variables satisfying Definition~\ref{def: order with high probability} and functions thereof.
\begin{enumerate}
    \item Let $X_1,X_2,\ldots,X_{P(n)}$ be random variables satisfying $X_i = \mathcal{O}_{n}(f(m,n))$, where $P(n)$ is a polynomial in $n$. Then, by Definition~\ref{def: order with high probability}, applying the union bound $P(n)$ times yields
\begin{equation}
    \max_{i=1,\ldots,P(n)} \vert X_i \vert = \mathcal{O}_{n}(f(m,n)). \label{eq: definition order with high probability property 1}
\end{equation}
\item Let $X$ be a real-valued random variable satisfying $X = \mathcal{O}_{n}(f(m,n))$ and define $Y = g(X)$, where $g:\mathbb{R} \rightarrow \mathbb{R}$ is a differentiable function in a neighborhood of zero. Suppose that $\lim_{n\rightarrow \infty} f(m,n) = 0$ and that the first derivative of $g(\cdot)$ is uniformly bounded by a global constant in a neighborhood of zero. Then, by Definition~\ref{def: order with high probability} and the Taylor expansion of $g(\cdot)$ around zero, it follows that
\begin{equation}
    Y = g(0) + \mathcal{O}_{n}(f(m,n)). \label{eq: definition order with high probability property 2}
\end{equation}
\end{enumerate}
We will use the above two properties of Definition~\ref{def: order with high probability} seamlessly throughout our proofs below.

\section{Proof of Lemma~\ref{lem: concentration of bilinear forms}} \label{appendix: proof of concentration of biliniear forms}
Under Assumption~\ref{assump: noise moment bound}, we can apply Theorem 2.1 in~\cite{erdHos2019random}, noting that Assumptions $A$ and $CD$ in~\cite{erdHos2019random} are satisfied in our setting since the noise entries $E_{ij}$ are independent and have zero means. Theorem 2.1 in~\cite{erdHos2019random} states that for any $\epsilon>0$ there exists $\delta_1, C_0, C_1(\epsilon,t)>0$ such that for all $t>0$, $i\in [m+n]$, and $z\in \{z: \operatorname{Im}(z) \geq (m+n)^{-\delta_1}, \;  \vert z \vert \leq (m+n)^{C_0}\}$, with probability at least $1 - C_1(\epsilon,D)n^{-t}$ we have that
\begin{equation}
    \vert \mathbf{a}^T R(z) \mathbf{b} - \mathbf{a}^T M \mathbf{b} \vert \leq \frac{(m+n)^{\epsilon-1/2}}{(1+\vert z\vert)^2}. \label{eq: bound for quadratic form of resolvent for general matrix-valued Dyson}
\end{equation}
Here, $C_0 \geq 100$ is universal constant, $C_1(\epsilon,t)$ is a constant depending on $\epsilon$, $t$, and on the universal constants in Assumption~\ref{assump: noise moment bound}, and $M\in \mathbb{C}^{(m+n)\times(m+n)}$ is a deterministic matrix that solves the matrix-valued Dyson equation
\begin{equation}
    \left(z I_{m+n} + \mathbb{E}[\mathcal{E} M \mathcal{E}]\right)M = -I_{m+n}. \label{eq: matrix-valued Dysin}
\end{equation}
The solution $M$ to this equation is a deterministic approximation to the resolvent $R(z)$ of $\mathcal{E}$. Since $E_{ij}$ are independent with zero means, it can be easily verified that the solution $M$ is a diagonal matrix whose main diagonal is precisely the solution $\mathbf{f}$ to the vector-valued Dyson equation~\eqref{eq: Dyson eq}, i.e., $M = D\{\mathbf{f}(z)\}$. In particular, setting $M$ as a diagonal matrix reduces the matrix-valued equation~\eqref{eq: matrix-valued Dysin} to the vector-valued equation~\eqref{eq: Dyson eq} for the main diagonal of $M$, which is solved by the vector $\mathbf{f}$. If we take $z = \imath\eta$, then $\mathbf{f}(z) = \imath \mathbf{g}(\eta)$, hence~\eqref{eq: bound for quadratic form of resolvent for general matrix-valued Dyson} becomes
\begin{equation}
    \vert \mathbf{a}^T R(z) \mathbf{b} - \mathbf{a}^T D\{\imath \mathbf{g}\} \mathbf{b} \vert 
    \leq \frac{(m+n)^{\epsilon-1/2}}{(1+\eta)^2} 
    \leq \frac{(2n)^{\epsilon-1/2}}{(1+\eta)^2} = \mathcal{O}_n(n^{\varepsilon-1/2}),
\end{equation}
where we used the fact that $m \leq n$.
Note that the aforementioned condition $z\in \{z: \operatorname{Im}(z) \geq (m+n)^{-\delta_1}, \;  \vert z \vert \leq (m+n)^{C_0}\}$ is satisfied for any fixed $\eta>0$ if $n$ is sufficiently large, namely if $n\geq n_0$ for some constant $n_0$ that may depend on $\eta$ and $\epsilon$ (since $\delta_1$ is determined by $\epsilon$). We can always represent the requirement $n\geq n_0$ in our Definition~\ref{def: order with high probability} by taking  $c{'}(t)$ to be large enough so that $1-c^{'}(t) n^{-t} \leq 0$ for all $n<n_0$, e.g., $c^{'}(t) = \max\{n_0^t, C_1(\epsilon,t)\}$. Consequently, the required result follows by our Definition~\ref{def: order with high probability}.

\section{Proof of Proposition~\ref{prop: formula for x and y in terms of g} } \label{appendix: proof of formula for x and y in terms of g}
Plugging $S=\mathbf{x}\mathbf{y}^T$ into the system of coupled equations~\eqref{eq: coupled Dyson eq imag axis}, we can write
\begin{equation}
    \mathbf{x} = \frac{1}{\mathbf{y}^T \mathbf{g}^{(2)}} \left( \frac{1}{\mathbf{g}^{(1)}} - \eta \right), \qquad\qquad \mathbf{y} = \frac{1}{\mathbf{x}^T \mathbf{g}^{(1)}} \left( \frac{1}{\mathbf{g}^{(2)}} - \eta \right).
\end{equation}
Therefore, by transposing the two equations above and multiplying them by $\mathbf{g}^{(1)}$ and $\mathbf{g}^{(2)}$, respectively, we obtain
\begin{equation}
    \mathbf{x}^T \mathbf{g}^{(1)} = \frac{n - \eta \Vert \mathbf{g}^{(1)} \Vert_1}{\mathbf{y}^T \mathbf{g}^{(2)}},\qquad\qquad 
    \mathbf{y}^T \mathbf{g}^{(2)} = \frac{m - \eta \Vert \mathbf{g}^{(2)} \Vert_1}{\mathbf{x}^T \mathbf{g}^{(1)}},
\end{equation}
implying that 
\begin{equation}
    (\mathbf{x}^T \mathbf{g}^{(1)}) (\mathbf{y}^T \mathbf{g}^{(2)}) = n - \eta \Vert \mathbf{g}^{(1)} \Vert_1 = m - \eta \Vert \mathbf{g}^{(2)} \Vert_1. \label{eq: x_T g_1 y_T g_2 formula}
\end{equation}
Consequently, using the definition of $\alpha$ from~\eqref{eq: formulas for x and y}, we can rewrite $\mathbf{x}$ and $\mathbf{y}$ as
\begin{equation}
    \mathbf{x} = \frac{\alpha}{\sqrt{\left( \mathbf{y}^T \mathbf{g}^{(2)}\right) \left( \mathbf{x}^T \mathbf{g}^{(1)} \right) }} \left( \frac{1}{\mathbf{g}^{(1)}} - \eta \right), \qquad\qquad 
    \mathbf{y} = \frac{\alpha^{-1}}{\sqrt{\left( \mathbf{y}^T \mathbf{g}^{(2)}\right) \left( \mathbf{x}^T \mathbf{g}^{(1)} \right) }} \left( \frac{1}{\mathbf{g}^{(2)}} - \eta \right),
\end{equation}
which together with~\eqref{eq: x_T g_1 y_T g_2 formula} completes the proof.

\section{Proof of Proposition~\ref{prop: formula for g_1_tilde and g_2_tilde}} \label{appendix: proof of formula for g_1_tilde and g_2_tilde}
Let $V = [V^{(1)},V^{(2)}]$, where $V^{(1)}\in\mathbb{R}^{n\times m}$ and $V^{(2)}\in\mathbb{R}^{n\times (n-m)}$, and denote $\Sigma = D\{\sigma_1,\ldots,\sigma_m\} \in \mathbb{R}^{m\times m}$. We can write the eigen-decomposition of $\mathcal{Y}$ as
\begin{equation}
    \mathcal{Y} = 
    \begin{bmatrix}
    \frac{1}{\sqrt{2}} U & \frac{1}{\sqrt{2}} U & \mathbf{0} \\
    \frac{1}{\sqrt{2}} V^{(1)} & - \frac{1}{\sqrt{2}} V^{(1)} & {V}^{(2)} 
    \end{bmatrix}
    \begin{bmatrix}
    \Sigma & \mathbf{0} & \mathbf{0} \\
    \mathbf{0} & -\Sigma & \mathbf{0} \\
    \mathbf{0} & \mathbf{0} & \mathbf{0}
    \end{bmatrix}
    \begin{bmatrix}
    \frac{1}{\sqrt{2}} U & \frac{1}{\sqrt{2}} U & \mathbf{0} \\
    \frac{1}{\sqrt{2}} V^{(1)} & - \frac{1}{\sqrt{2}} V^{(1)} & {V}^{(2)} 
    \end{bmatrix}^T,
\end{equation}
where $\mathbf{0}$ is a block of zeros of suitable size. Consequently, the resolvent of $\mathcal{Y}$ from~\eqref{eq: resolent of data} is given by
\begin{equation}
    \mathcal{R}(z) = \begin{bmatrix}
    \frac{1}{\sqrt{2}} U & \frac{1}{\sqrt{2}} U & \mathbf{0} \\
    \frac{1}{\sqrt{2}} V^{(1)} & - \frac{1}{\sqrt{2}} V^{(1)} & {V}^{(2)} 
    \end{bmatrix}
    \begin{bmatrix}
    (\Sigma-zI)^{-1} & \mathbf{0} & \mathbf{0} \\
    \mathbf{0} & (-\Sigma-zI)^{-1} & \mathbf{0} \\
    \mathbf{0} & \mathbf{0} & -z^{-1}I
    \end{bmatrix}
    \begin{bmatrix}
    \frac{1}{\sqrt{2}} U & \frac{1}{\sqrt{2}} U & \mathbf{0} \\
    \frac{1}{\sqrt{2}} V^{(1)} & - \frac{1}{\sqrt{2}} V^{(1)} & {V}^{(2)} 
    \end{bmatrix}^T.
\end{equation}
If we denote the top left $m\times m$ block of $\mathcal{R}$ by $\mathcal{R}^{(11)}$, then a direct calculation shows that
\begin{equation}
    \hat{\mathbf{g}}^{(1)}_i = \operatorname{Im}\{ \mathcal{R}^{(11)}_{ii}(\imath\eta) \} = \frac{1}{2}\sum_{k=1}^m U_{ik}^2 \operatorname{Im}\left\{ \frac{1}{\sigma_k - \imath \eta} + \frac{1}{-\sigma_k - \imath \eta} \right\} 
    = \sum_{k=1}^m U_{ik}^2 \frac{\eta}{\sigma_k^2 + \eta^2},
\end{equation}
for all $i\in [m]$. Similarly, denoting the bottom right $n \times n$ block $\mathcal{R}$ by $\mathcal{R}^{(22)}$, we have 
\begin{align}
    \hat{\mathbf{g}}^{(2)}_j &= \operatorname{Im}\{ \mathcal{R}^{(22)}_{jj}(\imath\eta) \} = \frac{1}{2}\sum_{k=1}^m (V_{jk}^{(1)})^2 \operatorname{Im}\left\{ \frac{1}{\sigma_k - \imath \eta} + \frac{1}{-\sigma_k - \imath \eta} \right\} + \frac{1}{\eta}\sum_{k=1}^{n-m} (V_{jk}^{(2)})^2 \\
    &= \frac{1}{\eta} + \sum_{k=1}^m V_{jk}^2 \left(\frac{\eta}{\sigma_k^2 + \eta^2} - \frac{1}{\eta}\right),
\end{align}
for all $j\in[n]$, where we used the fact that $\sum_{k=1}^{n-m} (V_{jk}^{(2)})^2 = 1 - \sum_{k=1}^m (V_{jk}^{(1)})^2$.

\section{Proof of Proposition~\ref{prop: x and y nonnegative} } \label{appendix: proof of x and y nonnegative}
Since $0 \leq \sigma_k < \infty$ and $U$ and $V$ are orthogonal matrices, it follows from~\eqref{eq: x_hat and y_hat def} that $0 < \hat{\mathbf{g}}_i^{(1)} \leq 1/\eta$ and $0 < \hat{\mathbf{g}}_j^{(2)} \leq 1/\eta$ for all $i\in[m]$ and $j\in[n]$. 
Therefore, according to~\eqref{eq: x_hat and y_hat def}, the only situation where the vectors $\hat{\mathbf{x}}$ and $\hat{\mathbf{y}}$ are ill-posed is if $\hat{\mathbf{g}}^{(1)} = \mathbf{1}_m/\eta$ or $\hat{\mathbf{g}}^{(2)} = \mathbf{1}_n/\eta$, respectively. According to~\eqref{eq: formulas for g_hat_1 and g_hat_2 in terms of the SVD}, this can only happen if $\sigma_1=\ldots=\sigma_m = 0$, i.e., if $Y$ is the zero matrix $\mathbf{0}_{m\times n}$. Otherwise, we must have $\hat{\mathbf{x}}_i \geq 0$ and $\hat{\mathbf{y}}_j \geq 0$ for all $i\in[m]$ and $j\in[n]$. We now consider the positivity of $\hat{\mathbf{x}}$ and $\hat{\mathbf{y}}$. Let us fix $i\in[m]$. If there exists $k\in [m]$ such that $\sigma_k >0$ and $U_{ik}^2>0$, then we obtain the strict inequality $\hat{\mathbf{g}}_i^{(1)} < 1/\eta$, which implies that $\hat{\mathbf{x}}_i >0$. Otherwise, we must have that $U_{ik} = 0$ for all $k\in[m]$ for which $\sigma_k>0$, asserting, by the SVD of $Y$, that $Y_{ij} = 0$ for all $j\in [n]$, i.e., the $i$th row of $Y$ is entirely zero. Conversely, if the $i$th row of $Y$ is entirely zero, then $U_{ik}=0$ for all $k\in [m]$ with $\sigma_k>0$, implying that $\hat{\mathbf{g}}^{(1)}_i = 1/\eta$ and hence $\hat{\mathbf{x}}_i = 0$. An analogous argument establishes that $\hat{\mathbf{y}}_j = 0$ if and only if the $j$th column of $Y$ is zero.

\section{Proof of Theorem~\ref{thm: robustness of the resolvent}} \label{appendix: proof of robustness of the resolvent}
We define the matrices $\mathcal{X}\in \mathbb{R}^{(m+n)\times(m+n)}$, $\mathcal{U}\in\mathbb{R}^{(m+n)\times 2r}$, and $\Sigma\in\mathbb{R}^{2r \times 2r}$ as
\begin{equation}
    \mathcal{X} = 
    \begin{bmatrix}
    \mathbf{0} & X \\
    X^T & \mathbf{0}
    \end{bmatrix}, 
    \quad 
    \mathcal{U} = \frac{1}{\sqrt{2}}
    \begin{bmatrix}
    \widetilde{U} & \widetilde{U} \\
    \widetilde{V} & -\widetilde{V}
    \end{bmatrix}, 
    \quad
    \Sigma = 
    \begin{bmatrix}
    D\{s_1,\ldots,s_r\} & \mathbf{0} \\
    \mathbf{0} & -D\{s_1,\ldots,s_r\}
    \end{bmatrix}.
\end{equation}
We can therefore write $\mathcal{X} = \mathcal{U} \Sigma \mathcal{U}^T$, and we have
\begin{equation}
    \mathcal{R}(z) = \left( \mathcal{X} + \mathcal{E}  - z I \right)^{-1} = \left( \mathcal{E} - z I + \mathcal{U} \Sigma \mathcal{U}^T\right)^{-1} = \left( R^{-1}(z) + \mathcal{U} \Sigma \mathcal{U}^T\right)^{-1}.
\end{equation}
Using the Woodbury matrix identity, we obtain
\begin{equation}
    \mathcal{R}(z) = R(z) + R(z)\mathcal{U} \left( \Sigma^{-1} + \mathcal{U}^T R(z) \mathcal{U} \right)^{-1} \mathcal{U}^T R(z). \label{eq: Woodbury expansion of R_cal}
\end{equation}
Denoting $\mathbf{e}_i$ as the $i$th standard basis vector in $\mathbb{R}^{m+n}$, we can write
\begin{equation}
    \mathcal{R}_{ii}(z) = \mathbf{e}_i^T\mathcal{R}(z) \mathbf{e}_i = \mathbf{e}_i^T{R}(z) \mathbf{e}_i + \left(\mathbf{e}_i^T R(z)\mathcal{U} \right) \left( \Sigma^{-1} + \mathcal{U}^T R(z) \mathcal{U} \right)^{-1} \left(\mathbf{e}_i^T R(z)\mathcal{U} \right)^T.
\end{equation}
Consequently, for $z=\imath \eta$, we have
\begin{align}
    \vert\hat{\mathbf{g}}_i - \mathbf{g}_i \vert = \vert \operatorname{Im}\{ \mathcal{R}_{ii}(\imath \eta) \} - \mathbf{g}_i \vert 
    \leq \vert \operatorname{Im}\{ {R}_{ii}(\imath \eta) \} - \mathbf{g}_i \vert 
    + \left\Vert \mathbf{e}_i^T R(\imath \eta)\mathcal{U} \right\Vert_2^2 \left\Vert \left( \Sigma^{-1} + \mathcal{U}^T R(\imath \eta) \mathcal{U} \right)^{-1} \right\Vert_2. \label{eq: woodbury expansion error}
\end{align}
Utilizing Lemma~\ref{lem: concentration of bilinear forms}, we can write
\begin{align}
    \mathbf{e}_i^T {R}(\imath \eta) \mathbf{e}_i &= \mathbf{e}_i^T D\{ \imath \mathbf{g} \} \mathbf{e}_i + \mathcal{O}_n(n^{\epsilon-1/2}), \label{eq: Woodbury expansion part 1}\\
    \mathbf{e}_i^T R(\imath \eta) \mathcal{U} &= \mathbf{e}_i^T D\{ \imath \mathbf{g} \} \mathcal{U} + \psi_i,
     \label{eq: Woodbury expansion part 2}\\
    \mathcal{U}^T R(\imath \eta) \mathcal{U} &= \mathcal{U}^T D\{\imath \mathbf{g} \} \mathcal{U} + \Psi, \label{eq: Woodbury expansion part 3} 
\end{align}
for all $i\in [m+n]$, where $\psi_{i}\in\mathbb{C}^{1\times 2r}$ and $\Psi\in\mathbb{C}^{2r\times 2r}$ satisfy $\psi_{ij}=\mathcal{O}_n(n^{\epsilon-1/2})$ and $\Psi_{ij}=\mathcal{O}_n(n^{\epsilon-1/2})$ for all $i,j\in [m+n]$. From~\eqref{eq: Woodbury expansion part 1}, we have
\begin{equation}
    \vert \operatorname{Im}\{ {R}_{ii}(\imath \eta) \} - \mathbf{g}_i \vert = \mathcal{O}_n(n^{\epsilon-1/2}),
\end{equation}
for all $i\in [m+n]$.
Additionally, From~\eqref{eq: Woodbury expansion part 2},
\begin{align}
    \left\Vert \mathbf{e}_i^T R(\imath \eta)\mathcal{U} \right\Vert_2 
    &\leq \left\Vert \mathbf{e}_i^T D\{ \imath \mathbf{g} \} \mathcal{U} \right\Vert_2 + \Vert \psi_i \Vert_2
    \leq \left\Vert \mathbf{e}_i^T D\{ \imath \mathbf{g} \} \mathcal{U} \right\Vert_2 + \sqrt{r} \max_j\{\vert \psi_{ij} \vert\} 
    \nonumber \\
    &= \sqrt{r} \max_j\{\vert \psi_{ij} \vert\} + \mathbf{g}_i\times 
    \begin{dcases}
         \sqrt{\sum_{k=1}^{r} \widetilde{U}_{ik}^2} , & 1\leq i \leq m, \\
       \sqrt{\sum_{k=1}^{r} \widetilde{V}_{\ell k}^2} , & m+1\leq i \leq m+n,
    \end{dcases}
    \nonumber \\
    &\leq \mathcal{O}_n\left( n^{\varepsilon - (1 - \delta_0)/2}\right) + \frac{1}{\eta} \times
    \begin{dcases}
        \sqrt{\sum_{k=1}^{r} \widetilde{U}_{ik}^2},   & 1\leq i \leq m, \\
        \sqrt{\sum_{k=1}^{r} \widetilde{V}_{\ell k}^2} , & m+1\leq i \leq m+n, 
    \end{dcases} 
\end{align}
where $\ell = i - m$ and we used Assumption~\ref{assump: rank growth rate} and Lemma~\ref{lem: boundedness of g} in the last inequality. We obtain
\begin{align}
    \left\Vert \mathbf{e}_i^T R(\imath \eta)\mathcal{U} \right\Vert_2^2 
    \leq \mathcal{O}_n\left( n^{2\varepsilon - 1 + \delta_0}\right) + \frac{2}{\eta^2} \times
    \begin{dcases}
        {\sum_{k=1}^{r} \widetilde{U}_{ik}^2},   , & 1\leq i \leq m, \\
        {\sum_{k=1}^{r} \widetilde{V}_{\ell k}^2} , & m+1\leq i \leq m+n. 
    \end{dcases} \label{eq: bound on second term in Woodbury expansion error}
\end{align}
Next, we define the matrices $A_R,A_I\in\mathbb{R}^{2r\times 2r}$ and $A\in\mathbb{C}^{2r\times 2r}$ as
\begin{equation}
    A_R = \Sigma^{-1}, \qquad A_I = \mathcal{U}^T D\{\mathbf{g}\} \mathcal{U}, \qquad A 
    = A_R + \imath A_I. \label{eq: A def in proof of data resolvent}
\end{equation}
Since the entries of $\mathbf{g}$ are strictly positive and the columns of $\mathcal{U}$ are orthonormal, the matrix $A_I$ is positive definite. Hence, we can apply Lemma~\ref{lem: operator norm bound on inverse of complex matrix} to $A$, obtaining that
\begin{equation}
        \Vert A^{-1}\Vert_2 \leq \frac{1}{\sqrt{\Vert A_R^{-1} \Vert_2^{-2} + \Vert A_I^{-1} \Vert_2^{-2}}},
\end{equation}
with
\begin{equation}
    \Vert A_R^{-1} \Vert_2^{-2} = \Vert \Sigma \Vert_2^{-2} = \frac{1}{s_1^2},
\end{equation}
and 
\begin{align}
    \Vert A_I^{-1} \Vert_2^{-2} 
    &= \Vert (\mathcal{U}^T D\{\mathbf{g}\} \mathcal{U})^{-1} \Vert_2^{-2} 
    = \frac{1}{\lambda^2_{\max} \left\{ (\mathcal{U}^T D\{\mathbf{g}\} \mathcal{U})^{-1} \right\} } 
    =  \lambda^2_{\min} \left\{ \mathcal{U}^T D\{\mathbf{g}\} \mathcal{U} \right\} \nonumber \\
    &= \left( \min_{\mathbf{a}\in\mathbb{R}^{2r},\;\Vert \mathbf{a}\Vert_2 
    = 1} \left\{ \mathbf{a}^T \mathcal{U}^T D\{\mathbf{g}\} \mathcal{U}\mathbf{a} \right\} \right)^2 
    \geq \left( \min_{\mathbf{b}\in\mathbb{R}^{m+n},\;\Vert \mathbf{b}\Vert_2 
    = 1} \left\{ \mathbf{b}^T D\{\mathbf{g}\} \mathbf{b} \right\} \right)^2 \nonumber \\
    &= \min_{i\in [m+n]} \mathbf{g}_i^2 \geq \left( \eta + \frac{C}{\eta} \right)^{-2},
\end{align}
where $\lambda_{\max}\{\cdot\}$ ($\lambda_{\min}\{\cdot\}$) denote the maximal (minimal) eigenvalue of a matrix, and we used Lemma~\ref{lem: boundedness of g} to obtain the last inequality. Consequently, we have 
\begin{equation}
    \Vert A^{-1}\Vert_2 \leq \frac{1}{\sqrt{s_1^{-2} + \left( \eta + \frac{C}{\eta} \right)^{-2}}}.
\end{equation}
Note that the above bound also implies that $\sigma_{\min} \{A\} \geq \sqrt{s_1^{-2} + \left( \eta + \frac{C}{\eta} \right)^{-2}}$, where $\sigma_{\min}\{\cdot\}$ denotes the smallest singular values of a matrix.
According to~\eqref{eq: Woodbury expansion part 3}, we can write
\begin{equation}
    \left\Vert \left( \Sigma^{-1} + \mathcal{U}^T R(\imath \eta) \mathcal{U} \right)^{-1} \right\Vert_2
    = \left\Vert \left( A + \Psi \right)^{-1} \right\Vert_2 = \frac{1}{\sigma_{\min}\left\{A + \Psi\right\}},
\end{equation}
and
\begin{align}
\sigma_{\min}\left\{A + \Psi\right\} 
&\geq \sigma_{\min}\{A\} - \sigma_{\max}\{\Psi\}
    \geq \sigma_{\min}\{A\} - \Vert \Psi \Vert_F \nonumber \\
    &\geq { \sigma_{\min}\{A\} - 2r \max_{ij}\{ \Psi_{ij}\} } 
    \geq \sqrt{s_1^{-2} + \left( \eta + \frac{C}{\eta} \right)^{-2}} + \mathcal{O}_n\left( n^{\varepsilon-1/2 + \delta_0} \right),
\end{align}
where $\Vert\cdot\Vert_F$ denotes the Frobenius norm and we used Assumption~\ref{assump: rank growth rate} and the notation $\sigma_{\max}\{\cdot\}$ to denote the largest singular value of a matrix. Since $n^{\varepsilon -1/2 + \delta_0} \rightarrow 0$ as $n\rightarrow \infty$ (according to Assumption~\ref{assump: rank growth rate}) and $\sqrt{s_1^{-2} + ( \eta + {C}/{\eta} )^{-2}} > (\eta + {C}/{\eta})^{-1} > 0$ for all $s_1>0$, we obtain that
\begin{equation}
    \left\Vert \left( \Sigma^{-1} + \mathcal{U}^T R(\imath \eta) \mathcal{U} \right)^{-1} \right\Vert_2 
    \leq \frac{1}{\sqrt{s_1^{-2} + \left( \eta + \frac{C}{\eta} \right)^{-2}} + \mathcal{O}_n\left( n^{\varepsilon-1/2 + \delta_0} \right)}
    \leq \frac{\mathcal{O}_n\left( 1 \right)}{\sqrt{s_1^{-2} + \left( \eta + \frac{C}{\eta} \right)^{-2}}}.
    \label{eq: bound on third term in Woodbury expansion error}
\end{equation}
Finally, plugging~\eqref{eq: bound on third term in Woodbury expansion error},~\eqref{eq: bound on second term in Woodbury expansion error}, and~\eqref{eq: Woodbury expansion part 1} into~\eqref{eq: woodbury expansion error}, and absorbing all constants into $\mathcal{O}_n\left(\cdot \right)$ terms, we obtain 
\begin{align}
    \vert \hat{\mathbf{g}}_i - \mathbf{g}_i \vert 
    &\leq \mathcal{O}_n\left( n^{\varepsilon-1/2}\right) + \mathcal{O}_n\left( n^{2\varepsilon - 1 + \delta_0}\right) \nonumber \\
    &\phantom{\leq \mathcal{O}_n\left( n^{\varepsilon-1/2}\right)\;} +\frac{\mathcal{O}_n\left( 1 \right)} {\sqrt{s_1^{-2} + \left( \eta + \frac{C}{\eta} \right)^{-2}}} \times 
    \begin{dcases}
        {\sum_{k=1}^{r} \widetilde{U}_{ik}^2}, & 1\leq i \leq m, \\
        {\sum_{k=1}^{r} \widetilde{V}_{\ell k}^2} , & m+1\leq i \leq m+n, 
    \end{dcases} \nonumber \\
    &= \mathcal{O}_n\left( n^{\varepsilon-1/2}\right) + 
    \frac{\mathcal{O}_n\left( 1 \right)} {\sqrt{s_1^{-2} + \left( \eta + \frac{C}{\eta} \right)^{-2}}} \times 
    \begin{dcases}
        {\sum_{k=1}^{r} \widetilde{U}_{ik}^2}, & 1\leq i \leq m, \\
        {\sum_{k=1}^{r} \widetilde{V}_{\ell k}^2} , & m+1\leq i \leq m+n, 
    \end{dcases}
\end{align}
for all $i\in [m+n]$, where $\ell = i-m$ and we used the fact that $2\varepsilon - 1 +\delta_0 < \varepsilon - 1/2$ for all $\varepsilon < 1/2 - \delta_0$ (see conditions in Theorem~\ref{thm: robustness of the resolvent}), recalling also that $\delta_0 \in [0,1/2)$ according to Assumption~\ref{assump: rank growth rate}. 

\section{Proof of Theorem~\ref{thm: convergence of estimated variance factors}} \label{appendix: proof of convergence of estimated variance factors}
Let us define $\hat{\mathbf{e}}\in\mathbb{R}^{m+n}$ according to
\begin{equation}
    \hat{\mathbf{g}}_i = \mathbf{g}_i + \hat{\mathbf{e}}_i, \label{eq: e_1 def}
\end{equation}
for all $i\in[m+n]$, where
\begin{equation}
    \Vert \hat{\mathbf{e}} \Vert_\infty  = \mathcal{O}_{n}\left( n^{\epsilon-1/2} + n^{-\delta_1}\right), \label{eq: bound on e}
\end{equation}
for any fixed $\epsilon>0$ according to Corollary~\ref{cor: convergence of g_tilde to g in l_infty}. We begin by proving the required bound on $\Vert  (\hat{\mathbf{x}} - \mathbf{x})/\mathbf{x} \Vert_\infty$. 
Using~\eqref{eq: x_hat and y_hat def} and~\eqref{eq: setting alpha=1}, we can write
\begin{align}
    \frac{\hat{\mathbf{x}}_i}{\mathbf{x}_i} = \left( \frac{\mathbf{g}_i^{(1)}}{\hat{\mathbf{g}}_i^{(1)}} \right) \left( \frac{1 - \eta \hat{\mathbf{g}}_i^{(1)}}{1 - \eta {\mathbf{g}_i^{(1)}}} \right) \sqrt{\frac{m - \eta \Vert {\mathbf{g}}^{(1)}\Vert_1}{m - \eta \Vert \hat{\mathbf{g}}^{(1)}\Vert_1 }}. \label{eq: relative error multiplicative term}
\end{align}
We now evaluate each of the three multiplicative terms in the expression above one by one, showing that they concentrate around $1$ with an error of $\mathcal{O}_{n}( n^{\epsilon-1/2} + n^{-\delta_1})$. Let us denote $\hat{\mathbf{e}}^{(1)} = [\hat{\mathbf{e}}_1,\ldots,\hat{\mathbf{e}}_m]^T$.
Using Definition~\ref{def: order with high probability} and its properties together with the boundedness of $\mathbf{g}$ from Lemma~\ref{lem: boundedness of g}, it can be verified that
\begin{equation}
    \frac{\mathbf{g}_i^{(1)}}{\hat{\mathbf{g}}_i^{(1)}} = \frac{\mathbf{g}^{(1)}_i}{\mathbf{g}^{(1)}_i + \hat{\mathbf{e}}_i^{(1)}} = 1 - \frac{\hat{\mathbf{e}}^{(1)}_i}{\mathbf{g}^{(1)}_i + \hat{\mathbf{e}}^{(1)}_i} = {1} + \mathcal{O}_{n}\left( n^{\epsilon-1/2} + n^{-\delta_1}\right), \label{eq: first multiplicative term bound}
\end{equation}
for all $i\in [m]$.
For the next term in~\eqref{eq: relative error multiplicative term}, let us write
\begin{equation}
    \frac{1 - \eta \hat{\mathbf{g}}_i^{(1)}}{1 - \eta {\mathbf{g}_i^{(1)}}} = 1 - \frac{\eta \hat{\mathbf{e}}_i^{(1)}}{1 - \eta {\mathbf{g}_i^{(1)}}}.
\end{equation}
According to Proposition~\ref{prop: formula for x and y in terms of g} (recalling that $\alpha=1$), we have 
\begin{equation}
    \mathbf{x}^T \mathbf{g}^{(1)} = \sqrt{m - \eta \Vert \mathbf{g}^{(1)} \Vert_1}. \label{eq: m-eta g_1_norm_1 expression}
\end{equation}
Plugging this back into the expression for $\mathbf{x}$ in~\eqref{eq: formulas for x and y} while utilizing Lammas~\ref{lem: boundedness of g} and~\ref{lem: boundedness of x and y}, gives
\begin{equation}
    1 - \eta \mathbf{g}_i^{(1)} = \mathbf{x}_i \mathbf{g}^{(1)}_i \left( {\mathbf{x}^T \mathbf{g}^{(1)}} \right) \geq c_1,
\end{equation}
where $c_1 > 0$ is a global constant. Overall, we obtain that
\begin{equation}
    \left\vert \frac{1 - \eta \hat{\mathbf{g}}_i^{(1)}}{1 - \eta {\mathbf{g}_i^{(1)}}} - 1 \right\vert 
    \leq \frac{\eta}{c_1} \vert \hat{\mathbf{e}}^{(1)}_i \vert = \mathcal{O}_{n}\left( n^{\epsilon-1/2} + n^{-\delta_1}\right), \label{eq: second multiplicative term bound}
\end{equation}
for all $i\in [m]$.
For the last term in~\eqref{eq: relative error multiplicative term}, we write its reciprocal as
\begin{equation}
    \sqrt{\frac{m - \eta \Vert \hat{\mathbf{g}}^{(1)}\Vert_1 }{m - \eta \Vert {\mathbf{g}}^{(1)}\Vert_1 }} 
    = \sqrt{1 + \eta \frac{\Vert {\mathbf{g}}^{(1)}\Vert_1 - \Vert \hat{\mathbf{g}}^{(1)}\Vert_1}{m - \eta \Vert {\mathbf{g}}^{(1)}\Vert_1 }}.
\end{equation}
According to Lemma~\ref{lem: boundedness of g} and~\eqref{eq: bound on e},
\begin{equation}
    \Vert {\mathbf{g}}^{(1)}\Vert_1 - \Vert \hat{\mathbf{g}}^{(1)}\Vert_1 = \sum_{i=1}^m {\mathbf{g}}_i^{(1)} - \sum_{i=1}^m \left( {\mathbf{g}}_i^{(1)} + \hat{\mathbf{e}}_i^{(1)}\right) = \sum_{i=1}^m \hat{\mathbf{e}}_i^{(1)} = \mathcal{O}_{n}\left( m (n^{\epsilon - 1/2} + n^{-\delta_1})\right),
\end{equation}
where we used the nonnegativity of $\mathbf{g}$ and $\hat{\mathbf{g}}$. In addition,~\eqref{eq: m-eta g_1_norm_1 expression} and Lemma~\ref{lem: boundedness of x and y} imply that
\begin{equation}
    m - \eta \Vert \mathbf{g}^{(1)} \Vert_1 = \left( \mathbf{x}^T \mathbf{g}^{(1)} \right)^2 \geq c_2 m,
\end{equation}
for a global constant $c_2>0$. Therefore, according to Definition~\ref{def: order with high probability} and its properties,
\begin{equation}
    \sqrt{\frac{m - \eta \Vert \hat{\mathbf{g}}^{(1)}\Vert_1 }{m - \eta \Vert {\mathbf{g}}^{(1)}\Vert_1 }} 
    = \sqrt{1 + \mathcal{O}_{n}\left( n^{\epsilon - 1/2}+  n^{-\delta_1}\right)} = 1 + \mathcal{O}_{n}\left( n^{\epsilon - 1/2} +  n^{-\delta_1}\right).
\end{equation}
Thus, we have
\begin{equation}
    \sqrt{\frac{m - \eta \Vert {\mathbf{g}}^{(1)}\Vert_1}{m - \eta \Vert \hat{\mathbf{g}}^{(1)}\Vert_1 }} = \frac{1}{1 + \mathcal{O}_{n}\left( n^{\epsilon - 1/2} + n^{-\delta_1}\right)} = 1 + \mathcal{O}_{n}\left( n^{\epsilon - 1/2} +  n^{-\delta_1}\right), \label{eq: third multiplicative term bound}
\end{equation}
for all $i\in [m]$.
Plugging~\eqref{eq: first multiplicative term bound},~\eqref{eq: second multiplicative term bound}, and~\eqref{eq: third multiplicative term bound} into~\eqref{eq: relative error multiplicative term} proves the required bound on $\Vert  (\hat{\mathbf{x}} - \mathbf{x})/\mathbf{x} \Vert_\infty$ after applying the union bound over $i=1,\ldots,m$. 

Next, we prove the bound on $\Vert  (\hat{\mathbf{y}} - \mathbf{y})/\mathbf{y} \Vert_\infty$. Analogously to~\eqref{eq: relative error multiplicative term}, we have
\begin{align}
    \frac{\hat{\mathbf{y}}_j}{\mathbf{y}_j} = \left( \frac{\mathbf{g}_j^{(2)}}{\hat{\mathbf{g}}_j^{(2)}} \right) \left( \frac{1 - \eta \hat{\mathbf{g}}_j^{(2)}}{1 - \eta {\mathbf{g}_j^{(2)}}} \right) \sqrt{\frac{n - \eta \Vert {\mathbf{g}}^{(2)}\Vert_1}{n - \eta \Vert \hat{\mathbf{g}}^{(2)}\Vert_1 }}, \label{eq: relative error multiplicative term for y_j}
\end{align}
for all $j\in[n]$. Let us denote $\hat{\mathbf{e}}^{(2)} = [\hat{\mathbf{e}}_{m+1},\ldots,\hat{\mathbf{e}}_{m+n}]^T$. For the first term in~\eqref{eq: relative error multiplicative term for y_j},  Definition~\ref{def: order with high probability} and its properties together with~\eqref{eq: bound on e} and Lemma~\ref{lem: boundedness of g}, imply that
\begin{equation}
    \frac{\mathbf{g}_j^{(2)}}{\hat{\mathbf{g}}_j^{(2)}} = \frac{\mathbf{g}^{(2)}_j}{\mathbf{g}^{(2)}_j + \hat{\mathbf{e}}_j^{(2)}} = 1 - \frac{\hat{\mathbf{e}}^{(2)}_j}{\mathbf{g}^{(2)}_j + \hat{\mathbf{e}}^{(2)}_j} = {1} + \mathcal{O}_{n}\left( n^{\epsilon-1/2} + n^{-\delta_1}\right), \label{eq: first multiplicative term bound for y_j}
\end{equation}
for all $j\in [n]$. 
For the second term in~\eqref{eq: relative error multiplicative term}, we can write
\begin{equation}
    \frac{1 - \eta \hat{\mathbf{g}}_j^{(2)}}{1 - \eta {\mathbf{g}_j^{(2)}}} = 1 - \frac{\eta \hat{\mathbf{e}}_j^{(2)}}{1 - \eta {\mathbf{g}_j^{(2)}}}.
\end{equation}
Using Proposition~\ref{prop: formula for x and y in terms of g}, we obtain
\begin{equation}
    \mathbf{y}^T \mathbf{g}^{(2)} = \sqrt{n - \eta \Vert \mathbf{g}^{(2)} \Vert_1}. \label{eq: m-eta g_2_norm_1 expression}
\end{equation}
Plugging~\eqref{eq: m-eta g_2_norm_1 expression} into the expression for $\mathbf{y}$ in~\eqref{eq: formulas for x and y} while utilizing Lammas~\ref{lem: boundedness of g} and~\ref{lem: boundedness of x and y}, asserts that
\begin{equation}
    1 - \eta \mathbf{g}_j^{(2)} = \mathbf{y}_j \mathbf{g}^{(2)}_j \left( {\mathbf{y}^T \mathbf{g}^{(2)}} \right) \geq c_3 \frac{m}{n},
\end{equation}
for some global constant $c_3 > 0$. Consequently, we obtain
\begin{equation}
    \left\vert \frac{1 - \eta \hat{\mathbf{g}}_j^{(2)}}{1 - \eta {\mathbf{g}_j^{(2)}}} - 1 \right\vert \leq \frac{\eta }{c_3} \left( \frac{n}{m}\right) \vert \hat{\mathbf{e}}^{(2)}_j \vert = \mathcal{O}_{n}\left( \frac{n^{\epsilon+1/2}}{m} + \frac{n^{1-\delta_1}}{m}\right), \label{eq: second multiplicative term bound for y_j}
\end{equation}
for all $j\in [n]$.
Next, we write the reciprocal of the third term in~\eqref{eq: relative error multiplicative term for y_j} as
\begin{equation}
    \sqrt{\frac{n - \eta \Vert \hat{\mathbf{g}}^{(2)}\Vert_1 }{n - \eta \Vert {\mathbf{g}}^{(2)}\Vert_1 }} 
    = \sqrt{1 + \eta \frac{\Vert {\mathbf{g}}^{(2)}\Vert_1 - \Vert \hat{\mathbf{g}}^{(2)}\Vert_1}{n - \eta \Vert {\mathbf{g}}^{(2)}\Vert_1 }}.
\end{equation}
According to Lemma~\ref{lem: boundedness of g} and~\eqref{eq: bound on e},
\begin{equation}
    \Vert {\mathbf{g}}^{(2)}\Vert_1 - \Vert \hat{\mathbf{g}}^{(2)}\Vert_1 = \sum_{j=1}^n {\mathbf{g}}_j^{(2)} - \sum_{j=1}^n \left( {\mathbf{g}}_j^{(2)} + \hat{\mathbf{e}}_j^{(2)}\right) = \sum_{j=1}^n  \hat{\mathbf{e}}_j^{(2)} = \mathcal{O}_{n}\left( n^{\epsilon + 1/2} + n^{1-\delta_1}\right),
\end{equation}
while~\eqref{eq: m-eta g_2_norm_1 expression} and Lemma~\ref{lem: boundedness of x and y} imply that
\begin{equation}
    n - \eta \Vert \mathbf{g}^{(2)} \Vert_1 = \left( \mathbf{y}^T \mathbf{g}^{(2)} \right)^2 \geq c_4 m,
\end{equation}
for a global constant $c_4>0$. Therefore, according to Assumption~\ref{assump: growth of m} combined with Definition~\ref{def: order with high probability} and its properties, it follows that
\begin{equation}
    \sqrt{\frac{m - \eta \Vert \hat{\mathbf{g}}^{(1)}\Vert_1 }{m - \eta \Vert {\mathbf{g}}^{(1)}\Vert_1 }} 
    = \sqrt{1 + \mathcal{O}_{n}\left( \frac{n^{\epsilon+1/2}}{m} + \frac{n^{1-\delta_1}}{m}\right)} = 1 + \mathcal{O}_{n}\left( \frac{n^{\epsilon+1/2}}{m} + \frac{n^{1-\delta_1}}{m}\right).
\end{equation}
Consequently,
\begin{equation}
    \sqrt{\frac{n - \eta \Vert {\mathbf{g}}^{(2)}\Vert_1}{n - \eta \Vert \hat{\mathbf{g}}^{(2)}\Vert_1 }} = \frac{1}{1 + \mathcal{O}_{n}\left( \frac{n^{\epsilon+1/2}}{m} + \frac{n^{1-\delta_1}}{m}\right)} = 1 + \mathcal{O}_{n}\left( \frac{n^{\epsilon+1/2}}{m} + \frac{n^{1-\delta_1}}{m}\right), \label{eq: third multiplicative term bound for y_j}
\end{equation}
for all $j\in [n]$.
Plugging~\eqref{eq: first multiplicative term bound for y_j},~\eqref{eq: second multiplicative term bound for y_j}, and~\eqref{eq: third multiplicative term bound for y_j} into~\eqref{eq: relative error multiplicative term for y_j} proves the required bound on $\Vert  (\hat{\mathbf{y}} - \mathbf{y})/\mathbf{y} \Vert_\infty$ after applying the union bound over $j=1,\ldots,n$. 

\section{Proof of Lemma~\ref{lem: g is close to h under incoherence}} \label{appendix: proof of g is close to h under incoherence}
According to~\eqref{eq: coupled Dyson eq imag axis rank one} we have
\begin{align}
    \frac{1}{\mathbf{h}^{(1)}} &= \eta + \mathbf{x}\mathbf{y}^T\mathbf{h}^{(2)} = \eta + D\{\mathbf{x}\} \widetilde{S} D\{\mathbf{y}\} \mathbf{h}^{(2)} 
    - D\{\mathbf{x}\} \left( \widetilde{S} -1\right) D\{\mathbf{y}\}\mathbf{h}^{(2)} \nonumber \\ 
    &= \eta + S \mathbf{h}^{(2)} - D\{\mathbf{x}\} \left( \widetilde{S} -1\right) \mathbf{w}^{(2)}.
\end{align}
Similarly,
\begin{align}
    \frac{1}{\mathbf{h}^{(2)}} &= \eta + \mathbf{y}\mathbf{x}^T\mathbf{h}^{(1)} = \eta + D\{\mathbf{y}\} \widetilde{S}^T D\{\mathbf{x}\} \mathbf{h}^{(1)} 
    - D\{\mathbf{y}\} \left( \widetilde{S} -1\right)^T D\{\mathbf{x}\}\mathbf{h}^{(1)} \nonumber \\ 
    &= \eta + S^T \mathbf{h}^{(1)} - D\{\mathbf{y}\} \left( \widetilde{S} -1\right)^T \mathbf{w}^{(1)}.
\end{align}
Combining both equations above we obtain
\begin{equation}
    \frac{1}{\mathbf{h}} = \eta + \mathcal{S} \mathbf{h} + \hat{\mathbf{e}},
\end{equation}
where
\begin{align}
    \Vert \hat{\mathbf{e}} \Vert_\infty &= \left\Vert 
    \begin{bmatrix}
    D\{\mathbf{x}\} \left( \widetilde{S} -1\right) \mathbf{w}^{(2)} \\
    D\{\mathbf{y}\} \left( \widetilde{S} -1\right)^T \mathbf{w}^{(1)}
    \end{bmatrix}
    \right\Vert_\infty \nonumber \\
    &\leq C_1^{'} \max\left\{ \frac{1}{\sqrt{m}} \left\Vert \left( \widetilde{S} -1\right) \mathbf{w}^{(2)} \right\Vert_\infty, \frac{\sqrt{m}}{n}\left\Vert \left( \widetilde{S} -1\right)^T \mathbf{w}^{(1)} \right\Vert_\infty \right\},
\end{align}
for some constant $C_1^{'}>0$, where we used Lemma~\ref{lem: boundedness of x and y for general S} (under Assumption~\ref{assump: variance boundedness}) in the last transition. Invoking Lemma~\ref{lem: stability of Dyson equation}, we have
\begin{equation}
    \left\Vert \frac{\mathbf{h} - \mathbf{g}}{\mathbf{g}} \right\Vert_\infty \leq \Vert \hat{\mathbf{e}} \Vert_\infty \max_i \{ \mathbf{h}_i\} \max \left\{ 1, \frac{2}{\eta \min_{i} \{ \mathbf{h}_i \}} \right\}.
\end{equation}
Note that Lemma~\ref{lem: boundedness of x and y for general S} also asserts that $\mathbf{x}_i \mathbf{y}_j \leq C_2^{'} n^{-1}$ for some constant $C_2^{'}>0$ and all $i\in [m]$ and $j\in[n]$.
Therefore, since $\mathbf{h}$ is the solution to~\eqref{eq: Dyson eq imag axis} for $S = \mathbf{x} \mathbf{y}^T$, Lemma~\ref{lem: boundedness of g} can be applied with $C = C_2^{'}$ to provide upper and lower bounds on the entries of $\mathbf{h}$. Specifically, we obtain
\begin{equation}
    \left( \eta + \frac{C_2^{'}}{\eta}\right)^{-1} \leq \mathbf{h}_i \leq \frac{1}{\eta}, \label{eq: boundedness of h}
\end{equation}
for all $i\in [m+n]$. Combining the previous three equations together with the lower bound on $\mathbf{g}_i$ from Lemma~\ref{lem: boundedness of g} gives
\begin{equation}
    \left\Vert {\mathbf{g} - \mathbf{h}} \right\Vert_\infty \leq C_2^{'} \max\left\{ \frac{1}{\sqrt{m}} \left\Vert \left( \widetilde{S} -1\right) \mathbf{w}^{(2)} \right\Vert_\infty, \frac{\sqrt{m}}{n}\left\Vert \left( \widetilde{S} -1\right)^T \mathbf{w}^{(1)} \right\Vert_\infty \right\},
\end{equation}
for some constant $C_2^{'}>0$. Since $\mathbf{w}^{(1)} = D\{\mathbf{x}\} \mathbf{h}^{(1)}$ and $\mathbf{w}^{(2)} = D\{\mathbf{y}\} \mathbf{h}^{(2)}$ (see~\eqref{eq: W_1 and W_2 def}), we have 
\begin{equation}
    \Vert \mathbf{w}^{(1)} \Vert_2 \leq \sqrt{m} \Vert \mathbf{x} \Vert_\infty \Vert \mathbf{h}^{(1)} \Vert_\infty \leq C_3^{'}, \qquad \Vert \mathbf{w}^{(2)} \Vert_2 \leq \sqrt{n} \Vert \mathbf{y} \Vert_\infty \Vert \mathbf{h}^{(2)} \Vert_\infty \leq C_3^{'} \sqrt{\frac{m}{n}},
\end{equation}
for some constant $C_3^{'}>0$, where we used~\eqref{eq: boundedness of h} and Lemma~\ref{lem: boundedness of x and y for general S}. Finally, we obtain
\begin{align}
    \left\Vert {\mathbf{g} - \mathbf{h}} \right\Vert_\infty &\leq C_2^{'} \max\left\{ \frac{\Vert \mathbf{w}^{(2)} \Vert_2}{\sqrt{m}} \left\Vert \left( \widetilde{S} -1\right) \frac{\mathbf{w}^{(2)}}{\Vert \mathbf{w}^{(2)} \Vert_2} \right\Vert_\infty, \frac{\sqrt{m}\Vert \mathbf{w}^{(1)} \Vert_2}{n}\left\Vert \left( \widetilde{S} -1\right)^T \frac{\mathbf{w}^{(1)}}{\Vert \mathbf{w}^{(1)} \Vert_2} \right\Vert_\infty \right\} \nonumber \\
    &\leq C_4^{'} \max\left\{ \frac{1}{\sqrt{n}} \left\Vert \left( \widetilde{S} -1\right) \frac{\mathbf{w}^{(2)}}{\Vert \mathbf{w}^{(2)} \Vert_2} \right\Vert_\infty, \frac{\sqrt{m}}{n}\left\Vert \left( \widetilde{S} -1\right)^T \frac{\mathbf{w}^{(1)}}{\Vert \mathbf{w}^{(1)} \Vert_2} \right\Vert_\infty \right\} \nonumber \\
    &= C_4^{'} \max\left\{ \frac{1}{\sqrt{n}} \left\Vert \left( \widetilde{S} -1\right) \frac{\mathbf{w}^{(2)} - \langle \mathbf{w}^{(2)} \rangle}{\Vert \mathbf{w}^{(2)} \Vert_2} \right\Vert_\infty, \frac{\sqrt{m}}{n}\left\Vert \left( \widetilde{S} -1\right)^T \frac{\mathbf{w}^{(1)} -\langle \mathbf{w}^{(1)} \rangle}{\Vert \mathbf{w}^{(1)} \Vert_2} \right\Vert_\infty \right\},
\end{align}
where we also used the fact that the vector of ones $\mathbf{1}$ is in the left and right null spaces of $(\widetilde{S}-1)$, hence $(\widetilde{S}-1) (\mathbf{w}^{(2)} -  \langle\mathbf{w}^{(2)} \rangle ) = (\widetilde{S}-1)\mathbf{w}^{(2)}$ and $(\widetilde{S}-1)^T (\mathbf{w}^{(1)} -  \langle\mathbf{w}^{(1)} \rangle ) = (\widetilde{S}-1)^T\mathbf{w}^{(1)}$.

\section{Proof of Proposition~\ref{prop: relation between x,y and w_1,w_2}} \label{appendix: proof of relation between x,y and w_1,w_2}
Using~\eqref{eq: W_1 and W_2 def} we can write
\begin{equation}
    \frac{\mathbf{w}_i^{(1)}}{\mathbf{w}^{(1)}_j} - 1
    = \frac{\mathbf{x}_j^{-1}\eta + a}{\mathbf{x}_j^{-1}\eta + a} - 1 
    = \frac{\eta (\mathbf{x}_i - \mathbf{x}_j)}{\mathbf{x}_i \mathbf{x}_j (\mathbf{x}_i^{-1} \eta + a)}
    = \frac{\eta (\mathbf{x}_i - \mathbf{x}_j)}{ \mathbf{x}_j \eta + a \mathbf{x}_i}.
\end{equation}
Therefore, $\mathbf{w}_i^{(1)} - \mathbf{w}_j^{(1)}$ has the same sign as $\mathbf{x}_i - \mathbf{x}_j$. If $\mathbf{x}_i > \mathbf{x}_j$, then using the fact that $\mathbf{x}_i>0$ and $a>0$, we have
\begin{equation}
    0 < \frac{\mathbf{w}_i^{(1)}}{\mathbf{w}^{(1)}_j} - 1 < \frac{\mathbf{x}_i - \mathbf{x}_j}{\mathbf{x}_j}.
\end{equation}
Similarly, if $\mathbf{x}_i < \mathbf{x}_j$, we have
\begin{equation}
     0 > \frac{\mathbf{w}_i^{(1)}}{\mathbf{w}^{(1)}_j} - 1 > \frac{\mathbf{x}_i - \mathbf{x}_j}{\mathbf{x}_j}.
\end{equation}
Overall, we obtained that
\begin{equation}
    \left\vert \frac{\mathbf{w}_i^{(1)}}{\mathbf{w}^{(1)}_j} - 1 \right\vert \leq \left\vert \frac{\mathbf{x}_i - \mathbf{x}_j}{\mathbf{x}_j} \right\vert,
\end{equation}
for all $i,j\in[m]$.
An analogous argument holds if we replace $\mathbf{w}_i^{(1)},\mathbf{w}_j^{(1)},\mathbf{x}_i,\mathbf{x}_j$ above with $\mathbf{w}_i^{(2)},\mathbf{w}_j^{(2)},\mathbf{y}_i,\mathbf{y}_j$, respectively, and consider all $i,j\in [n]$. 

\section{Proof of Theorem~\ref{thm: convergence of estimated scaling fators for generl S}} \label{appendix: proof of convergence of estimated scaling fators for generl S}
According to Assumption~\ref{assump: decay rate of g-h} and Corollary~\ref{cor: convergence of g_tilde to g in l_infty}, we have 
\begin{equation}
    \Vert \hat{\mathbf{g}} - \mathbf{h} \Vert_\infty \leq \Vert \hat{\mathbf{g}} - \mathbf{g} \Vert_\infty + \Vert {\mathbf{g}} - \mathbf{h} \Vert_\infty = \mathcal{O}_n\left(\max\left\{n^{\epsilon-1/2},n^{-\delta_1},n^{-\delta_3}\right\}\right).
\end{equation}
Therefore, the proof of Theorem~\ref{thm: convergence of estimated variance factors} goes through if we replace the right-hand side of~\eqref{eq: bound on e} with $\mathcal{O}_n\left(\max\left\{n^{\epsilon-1/2},n^{-\delta_1},n^{-\delta_3}\right\}\right)$ and replace all uses of Lemma~\ref{lem: boundedness of x and y} with Lemma~\ref{lem: boundedness of x and y for general S} (which provides essentially the same result as Lemma~\ref{lem: boundedness of x and y} but with different constants and for general variance matrices $S$).

\section{Proof of Theorem~\ref{thm: MP law for E_hat}} \label{appendix: proof of MP law for E_hat}
According to Theorem~\ref{thm: convergence of estimated scaling fators for generl S}, we have
\begin{align}
    \frac{1}{\sqrt{\hat{\mathbf{x}}_i}}
    = \frac{1}{\sqrt{{\mathbf{x}}_i}} \sqrt{\frac{\mathbf{x}_i}{\hat{\mathbf{x}}_i}} 
    = \frac{1}{\sqrt{{\mathbf{x}}_i}} \left( \frac{1}{1 + \mathcal{O}_n \left(\max\left\{n^{\epsilon-1/2},n^{-\delta_1},n^{-\delta_3}\right\}\right)} \right)
    = \frac{1 + \zeta_i^{(1)}}{\sqrt{{\mathbf{x}}_i}},
\end{align}
for all $i\in[m]$, where $\zeta^{(1)}\in\mathbb{R}^{m}$ satisfies
\begin{equation}
    \Vert \zeta^{(1)} \Vert_\infty = \mathcal{O}_n (\max\{n^{\epsilon-1/2},n^{-\delta_1},n^{-\delta_3}\}),
\end{equation}
using the properties of Definition~\ref{def: order with high probability}.
Analogously, we have
\begin{equation}
    \frac{1}{\sqrt{\hat{\mathbf{y}}_j}} = \frac{1 + \zeta_j^{(2)}}{\sqrt{{\mathbf{y}}_j}},
\end{equation}
for all $j\in [n]$, where $\zeta^{(2)} \in \mathbb{R}^n$ satisfies
\begin{equation}
     \Vert \zeta^{(2)} \Vert_\infty 
     = \mathcal{O}_n \left(\max\left\{\frac{n^{\epsilon+1/2}}{m}, \frac{n^{1-\delta_1}}{m}, \frac{n^{1-\delta_3}}{m} \right\}\right) 
     = \max\{n^{\epsilon-1/2},n^{-\delta_1},n^{-\delta_3}\},
\end{equation}
where we used the fact that $m$ is proportional to $n$ (see the assumptions in Theorem~\ref{thm: MP law for E_hat}). Let us write
\begin{align}
    \hat{E} &= D\left\{\frac{1}{\sqrt{\hat{\mathbf{x}}}}\right\} E D\left\{\frac{1}{\sqrt{\hat{\mathbf{y}}}}\right\} 
    = D\left\{\frac{1}{\sqrt{{\mathbf{x}}}}\right\} E D\left\{\frac{1}{\sqrt{{\mathbf{y}}}}\right\} 
    + D\left\{\frac{\zeta^{(1)}}{\sqrt{\mathbf{x}}}\right\} E D\left\{\frac{1}{\sqrt{{\mathbf{y}}}}\right\} \nonumber \\
    &+ D\left\{\frac{1}{\sqrt{{\mathbf{x}}}}\right\} E D\left\{\frac{\zeta^{(2)}}{\sqrt{\mathbf{y}}}\right\} 
    + D\left\{\frac{\zeta^{(1)}}{\sqrt{\mathbf{x}}}\right\} E D\left\{\frac{\zeta^{(2)}}{\sqrt{\mathbf{y}}}\right\}.
\end{align}
It follows that
\begin{equation}
    \frac{1}{\sqrt{n}}\Vert \hat{E} - \widetilde{E} \Vert_2 \leq \Vert E \Vert_2 \frac{  {\Vert \zeta^{(1)} \Vert_\infty  } + { \Vert \zeta^{(2)} \Vert_\infty  } + \Vert \zeta^{(1)} \Vert_\infty \Vert \zeta^{(2)} \Vert_\infty }{\sqrt{n \min_i \{{\mathbf{x}_i}\} \min_j \{{\mathbf{y}_j}\}} },
\end{equation}
where we used the fact that the operator norm of a diagonal matrix is the maximal absolute value of its diagonal entries. To bound $\Vert E \Vert_2$ we will apply Theorem 2.4 part II in~\cite{alt2017local} to the matrix $E$. First, we show that conditions (A)--(D) in~\cite{alt2017local} hold. Condition (D) in~\cite{alt2017local} holds since $m,n\rightarrow \infty$ and $m/n\rightarrow \gamma \in (0,1]$, hence $ 0.9 \gamma \leq  m/n \leq 1.1\gamma$ for all sufficiently large $n$. Second, our Assumptions~\ref{assump: noise moment bound} and~\ref{assump: variance boundedness} imply that $cn^{-1} \leq S_{ij} \leq \mu_2n^{-1}$. This fact, together with the fact that $m$ is proportional to $n$, imply that Conditions (A) and (B) in~\cite{alt2017local} hold (see Remark 2.8 in~\cite{alt2017local}). Lastly, Condition (C) in~\cite{alt2017local} holds since our Assumptions~\ref{assump: noise moment bound} and~\ref{assump: variance boundedness} assert that
\begin{equation}
    \vert E_{ij} \vert^k \leq \mu_k n^{-k/2} \leq {\mu_k} c^{-k/2} (c n^{-1})^{k/2} \leq \widetilde{\mu}_k S_{ij}^{k/2},
\end{equation}
for all $k\in\mathbb{N}$, where $\widetilde{\mu}_k = {\mu_k} c^{-k/2} > 0$. Therefore, applying part II of Theorem 2.4 in~\cite{alt2017local} to the matrix $E$, we obtain that for any $\varepsilon^*>0$, 
\begin{equation}
    \Vert E \Vert_2^2 \leq 4 \mu_2 + \varepsilon*,
\end{equation}
with probability at least $1 - n^{-t}$, for all $t>0$ and $n \geq n_0(t)$, where $n_0(t)$ is some constant that depends also on $t$. To obtain this probabilistic bound, we also used part I of Lemma 2.1 in~\cite{alt2017local}, which states that the support of the measure appearing in Theorem 2.4 of~\cite{alt2017local} is confined to the interval $[0,4\mu_2]$. Note that if an event holds with probability at least $1 - n^{-t}$ for all $t>0$ and $n\geq n_0(t)$, then it also holds with probability at least $1 - c^{'}(t)n^{-t}$ for all $t>0$ and all $n$, where $c^{'}(t) = \max\{[n_0(t)]^t,1\}$ (since $1 - c^{'}(t) n^{-t} \leq 0 $ for all $n < n_0(t)$ and $1-n^{-t} \geq 1-c^{'}(t)n^{-t}$ for all $n\geq n_0(t)$). Consequently, using Definition~\ref{def: order with high probability}, we have
\begin{equation}
    \Vert E \Vert_2 = \mathcal{O}_n(1). \label{eq: E operator norm bound}
\end{equation}
Utilizing the above together with Lemma~\ref{lem: boundedness of x and y for general S}, we arrive at
\begin{equation}
    \frac{1}{\sqrt{n}}\Vert \hat{E} - \widetilde{E} \Vert_2 = \mathcal{O}_n(\max\{n^{\epsilon-1/2},n^{-\delta_1},n^{-\delta_3}\}).
\end{equation}
Using Weyl's inequality for singular values (see Theorem 3.3.16 in~\cite{horn1994topics}), we have
\begin{equation}
    \left\vert \sigma_i \{\frac{1}{\sqrt{n}}\hat{E}\} - \sigma_i \{\frac{1}{\sqrt{n}}\widetilde{E}\} \right\vert =  \mathcal{O}_n(\max\{n^{\epsilon-1/2},n^{-\delta_1},n^{-\delta_3}\}),
\end{equation}
for all $i\in [m]$, where $\sigma_i\{\cdot\}$ denotes the $i$th singular value of a matrix. Consequently, 
\begin{align}
    \lambda_i \{ \frac{1}{n}\hat{E}\hat{E}^T \} 
    &= \left( \sigma_i \{\frac{1}{\sqrt{n}}\hat{E}\} \right)^2 
    = \left( \sigma_i \{\frac{1}{\sqrt{n}}\widetilde{E}\} + \mathcal{O}_n(\max\{n^{\epsilon-1/2},n^{-\delta_1},n^{-\delta_3}\}) \right)^2 \nonumber \\
    &= \lambda_i \{ \frac{1}{n}\widetilde{E}\widetilde{E}^T \} + \sigma_i \{\frac{1}{\sqrt{n}}\widetilde{E}\} \mathcal{O}_n(\max\{n^{\epsilon-1/2},n^{-\delta_1},n^{-\delta_3}\}) + \mathcal{O}_n(\max\{n^{2\epsilon-1},n^{-4\delta_1},n^{-\delta_3}\}) \nonumber \\
    &= \lambda_i \{ \frac{1}{n}\widetilde{E}\widetilde{E}^T \} + \mathcal{O}_n(\max\{n^{\epsilon-1/2},n^{-\delta_1},n^{-\delta_3}\}),
\end{align}
where $\lambda_i\{\cdot\}$ denotes the $i$th largest eigenvalue of a symmetric matrix, and we also used the fact that according to~\eqref{eq: E operator norm bound} and Lemma~\ref{lem: boundedness of x and y for general S},
\begin{equation}
    \sigma_i \{\frac{1}{\sqrt{n}}\widetilde{E}\} 
    \leq \left\Vert \frac{1}{\sqrt{n}}\widetilde{E} \right\Vert_2 
    = \left\Vert \frac{1}{\sqrt{n}} D\{\frac{1}{\sqrt{\mathbf{x}}}\} E D\{\frac{1}{\sqrt{\mathbf{y}}}\} \right\Vert_2 \leq  \frac{  \Vert E \Vert_2 }{\sqrt{n \min_i \{{\mathbf{x}_i}\} \min_j \{{\mathbf{y}_j}\}} } = \mathcal{O}_n(1),
\end{equation}
for all $i\in[m]$. Overall, we have
\begin{equation}
    \max_{i\in[m]} \left\vert \lambda_i \{ \frac{1}{n}\hat{E}\hat{E}^T \}  - \lambda_i \{ \frac{1}{n}\widetilde{E}\widetilde{E}^T \} \right\vert 
    = \mathcal{O}_n(\max\{n^{\epsilon-1/2},n^{-\delta_1},n^{-\delta_3}\}) \underset{n\rightarrow\infty }{\overset{\text{a.s.}}{\longrightarrow}} 0, \label{eq: convergence of eigenvalues}
\end{equation}
where $\overset{\text{a.s.}}{\longrightarrow}$ refers to almost sure convergence, which follows from Definition~\ref{def: order with high probability} taking any $t>1$. Next, we apply Proposition 4.1 in~\cite{landa2022biwhitening} to the matrix $\sqrt{n} Y$. Note that the conditions in Proposition 4.1 in~\cite{landa2022biwhitening} are satisfied by our Assumptions~\ref{assump: noise moment bound} and~\ref{assump: variance boundedness}. Therefore, Proposition 4.1 in~\cite{landa2022biwhitening} asserts that $F_{\widetilde{\Sigma}} {\overset{\text{a.s.}}{\longrightarrow}} F_{\gamma,1}$ and $\lambda_1\{\widetilde{\Sigma}\} {\overset{\text{a.s.}}{\longrightarrow}} \beta_+ = (1+\sqrt{\gamma})^2$. Consequently, according to~\eqref{eq: convergence of eigenvalues}, we also have that $F_{\hat{\Sigma}} {\overset{\text{a.s.}}{\longrightarrow}} F_{\gamma,1}$ and $\lambda_1\{\hat{\Sigma}\} {\overset{\text{a.s.}}{\longrightarrow}} \beta_+ = (1+\sqrt{\gamma})^2$, which concludes the proof.

\section{Proof of Proposition~\ref{prop: optimal approximate ML estimation}} \label{appendix: proof of optimal approximate ML estimation}
We have 
\begin{equation}
    \mathbb{E}[\mathcal{L}_A(X)] 
    = \frac{1}{2} \sum_{i=1}^m\sum_{j=1}^n \left[ \frac{\mathbb{E}\vert X_{ij} - Y_{ij} \vert^2}{A_{ij}} + \log (2\pi A_{ij})\right] 
    = \frac{1}{2} \sum_{i=1}^m\sum_{j=1}^n \left[ \frac{S_{ij}}{A_{ij}} + \log (A_{ij}) \right] + \frac{mn}{2}\log(2\pi).
\end{equation}
For any matrix $A$ of rank one, we can write $A = \mathbf{a}\mathbf{b}^T$ for some $\mathbf{a}\in\mathbb{R}^m$ and $\mathbf{b}\in\mathbb{R}^n$. Therefore, minimizing $\mathbb{E}[\mathcal{L}_A(X)] $ over all positive rank-one matrices $A\in\mathbb{R}^{m\times n}$ is equivalent to minimizing
\begin{equation}
    J(\mathbf{a},\mathbf{b}) = \sum_{i=1}^m\sum_{j=1}^n \left[ \frac{S_{ij}}{\mathbf{a}_{i} \mathbf{b}_j} + \log (\mathbf{a}_{i} \mathbf{b}_j) \right] 
    = \sum_{i=1}^m\sum_{j=1}^n \frac{S_{ij}}{\mathbf{a}_{i} \mathbf{b}_j} + n \sum_{i=1}^m \log (\mathbf{a}_{i}) + m \sum_{j=1}^n\log (\mathbf{b}_j),
\end{equation}
over all positive vectors $\mathbf{a}\in\mathbb{R}^m$ and $\mathbf{b}\in\mathbb{R}^n$. Defining ${\mathbf{a}_i} = e^{-\eta_i}$, ${\mathbf{b}_j} = e^{-\xi_j}$, minimizing $J(\mathbf{a},\mathbf{b})$ is equivalent to minimizing
\begin{equation}
    \widetilde{J}(\eta,\xi) = \sum_{i=1}^m\sum_{j=1}^n e^{\eta_i} S_{ij} e^{\xi_j} - n \sum_{i=1}^m \eta_i - m \sum_{j=1}^n\xi_j,
\end{equation}
over all $\eta\in\mathbb{R}^m$ and $\xi\in\mathbb{R}^n$. It is well known that the minimization of $\widetilde{J}$ is closely related to the problem of matrix scaling; see, e.g.,~\cite{idel2016review} and references therein. Specifically, since $\widetilde{J}(\eta,\xi)$ is convex in $(\eta,\xi)\in\mathbb{R}^m\times \mathbb{R}^n$, it has a global minimum if there exist $\eta$ and $\xi$ that satisfy
\begin{equation}
    \frac{\partial \widetilde{J}}{\partial \eta_i} = \sum_{j=1}^n e^{\eta_i} S_{ij} e^{\xi_j} -n = 0, \qquad \frac{\partial \widetilde{J}}{\partial \xi_j} = \sum_{i=1}^m e^{\eta_i} S_{ij} e^{\xi_j} -m = 0,
\end{equation}
for all $i\in[m]$ and $j\in [n]$. This implies that $(\mathbf{a},\mathbf{b})$ is a global minimizer of $J(\mathbf{a},\mathbf{b})$ if it satisfies
\begin{equation}
    n = \sum_{j=1}^n \frac{S_{ij}}{\mathbf{a}_i\mathbf{b}_j} \qquad m = \sum_{i=1}^m \frac{S_{ij}}{\mathbf{a}_i\mathbf{b}_j},
\end{equation}
for all $i\in[m]$ and $j\in [n]$, or in other words, if $(D\{\mathbf{a}\})^{-1} S (D\{\mathbf{b}\})^{-1}$ is doubly regular (see Definition~\ref{def: doubly regular matrix}).
According to the existence and uniqueness statements in Proposition~\ref{prop: matrix scaling}, such $\mathbf{a}$ and $\mathbf{b}$ exist and must satisfy $\mathbf{a}\mathbf{b}^T = \mathbf{x}\mathbf{y}^T$, which is uniquely defined. Hence, we conclude that $\mathbf{x}\mathbf{y}^T$ is the unique minimizer of $\mathbb{E}[\mathcal{L}_A(X)] $ over all positive matrices $A\in\mathbb{R}^{m\times n}$ of rank one.

If we alleviate the rank-one requirement on $A$, then minimizing $\mathbb{E}[\mathcal{L}_A(X)] $ is equivalent to minimizing
\begin{equation}
    \frac{S_{ij}}{A_{ij}} + \log(A_{ij}),
\end{equation}
for each pair $i,j$ separately over $A_{ij}\in (0,\infty)$. Defining $A_{ij} = e^{-\zeta_{ij}}$, we obtain a strictly convex minimization problem in $\zeta_{ij} \in \mathbb{R}$. Then, a straight-forward calculation shows that the minimizer is $A_{ij} = e^{-\zeta_{ij}} = S_{ij}$.

\end{appendices}

\bibliographystyle{plain}
\bibliography{mybib}

\begin{thebibliography}{10}

\bibitem{ajanki2019quadratic}
Oskari Ajanki, L{\'a}szl{\'o} Erd{\H{o}}s, and Torben Kr{\"u}ger.
\newblock {\em Quadratic vector equations on complex upper half-plane}, volume
  261.
\newblock American Mathematical Society, 2019.

\bibitem{alt2017local}
Johannes Alt, L{\'a}szl{\'o} Erd{\H{o}}s, Torben Kr{\"u}ger, et~al.
\newblock Local law for random gram matrices.
\newblock {\em Electronic Journal of Probability}, 22, 2017.

\bibitem{baik2005phase}
Jinho Baik, G{\'e}rard~Ben Arous, Sandrine P{\'e}ch{\'e}, et~al.
\newblock Phase transition of the largest eigenvalue for nonnull complex sample
  covariance matrices.
\newblock {\em The Annals of Probability}, 33(5):1643--1697, 2005.

\bibitem{baik2006eigenvalues}
Jinho Baik and Jack~W Silverstein.
\newblock Eigenvalues of large sample covariance matrices of spiked population
  models.
\newblock {\em Journal of Multivariate Analysis}, 97(6):1382--1408, 2006.

\bibitem{ding2021singular}
Zhigang Bao, Xiucai Ding, , and Ke~Wang.
\newblock {Singular vector and singular subspace distribution for the matrix
  denoising model}.
\newblock {\em The Annals of Statistics}, 49(1):370 -- 392, 2021.

\bibitem{benaych2011eigenvalues}
Florent Benaych-Georges and Raj~Rao Nadakuditi.
\newblock The eigenvalues and eigenvectors of finite, low rank perturbations of
  large random matrices.
\newblock {\em Advances in Mathematics}, 227(1):494--521, 2011.

\bibitem{burgess2019spatial}
Darren~J Burgess.
\newblock Spatial transcriptomics coming of age.
\newblock {\em Nature Reviews Genetics}, 20(6):317--317, 2019.

\bibitem{cattell1966scree}
Raymond~B Cattell.
\newblock The scree test for the number of factors.
\newblock {\em Multivariate behavioral research}, 1(2):245--276, 1966.

\bibitem{chatterjee2015matrix}
Sourav Chatterjee et~al.
\newblock Matrix estimation by universal singular value thresholding.
\newblock {\em The Annals of Statistics}, 43(1):177--214, 2015.

\bibitem{chen2019single}
Geng Chen, Baitang Ning, and Tieliu Shi.
\newblock Single-cell rna-seq technologies and related computational data
  analysis.
\newblock {\em Frontiers in genetics}, 10:317, 2019.

\bibitem{choi2017selecting}
Yunjin Choi, Jonathan Taylor, and Robert Tibshirani.
\newblock Selecting the number of principal components: Estimation of the true
  rank of a noisy matrix.
\newblock {\em The Annals of Statistics}, pages 2590--2617, 2017.

\bibitem{choudhary2022comparison}
Saket Choudhary and Rahul Satija.
\newblock Comparison and evaluation of statistical error models for scrna-seq.
\newblock {\em Genome biology}, 23(1):27, 2022.

\bibitem{cochran1977statistically}
Robert~N Cochran and Frederick~H Horne.
\newblock Statistically weighted principal component analysis of rapid scanning
  wavelength kinetics experiments.
\newblock {\em Analytical Chemistry}, 49(6):846--853, 1977.

\bibitem{cole2019performance}
Michael~B Cole, Davide Risso, Allon Wagner, David DeTomaso, John Ngai,
  Elizabeth Purdom, Sandrine Dudoit, and Nir Yosef.
\newblock Performance assessment and selection of normalization procedures for
  single-cell rna-seq.
\newblock {\em Cell systems}, 8(4):315--328, 2019.

\bibitem{ding2019singular}
Xiucai Ding.
\newblock Singular vector distribution of sample covariance matrices.
\newblock {\em Advances in applied probability}, 51(1):236--267, 2019.

\bibitem{ding2020high}
Xiucai Ding.
\newblock {High dimensional deformed rectangular matrices with applications in
  matrix denoising}.
\newblock {\em Bernoulli}, 26(1):387 -- 417, 2020.

\bibitem{ding2023local}
Xiucai Ding and Hong~Chang Ji.
\newblock {Local laws for multiplication of random matrices}.
\newblock {\em The Annals of Applied Probability}, 33(4):2981 -- 3009, 2023.

\bibitem{ding2021spiked}
Xiucai Ding and Fan Yang.
\newblock {Spiked separable covariance matrices and principal components}.
\newblock {\em The Annals of Statistics}, 49(2):1113 -- 1138, 2021.

\bibitem{ding2022tracy}
Xiucai Ding and Fan Yang.
\newblock Tracy-widom distribution for heterogeneous gram matrices with
  applications in signal detection.
\newblock {\em IEEE Transactions on Information Theory}, 68(10):6682--6715,
  2022.

\bibitem{dobriban2020permutation}
Edgar Dobriban et~al.
\newblock Permutation methods for factor analysis and pca.
\newblock {\em The Annals of Statistics}, 48(5):2824--2847, 2020.

\bibitem{dobriban2019deterministic}
Edgar Dobriban and Art~B Owen.
\newblock Deterministic parallel analysis: an improved method for selecting
  factors and principal components.
\newblock {\em Journal of the Royal Statistical Society: Series B (Statistical
  Methodology)}, 81(1):163--183, 2019.

\bibitem{donoho2023screenot}
David Donoho, Matan Gavish, and Elad Romanov.
\newblock Screenot: Exact mse-optimal singular value thresholding in correlated
  noise.
\newblock {\em The Annals of Statistics}, 51(1):122--148, 2023.

\bibitem{erdHos2019random}
L{\'a}szl{\'o} Erd{\H{o}}s, Torben Kr{\"u}ger, and Dominik Schr{\"o}der.
\newblock Random matrices with slow correlation decay.
\newblock In {\em Forum of Mathematics, Sigma}, volume~7. Cambridge University
  Press, 2019.

\bibitem{fan2020estimating}
Jianqing Fan, Jianhua Guo, and Shurong Zheng.
\newblock Estimating number of factors by adjusted eigenvalues thresholding.
\newblock {\em Journal of the American Statistical Association}, pages 1--10,
  2020.

\bibitem{fan2014principal}
Jianqing Fan, Qiang Sun, Wen-Xin Zhou, and Ziwei Zhu.
\newblock Principal component analysis for big data.
\newblock {\em Wiley StatsRef: Statistics Reference Online}, pages 1--13, 2014.

\bibitem{foi2009clipped}
Alessandro Foi.
\newblock Clipped noisy images: Heteroskedastic modeling and practical
  denoising.
\newblock {\em Signal Processing}, 89(12):2609--2629, 2009.

\bibitem{gavish2014optimal}
Matan Gavish and David~L Donoho.
\newblock The optimal hard threshold for singular values is $4 / \sqrt{3} $.
\newblock {\em IEEE Transactions on Information Theory}, 60(8):5040--5053,
  2014.

\bibitem{gavish2017optimal}
Matan Gavish and David~L Donoho.
\newblock Optimal shrinkage of singular values.
\newblock {\em IEEE Transactions on Information Theory}, 63(4):2137--2152,
  2017.

\bibitem{gavish2022matrix}
Matan Gavish, William Leeb, and Elad Romanov.
\newblock Matrix denoising with partial noise statistics: Optimal singular
  value shrinkage of spiked f-matrices.
\newblock {\em arXiv preprint arXiv:2211.00986}, 2022.

\bibitem{geman1980limit}
Stuart Geman.
\newblock A limit theorem for the norm of random matrices.
\newblock {\em The Annals of Probability}, pages 252--261, 1980.

\bibitem{goes2020robust}
John Goes, Gilad Lerman, and Boaz Nadler.
\newblock Robust sparse covariance estimation by thresholding tyler’s
  m-estimator.
\newblock {\em The Annals of Statistics}, 48(1):86--110, 2020.

\bibitem{hafemeister2019normalization}
Christoph Hafemeister and Rahul Satija.
\newblock Normalization and variance stabilization of single-cell rna-seq data
  using regularized negative binomial regression.
\newblock {\em Genome biology}, 20(1):1--15, 2019.

\bibitem{hoeffding1963probability}
Wassily Hoeffding.
\newblock Probability inequalities for sums of bounded random variables.
\newblock {\em Journal of the American statistical association},
  58(301):13--30, 1963.

\bibitem{hoff2007model}
Peter~D Hoff.
\newblock Model averaging and dimension selection for the singular value
  decomposition.
\newblock {\em Journal of the American Statistical Association},
  102(478):674--685, 2007.

\bibitem{hong2018asymptotic}
David Hong, Laura Balzano, and Jeffrey~A Fessler.
\newblock Asymptotic performance of pca for high-dimensional heteroscedastic
  data.
\newblock {\em Journal of Multivariate Analysis}, 167:435--452, 2018.

\bibitem{hong2018optimally}
David Hong, Jeffrey~A Fessler, and Laura Balzano.
\newblock Optimally weighted pca for high-dimensional heteroscedastic data.
\newblock {\em arXiv preprint arXiv:1810.12862}, 2018.

\bibitem{hong2020selecting}
David Hong, Yue Sheng, and Edgar Dobriban.
\newblock Selecting the number of components in pca via random signflips.
\newblock {\em arXiv preprint arXiv:2012.02985}, 2020.

\bibitem{horn1994topics}
Roger~A Horn, Roger~A Horn, and Charles~R Johnson.
\newblock {\em Topics in matrix analysis}.
\newblock Cambridge university press, 1994.

\bibitem{idel2016review}
Martin Idel.
\newblock A review of matrix scaling and sinkhorn's normal form for matrices
  and positive maps.
\newblock {\em arXiv preprint arXiv:1609.06349}, 2016.

\bibitem{irani2000factorization}
Michal Irani and P~Anandan.
\newblock Factorization with uncertainty.
\newblock In {\em ECCV (1)}, pages 539--553, 2000.

\bibitem{johanson2018genome}
Timothy~M Johanson, Hannah~D Coughlan, Aaron~TL Lun, Naiara~G Bediaga, Gaetano
  Naselli, Alexandra~L Garnham, Leonard~C Harrison, Gordon~K Smyth, and Rhys~S
  Allan.
\newblock Genome-wide analysis reveals no evidence of trans chromosomal
  regulation of mammalian immune development.
\newblock {\em PLoS Genetics}, 14(6):e1007431, 2018.

\bibitem{johnstone2018pca}
Iain~M Johnstone and Debashis Paul.
\newblock Pca in high dimensions: An orientation.
\newblock {\em Proceedings of the IEEE}, 106(8):1277--1292, 2018.

\bibitem{johnstone2017roy}
IM~Johnstone and Boaz Nadler.
\newblock Roy’s largest root test under rank-one alternatives.
\newblock {\em Biometrika}, 104(1):181--193, 2017.

\bibitem{ke2021estimation}
Zheng~Tracy Ke, Yucong Ma, and Xihong Lin.
\newblock Estimation of the number of spiked eigenvalues in a covariance matrix
  by bulk eigenvalue matching analysis.
\newblock {\em Journal of the American Statistical Association}, pages 1--19,
  2021.

\bibitem{kharchenko2021triumphs}
Peter~V Kharchenko.
\newblock The triumphs and limitations of computational methods for scrna-seq.
\newblock {\em Nature methods}, 18(7):723--732, 2021.

\bibitem{kritchman2008determining}
Shira Kritchman and Boaz Nadler.
\newblock Determining the number of components in a factor model from limited
  noisy data.
\newblock {\em Chemometrics and Intelligent Laboratory Systems}, 94(1):19--32,
  2008.

\bibitem{landa2022scaling}
Boris Landa.
\newblock Scaling positive random matrices: concentration and asymptotic
  convergence.
\newblock {\em Electronic Communications in Probability}, 27:1--13, 2022.

\bibitem{landa2022biwhitening}
Boris Landa, Thomas~TCK Zhang, and Yuval Kluger.
\newblock Biwhitening reveals the rank of a count matrix.
\newblock {\em SIAM Journal on Mathematics of Data Science}, 4(4):1420--1446,
  2022.

\bibitem{leeb2022optimal}
William Leeb.
\newblock Optimal singular value shrinkage for operator norm loss: Extending to
  non-square matrices.
\newblock {\em Statistics \& Probability Letters}, 186:109472, 2022.

\bibitem{leeb2021optimal}
William Leeb and Elad Romanov.
\newblock Optimal spectral shrinkage and pca with heteroscedastic noise.
\newblock {\em IEEE Transactions on Information Theory}, 67(5):3009--3037,
  2021.

\bibitem{leeb2021matrix}
William~E Leeb.
\newblock Matrix denoising for weighted loss functions and heterogeneous
  signals.
\newblock {\em SIAM Journal on Mathematics of Data Science}, 3(3):987--1012,
  2021.

\bibitem{linderman2022zero}
George~C Linderman, Jun Zhao, Manolis Roulis, Piotr Bielecki, Richard~A
  Flavell, Boaz Nadler, and Yuval Kluger.
\newblock Zero-preserving imputation of single-cell rna-seq data.
\newblock {\em Nature communications}, 13(1):192, 2022.

\bibitem{marvcenko1967distribution}
Vladimir~A Mar{\v{c}}enko and Leonid~Andreevich Pastur.
\newblock Distribution of eigenvalues for some sets of random matrices.
\newblock {\em Mathematics of the USSR-Sbornik}, 1(4):457, 1967.

\bibitem{maynard2021transcriptome}
Kristen~R Maynard, Leonardo Collado-Torres, Lukas~M Weber, Cedric Uytingco,
  Brianna~K Barry, Stephen~R Williams, Joseph~L Catallini, Matthew~N Tran,
  Zachary Besich, Madhavi Tippani, et~al.
\newblock Transcriptome-scale spatial gene expression in the human dorsolateral
  prefrontal cortex.
\newblock {\em Nature neuroscience}, 24(3):425--436, 2021.

\bibitem{morris1982natural}
Carl~N Morris.
\newblock Natural exponential families with quadratic variance functions.
\newblock {\em The Annals of Statistics}, pages 65--80, 1982.

\bibitem{nadakuditi2014optshrink}
Raj~Rao Nadakuditi.
\newblock Optshrink: An algorithm for improved low-rank signal matrix denoising
  by optimal, data-driven singular value shrinkage.
\newblock {\em IEEE Transactions on Information Theory}, 60(5):3002--3018,
  2014.

\bibitem{nadler2008finite}
Boaz Nadler et~al.
\newblock Finite sample approximation results for principal component analysis:
  A matrix perturbation approach.
\newblock {\em The Annals of Statistics}, 36(6):2791--2817, 2008.

\bibitem{owen2009bi}
Art~B Owen, Patrick~O Perry, et~al.
\newblock Bi-cross-validation of the svd and the nonnegative matrix
  factorization.
\newblock {\em The Annals of Applied Statistics}, 3(2):564--594, 2009.

\bibitem{paul2007asymptotics}
Debashis Paul.
\newblock Asymptotics of sample eigenstructure for a large dimensional spiked
  covariance model.
\newblock {\em Statistica Sinica}, pages 1617--1642, 2007.

\bibitem{salmon2014poisson}
Joseph Salmon, Zachary Harmany, Charles-Alban Deledalle, and Rebecca Willett.
\newblock Poisson noise reduction with non-local pca.
\newblock {\em Journal of mathematical imaging and vision}, 48(2):279--294,
  2014.

\bibitem{shen2005analysis}
Haipeng Shen and Jianhua~Z Huang.
\newblock Analysis of call centre arrival data using singular value
  decomposition.
\newblock {\em Applied Stochastic Models in Business and Industry},
  21(3):251--263, 2005.

\bibitem{sinkhorn1967diagonal}
Richard Sinkhorn.
\newblock Diagonal equivalence to matrices with prescribed row and column sums.
\newblock {\em The American Mathematical Monthly}, 74(4):402--405, 1967.

\bibitem{sinkhorn1967concerning}
Richard Sinkhorn and Paul Knopp.
\newblock Concerning nonnegative matrices and doubly stochastic matrices.
\newblock {\em Pacific Journal of Mathematics}, 21(2):343--348, 1967.

\bibitem{srebro2003weighted}
Nathan Srebro and Tommi Jaakkola.
\newblock Weighted low-rank approximations.
\newblock In {\em Proceedings of the 20th international conference on machine
  learning (ICML-03)}, pages 720--727, 2003.

\bibitem{su2022adaptive}
Pei-Chun Su and Hau-Tieng Wu.
\newblock Adaptive optimal shrinkage of singular values under colored and
  dependent noise.
\newblock {\em arXiv preprint arXiv:2207.03466}, 2022.

\bibitem{tamuz2005correcting}
Omer Tamuz, Tsevi Mazeh, and Shay Zucker.
\newblock Correcting systematic effects in a large set of photometric light
  curves.
\newblock {\em Monthly Notices of the Royal Astronomical Society},
  356(4):1466--1470, 2005.

\bibitem{tyler1987distribution}
David~E Tyler.
\newblock A distribution-free m-estimator of multivariate scatter.
\newblock {\em The annals of Statistics}, pages 234--251, 1987.

\bibitem{vershynin2018high}
Roman Vershynin.
\newblock {\em High-dimensional probability: An introduction with applications
  in data science}, volume~47.
\newblock Cambridge university press, 2018.

\bibitem{vieth2019systematic}
Beate Vieth, Swati Parekh, Christoph Ziegenhain, Wolfgang Enard, and Ines
  Hellmann.
\newblock A systematic evaluation of single cell rna-seq analysis pipelines.
\newblock {\em Nature communications}, 10(1):1--11, 2019.

\bibitem{wallach2006topic}
Hanna~M Wallach.
\newblock Topic modeling: beyond bag-of-words.
\newblock In {\em Proceedings of the 23rd international conference on Machine
  learning}, pages 977--984, 2006.

\bibitem{wold1978cross}
Svante Wold.
\newblock Cross-validatory estimation of the number of components in factor and
  principal components models.
\newblock {\em Technometrics}, 20(4):397--405, 1978.

\bibitem{yang2019edge}
Fan Yang.
\newblock {Edge universality of separable covariance matrices}.
\newblock {\em Electronic Journal of Probability}, 24(none):1 -- 57, 2019.

\bibitem{yin1988limit}
Yong-Qua Yin, Zhi-Dong Bai, and Pathak~R Krishnaiah.
\newblock On the limit of the largest eigenvalue of the large dimensional
  sample covariance matrix.
\newblock {\em Probability theory and related fields}, 78(4):509--521, 1988.

\bibitem{zhang2014novel}
Teng Zhang and Gilad Lerman.
\newblock A novel m-estimator for robust pca.
\newblock {\em The Journal of Machine Learning Research}, 15(1):749--808, 2014.

\bibitem{zheng2017massively}
Grace~XY Zheng, Jessica~M Terry, Phillip Belgrader, Paul Ryvkin, Zachary~W
  Bent, Ryan Wilson, Solongo~B Ziraldo, Tobias~D Wheeler, Geoff~P McDermott,
  Junjie Zhu, et~al.
\newblock Massively parallel digital transcriptional profiling of single cells.
\newblock {\em Nature Communications}, 8(1):1--12, 2017.

\end{thebibliography}

\end{document}